\documentclass[10pt,a4paper]{article}
\usepackage[utf8]{inputenc}
\usepackage[T1]{fontenc}
\usepackage{geometry}
\usepackage{amsthm}
\usepackage[french,english]{babel}
\usepackage{amssymb}
\usepackage{lmodern}
\usepackage{amsmath,amsfonts}
\usepackage{mathrsfs} 
\usepackage{dsfont}
\usepackage{stmaryrd}
\DeclareMathOperator{\dive}{div} 
\usepackage{graphicx}
\usepackage{fullpage}
\usepackage[nottoc,notlot,notlof]{tocbibind}
\usepackage[colorlinks=true,urlcolor=black,linkcolor=black,citecolor=black]{hyperref}
\numberwithin{equation}{section}
\usepackage{mathtools,array}
\usepackage{verbatimbox}
\usepackage{color}
\usepackage{bm}
\usepackage{tikz-cd}
\usepackage{pst-node}
\usetikzlibrary{matrix}
\usepackage{mathtools}
\usepackage{verbatimbox}
\usepackage{diagbox}
\usepackage{array}
\newcolumntype{C}{>{$\displaystyle} c <{$}}
\usepackage{makecell}
\usepackage{float}
\usepackage{tabu}
\usepackage{cancel}
\usepackage{lscape}
\usepackage[babel=true]{csquotes}
\usepackage{bm}
\usepackage{xcolor}
\usepackage[new]{old-arrows}
\usepackage{stackengine}
\makeatletter
\def\env@dmatrix{\hskip -\arraycolsep
	\let\@ifnextchar\new@ifnextchar
	\def\arraystretch{2}%
	\array{*{\c@MaxMatrixCols}{>{\displaystyle}c}}%
}
\usepackage{soul}

\makeatother

\DeclareFontShape{OMX}{cmex}{m}{n}{
	<-7.5> cmex7
	<7.5-8.5> cmex8
	<8.5-9.5> cmex9
	<9.5-> cmex10
}{}
\SetSymbolFont{largesymbols}{normal}{OMX}{cmex}{m}{n}
\SetSymbolFont{largesymbols}{bold}  {OMX}{cmex}{m}{n}

\begin{document}

	\renewcommand{\thefootnote}{\fnsymbol{footnote}}
	
	\title{Weighted Eigenvalue Problems for Fourth-Order Operators in Degenerating Annuli}

	\author{Alexis Michelat\footnote{Institute of Mathematics, EPFL B, Station 8, CH-1015 Lausanne, Switzerland.\hspace{.5em} \href{alexis.michelat@epfl.ch}{alexis.michelat@epfl.ch}}\and \;Tristan \selectlanguage{french}Rivière\selectlanguage{english}\footnote{Department of Mathematics, ETH Zentrum, CH-8093 Zürich, Switzerland.\hspace{2.85em}\href{mailto:tristan.riviere@math.ethz.ch}{tristan.riviere@math.ethz.ch}}}
	\date{\today}
	
	\maketitle
	
	\vspace{-0.5em}
	
	\begin{abstract} 
    We obtain a nigh optimal estimate for the first eigenvalue of two natural weighted problems associated to the bilaplacian (and of a continuous family of fourth-order elliptic operators in dimension $2$) in degenerating annuli—that are central objects in bubble tree analysis—in all dimension. The estimate depends only on the conformal class of the annulus. We also show that in dimension $2$ and dimension $4$, the first eigenfunction (of the first problem) is never radial provided that the conformal class of the annulus is large enough. The other result is a weighted Poincaré-type inequality in annuli for those fourth-order operators. Applications to Morse theory are given.
	\end{abstract}

	\tableofcontents
	\vspace{0cm}
	\begin{center}
	\end{center}

	\theoremstyle{plain}
	\newtheorem*{theorem*}{Theorem}
	\newtheorem{theorem}{Theorem}[section]
	\newenvironment{theorembis}[1]
	{\renewcommand{\thetheorem}{\ref{#1}$'$}%
		\addtocounter{theorem}{-1}%
		\begin{theorem}}
		{\end{theorem}}
	\renewcommand*{\thetheorem}{\Alph{theorem}}
	\newtheorem{lemme}[theorem]{Lemma}
	\newtheorem*{lemme*}{Lemma}
	\newtheorem{propdef}[theorem]{Definition-Proposition}
	\newtheorem*{propdef*}{Definition-Proposition}
	\newtheorem{prop}[theorem]{Proposition}
	\newtheorem{cor}[theorem]{Corollary}
	\theoremstyle{definition}
	\newtheorem*{definition}{Definition}
	\newtheorem{defi}[theorem]{Definition}
	\newtheorem{rem}[theorem]{Remark}
	\newtheorem*{rem*}{Remark}
	\newtheorem{rems}[theorem]{Remarks}
	\newtheorem{remimp}[theorem]{Important Remark}
	\newtheorem{exemple}[theorem]{Example}
	\newtheorem{defi2}{Definition}
	\newtheorem{propdef2}[defi2]{Proposition-Definition}
	\newtheorem{remintro}[defi2]{Remark}
	\newtheorem{remsintro}[defi2]{Remarks}
	\newtheorem{conj}{Conjecture}
	\newtheorem{question}{Open Question}
	\renewcommand\hat[1]{%
		\savestack{\tmpbox}{\stretchto{%
				\scaleto{%
					\scalerel*[\widthof{\ensuremath{#1}}]{\kern-.6pt\bigwedge\kern-.6pt}%
					{\rule[-\textheight/2]{1ex}{\textheight}}
				}{\textheight}%
			}{0.5ex}}%
		\stackon[1pt]{#1}{\tmpbox}
	}
	\parskip 1ex
	\newcommand{\totimes}{\ensuremath{\,\dot{\otimes}\,}}
	\newcommand{\vc}[3]{\overset{#2}{\underset{#3}{#1}}}
	\newcommand{\conv}[1]{\ensuremath{\underset{#1}{\longrightarrow}}}
	\newcommand{\A}{\ensuremath{\vec{A}}}
	\newcommand{\B}{\ensuremath{\vec{B}}}
	\newcommand{\C}{\ensuremath{\mathbb{C}}}
	\newcommand{\D}{\ensuremath{\nabla}}
	\newcommand{\Disk}{\ensuremath{\mathbb{D}}}
	\newcommand{\E}{\ensuremath{\vec{E}}}
	\newcommand{\I}{\ensuremath{\mathbb{I}}}
	\newcommand{\Q}{\ensuremath{\vec{Q}}}
	\newcommand{\loc}{\ensuremath{\mathrm{loc}}}
	\newcommand{\z}{\ensuremath{\bar{z}}}
	\newcommand{\hh}{\ensuremath{\mathscr{H}}}
	\newcommand{\h}{\ensuremath{\vec{h}}}
	\newcommand{\vol}{\ensuremath{\mathrm{vol}}}
	\newcommand{\hs}[3]{\ensuremath{\left\Vert #1\right\Vert_{\mathrm{H}^{#2}(#3)}}}
	\newcommand{\R}{\ensuremath{\mathbb{R}}}
	\renewcommand{\P}{\ensuremath{\mathbb{P}}}
	\newcommand{\N}{\ensuremath{\mathbb{N}}}
	\newcommand{\Z}{\ensuremath{\mathbb{Z}}}
	\newcommand{\p}[1]{\ensuremath{\partial_{#1}}}
	\newcommand{\Res}{\ensuremath{\mathrm{Res}}}
	\newcommand{\lp}[2]{\ensuremath{\mathrm{L}^{#1}(#2)}}
	\renewcommand{\wp}[3]{\ensuremath{\left\Vert #1\right\Vert_{\mathrm{W}^{#2}(#3)}}}
	\newcommand{\wpn}[3]{\ensuremath{\Vert #1\Vert_{\mathrm{W}^{#2}(#3)}}}
	\newcommand{\np}[3]{\ensuremath{\left\Vert #1\right\Vert_{\mathrm{L}^{#2}(#3)}}}
	\newcommand{\hp}[3]{\ensuremath{\left\Vert #1\right\Vert_{\mathrm{H}^{#2}(#3)}}}
	\newcommand{\ck}[3]{\ensuremath{\left\Vert #1\right\Vert_{\mathrm{C}^{#2}(#3)}}}
	\newcommand{\hardy}[2]{\ensuremath{\left\Vert #1\right\Vert_{\mathscr{H}^{1}(#2)}}}
	\newcommand{\lnp}[3]{\ensuremath{\left| #1\right|_{\mathrm{L}^{#2}(#3)}}}
	\newcommand{\npn}[3]{\ensuremath{\Vert #1\Vert_{\mathrm{L}^{#2}(#3)}}}
	\newcommand{\nc}[3]{\ensuremath{\left\Vert #1\right\Vert_{C^{#2}(#3)}}}
	\renewcommand{\Re}{\ensuremath{\mathrm{Re}\,}}
	\renewcommand{\Im}{\ensuremath{\mathrm{Im}\,}}
	\newcommand{\dist}{\ensuremath{\mathrm{dist}}}
	\newcommand{\diam}{\ensuremath{\mathrm{diam}\,}}
	\newcommand{\leb}{\ensuremath{\mathscr{L}}}
	\newcommand{\supp}{\ensuremath{\mathrm{supp}\,}}
	\renewcommand{\phi}{\ensuremath{\vec{\Phi}}}
	\renewcommand{\H}{\ensuremath{\vec{H}}}
	\renewcommand{\L}{\ensuremath{\vec{L}}}
	\renewcommand{\lg}{\ensuremath{\mathscr{L}_g}}
	\renewcommand{\ker}{\ensuremath{\mathrm{Ker}}}
	\renewcommand{\epsilon}{\ensuremath{\varepsilon}}
	\renewcommand{\bar}{\ensuremath{\overline}}
	\newcommand{\s}[2]{\ensuremath{\langle #1,#2\rangle}}
	\newcommand{\pwedge}[2]{\ensuremath{\,#1\wedge#2\,}}
	\newcommand{\bs}[2]{\ensuremath{\left\langle #1,#2\right\rangle}}
	\newcommand{\scal}[2]{\ensuremath{\langle #1,#2\rangle}}
	\newcommand{\sg}[2]{\ensuremath{\left\langle #1,#2\right\rangle_{\mkern-3mu g}}}
	\newcommand{\n}{\ensuremath{\vec{n}}}
	\newcommand{\ens}[1]{\ensuremath{\left\{ #1\right\}}}
	\newcommand{\lie}[2]{\ensuremath{\left[#1,#2\right]}}
	\newcommand{\g}{\ensuremath{g}}
	\newcommand{\dzeta}{\ensuremath{\det\hphantom{}_{\kern-0.5mm\zeta}}}
	\newcommand{\e}{\ensuremath{\vec{e}}}
	\newcommand{\f}{\ensuremath{\vec{f}}}
	\newcommand{\ig}{\ensuremath{|\vec{\mathbb{I}}_{\phi}|}}
	\newcommand{\ik}{\ensuremath{\left|\mathbb{I}_{\phi_k}\right|}}
	\newcommand{\w}{\ensuremath{\vec{w}}}
	\newcommand{\hooklongrightarrow}{\lhook\joinrel\longrightarrow}
	\renewcommand{\tilde}{\ensuremath{\widetilde}}
	\newcommand{\vg}{\ensuremath{\mathrm{vol}_g}}
	\newcommand{\im}{\ensuremath{\mathrm{W}^{2,2}_{\iota}(\Sigma,N^n)}}
	\newcommand{\imm}{\ensuremath{\mathrm{W}^{2,2}_{\iota}(\Sigma,\R^3)}}
	\newcommand{\timm}[1]{\ensuremath{\mathrm{W}^{2,2}_{#1}(\Sigma,T\R^3)}}
	\newcommand{\tim}[1]{\ensuremath{\mathrm{W}^{2,2}_{#1}(\Sigma,TN^n)}}
	\renewcommand{\d}[1]{\ensuremath{\partial_{x_{#1}}}}
	\newcommand{\dg}{\ensuremath{\mathrm{div}_{g}}}
	\renewcommand{\Res}{\ensuremath{\mathrm{Res}}}
	\newcommand{\un}[2]{\ensuremath{\bigcup\limits_{#1}^{#2}}}
	\newcommand{\res}{\mathbin{\vrule height 1.6ex depth 0pt width
			0.13ex\vrule height 0.13ex depth 0pt width 1.3ex}}
    \newcommand{\antires}{\mathbin{\vrule height 0.13ex depth 0pt width 1.3ex\vrule height 1.6ex depth 0pt width 0.13ex}}
	\newcommand{\ala}[5]{\ensuremath{e^{-6\lambda}\left(e^{2\lambda_{#1}}\alpha_{#2}^{#3}-\mu\alpha_{#2}^{#1}\right)\left\langle \nabla_{\e_{#4}}\vec{w},\vec{\mathbb{I}}_{#5}\right\rangle}}
	\setlength\boxtopsep{1pt}
	\setlength\boxbottomsep{1pt}
	\newcommand\norm[1]{%
		\setbox1\hbox{$#1$}%
		\setbox2\hbox{\addvbuffer{\usebox1}}%
		\stretchrel{\lvert}{\usebox2}\stretchrel*{\lvert}{\usebox2}%
	}
	\allowdisplaybreaks
	\newcommand*\mcup{\mathbin{\mathpalette\mcapinn\relax}}
	\newcommand*\mcapinn[2]{\vcenter{\hbox{$\mathsurround=0pt
				\ifx\displaystyle#1\textstyle\else#1\fi\bigcup$}}}
	\def\Xint#1{\mathchoice
		{\XXint\displaystyle\textstyle{#1}}%
		{\XXint\textstyle\scriptstyle{#1}}%
		{\XXint\scriptstyle\scriptscriptstyle{#1}}%
		{\XXint\scriptscriptstyle\scriptscriptstyle{#1}}%
		\!\int}
	\def\XXint#1#2#3{{\setbox0=\hbox{$#1{#2#3}{\int}$ }
			\vcenter{\hbox{$#2#3$ }}\kern-.58\wd0}} 
	\def\ddashint{\Xint=}
	\newcommand{\dashint}[1]{\ensuremath{{\Xint-}_{\mkern-10mu #1}}}
	\newcommand\ccancel[1]{\renewcommand\CancelColor{\color{red}}\cancel{#1}}
	\newcommand\colorcancel[2]{\renewcommand\CancelColor{\color{#2}}\cancel{#1}}
	\newcommand{\abs}[1]{\left\lvert #1 \right \rvert}
	
	\renewcommand{\thetheorem}{\thesection.\arabic{theorem}}

 \section{Introduction}

 \subsection{Motivation and Main Results}

 In \cite{riviere_morse_scs}, a new approach is developed to show the upper semi-continuity of the extended Morse index—the number of negative or null eigenvalues—of sequences of critical point of geometric functionals with uniformly bounded energy. One of the main steps in the argument is to derive suitable weighted estimates in degenerating annuli. Namely, if $0<a<b<\infty$ and $\Omega=B_b\setminus\bar{B}_a(0)$, then for all $u\in W^{1,2}_0(\Omega)$, Da Lio-Gianocca-Rivière showed  (\cite[Lemma IV.$1$]{riviere_morse_scs}) that for all $0<\beta<1$ the follows estimate holds
 \begin{align*}
     \int_{\Omega}|\D u|^2dx\geq C_{\beta}\int_{\Omega}\frac{u^2}{|x|^2}\left(\left(\frac{|x|}{b}\right)^{\beta}+\left(\frac{a}{|x|}\right)^{\beta}+\frac{1}{\log^2\left(\frac{b}{a}\right)}\right)dx,
 \end{align*}
 where $0<C_{\beta}<\infty$ is a universal constant. More precisely, the more precise inequality holds for all $u\in W^{1,2}_0(\Omega)$ (see \cite{riviere_morse_scs})
 \begin{align*}
     \int_{\Omega}|\D u|^2dx\geq \frac{\pi^2}{\log^2\left(\frac{b}{a}\right)}\int_{\Omega}\frac{u^2}{|x|^2}dx.
 \end{align*}
 This estimate is central to prove that the neck regions—associated to  bubble trees of fourth-order equations such as Willmore surfaces (in dimension $2$) and the (intrinsic or extrinsic) biharmonic maps in dimension $4$—contribute positively to the second derivative. For more details, refer to the introduction of \cite{morse_willmore_I}. It is equivalent to the estimate of the first positive eigenvalue for the problem
 \begin{align*}
     -\Delta u=\frac{\lambda_1}{|x|^2}u\qquad \text{in}\;\, \Omega,
 \end{align*}
 that is related to the Hardy's inequality in dimension $d\geq 3$ (\cite{hardy}). Due to the singularity of the weight $\dfrac{1}{|x|^2}$, this inequality does not hold in $\R^2$, but the analysis of \cite{riviere_morse_scs} furnishes the precise asymptotic behaviour of the degeneracy of this constant on annuli whose conformal classes diverges to $\infty$.  
 
 In this article, we generalise this estimate for a natural family of fourth-order operators in dimension $2$, and in all dimension in the case of the bilaplacian $\Delta^2$. Namely, we prove that for all $m\geq 1$, and for all $1/2<\beta<1$, 
 \begin{align*}
     \int_{\Omega}\left(\Delta u+2(m-1)\frac{x}{|x|^2}\cdot \D u+\frac{(m-1)^2}{|x|^2}u\right)^2dx\geq C_{\beta,m}\int_{\Omega}\frac{u^2}{|x|^4}\left(\left(\frac{|x|}{b}\right)^{4\beta}+\left(\frac{a}{|x|}\right)^{4\beta}+\frac{1}{\log^2\left(\frac{b}{a}\right)}\right)dx
 \end{align*}
 and for all $\sqrt{2}-1<\beta<1$
 \begin{align*}
     \int_{\Omega}\left(\Delta u+2(m-1)\frac{x}{|x|^2}\cdot \D u+\frac{(m-1)^2}{|x|^2}u\right)^2\geq C_{\beta,m}\int_{\Omega}\frac{|\D u|^2}{|x|^2}\left(\left(\frac{|x|}{b}\right)^{2\beta}+\left(\frac{a}{|x|}\right)^{2\beta}+\frac{1}{\log^2\left(\frac{b}{a}\right)}\right)dx.
 \end{align*}
 Let us point out that our inequality is more general provided that $m>1$.
 The operator
 \begin{align*}
     \leb_m=\Delta+2(m-1)\frac{x}{|x|^2}\cdot \D+\frac{(m-1)^2}{|x|^2}=|x|^{1-m}\Delta\left(|x|^{m-1}\,\cdot\,\right)
 \end{align*}
 naturally appears when on studies the Morse index stability of branched surfaces (and its adjoint plays an important role in the study of the index of branched Willmore spheres \cite{index4}). At a branch point in a neck region, thanks to the analysis of \cite{pointwise} the metric admits the following expansion in a conformal chart
 \begin{align*}
     g=e^{2u}|z|^{2m-2}|dz|^2,
 \end{align*}
     where $u$ is a bounded function and $m\geq 1$ is a positive integer. Since the Willmore energy is a non-linear biharmonic equation, one can easily show that the dominant term of the second derivative (we do not give a precise definition of this quadratic functional and refer to the introduction of \cite{morse_willmore_I} for more details) is given by
 \begin{align*}
     \frac{1}{2}\int_{\Omega}(\Delta_gu+|A|^2u)^2d\vg,
 \end{align*}
 where $A$ is the second fundamental form. Using the previous expansion, we get
 \begin{align*}
     \int_{\Omega}(\Delta_gu)^2d\vg=\int_{\Omega}e^{-2u}(|x|^{1-m}\Delta u)^2dx.
 \end{align*}
 Since we want to show that this positive term controls the lower-order terms of the second derivative, making the change of variable $u=|x|^{m-1}v$, we are led to the previous estimates. In the case of biharmonic maps, we show a general eigenvalue estimate in all dimension.

 We can now state the main theorems more precisely. 
\renewcommand*{\thetheorem}{\Alph{theorem}}
 \begin{theorem}\label{intro_main10}
     Let $m\geq 1$, let $0<a<b<\infty$, and $\Omega=B_b\setminus\bar{B}_a(0)\subset \R^2$. Consider the following minimisation problem
     \begin{align}\label{intro_eigenvalue0}
         \lambda_m=\inf\ens{\int_{\Omega}(\leb_mu)^2dx:u\in W^{2,2}_0(\Omega)\;\,\text{and}\;\,  \int_{\Omega}\frac{u^2}{|x|^4}dx=1}.
     \end{align}
     Then, there exists a unique solution $u$ to \eqref{intro_eigenvalue0}, that satisfies the equation
     \begin{align*}
         \Delta^2u+\frac{2(m^2-1)}{|x|^2}\Delta u-4(m^2-1)\left(\frac{x}{|x|^2}\right)^t\cdot\D^2u\cdot\left(\frac{x}{|x|^2}\right)+\frac{(m^2-1)^2}{|x|^4}u=\frac{\lambda_m}{|x|^4}u,
     \end{align*}
     and there exists a constant $R_{1,m}<\infty$ such that provided that 
     \begin{align}
         \log\left(\frac{b}{a}\right)\geq R_{1,m},
     \end{align}
     then
     \begin{align}
         \left(4m^2+\frac{\pi^2}{\log^2\left(\frac{b}{a}\right)}\right)\frac{\pi^2}{\log^2\left(\frac{b}{a}\right)}<\lambda_m<\left(4m^2+\frac{4\pi^2}{\log^2\left(\frac{b}{a}\right)}\right)\frac{4\pi^2}{\log^2\left(\frac{b}{a}\right)}.
     \end{align}
     Furthermore, if $m\geq 1$ is an integer, there exists $R_{2,m}<\infty$ such that the condition
     \begin{align*}
         \log\left(\frac{b}{a}\right)\geq R_{2,m}
     \end{align*}
     implies that the solution $u$ takes the form $u(r,\theta)=\Re\left(u_0(r)e^{im\theta}\right)$, where $u_0$ is a solution of the following ordinary differential equation
     \begin{align}
         f''''+\frac{2}{r}f'''-(4m^2-1)\frac{1}{r^2}f''+(4m^2-1)\frac{1}{r^3}f'-(4m^2-1)\frac{1}{r^4}f=\frac{\lambda_m}{r^4}f.
     \end{align}
 \end{theorem}
 \setcounter{theorem}{0}
\renewcommand*{\thetheorem}{\thesection.\arabic{theorem}}
 \begin{rem}
 \begin{enumerate}
     \item[($1$)] In contrast to the previously quoted result from \cite{riviere_morse_scs}, the bound on the first eigenvalue of $\leb_m^{\ast}\leb_m$ for the weighted problem is only true asymptotically, and the optimal solution is \emph{not} radial, even for $m=1$ where $\leb_m^{\ast}\leb_m=\Delta^2$ is the bi-Laplacian! This result is reminiscent to those on Ginzburg-Landau vortices where the solution is radial up to an $e^{i\,\theta}$ phase (\cite{BBH,GLquanta}). 
        \item[($2$)] The bound on the conformal class is far from optimal, but we are mostly interested in the estimate when the conformal class diverges. 
        \item[($3$)] More precisely, we have the bound
        \begin{align*}
            \left(4m^2+\frac{\pi^2}{\log^2\left(\frac{b}{a}\right)}\right)\frac{\pi^2}{\log^2\left(\frac{b}{a}\right)}<\lambda_m<\left(4m^2+\frac{4\pi^2}{\log^2\left(\frac{b}{a}\right)}\right)\frac{4\pi^2}{\log^2\left(\frac{b}{a}\right)}.
        \end{align*}
 \end{enumerate}
 \end{rem}

The proof is based on a delicate ODE analysis of all frequencies in Fourier space (the same approach works in higher dimension if we replace the $e^{ik\theta}$ functions ($k\in\Z$) by spherical harmonics), that notably proves that the optimal solution is \emph{not} radial, but has a phase $e^{im\theta}$ provided that $m\geq 1$ is integer and the conformal class is large enough. It amounts at studying a family of ODEs and estimating the first positive zero of an explicit (and non-trivial!) function of the conformal class. Even the case $m=1$ of the bilaplacian was not known, and furnishes an interesting counter-example to the classical  Gidas-Ni-Nirenberg theorem (\cite{gidas}) that shows under general hypotheses that the first eigenfunction must be radial (the theorem does not apply here since our domain is not simply connected).

 \begin{cor}
     Let $m\geq 1$, let $0<a<b<\infty$ and $\Omega=B_b\setminus\bar{B}_a(0)$. Assume that 
     \begin{align}
         \log\left(\frac{b}{a}\right)\geq \frac{\pi\sqrt{2}}{\sqrt{2m-1}}
     \end{align}
     Then, for all $u\in W^{2,2}_0(\Omega)$, we have
     \begin{align}
         \int_{\Omega}\left(\Delta u+2(m-1)\frac{x}{|x|^2}\cdot \D u+\frac{(m-1)^2}{|x|^2}u\right)^2dx\geq \left(4m^2+\frac{\pi^2}{\log^2\left(\frac{b}{a}\right)}\right)\frac{\pi^2}{\log^2\left(\frac{b}{a}\right)}\int_{\Omega}\frac{u^2}{|x|^4}dx.
     \end{align}
 \end{cor}
 \begin{rem}
    \begin{enumerate}
        \item[$1$.] In the case of the bilaplacian $\leb_1^{\ast}\leb_1=\Delta^2$, the condition on $a<b$ becomes
    \begin{align*}
        \frac{b}{a}\geq e^{\pi\sqrt{2}}=85.019695\cdots
    \end{align*}
    \item[$2$.] This kind of Rellich inequality for general elliptic operators with regular singularities have already been considered in higher dimension $d\geq 5$ (\cite{rellich_general_operator}) and in dimension $d=3$ for the subclass of functions $W^{2,2}(\R^3)\cap\ens{u:u(0)=0}$ that vanish at the origin.
    \end{enumerate}
 \end{rem}

 A similar result holds for the second minimisation problem.
 \setcounter{theorem}{1}
\renewcommand*{\thetheorem}{\Alph{theorem}}
 \begin{theorem}\label{intro_main1}
     Let $m\geq 1$, let $0<a<b<\infty$, and $\Omega=B_b\setminus\bar{B}_a(0)\subset \R^2$. Consider the following minimisation problem
     \begin{align}\label{intro_eigenvalue}
         \mu_m=\inf\ens{\int_{\Omega}(\leb_mu)^2dx:u\in W^{2,2}_0(\Omega)\;\,\text{and}\;\,  \int_{\Omega}\frac{|\D u|^2}{|x|^2}dx=1}.
     \end{align}
     Then, there exists a unique solution $u$ to \eqref{intro_eigenvalue} that satisfies the equation
      \begin{align*}
         \Delta^2u+\frac{2(m^2-1)}{|x|^2}\Delta u-4(m^2-1)\left(\frac{x}{|x|^2}\right)^t\cdot\D^2u\cdot\left(\frac{x}{|x|^2}\right)+\frac{(m^2-1)^2}{|x|^4}u=-\frac{\mu_m}{|x|^2}\left(\Delta u-2\frac{x}{|x|^2}\cdot \D u\right),
     \end{align*}
     and there exists a constant $R_{1,m}^{\ast}<\infty$ such that provided that 
     \begin{align}
         \log\left(\frac{b}{a}\right)\geq R_{1,m}^{\ast},
     \end{align}
     then
     \begin{align}
         \dfrac{\left(4m^2+\dfrac{\pi^2}{\log^2\left(\frac{b}{a}\right)}\right)\dfrac{\pi^2}{\log^2\left(\frac{b}{a}\right)}}{4(m^2+1)+\dfrac{2\pi^2}{\log^2\left(\frac{b}{a}\right)}}<\mu_m<\dfrac{\left(4m^2+\dfrac{2\pi^2}{\log^2\left(\frac{b}{a}\right)}\right)\dfrac{2\pi^2}{\log^2\left(\frac{b}{a}\right)}}{m^2+1+\dfrac{2\pi^2}{\log^2\left(\frac{b}{a}\right)}}.
     \end{align}
     Furthermore, if $m\geq 1$ is an integer, there exists $R_{2,m}^{\ast}<\infty$ such that the condition
     \begin{align*}
         \log\left(\frac{b}{a}\right)\geq R_{2,m}^{\ast}
     \end{align*}
     the solution $u$ takes the form $u(r,\theta)=\Re\left(u_0(r)e^{im\theta}\right)$, where $u_0$ is a solution of the following ordinary differential equation
     \begin{align}
         f''''+\frac{2}{r}f'''-(4m^2-1)\frac{1}{r^2}f''+(4m^2-1)\frac{1}{r^3}f'-(4m^2-1)\frac{1}{r^4}f=-\frac{\mu_m}{r^4}\left(f''-\frac{1}{r}f'-\frac{m^2}{r^2}f\right).
     \end{align}
 \end{theorem}
\renewcommand*{\thetheorem}{\thesection.\arabic{theorem}}
 \begin{cor}
     Let $m\geq 1$, let $0<a<b<\infty$, and $\Omega=B_b\setminus\bar{B}_a(0)\subset \R^2$. Assume that 
     \begin{align}
         \log\left(\frac{b}{a}\right)\geq R_{1,m}^{\ast}.
     \end{align}
     Then, for all $u\in W^{2,2}_0(\Omega)$, we have
     \begin{align}
         \int_{\Omega}\left(\Delta u+2(m-1)\frac{x}{|x|^2}\cdot \D u+\frac{(m-1)^2}{|x|^2}u\right)^2dx\geq \dfrac{\left(4m^2+\dfrac{\pi^2}{\log^2\left(\frac{b}{a}\right)}\right)\dfrac{\pi^2}{\log^2\left(\frac{b}{a}\right)}}{4(m^2+1)+\dfrac{2\pi^2}{\log^2\left(\frac{b}{a}\right)}}\int_{\Omega}\frac{|\D u|^2}{|x|^2}dx. 
     \end{align}
 \end{cor}

 Finally, the general result for the bilaplacian in all dimension reads as follows. For simplicity, we only state the inequality that comes out of corollary from the eigenvalue estimate.

  \setcounter{theorem}{2}
\renewcommand*{\thetheorem}{\Alph{theorem}}
 \begin{theorem}\label{poincare_gen00}
        Let $d\geq 3$, let $0<a<b<\infty$, and let $\Omega=B_b\setminus\bar{B}_a(0)\subset \R^d$. Then, there exists $R_d<\infty$ such that the hypothesis
        \begin{align*}
            \log\left(\frac{b}{a}\right)\geq R_d
        \end{align*}
        implies that for all  $u\in W^{2,2}_0(\Omega)$, we have
        \begin{align}\label{poincare_dim_gen00}
            \int_{\Omega}(\Delta u)^2dx\geq \left(\frac{d^2}{4}+\frac{\pi^2}{\log^2\left(\frac{b}{a}\right)}\right)\left(\frac{(d-4)^2}{4}+\frac{\pi^2}{\log^2\left(\frac{b}{a}\right)}\right)\int_{\Omega}\frac{u^2}{|x|^4}dx.
        \end{align}
        Furthermore, for all $u\in W^{2,2}_0(\Omega)$ and $d\geq 5$, we have
        \begin{align}\label{poincare_dim_gen10}
            \int_{\Omega}(\Delta u)^2dx\geq \frac{d^2(d-4)^2+((d-2)^2+4)\dfrac{8\pi^2}{\log^2\left(\frac{b}{a}\right)}+\dfrac{16\pi^4}{\log^4\left(\frac{b}{a}\right)}}{4(d-4)^2+\dfrac{16\pi^2}{\log^2\left(\frac{b}{a}\right)}}\int_{\Omega}\frac{|\D u|^2}{|x|^2}dx.
        \end{align}
        For $d=3$, and for all $u\in W^{2,2}_0(\Omega)\subset W^{2,2}(\Omega)$, we have
        \begin{align*}
            \int_{\Omega}(\Delta u)^2dx\geq \frac{25+\dfrac{108\pi^2}{\log^2\left(\frac{b}{a}\right)}+\dfrac{16\pi^4}{\log^2\left(\frac{b}{a}\right)}}{36+\dfrac{16\pi^2}{\log^2\left(\frac{b}{a}\right)}}\int_{\Omega}\frac{|\D u|^2}{|x|^2}dx.
        \end{align*}
        Finally, for $d=4$, and for all $u\in W^{2,2}_0(\Omega)\subset W^{2,2}(\R^4)$, we have
        \begin{align*}
            \int_{\Omega}(\Delta u)^2dx\geq \frac{9+\dfrac{10\pi^2}{\log^2\left(\frac{b}{a}\right)}+\dfrac{\pi^4}{\log^2\left(\frac{b}{a}\right)}}{3+\dfrac{\pi^2}{\log^2\left(\frac{b}{a}\right)}}\int_{\Omega}\frac{|\D u|^2}{|x|^2}dx.
        \end{align*}
    \end{theorem}
\renewcommand*{\thetheorem}{\thesection.\arabic{theorem}}
The first inequality was already known in the full space $\Omega=\R^d$ in the case $d\geq 5$ and bears the name of Rellich inequality (\cite{rellich}). Furthermore, our analysis allows us to recover (thanks to an explicit upper bound of the first eigenvalue) the optimal constant
in the inequality
\begin{align*}
    \int_{\R^d}(\Delta u)^2dx\geq \left(\frac{d(d-4)}{4}\right)^2\int_{\R^d}\frac{u^2}{|x|^4}dx
\end{align*}
valid for $d\geq 5$ and $u\in W^{2,2}(\R^d)$ (refer to \cite{hardy-rellich1}, \cite{hardy-rellich2}, and the survey paper \cite{survey_cazacu} for some history of those inequalities. For $d=3$, we deduce that for all $u\in C^{\infty}_c(\R^3\setminus\ens{0})$
\begin{align}\label{dim3_ineq}
    \int_{\R^3}(\Delta u)^2dx\geq \frac{9}{16}\int_{\R^3}\frac{u^2}{|x|^4}dx.
\end{align}
By the Sobolev embedding $W^{2,2}(\R^3)\hookrightarrow W^{1,6}(\R^3)\hookrightarrow C^{0,\frac{1}{2}}(\R^3)$, we deduce that
\begin{align*}
    \bar{C^{\infty}_c(\R^3\setminus\ens{0})}^{W^{2,2}}=W^{2,2}(\R^3)\cap\ens{u:u(0)=0},
\end{align*}
and we deduce that for all $u\in W^{2,2}(\R^3)$ such that $u(0)=0$, the optimal inequality \eqref{dim3_ineq} holds (which recovers a result of Rellich \cite{rellich}; see also \cite{rellich_general_operator}). As expected, for $d=2$ or $d=4$, we do not get any new inequality (in the first case, the integral would not be finite in general, and in the second case, by density of $C^{\infty}_c(\R^4\setminus\ens{0})$ in $W^{2,2}(\R^4)$, it would imply an inequality for all functions $u\in W^{2,2}(\R^4)$, which is impossible since $|x|^{-4}\notin {L}^1_{\mathrm{loc}}(\R^4)$).

The second inequality bears the name Hardy-Rellich inequality, and we deduce that for all $d\geq 3$ and for all $u\in W^{2,2}(\R^d)$
\begin{align}
    \int_{\R^d}(\Delta u)^2dx\geq c_d\int_{\R^d}\frac{|\D u|^2}{|x|^2}dx
\end{align}
where 
\begin{align*}
    c_d=\left\{\begin{alignedat}{2}
        &\frac{25}{36}\qquad&& \text{for}\;\, d=3\\
        &3\qquad&& \text{for}\;\, d=4\\
        &\frac{d^2}{4}\qquad&& \text{for}\;\, d\geq 5,
    \end{alignedat}\right.
\end{align*}
which is the optimal constant in the Hardy-Rellich inequality (\cite{survey_cazacu}). Furthermore, we recover the fact that if $d=3,4$, the minimiser of the inequality (in the annulus) is of the form $u(r,\theta)=u_0(r)Y_1(\theta)$, where $Y_1$ is the first spherical harmonic on $S^d$, while the minimiser is radial for $d\geq 5$. We indeed show that for $d\neq 4$, the first eigenvalue can be bounded from below by
\begin{align*}
    \lambda_1>\inf_{n\geq 0}\ens{\frac{\left(d(d-4)+4n(n+d-2)\right)^2+((d-2)^2+4+4n(n+d-2))\dfrac{8\pi^2}{\log^2\left(\frac{b}{a}\right)}+\dfrac{16\pi^4}{\log^2\left(\frac{b}{a}\right)}}{4(d-4)^2+16n(n+d-2)+\dfrac{16\pi^2}{\log^2\left(\frac{b}{a}\right)}}}
\end{align*}
which is minimal for $n=0$ if $d\geq 5$, and for $n=1$ if $d=3$ (provided that the conformal class is large enough). For $d=4$, the formula above is valid for $n\geq 1$, but due to an exceptional algebraic structure, we prove that for $n=0$, the first eigenvalue of the associated ODE is equal (this is the only case where one gets an exact expression of the first eigenvalue) to
\begin{align*}
    \mu_0=4+\frac{4\pi^2}{\log^2\left(\frac{b}{a}\right)}
\end{align*}
which is strictly larger than 
\small
\begin{align*}
    \inf_{n\geq 1}\ens{\frac{16n^2(n+2)^2+(8+4n(n+2))\dfrac{8\pi^2}{\log^2\left(\frac{b}{a}\right)}+\dfrac{16\pi^4}{\log^2\left(\frac{b}{a}\right)}}{16n(n+2)+\dfrac{16\pi^2}{\log^2\left(\frac{b}{a}\right)}}}=\frac{9+\dfrac{10\pi^2}{\log^2\left(\frac{b}{a}\right)}+\dfrac{\pi^4}{\log^2\left(\frac{b}{a}\right)}}{3+\dfrac{\pi^2}{\log^2\left(\frac{b}{a}\right)}}=3+O\left(\frac{1}{\log^2\left(\frac{b}{a}\right)}\right)
\end{align*}
\normalsize
if the conformal class is large enough.

    \subsection{Applications}

    As we mentioned it in the beginning of the introduction, the main motivation behind this work is applications to Morse theory. More precisely, the estimate proven for the bilaplacian $\Delta^2$ in dimension $2$ allow one to prove the upper semi-continuity of the extended Morse index in the case of immersions (\cite{morse_willmore_I}), whiles the estimate on $\leb_m^{\ast}\leb_m$ (for $m\geq 2$) should allow us to treat the case of branched immersions (\cite{morse_willmore_II})—there are many technical difficulties in this case, starting by the definition of the Morse index—and later on to adapt our methods to the viscosity method (\cite{morse_viscosity}). Furthermore, the results in dimension $4$ for the bilaplacian should easily give us the two-sided Morse index stability estimate for (intrinsic or extrinsic) biharmonic maps into manifolds (\cite{morse_biharmonic}). Finally, it should be possible to use the eigenvalues estimates to prove Morse index stability results for the Yang-Mills functional. Indeed, in an appropriate gauge (\cite{uhlenbeck_yang_mills,yang_mills_riviere_tian_birthday}), the highest order term in the Yang-Mills functional becomes the $L^2$ norm of the Laplacian of a suitable function, which should allow one to prove that variations located in the neck have a positive contributions in the second derivative (that remains to be computed).

\section{First Eigenvalue Problem in Dimension 2}

 \subsection{Basic Properties of the Family of Differential Operators}

   As we mentioned it in the introduction, we will study some fine properties of the following second-order elliptic differential operator with \emph{regular} singularities (\cite{PR})
   \begin{align}
       \mathscr{L}_m=\Delta +2(m-1)\frac{x}{|x|^2}\cdot \D +\frac{(m-1)^2}{|x|^2}.
   \end{align}

As we mentioned in the introduction, its adjoint we already studied in \cite{index4}. Let us recall some basic properties of $\leb_m^{\ast}$. 
    
    \begin{lemme}[\cite{index4}]\label{indicielles}
    	$0<a<b<\infty$, and let $\Omega=B_b\setminus\bar{B}_a(0)$. Define for all for all $m\geq 1$ the second order elliptic differential operator
    	\begin{align*}
    	\widetilde{\mathscr{L}}_m=\Delta-2(m+1)\frac{x}{|x|^2}\cdot \D +\frac{(m+1)^2}{|x|^2}.
    	\end{align*} 
    	Let $u\in W^{2,2}(\Omega)$ be such that $u=\partial_{\nu}u=0$ Then we have the identity
    	\begin{align}\label{ipp0}
    		&\int_{\Omega}\left(\widetilde{\mathscr{L}}_mu\right)^2dx=\int_{\Omega}\left(\Delta u-2(m+1)\frac{x}{|x|^2}\cdot \D u+(m+1)^2\frac{u}{|x|^2}\right)^2dx\\
    		&=\int_{\Omega}\left(\Delta u+(m+1)(m-1)\frac{u}{|x|^2}\right)^2dx+4(m+1)(m-1)\int_{\Omega}\left(\frac{x}{|x|^2}\cdot \D u-\frac{u}{|x|^2}\right)^2dx.\nonumber
    	\end{align}
    \end{lemme}
    \begin{rem}
        With the previous formula, we have $\mathscr{L}_m=\widetilde{\mathscr{L}}_{m-2}^{\ast}$, since
        \begin{align*}
            \frac{x}{|x|^2}\cdot \D=\D\log|x|\cdot \D,
        \end{align*}
        and $\log|\,\cdot\,|$ is a harmonic function on $\R^2\setminus\ens{0}$.
    \end{rem}

    We will also need the following computation from \cite{index4}.
	\begin{align}\label{operator0}
	\widetilde{\mathscr{L}}_m^{\ast}\widetilde{\mathscr{L}}_m&=\Delta^2+\frac{2(m+1)(m-1)}{|x|^2}\Delta -4(m+1)(m-1)\left(\frac{x}{|x|^2}\right)^t\cdot \D^2(\,\cdot\,)\cdot\left(\frac{x}{|x|^2}\right)+\frac{(m+1)^2(m-1)^2}{|x|^4},
	\end{align}
	and we indeed recover $\widetilde{\mathscr{L}}_1^{\ast}\widetilde{\mathscr{L}}_1=\Delta^2$ in the case $m=1$.

   \subsection{Computation of the Fourth-Order Differential Operators}

   The goal of this section is to prove the following result:
   \begin{align*}
       \leb_m^{\ast}\leb_m=\widetilde{\leb}_{m}^{\ast}\widetilde{\leb}_{m}.
   \end{align*}
   By the formulae  (from \cite{index4})
   \begin{align*}
       \Delta \left(\frac{1}{|x|^2}(\,\cdot\,)\right)&=\frac{1}{|x|^2}\left(\Delta-4\frac{x}{|x|^2}\cdot \D+\frac{4}{|x|^2}\right)u\\
 	\Delta\left(\frac{x}{|x|^2}\cdot \D (\,\cdot\,)\right)&=\frac{x}{|x|^2}\cdot \D\Delta +\frac{2}{|x|^2}\Delta -4\left(\frac{x}{|x|^2}\right)^t\cdot\D^2 (\,\cdot\,)\cdot\left(\frac{x}{|x|^2}\right),
   \end{align*}
   we get
   \begin{align*}
       \Delta(\mathscr{L}_m)&=\Delta^2+2(m-1)\frac{x}{|x|^2}\cdot \D \Delta+4(m-1)\frac{1}{|x|^2}\Delta-8(m-1)\left(\frac{x}{|x|^2}\right)^t\cdot \D^2(\,\cdot\,)\cdot\left(\frac{x}{|x|^2}\right)\\
       &+\frac{(m-1)^2}{|x|^2}\left(\Delta-4\frac{x}{|x|^2}\cdot \D+\frac{4}{|x|^2}\right)\\
       &=\Delta^2+2(m-1)\frac{x}{|x|^2}\cdot \D\Delta+(m-1)(m+3)\frac{1}{|x|^2}\Delta -8(m-1)\left(\frac{x}{|x|^2}\right)^t\cdot\D^2(\,\cdot\,)\cdot \left(\frac{x}{|x|^2}\right)\\
       &-4(m-1)^2\frac{x}{|x|^4}\cdot \D+\frac{4(m-1)^2}{|x|^4}.
   \end{align*}
   Likewise, by the identities
   \begin{align*}
       &\frac{x}{|x|^2}\cdot \D\left(\frac{x}{|x|^2}\cdot \D \right)=\left(\frac{x}{|x|^2}\right)^t\cdot \D^2 (\,\cdot\,)\cdot \left(\frac{x}{|x|^2}\right)-\frac{x}{|x|^4}\cdot \D\\
       &\frac{x}{|x|^2}\cdot \D\left(\frac{1}{|x|^2}(\,\cdot\,)\right)=\frac{x}{|x|^4}\cdot \D-\frac{2}{|x|^4},
   \end{align*}
   we get
   \begin{align*}
       \frac{x}{|x|^2}\cdot \D(\mathscr{L}_m)&=\frac{x}{|x|^2}\cdot \D\Delta +2(m-1)\left(\frac{x}{|x|^2}\right)^t\cdot \D^2(\,\cdot\,)\cdot \left(\frac{x}{|x|^2}\right)-2(m-1)\frac{x}{|x|^4}\cdot \D\\
       &+(m-1)^2\frac{x}{|x|^4}\cdot \D-\frac{2(m-1)^2}{|x|^4}
   \end{align*}
   Trivially, we have
   \begin{align*}
       \frac{1}{|x|^2}\mathscr{L}_m=\frac{1}{|x|^2}\Delta+2(m-1)\frac{x}{|x|^4}\cdot \D+\frac{(m-1)^2}{|x|^4}.
   \end{align*}
   Using the identity $\dfrac{x}{|x|^2}=\D\log|x|$, we obyain
   \begin{align*}
       \mathscr{L}_m^{\ast}=\Delta-2(m-1)\frac{x}{|x|^2}\cdot \D+\frac{(m-1)^2}{|x|^2}.
   \end{align*}
   Finally, we have
   \small
   \begin{align}\label{operator1}
       &\mathscr{L}_m^{\ast}\mathscr{L}_m=\Delta^2+\colorcancel{2(m-1)\frac{x}{|x|^2}\cdot \D\Delta}{red}+(m-1)(m+3)\frac{1}{|x|^2}\Delta -8(m-1)\left(\frac{x}{|x|^2}\right)^t\cdot\D^2(\,\cdot\,)\cdot \left(\frac{x}{|x|^2}\right)\nonumber\\
       &-\colorcancel{4(m-1)^2\frac{x}{|x|^4}\cdot \D}{blue}+\frac{4(m-1)^2}{|x|^4}
       -\colorcancel{2(m-1)\frac{x}{|x|^2}\cdot \D\Delta}{red}-4(m-1)^2\left(\frac{x}{|x|^2}\right)^t\cdot \D^2(\,\cdot\,)\cdot\left(\frac{x}{|x|^2}\right)\nonumber\\
       &+\colorcancel{4(m-1)^2\frac{x}{|x|^4}\cdot \D
       }{blue}-\colorcancel{2(m-1)^3\frac{x}{|x|^4}\cdot \D}{blue}+\frac{4(m-1)^3}{|x|^2}+\frac{(m-1)^2}{|x|^4}\Delta +\colorcancel{2(m-1)^3\frac{x}{|x|^4}\cdot \D}{blue}+\frac{(m-1)^4}{|x|^2}\nonumber\\
       &=\Delta^2+2(m+1)(m-1)\frac{1}{|x|^2}\Delta -4(m+1)(m-1)\left(\frac{x}{|x|^2}\right)^t\cdot\D^2(\,\cdot\,)\cdot\left(\frac{x}{|x|^2}\right)+\frac{(m+1)^2(m-1)^2}{|x|^4}.
   \end{align}
   \normalsize
   Comparing \eqref{operator0} and \eqref{operator1}, we deduce that 
   \begin{align}\label{operator2}
       \mathscr{L}_m^{\ast}\mathscr{L}_m=\widetilde{\mathscr{L}}_{m}^{\ast}\widetilde{\mathscr{L}}_m=\widetilde{\leb}_{m-2}\widetilde{\leb}_{m-2}^{\ast}.
   \end{align}

   \subsection{Estimates on the Second-Order Differential Operators}
   
   Thanks to \eqref{ipp0}, if $0<a<b<\infty$, $\Omega=B_b\setminus\bar{B}_a(0)$, and $u\in W^{2,2}_0(\Omega)$, then we get by \eqref{operator1}
   \begin{align}\label{ipp_Lm}
      \int_{\Omega}(\mathscr{L}_mu)^2dx&=\int_{\Omega}\left(\Delta u+(m^2-1)\frac{u}{|x|^2}\right)^2dx    +4(m^2-1)\int_{\Omega}\left(\frac{x}{|x|^2}\cdot \D u-\frac{u}{|x|^2}\right)^2dx.
   \end{align}
   
   \subsection{Expression of the Second-Order Differential Operators in Polar Coordinates}
    
   Using polar coordinates
   \begin{align*}
       \begin{pmatrix}
       \p{x_1}f\\
       \p{x_2}f
       \end{pmatrix}=\begin{pmatrix}
       \cos(\theta)\p{r}f-\sin(\theta)\dfrac{1}{r}\p{\theta}f\vspace{0.5em}\\
       \sin(\theta)\p{r}f+\cos(\theta)\dfrac{1}{r}\p{\theta}f
       \end{pmatrix},
   \end{align*}
   we get 
   \begin{align*}
       \p{x_1}^2f
       &=\cos^2(\theta)\p{r}^2f+\sin^2(\theta)\frac{1}{r^2}\p{\theta}^2f-2\cos(\theta)\sin(\theta)\frac{1}{r}\p{r,\theta}^2f
       +\sin^2(\theta)\frac{1}{r}\p{r}f+2\cos(\theta)\sin(\theta)\frac{1}{r^2}\p{\theta}f\\
       \p{x_1,x_2}^2f
       &=\cos(\theta)\sin(\theta)\left(\p{r}^2f-\frac{1}{r^2}\p{\theta}^2f\right)+\left(\cos^2(\theta)-\sin^2(\theta)\right)\frac{1}{r}\p{r,\theta}^2f-\cos(\theta)\sin(\theta)\frac{1}{r}\p{r}f\nonumber\\
       &+\left(\sin^2(\theta)-\cos^2(\theta)\right)\frac{1}{r^2}\p{\theta}f\\
       \p{x_2}^2f
       &=\sin^2(\theta)\p{r}^2f+\cos^2(\theta)\frac{1}{r^2}\p{\theta}^2f+2\cos(\theta)\sin(\theta)\frac{1}{r}\p{r,\theta}^2f+\cos^2(\theta)\frac{1}{r}\p{r}f-2\cos(\theta)\sin(\theta)\frac{1}{r^2}\p{\theta}f.
   \end{align*}
   Therefore, we have
   \begin{align*}
       &|x|^2\left(\frac{x}{|x|^2}\right)^t\cdot\D^2f\cdot\left(\frac{x}{|x|^2}\right)=\frac{1}{|x|^2}\sum_{i,j=1}^2x_i\,x_j\,\p{x_i,x_j}^2u\\
       &=\cos^2(\theta)\left(\cos^2(\theta)\p{r}^2f+\colorcancel{\sin^2(\theta)\frac{1}{r^2}\p{\theta}^2f}{red}-\colorcancel{2\cos(\theta)\sin(\theta)\frac{1}{r}\p{r,\theta}^2f}{blue}
       +\colorcancel{\sin^2(\theta)\frac{1}{r}\p{r}f}{red}+\colorcancel{2\cos(\theta)\sin(\theta)\frac{1}{r^2}\p{\theta}f}{blue}\right)\\
       &+2\cos(\theta)\sin(\theta)\left(\cos(\theta)\sin(\theta)\left(\p{r}^2f-\colorcancel{\frac{1}{r^2}\p{\theta}^2f}{red}\right)+\colorcancel{\left(\cos^2(\theta)-\sin^2(\theta)\right)\frac{1}{r}\p{r,\theta}^2f}{blue}-\colorcancel{\cos(\theta)\sin(\theta)\frac{1}{r}\p{r}f}{red}\right.\\
       &\left.+\colorcancel{\left(\sin^2(\theta)-\cos^2(\theta)\right)\frac{1}{r^2}\p{\theta}f}{blue}\right)
       +\sin^2(\theta)\left(\sin^2(\theta)\p{r}^2f+\colorcancel{\cos^2(\theta)\frac{1}{r^2}\p{\theta}^2f}{red}+\colorcancel{2\cos(\theta)\sin(\theta)\frac{1}{r}\p{r,\theta}^2f}{blue}\right.\\
       &\left.+\colorcancel{\cos^2(\theta)\frac{1}{r}\p{r}f}{red}-\colorcancel{2\cos(\theta)\sin(\theta)\frac{1}{r^2}\p{\theta}f}{blue}\right)
       =\p{r}^2f,
   \end{align*}
   so that 
   \begin{align}\label{operator3}
       \left(\frac{x}{|x|^2}\right)^t\cdot\D^2f\cdot\left(\frac{x}{|x|^2}\right)=\frac{1}{r^2}\p{r}^2f
   \end{align}
   where we used the identity
  $
       \cos^4(\theta)+2\cos^2(\theta)\sin^2(\theta)+\sin^4(\theta)=(\cos^2(\theta)+\sin^2(\theta))^2=1. 
 $
   By the identity
   $
       \Delta=\p{r}^2+\dfrac{1}{r}\p{r}+\dfrac{1}{r^2}\p{\theta}^2,
   $
   we get
   \begin{align*}
       &\p{r}\Delta=\p{r}^3+\frac{1}{r}\p{r}^2-\frac{1}{r^2}\p{r}+\frac{1}{r^2}\p{r}\p{\theta}^2-\frac{2}{r^3}\p{\theta}^2\\
       &\p{r}^2\Delta=\p{r}^4+\frac{1}{r}\p{r}^3-\frac{2}{r^2}\p{r}^2+\frac{2}{r^3}\p{r}+\frac{1}{r^2}\p{r}^2\p{\theta}^2-\frac{4}{r^3}\p{r}\p{\theta}^2+\frac{6}{r^4}\p{\theta}^2\\
       &\frac{1}{r^2}\p{\theta}^2\Delta=\frac{1}{r^2}\p{r}^2\p{\theta}^2+\frac{1}{r^3}\p{r}\p{\theta}^2+\frac{1}{r^4}\p{\theta}^4.
    \end{align*}
    Therefore, we finally deduce that
    \begin{align*}
        \Delta^2&=\p{r}^4+\frac{1}{r}\p{r}^3-\frac{2}{r^2}\p{r}^2+\frac{2}{r^3}\p{r}+\frac{1}{r^2}\p{r}^2\p{\theta}^2-\frac{4}{r^3}\p{r}\p{\theta}^2+\frac{6}{r^4}\p{\theta}^2\\
        &+\frac{1}{r}\p{r}^3+\frac{1}{r^2}\p{r}^2-\frac{1}{r^3}\p{r}+\frac{1}{r^3}\p{r}\p{\theta}^2-\frac{2}{r^4}\p{r}\p{\theta}^2
        +\frac{1}{r^2}\p{r}^2\p{\theta}^2+\frac{1}{r^3}\p{r}\p{\theta}^2+\frac{1}{r^4}\p{\theta}^4\\
        &=\p{r}^4+\frac{2}{r}\p{r}^3-\frac{1}{r^2}\p{r}^2+\frac{1}{r^3}\p{r}+\frac{1}{r^4}\p{\theta}^4+\frac{2}{r^2}\p{r}^2\p{\theta}^2-\frac{2}{r^3}\p{r}\p{\theta}^2+\frac{4}{r^4}\p{\theta}^2
    \end{align*}
    Projecting on the eigenspace corresponding to the eigenvalue $\lambda=n^2$, we get the operator
    \begin{align*}
        \Pi_{n^2}\left(\Delta^2\right)=\p{r}^4+\frac{2}{r}\p{r}^3-\frac{2\,n^2+1}{r^2}\p{r}^2+\frac{2\,n^2+1}{r^3}\p{r}+\frac{n^4-4n^2}{r^4}.
    \end{align*}
    Since
    \begin{align*}
        \Pi_{n^2}(\Delta)=\p{r}^2+\frac{1}{r}\p{r}-\frac{n^2}{r^2},
    \end{align*}
    we finally get by \eqref{operator3}
    \begin{align}\label{projection_m}
        &\Pi_{n^2}(\mathscr{L}_m^{\ast}\leb_m)=\Pi_{n^2}\left(\Delta^2+2(m+1)(m-1)\frac{1}{r^2}\Delta -4(m+1)(m-1)\frac{1}{r^2}\p{r}^2+\frac{(m+1)^2(m-1)^2}{r^4}\right)\nonumber\\
        &=\p{r}^4+\frac{2}{r}\p{r}^3-\frac{2\,n^2+1}{r^2}\p{r}^2+\frac{2\,n^2+1}{r^3}\p{r}+\frac{n^4-4n^2}{r^4}\nonumber\\
        &+2(m+1)(m-1)\frac{1}{r^2}\p{r}^2+2(m+1)(m-1)\frac{1}{r^3}\p{r}-2n^2(m+1)(m-1)\frac{1}{r^4}-4(m+1)(m-1)\frac{1}{r^2}\p{r}^2\nonumber\\
        &+\frac{(m+1)^2(m-1)^2}{r^4}\nonumber\\
        &=\p{r}^4+\frac{2}{r}\p{r}^3-(2(m+1)(m-1)+2\,n^2+1)\frac{1}{r^2}\p{r}^2+(2(m+1)(m-1)+2\,n^2+1)\frac{1}{r^3}\p{r}\nonumber\\
        &+\left((m+1)^2(m-1)^2-2n^2(m+1)(m-1)+n^2(n^2-4)\right)\frac{1}{r^4}.
    \end{align}

     \subsection{Minimisation Problem and Main Theorem}

    \subsubsection{Existence and Non-Triviality of Minimisers}

    Let $m\geq 1$ be a fixed real number, and consider for all $0<a<b<\infty$ the following minimisation problem:
    \begin{align}\label{min}
        \lambda_m=\inf\left\{\int_{B_{b}\setminus\bar{B}_a(0)}\left(\mathscr{L}_m u\right)^2dx: u\in W^{2,2}_0(B_b\setminus\bar{B}_a(0))\quad \text{and}\quad\int_{B_b\setminus\bar{B}_a(0)}\frac{u^2}{|x|^4}dx=1\right\}. 
    \end{align}
    From now on, we write $\Omega=B_b\setminus\bar{B}_a(0)$. 
    For $m=1$, as $\mathscr{L}_1=\Delta$,  there exists a constant $C=C(\Omega)$ such that for all $u\in W^{2,2}_0(\Omega)$,
    \begin{align*}
        \int_{\Omega}\left(|\D u|^2+u^2\right)dx\leq C\int_{\Omega}(\Delta u)^2dx.
    \end{align*}
    Indeed, by the standard Poincaré inequality, there exists a universal constant $C(\Omega)<\infty$ such that for all $u\in W^{2,2}_0(\Omega)\subset W^{1,2}_0(\Omega)$, we have
    \begin{align}\label{poincaré}
        \int_{\Omega}u^2dx\leq C(\Omega)\int_{\Omega}|\D u|^2dx.
    \end{align}
    Integrating by parts, and using Cauchy-Schwarz inequality and \eqref{poincaré}, we get
    \begin{align*}
        \int_{\Omega}|\D u|^2dx=-\int_{\Omega}u\,\Delta u\,dx\leq\left(\int_{\Omega}u^2dx\right)^{\frac{1}{2}}\left(\int_{\Omega}(\Delta u)^2dx\right)^{\frac{1}{2}}\leq \left(C(\Omega)\int_{\Omega}|\D u|^2dx\right)^{\frac{1}{2}} \left(\int_{\Omega}(\Delta u)^2dx\right)^{\frac{1}{2}}. 
    \end{align*}
    Therefore, provided that $u$ is not constant (\emph{i.e.} $u\neq 0$), we get
    \begin{align*}
        \int_{\Omega}|\D u|^2\leq C(\Omega)\int_{\Omega}(\Delta u)^2dx,
    \end{align*}
    which implies that
    \begin{align}\label{poincaré_laplacien}
        \int_{\Omega}\left(|\D u|^2+u^2\right)dx\leq C(\Omega)(1+C(\Omega))\int_{\Omega}(\Delta u)^2dx.
    \end{align}
    Finally, by integrating by parts, we have for all $u\in W^{2,2}_0(\Omega)$
    \begin{align*}
        \int_{\Omega}|\p{z}^2u|^2|dz|^2&=\int_{\Omega}\p{z}^2u\,\p{\z}^2u\,|dz|^2=-\int_{\Omega}\p{\z}(\p{z}^2u)\p{\z}u\,|dz|^2=-\int_{\Omega}\p{z}(\p{z\z}^2u)\p{\z}u\,|dz|^2=\int_{\Omega}|\p{z\z}^2u|^2dx\\
        &=\frac{1}{4}\int_{\Omega}(\Delta u)^2dx,
    \end{align*}
    using the well-known formula $\Delta=4\,\p{z\z}^2$. Therefore, we have 
    \begin{align}\label{ipp_laplacien}
        \int_{\Omega}|\D^2u|^2dx=\int_{\Omega}\left(2|\p{z}^2u|^2+2|\p{z\z}^2u|^2\right)|dz|^2=\int_{\Omega}(\Delta u)^2dx.
    \end{align}
    Gathering \eqref{poincaré_laplacien} and \eqref{ipp_laplacien}, we deduce that for all $u\in W^{2,2}_0(\Omega)$
    \begin{align}\label{complete_poincaré}
        \int_{\Omega}\left(|\D^2u|^2+|\D u|^2+u^2\right)dx\leq \left(1+C(\Omega)\right)^2\int_{\Omega}(\Delta u)^2dx.
    \end{align}
    where $C(\Omega)<\infty$ is the Poincaré constant from \eqref{poincaré}.
    Therefore, $E(u)=\np{\Delta u}{2}{\Omega}^2$ is a coercive quadratic form, which shows by the Lax-Milgram theorem (\cite[Corollaire V.$8$]{brezis}) that there exists a unique minimiser to this problem. Let us denote it by $u$. It satisfies for some $\lambda_1\geq 0$  the equation
    \begin{align}
        \left\{\begin{alignedat}{2}
            \Delta^2u&=\frac{\lambda_1}{|x|^4} u\quad&& \text{in}\;\, \Omega\\
            u&=0\quad&& \text{on}\;\,\partial\Omega\\
            \partial_{\nu}u&=0\quad &&\text{on}\;\, \partial\Omega.
        \end{alignedat}\right.
    \end{align}
    Integrating by parts, we deduce that
    \begin{align}\label{forth}
        \int_{\Omega}(\Delta u)^2dx=\int_{\Omega}u\,\Delta^2u\,dx=\lambda_1\int_{\Omega}\frac{u^2}{|x|^4}dx=\lambda_1,
    \end{align}
    since
    \begin{align*}
        \int_{\Omega}\frac{u^2}{|x|^4}dx=1
    \end{align*}
    by construction (more precisely, thanks to the compact Sobolev embedding $W^{2,2}_0(\Omega)\hookrightarrow L^2(\Omega)$; see the computations below for more details in general).

    Now let us generalise \eqref{complete_poincaré} to the general case $m\geq 1$.
    \begin{prop}
    Let $0<a<b<\infty$, $m\geq 1$, and $\Omega=B_b\setminus\bar{B}_a(0)$. Then, there exists a universal constant $C=C(m,\Omega)$ such that for all $u\in W^{2,2}_0(\Omega)$, we have
    \begin{align}\label{generalised_poincaré}
        \int_{\Omega}\left(|\D^2u|^2+|\D u|^2+u^2\right)dx\leq C\int_{\Omega}\left(\Delta u+2(m-1)\frac{x}{|x|^2}\cdot \D u+(m-1)^2\frac{u}{|x|^2}\right)^2dx.
    \end{align}
    \end{prop}
    \begin{proof}
        We have already treated the case $m=1$, so we can assume that $m>1$. Thanks to \eqref{ipp_Lm}, we have for all $u\in W^{2,2}_0(\Omega)$
        \begin{align}\label{formula_ipp_lmbis}
            \int_{\Omega}\left(\leb_mu\right)^2dx=\int_{\Omega}\left(\Delta u+(m^2-1)\frac{u}{|x|^2}\right)^2dx+4(m^2-1)\int_{\Omega}\left(\frac{x}{|x|^2}\cdot \D u-\frac{u}{|x|^2}\right)^2dx.
        \end{align}
        Thanks to the Poincaré inequality, there exists $C(\Omega)<\infty$ such that for all $u\in W^{2,2}_0(\Omega)\subset W^{1,2}_0(\Omega)$, we have
        \begin{align*}
            \int_{\Omega}u^2dx\leq C(\Omega)\int_{\Omega}|\D u|^2dx.
        \end{align*}
        Therefore, the following elementary inequality holds 
        \begin{align}\label{bound_poincaré}
            \int_{\Omega}\frac{u^2}{|x|^4}dx\leq \frac{1}{a^4}\int_{\Omega}u^2dx\leq \frac{C(\Omega)}{a^4}\int_{\Omega}|\D u|^2dx.
        \end{align}
        Therefore, we deduce by \eqref{bound_poincaré}, \eqref{formula_ipp_lmbis}, and Minkowski's inequality that
        \begin{align*}
            \np{\Delta u}{2}{\Omega}&\leq (m^2-1)\np{\frac{u}{|x|^2}}{2}{\Omega}+\left(\int_{\Omega}\left(\Delta u +(m^2-1)\frac{u}{|x|^2}\right)^2dx\right)^{\frac{1}{2}}\\
            &\leq (m^2-1)\frac{\sqrt{C(\Omega)}}{a^2}\np{\D u}{2}{\Omega}+\np{\leb_m u}{2}{\Omega}.
        \end{align*}
        Finally, using \eqref{poincaré_laplacien} and \eqref{ipp_laplacien}, we deduce that there exists $C(\Omega)<\infty$ such that for all $u\in W^{2,2}_0(\Omega)$,
        \begin{align*}
            \wp{u}{2,2}{\Omega}\leq \sqrt{C(\Omega)}\np{\leb_mu}{2}{\Omega},
        \end{align*}
        which concludes the proof of the proposition.
    \end{proof}

    As previously, we deduce that there exists a minimiser $u$ to \eqref{min}. Indeed, take a minimising sequence $\ens{u_k}_{k\in\N}\subset W^{2,2}_0(\Omega)$, where 
    \begin{align*}
        &\int_{\Omega}\frac{u_k^2}{|x|^4}dx=1\\
        &\int_{\Omega}\left(\leb_mu_k\right)^2dx\conv{k\rightarrow \infty}\lambda_m.
    \end{align*}
    Then, by inequality \eqref{generalised_poincaré}, $\ens{u_k}_{k\in\N}$ is a bounded sequence in $W^{2,2}(\Omega)$. In particular, we deduce that there exists $u\in W^{2,2}_0(\Omega)$ such that
    \begin{align*}
        &u_k\underset{k\rightarrow \infty}{\rightharpoonup}u\;\,\text{weakly in}\;\, W^{2,2}_0(\Omega). 
    \end{align*}
    Thanks to the compact Sobolev embedding of $W^{1,2}(\Omega)\hooklongrightarrow L^q(\Omega)$ for all $q<\infty$, we deduce that 
    \begin{align*}
        \lim_{k\rightarrow \infty}\wp{u_k-u}{1,2}{\Omega}=0.
    \end{align*}
    In particular, the constraint
    \begin{align}\label{constraint}
        \int_{\Omega}\frac{u^2}{|x|^4}dx=\lim_{k\rightarrow\infty}\int_{\Omega}\frac{u_k^2}{|x|^4}=1
    \end{align}
    is satisfied, and Fatou's lemma shows that
    \begin{align}\label{bound}
        \int_{\Omega}\left(\leb_mu\right)^2dx\leq \liminf_{k\rightarrow  \infty}\int_{\Omega}\left(\leb_mu\right)^2dx=\lambda_m.
    \end{align}
    By the constraint \eqref{constraint} and the bound \eqref{bound}, we deduce that $u$ is a minimiser and that
    \begin{align*}
        \int_{\Omega}\left(\leb_mu\right)^2dx=\lambda_m.
    \end{align*}
    The solution $u$ is a solution to the following system of equations
    \begin{align}\label{sys_u_lm}
        \left\{\begin{alignedat}{2}
            \mathscr{L}_m^{\ast}\mathscr{L}_mu&=\frac{\lambda}{|x|^4} u\quad&& \text{in}\;\, \Omega\\
            u&=0\quad&& \text{on}\;\,\partial\Omega\\
            \partial_{\nu}u&=0\quad &&\text{on}\;\, \partial\Omega.
        \end{alignedat}\right.
    \end{align}
    Integrating by parts, we recover
    \begin{align}
        \int_{B_b\setminus \bar{B}_a(0)}(\mathscr{L}_mu)^2dx=\lambda_m\int_{B_b\setminus \bar{B}_a(0)}\frac{u^2}{|x|^4}dx=\lambda_m.
    \end{align}
    \begin{lemme}\label{lambda_positive}
        Under the previous hypothesis, for all $m\geq 1$, we have $\lambda_m>0$.
    \end{lemme}
    \begin{proof}
    Indeed, if $\lambda_m=0$, then $\leb_m=0$, and inequality \eqref{generalised_poincaré} shows that $u=0$, which contradicts the identity 
    \begin{align*}
        \int_{\Omega}\frac{u^2}{|x|^4}dx=1.
    \end{align*}
    \end{proof}

    \subsubsection{Strategy of the Proof}

    The proof goes as follows. 
    We study the solutions of the equation
    \begin{align}\label{main_equation}
        \leb_m^{\ast}\leb_mf=\frac{\lambda}{|x|^4}f
    \end{align}
    projected on the eigenspace of the eigenvector $e^{i\,n\,\theta}$ for $n\in \Z$. That leads us to study thanks to \eqref{projection_m} the solutions of the following ordinary differential equation
    \begin{align}\label{ode_m_n_0}
        f''''+\frac{2}{r}f'''-\left(2m^2+2n^2-1\right)\frac{1}{r^2}f''+\left(2m^2+2n^2-1\right)\frac{1}{r^3}f'+\left((m^2-n^2-1)^2-4n^2)\right)\frac{1}{r^4}f=\frac{\lambda}{r^4}f.
    \end{align}
    with the boundary conditions
    \begin{align}\label{boundary_m_n_0}
    \left\{\begin{alignedat}{1}
        &f(a)=f(b)=0\\
        &f'(a)=f'(b)=0,
        \end{alignedat}\right.
    \end{align}
    that are equivalent to $u=\partial_{\nu}u=0$ on $\partial \Omega$ for $u=f(r)e^{i\,n\,\theta}$. In order to get the previous equation, we used the elementary identities:
    \begin{align*}
        2(m+1)(m-1)+2n^2+1&=2(m^2-1)+2n^2+1=2m^2+2n^2-1\\
        (m+1)^2(m-1)^2-2n^2(m+1)(m-1)+n^2(n^2-4)&=(m^2-1)^2-2n^2(m^2-1)+n^2(n^2-4)\\
        &=(m^2-n^2-1)^2-4n^2
    \end{align*}
    Then, a careful analysis gives us a bound on the minimal coefficient $\lambda=\lambda_{m,n}$ of this equation, once we make a change of variable that makes it linear (this is where the regular singularities of our operator play a crucial role).

    \subsection{Linearisation of the Ordinary Differential Equation}

    Let $m>0$, $n\in \Z$, and consider non-trivial solutions of the following ordinary differential equation:
    \begin{align}\label{eq_f_m_n}
        &\p{r}^4f+\frac{2}{r}\p{r}^3f-(2(m+1)(m-1)+2\,n^2+1)\frac{1}{r^2}\p{r}^2f+(2(m+1)(m-1)+2\,n^2+1)\frac{1}{r^3}\p{r}f\nonumber\\
        &+\left((m+1)^2(m-1)^2-2n^2(m+1)(m-1)+n^2(n^2-4)\right)\frac{1}{r^4}f=\frac{\lambda}{r^4}f,
    \end{align}
    where $\lambda\geq 0$.
    Notice that
    \begin{align*}
        &(m+1)^2(m-1)^2-2n^2(m+1)(m-1)+n^2(n^2-4)\\
        &=(m^2-1)^2-2n^2(m^2-1)+n^2(n^2-4)=(m^2-n^2-1)^2-4n^2.
    \end{align*}
    Make the change of variable $f(r)=Y(\log(r))$. Then, we have
    \begin{align}\label{change_var_log}
    \left\{\begin{alignedat}{1}
        &\p{r}f(r)=\frac{1}{r}Y'(\log(r))\\
        &\p{r}^2f(r)=\frac{1}{r^2}Y''(\log(r))-\frac{1}{r^2}Y'(\log(r))\\
        &\p{r}^3f(r)=\frac{1}{r^3}Y'''(\log(r))-\frac{3}{r^3}Y''(\log(r))+\frac{2}{r^3}Y'(\log(r))\\
        &\p{r}^4f(r)=\frac{1}{r^4}Y''''(\log(r))-\frac{6}{r^4}Y'''(\log(r))+\frac{11}{r^4}Y''(\log(r))-\frac{6}{r^4}Y'(\log(r)).
        \end{alignedat}\right.
    \end{align}
    Therefore, \eqref{eq_f_m_n} and \eqref{change_var_log} show that
    \begin{align}\label{change_var_log2}
        &f''''+\frac{2}{r}f'''-(2m^2+2n^2-1)\frac{1}{r^2}f''+(2m^2+2n^2-1)\frac{1}{r^3}f'\nonumber\\
        &=\frac{1}{r^4}\bigg(Y''''(\log(r))-6\,Y'''(\log(r))+11\,Y''(\log(r))-6\,Y'(\log(r))\nonumber\\
        &+2\,Y'''(\log(r))-6\,Y''(\log(r))+4\,Y'(\log(r))
        -(2m^2+2n^2-1)Y''(\log(r))+(2m^2+2n^2-1)Y'(\log(r))\nonumber\\
        &+(2m^2+2n^2-1)Y'(\log(r))\bigg)
        =\frac{1}{r^4}\bigg(Y''''(\log(r))-4\,Y'''(\log(r))-(2m^2+2n^2-6)Y''(\log(r))\nonumber\\
        &+4(m^2+n^2-1)Y'(\log(r))\bigg)
        =\frac{\lambda-\left((m^2-n^2-1)^2-4n^2\right)}{r^4}Y(\log(r)).
    \end{align}
    If $r=e^t$, then the equation is equivalent to
    \begin{align}\label{equa_diff_m_n0}
        &Y''''(t)-4\,Y'''(t)-2(m^2+n^2-3)Y''(t)+4(m^2+n^2-1)Y'(t)\nonumber\\
        &=\left(\lambda-\left((m^2-n^2-1)^2-4n^2\right)\right)Y(t).
    \end{align}
    Its characteristic polynomial is given by
    \begin{align*}
        P(X)=X^4-4X^3-2(m^2+n^2-3)X^2+4(m^2+n^2-1)X+(m^2-n^2-1)^2-4n^2-\lambda,
    \end{align*}
    and the roots of this polynomial are given by 
    \begin{align}\label{racines_m_n}
    \left\{\begin{alignedat}{1}
        r_1&=1+\sqrt{m^2+n^2+\sqrt{\lambda+4m^2n^2}}\\
        r_2&=1-\sqrt{m^2+n^2-\sqrt{\lambda+4m^2n^2}}\\
        r_3&=1+\sqrt{m^2+n^2+\sqrt{\lambda+4m^2n^2}}\\
        r_4&=1-\sqrt{m^2+n^2-\sqrt{\lambda+4m^2n^2}}.
        \end{alignedat}\right.
    \end{align}
    Indeed, if we make the substitution $X=Y+1$, we obtain
    \begin{align}\label{change_var_polynomial}
        &Q(Y)=P(Y+1)=Y^4+4Y^3+6Y^2+4Y+1
        -4\left(Y^3+3Y^2+3Y+1\right)\nonumber\\
        &-2(m^2+n^2-3)\left(Y^2+2Y+1\right)
        +4(m^2+n^2-1)(Y+1)+(m^2-n^2-1)^2-4n^2-\lambda\nonumber\\
        &=Y^4-2(m^2+n^2)Y^2+1+2(m^2+n^2-1)+(m^2-n^2-1)^2-4n^2-\lambda\nonumber\\
        &=Y^4-2(m^2+n^2)Y^2+(m^2-n^2)^2-\lambda
    \end{align}
    which is a biquadratic equation!
    Since $(m^2-n^2)^2=(m^2+n^2)^2-4m^2n^2$, we deduce that the discriminant of the polynomial $R(Z)=Q(\sqrt{Z})$ is given by
    \begin{align*}
        D_{m,n}=4(m^2+n^2)^2-4\left((m^2+n^2)^2-\lambda-4m^2n^2\right)=4(\lambda+4m^2n^2),
    \end{align*}
    which shows that the roots of $R$ are given by
    \begin{align*}
        s_{\pm}=m^2+n^2\pm\sqrt{\lambda+4m^2n^2},
    \end{align*}
    and we obtain the formula \eqref{racines_m_n} for the roots of $P$ (using that $Z=Y^2$ and $X=Y+1$).

    \subsection{Proof of the Main Results}
    
    \begin{theorem}\label{theoreme_ode_m_n}
        Let $m>0$, and $n\in\Z$. Let $0<a<b<\infty$ be such that
        \begin{align}\label{log_bound1}
            \log\left(\frac{b}{a}\right)>\frac{5}{\sqrt{2m^2+2n^2}}.
        \end{align}        
        For all $\lambda>0$, consider the following linear differential equation
        \begin{align}\label{equa_diff_m_n}
        &Y''''(t)-4\,Y'''(t)-2(m^2+n^2-3)Y''(t)+4(m^2+n^2-1)Y'(t)\nonumber\\
        &=\left(\lambda-\left((m^2-n^2-1)^2-4n^2\right)\right)Y(t)=0.
        \end{align}
        with boundary conditions
        \begin{align}\label{boundary_conditions_m_n}
            \left\{\begin{alignedat}{2}
                Y(\log(a))&=Y(\log(b))&&=0\\
                Y'(\log(a))&=Y'(\log(b))&&=0
        \end{alignedat}\right.
        \end{align}
        There exists $\lambda>0$ such that the system \eqref{equa_diff_m_n}-\eqref{boundary_conditions_m_n} admits a non-trivial solution $Y$. Furthermore, the minimal value $\lambda_{m,n}>0$ satisfies the following estimate 
        \small
        \begin{align}\label{estimate_log_lambda_m_n}
            \left((m-n)^2+\frac{\pi^2}{\log^2\left(\frac{b}{a}\right)}\right)\left((m+n)^2+\frac{\pi^2}{\log^2\left(\frac{b}{a}\right)}\right)<\lambda_{m,n}<\left((m-n)^2+\frac{4\pi^2}{\log^2\left(\frac{b}{a}\right)}\right)\left((m+n)^2+\frac{4\pi^2}{\log^2\left(\frac{b}{a}\right)}\right).
        \end{align}
        \normalsize
        Assuming that $m\in \N\setminus\ens{0}$ and
        \begin{align}
            \log\left(\frac{b}{a}\right)\geq \frac{\pi\sqrt{2}}{\sqrt{2m-1}},
        \end{align}
        we have
        \begin{align}\label{bound_all_eigenvalues}
            \inf_{n\in\Z}\lambda_{m,n}>\left(4m^2+\frac{\pi^2}{\log^2\left(\frac{b}{a}\right)}\right)\frac{\pi^2}{\log^2\left(\frac{b}{a}\right)}.
        \end{align}
    \end{theorem}
    \begin{proof}
        Notice that
        \begin{align*}
            m^2+n^2-\sqrt{\lambda+4m^2n^2}\leq 0
        \end{align*}
        if and only if
        \begin{align*}
            m^4+2m^2n^2+n^4\leq \lambda+4m^2n^2,
        \end{align*}
        or
        \begin{align*}
            \lambda \geq m^4-2m^2n^2+n^4=(m^2-n^2)^2.
        \end{align*}
        Therefore, we need to distinguish three cases.
        
        \textbf{Case 1:} $\lambda<(m^2-n^2)^2$.

        Then, $P$ has four distinct real roots, and the solution $Y$ is given by
    \begin{align*}
        Y(t)=\mu_1e^{r_1t}+\mu_2e^{r_2t}+\mu_3e^{r_3t}+\mu_4e^{r_4t}.
    \end{align*}
    Since
    \begin{align*}
        Y'(t)=r_1\,\mu_1e^{r_1t}+r_2\,\mu_2e^{r_2t}+r_3\,\mu_3e^{r_3t}+r_4\,\mu_4e^{r_2t},
    \end{align*}
    the boundary conditions \eqref{boundary_conditions_m_n} show that the following linear system holds true:
    \begin{align*}
        A\,\mu=\begin{pmatrix}
            a^{r_1} & a^{r_2} & a^{r_3} & a^{r_4}\\
            b^{r_1} & b^{r_2} & b^{r_3} & b^{r_4}\\
            r_1\,a^{r_1} & r_2\,a^{r_2} & r_3\,a^{r_3} & r_4\,a^{r_4}\\
            r_1\,b^{r_2} & r_2\,b^{r_2} & r_3\,b^{r_3} & r_4\,b^{r_4}
        \end{pmatrix}\begin{pmatrix}
            \mu_1\\
            \mu_2\\
            \mu_3\\
            \mu_4
        \end{pmatrix}=0. 
    \end{align*}
    Now write for simplicity
    \begin{align*}
    \left\{\begin{alignedat}{1}
        &r_1=1+\lambda_1\\
        &r_2=1-\lambda_1\\
        &r_3=1+\lambda_2\\
        &r_4=1-\lambda_2,
        \end{alignedat}\right.
    \end{align*}
    where $\lambda_1=\sqrt{m^2+n^2+\sqrt{\lambda+4m^2n^2}}$, and $\lambda_2=\sqrt{m^2+n^2-\sqrt{\lambda+4m^2n^2}}$.

    In order to make computations more readable, define for $X,Y\in \C\setminus\ens{0}$ and $\alpha,\beta,\gamma,\delta\in \C$
    \begin{align*}
        B=\begin{pmatrix}
            X^{\alpha} & X^{\beta} & X^{\gamma} & X^{\delta}\\
            Y^{\alpha} & Y^{\beta} & Y^{\gamma} & Y^{\delta}\\
            \alpha\,X^{\alpha} & \beta\,X^{\beta} & \gamma\,X^{\gamma} & \delta\,X^{\delta}\\
            \alpha\,Y^{\alpha} & \beta\,Y^{\beta} & \gamma\,Y^{\gamma} & \delta\,Y^{\delta}
        \end{pmatrix}.
    \end{align*}
    We have
    \begin{align*}
        \det(B)&=X^{\alpha}\begin{vmatrix}
            Y^{\beta} & Y^{\gamma} & Y^{\delta}\\
            \beta\,X^{\beta} & \gamma\,X^{\gamma} & \delta\,X^{\delta}\\ 
            \beta\,Y^{\beta} & \gamma\,Y^{\gamma} & \delta\,Y^{\delta}
            \end{vmatrix}
            -X^{\beta}\begin{vmatrix}
            Y^{\alpha} & Y^{\gamma} & Y^{\delta}\\
            \alpha\,X^{\alpha} & \gamma\,X^{\gamma} & \delta\,X^{\delta}\\ 
            \alpha\,Y^{\alpha} & \gamma\,Y^{\gamma} & \delta\,Y^{\delta}
            \end{vmatrix}\\
            &+X^{\gamma}\begin{vmatrix}
            Y^{\alpha} & Y^{\beta} & Y^{\delta}\\
            \alpha\,X^{\alpha} & \beta\,X^{\beta} & \delta\,X^{\delta}\\ 
            \alpha\,Y^{\alpha} & \beta\,Y^{\beta} & \delta\,Y^{\delta}
            \end{vmatrix}
            -X^{\delta}\begin{vmatrix}
            Y^{\alpha} & Y^{\beta} & Y^{\gamma}\\
            \alpha\,X^{\alpha} & \beta\,X^{\beta} & \gamma\,X^{\gamma}\\ 
            \alpha\,Y^{\alpha} & \beta\,Y^{\beta} & \gamma\,Y^{\gamma}
            \end{vmatrix}\\
            &=X^{\alpha}\Big(\beta(\gamma-\delta)X^{\beta}Y^{\gamma+\delta}+\gamma(\delta-\beta)X^{\gamma}Y^{\beta+\delta}+\delta(\beta-\gamma)X^{\delta}Y^{\beta+\gamma}\Big)\\
            &-X^{\beta}\Big(\alpha(\gamma-\delta)X^{\alpha}Y^{\gamma+\delta}+\gamma(\delta-\alpha)X^{\gamma}Y^{\alpha+\delta}+\delta(\alpha-\gamma)X^{\delta}Y^{\alpha+\gamma}\Big)\\
            &+X^{\gamma}\Big(\alpha(\beta-\delta)X^{\alpha}Y^{\beta+\delta}+\beta(\delta-\alpha)X^{\beta}Y^{\alpha+\delta}+\delta(\alpha-\beta)X^{\delta}Y^{\alpha+\gamma}\Big)\\
            &-X^{\delta}\Big(\alpha(\beta-\gamma)X^{\alpha}Y^{\beta+\gamma}+\beta(\gamma-\alpha)X^{\beta}Y^{\alpha+\gamma}+\gamma(\alpha-\beta)X^{\gamma}Y^{\alpha+\beta}\Big)\\
            &=(\beta-\alpha)(\gamma-\delta)X^{\alpha+\beta}Y^{\gamma+\delta}+(\gamma-\alpha)(\delta-\beta)X^{\alpha+\gamma}Y^{\beta+\delta}+(\delta-\alpha)(\beta-\gamma)X^{\alpha+\delta}Y^{\beta+\gamma}\\
            &+(\beta-\gamma)(\delta-\alpha)X^{\beta+\gamma}Y^{\alpha+\delta}+(\beta-\delta)(\alpha-\gamma)X^{\beta+\delta}Y^{\alpha+\gamma}+(\delta-\gamma)(\alpha-\beta)X^{\gamma+\delta}Y^{\alpha+\beta}.
    \end{align*}
    Now, using
    \begin{align*}
        \left\{\begin{alignedat}{1}
            \alpha=1+\lambda_1\\
            \beta=1-\lambda_1\\
            \gamma=1+\lambda_2\\
            \delta=1-\lambda_2
        \end{alignedat}\right.
    \end{align*}
    the expression becomes
            \small
    \begin{align*}
        &\det(B)=(-2\lambda_1)(2\lambda_2)X^2Y^2+(\lambda_2-\lambda_1)(-\lambda_2+\lambda_1)X^{2+\lambda_1+\lambda_2}Y^{2-\lambda_1-\lambda_2}\\
        &+(-\lambda_2-\lambda_1)(-\lambda_1-\lambda_2)X^{2+\lambda_1-\lambda_2}Y^{2-\lambda_1+\lambda_2}
        +(-\lambda_1-\lambda_2)(-\lambda_2-\lambda_1)X^{2-\lambda_1+\lambda_2}Y^{2+\lambda_1-\lambda_2}\\
        &+(-\lambda_1+\lambda_2)(\lambda_1-\lambda_2)X^{2-\lambda_1-\lambda_2}Y^{2+\lambda_1+\lambda_2}+(-2\lambda_2)(2\lambda_1)X^2Y^2\\
        &=-(\lambda_1-\lambda_2)^2\left(X^{2+\lambda_1+\lambda_2}Y^{2-\lambda_1-\lambda_2}+X^{2-\lambda_1-\lambda_2}Y^{2+\lambda_1+\lambda_2}\right)\\
        &+(\lambda_1+\lambda_2)^2\left(X^{2+\lambda_1-\lambda_2}Y^{2-\lambda_1+\lambda_2}+X^{2-\lambda_1+\lambda_2}\right)-8\lambda_1\lambda_2X^2Y^2\\
        &=X^{2}Y^2\left((\lambda_1+\lambda_2)^2\left(\left(\frac{X}{Y}\right)^{\lambda_1-\lambda_2}+\left(\frac{Y}{X}\right)^{\lambda_1-\lambda_2}\right)-\left(\lambda_1-\lambda_2\right)^2\left(\left(\frac{X}{Y}\right)^{\lambda_1+\lambda_2}+\left(\frac{Y}{X}\right)^{\lambda_1+\lambda_2}\right)-8\lambda_1\lambda_2\right).
    \end{align*}
    \normalsize
    Therefore, we have
    \small
    \begin{align}\label{eq_det}
        \det(A)=a^2b^2\left((\lambda_1+\lambda_2)^2\left(\left(\frac{a}{b}\right)^{\lambda_1-\lambda_2}+\left(\frac{b}{a}\right)^{\lambda_1-\lambda_2}\right)-\left(\lambda_1-\lambda_2\right)^2\left(\left(\frac{a}{b}\right)^{\lambda_1+\lambda_2}+\left(\frac{b}{a}\right)^{\lambda_1+\lambda_2}\right)-8\lambda_1\lambda_2\right).
    \end{align}
    \normalsize
    This leads us to study (for all $\lambda_1>\lambda_2>0$) the function
    \begin{align*}
        g(x)=\left(\lambda_1+\lambda_2\right)^2\left(x^{\lambda_1-\lambda_2}+\frac{1}{x^{\lambda_1-\lambda_2}}\right)-\left(\lambda_1-\lambda_2\right)^2\left(x^{\lambda_1+\lambda_2}+\frac{1}{x^{\lambda_1+\lambda_2}}\right)-8\lambda_1\lambda_2.
    \end{align*}
    Notice that 
    \begin{align*}
        g(1)=2(\lambda_1+\lambda_2)^2-2(\lambda_1-\lambda_2)^2-8\lambda_1\lambda_2=0.
    \end{align*}
    We will show that $g(x)<0$ for all $x>1$ (or equivalently $0<x<1$ by symmetry). In fact, we find it make convenient to work with
    \begin{align*}
        f(x)=x^{\lambda_1+\lambda_2}g(x)=\left(\lambda_1+\lambda_2\right)^2\left(x^{2\lambda_1}+x^{2\lambda_2}\right)-(\lambda_1-\lambda_2)^2\left(x^{2\lambda_1+2\lambda_2}+1\right)-8\lambda_1\lambda_2x^{\lambda_1+\lambda_2}
    \end{align*}
    For notational convenience, write $\alpha=\lambda_1$ and $\beta=\lambda_2$, so that
    \begin{align*}
        f(x)=(\alpha+\beta)^2(x^{2\alpha}+x^{2\beta})-(\alpha-\beta)^2(x^{2\alpha+2\beta}+1)-8\alpha\beta\, x^{\alpha+\beta}.
    \end{align*}
    We have
    \begin{align*}
        xf'(x)&=(\alpha+\beta)^2(2\alpha\, x^{2\alpha}+2\beta\,x^{2\beta})-2(\alpha+\beta)(\alpha-\beta)^2x^{2\alpha+2\beta}-8\alpha\beta(\alpha+\beta)x^{\alpha+\beta}\\
        x\frac{d}{dx}\left(xf'(x)\right)&=(\alpha+\beta)^2(4\alpha^2x^{2\alpha}+4\beta^2x^{2\beta})-4(\alpha+\beta)^2(\alpha-\beta)^2x^{2\alpha+2\beta}-8\alpha\beta(\alpha+\beta)^2x^{\alpha+\beta}\\
        &=4(\alpha+\beta)^2\left(\alpha^2x^{2\alpha}+\beta^2x^{2\beta}-2\alpha\beta\, x^{\alpha+\beta}-(\alpha-\beta)^2x^{2\alpha+2\beta}\right)\\
        &=4(\alpha+\beta)^2\left((\alpha\, x^{\alpha}-\beta\, x^{\beta})^2-(\alpha-\beta)^2x^{2\beta+2\beta}\right)\\
        &=4(\alpha+\beta)^2\left(\alpha\, x^{\alpha}-\beta\, x^{\beta}+(\alpha-\beta)x^{\alpha+\beta}\right)\left(\alpha\, x^{\alpha}-\beta\, x^{\beta}-(\alpha-\beta)x^{\alpha+\beta}\right)\\
        &=-4(\alpha+\beta)^2x^{\beta}(\alpha\, x^{\alpha}-\beta\, x^{\beta}+(\alpha-\beta)x^{\alpha+\beta})((\alpha-\beta)x^{\alpha}-\alpha\, x^{\alpha-\beta}+\beta).
    \end{align*}
    We are led to study the function
    \begin{align*}
        h(x)=(\alpha-\beta)x^{\alpha}-\alpha\, x^{\alpha-\beta}+\beta.
    \end{align*}
    We have
    \begin{align}\label{var_lemma1}
        x\, h'(x)=\alpha(\alpha-\beta)x^{\alpha}-\alpha(\alpha-\beta)x^{\alpha-\beta}>0
    \end{align}
    for all $x>1$. Therefore, $h$ is a strictly increasing function on $(1,\infty)$, and since $h(1)=0$, we deduce that $h(x)>0$ for all $x>1$, which implies that 
    \begin{align*}
        x\dfrac{d}{dx}(xf'(x))<0\qquad\text{for all}\;\,x>1.
    \end{align*}
    Therefore, $x\mapsto xf'(x)$ is strictly decreasing on $(1,\infty)$, and since 
    \begin{align*}
        f'(1)=2(\alpha+\beta)^3-2(\alpha+\beta)(\alpha-\beta)^2-8\alpha\beta(\alpha+\beta)=2(\alpha+\beta)\left((\alpha+\beta)^2-(\alpha-\beta)^2-4\alpha\beta\right)=0,
    \end{align*}
    we deduce that $f'(x)<0$ for all $x>1$, which finally implies that $f$ is a strictly decreasing function on $(1,\infty)$. Therefore, we have
    \begin{align*}
        f(x)<0\qquad\text{for all}\;\,x>1.
    \end{align*}
    Coming back to \eqref{eq_det}, we deduce that $\det(A)<0$, which implies that $Y=0$.

        \textbf{Case 2:} $\lambda=(m^2-n^2)^2$.

        Then, the roots are given by $r_3=r_4=1$, and $r_1=1+\sqrt{2m^2+2n^2}=1+\lambda_1$, $r_2=1-\sqrt{2m^2+2n^2}=1-\lambda_1$. Therefore, the solution $Y$ is given by
        \begin{align*}
            Y(t)=\mu_1e^{(1+\lambda_1)t}+\mu_2e^{(1-\lambda_1)t}+\mu_3e^t+\mu_4\,t\,e^t.
        \end{align*}
        The boundary conditions \eqref{boundary_conditions_m_n} are equivalent to
        \begin{align}
            A\begin{pmatrix}
                \mu_1\\
                \mu_2\\
                \mu_3\\
                \mu_4
            \end{pmatrix}=\begin{pmatrix}
                a^{1+\lambda_1} & a^{1-\lambda_1} & a & a\log(a)\\
                b^{1+\lambda_1} & b^{1-\lambda_1} & b & b\log(b)\\
                (1+\lambda_1)a^{1+\lambda_1} & (1-\lambda_1)a^{1-\lambda_1} & a & (a+1)\log(a)\\
                (1+\lambda_1)b^{1+\lambda_1} & (1-\lambda_1)b^{1-\lambda_1} & b & (b+1)\log(b)
            \end{pmatrix}\begin{pmatrix}
                \mu_1\\
                \mu_2\\
                \mu_3\\
                \mu_4
            \end{pmatrix}=0.
        \end{align}
        
        Let $X,Y\in \C\setminus\ens{0}$, and $\alpha,\beta,\gamma\in \C$, and define the following matrix
        \begin{align*}
            B=\begin{pmatrix}
                X^{\alpha} & X^{\beta} & X^{\gamma} & \log(X)X^{\gamma}\\
                Y^{\alpha} & Y^{\beta} & Y^{\gamma} & \log(Y)Y^{\gamma}\\
                \alpha X^{\alpha} & \beta X^{\beta} & \gamma X^{\gamma} & (1+\gamma\log(X))X^{\gamma}\\
                \alpha Y^{\alpha} & \beta Y^{\beta} & \gamma Y^{\gamma} & (1+\gamma\log(Y))Y^{\gamma}
            \end{pmatrix}.
        \end{align*}
        Expanding on the fourth column, we get
        \begin{align}\label{determinant_equality_case1}
            &\det(B)=-\log(X)X^{\gamma}\begin{vmatrix}
                Y^{\alpha} & Y^{\beta} & Y^{\gamma}\\
                \alpha X^{\alpha} & \beta X^{\beta} & \gamma X^{\gamma}\nonumber\\
                \alpha Y^{\alpha} & \beta Y^{\beta} & \gamma Y^{\gamma}
            \end{vmatrix}+\log(Y)Y^{\gamma}\begin{vmatrix}
                X^{\alpha} & X^{\beta} & X^{\gamma}\\
                \alpha X^{\alpha} & \beta X^{\beta} & \gamma X^{\gamma}\nonumber\\
                \alpha Y^{\alpha} & \beta Y^{\beta} & \gamma Y^{\gamma}
            \end{vmatrix}\\
            &-\left(1+\gamma\log(X)\right)X^{\gamma}\begin{vmatrix}
                X^{\alpha} & X^{\beta} & X^{\gamma}\\
                Y^{\alpha} & Y^{\beta} & Y^{\gamma}\\
                \alpha Y^{\alpha} & \beta Y^{\beta} & \gamma Y^{\gamma}
            \end{vmatrix}+\left(1+\gamma\log(Y)\right)Y^{\gamma}
            \begin{vmatrix}
                X^{\alpha} & X^{\beta} & X^{\gamma}\\
                Y^{\alpha} & Y^{\beta} & Y^{\gamma}\\
                \alpha X^{\alpha} & \beta X^{\beta} & \gamma X^{\gamma}
            \end{vmatrix}\nonumber\\
            &=-\log(X)\left(\beta(\gamma-\alpha)X^{\beta+\gamma}Y^{\alpha+\gamma}+\colorcancel{\gamma(\alpha-\beta)X^{2\gamma}Y^{\alpha+\beta}}{red}+\alpha(\beta-\gamma)X^{\alpha+\gamma}Y^{\beta+\gamma}\right)\nonumber\\
            &-\log(X)\left(\gamma(\gamma-\beta)X^{\alpha+\gamma}Y^{\beta+\gamma}+\gamma(\alpha-\gamma)X^{\beta+\gamma}Y^{\alpha+\gamma}+\colorcancel{\gamma(\beta-\alpha)X^{2\gamma}Y^{\alpha+\beta}}{red}\right)\nonumber\\
            &+\log(Y)\left(\colorcancel{\gamma(\beta-\alpha)X^{\alpha+\beta}Y^{2\gamma}}{blue}+\alpha(\gamma-\beta)X^{\beta+\gamma}Y^{\alpha+\gamma}+\beta(\alpha-\gamma)X^{\alpha+\gamma}Y^{\beta+\gamma}\right)\nonumber\\
            &+\log(Y)\left(\gamma(\gamma-\alpha)X^{\alpha+\gamma}Y^{\beta+\gamma}+\colorcancel{\gamma(\alpha-\beta)X^{\alpha+\beta}Y^{2\gamma}}{blue}+\gamma(\beta-\gamma)X^{\beta+\gamma}Y^{\alpha+\gamma}\right)\nonumber\\
            &-\left((\gamma-\beta)X^{\alpha+\gamma}Y^{\beta+\gamma}+(\alpha-\gamma)X^{\beta+\gamma}Y^{\alpha+\gamma}+(\beta-\alpha)X^{2\gamma}Y^{\alpha+\beta}\right)\nonumber\\
            &+(\gamma-\alpha)X^{\alpha+\gamma}Y^{\beta+\gamma}+(\alpha-\beta)X^{\alpha+\beta}Y^{2\gamma}+(\beta-\gamma)X^{\beta+\gamma}Y^{\alpha+\gamma}\nonumber\\
            &=-\log(X)\Big((\beta-\gamma)(\gamma-\alpha)X^{\beta+\gamma}Y^{\beta+\gamma}+(\alpha-\gamma)(\beta-\gamma)X^{\alpha+\gamma}Y^{\beta+\gamma}\Big)\nonumber\\
            &+\log(Y)\Big((\alpha-\gamma)(\gamma-\beta)X^{\beta+\gamma}Y^{\alpha+\gamma}+(\beta-\gamma)(\alpha-\gamma)X^{\alpha+\gamma}Y^{\beta+\gamma}\Big)\nonumber\\
            &+(\beta-\alpha)\Big(X^{\alpha+\gamma}Y^{\beta+\gamma}+X^{\beta+\gamma}Y^{\alpha+\gamma}-\left(X^{2\gamma}Y^{\alpha+\beta}+X^{\alpha+\beta}Y^{2\gamma}\right)\Big)\nonumber\\
            &=(\alpha-\gamma)(\beta-\gamma)\left(X^{\alpha+\gamma}Y^{\beta+\gamma}-X^{\beta+\gamma}Y^{\alpha+\gamma}\right)\log\left(\frac{Y}{X}\right)\nonumber\\
            &+(\alpha-\beta)\Big(X^{2\gamma}Y^{\alpha+\beta}+X^{\alpha+\beta}Y^{2\gamma}-\big(X^{\alpha+\gamma}Y^{\beta+\gamma}+X^{\beta+\gamma}Y^{\alpha+\gamma}\big)\Big).
        \end{align}
        Taking now
        \begin{align}\label{determinant_equality_case2}
        \left\{\begin{alignedat}{1}
            &\alpha=1+\sqrt{2m^2+2n^2}=1+\lambda_1\\
            &\beta=1-\sqrt{2m^2+2n^2}=1-\lambda_1\\
            &\gamma=1
            \end{alignedat}\right.
        \end{align}
        we get (if $X=a$ and $Y=b$)
        \begin{align}\label{determinant_equality_case4}
            \det(A)&=-\lambda_1^2\left(a^{2+\lambda_1}b^{2-\lambda_1}-a^{2-\lambda_1}b^{2+\lambda_1}\right)\log\left(\frac{b}{a}\right)+2\lambda_1\left(a^2b^2+a^2b^2-\left(a^{2+\lambda_1}b^{2-\lambda_1}+a^{2-\lambda_1}b^{2+\lambda_1}\right)\right)\nonumber\\
            &=\lambda_1(ab)^2\left(4-2\left(\left(\frac{b}{a}\right)^{\lambda_1}+\bigg(\frac{a}{b}\bigg)^{\lambda_1}\right)+\lambda_1\left(\left(\frac{b}{a}\right)^{\lambda_1}-\bigg(\frac{a}{b}\bigg)^{\lambda_1}\right)\log\left(\frac{b}{a}\right)\right).
        \end{align}
        We are therefore led to consider on $\R_+$ the function 
        \begin{align*}
            g(x)=2\left((1+x)^{\beta}+\frac{1}{(1+x)^{\beta}}\right)-\beta\left((1+x)^{\beta}-\frac{1}{(1+x)^{\beta}}\right)\log(1+x)-4
        \end{align*}
        where $\beta=\sqrt{2m^2+2n^2}\geq \sqrt{2}$. Notice that $f(0)=0$. Then, we compute
        \begin{align*}
            (1+x)g'(x)
            &=\beta\left((1+x)^{\beta}-\frac{1}{(1+x)^{\beta}}\right)-\beta^2\left((1+x)^{\beta}+\frac{1}{(1+x)^{\beta}}\right)\log(1+x)\\
            &=\frac{\beta}{2(1+x)^{\beta}}\left(2((1+x)^{2\beta}-1)-2\beta\left((1+x)^{2\beta}+1\right)\log(1+x)\right).
        \end{align*}
        Letting $\alpha=2\beta$, we are now led to study the function
        Let $\alpha\geq \sqrt{2}$ and consider
        \begin{align*}
            f(x)&=2\left((1+x)^{\alpha}-1\right)-\alpha\left((1+x)^{\alpha}+1\right)\log(1+x).
        \end{align*}
        We have
        \begin{align*}
            f'(x)&=2\alpha(1+x)^{\alpha-1}-\alpha^2(1+x)^{\alpha-1}\log(1+x)-\alpha(1+x)^{\alpha-1}\\
            &=\alpha(1+x)^{\alpha-1}(1-\alpha\log(1+x))-\frac{\alpha}{1+x}=-\alpha(1+x)^{\alpha-1}\left(\frac{1}{(1+x)^{\alpha}}+\alpha\log(1+x)-1\right).
        \end{align*}
        Now, let
        \begin{align*}
            h(x)=\frac{1}{(1+x)^{\alpha}}+\alpha\log(1+x)-1.
        \end{align*}
        We have
        \begin{align*}
            h'(x)=-\frac{\alpha}{(1+x)^{\alpha+1}}+\frac{\alpha}{1+x}=\frac{\alpha}{(1+x)^{\alpha+1}}\left((1+x)^{\alpha}-1\right)\geq 0.
        \end{align*}
        Therefore, $h$ is an increasing function, which shows that
        \begin{align*}
            h(x)>h(0)=0
        \end{align*}
        for all $x> 0$. In particular, $f'(x)\leq 0$, which shows that $f$ is decreasing and $f(x)\leq f(0)=0$ (and even $f(x)<0$ for all $x>0$). Therefore, we deduce that $g'(x)<0$ for all $x>0$, which finally implies that $\det(A)<0$. As a consequence, our solution $Y$ vanishes identically.
        \textbf{Case 3:} $\lambda>(m^2-n^2)^2$.

         Then we have 
    \begin{align*}
    \left\{\begin{alignedat}{1}
        r_1&=1+\sqrt{\sqrt{\lambda+4m^2n^2}+m^2+n^2}\\
        r_2&=1-\sqrt{\sqrt{\lambda+4m^2n^2}+m^2+n^2}\\
        r_3&=1+i\sqrt{\sqrt{\lambda+4m^2n^2}-(m^2+n^2)}\\
        r_4&=1-i\sqrt{\sqrt{\lambda+4m^2n^2}-(m^2+n^2)},
        \end{alignedat}\right.
    \end{align*}
    and if $\lambda_1=\sqrt{\sqrt{\lambda+4m^2n^2}+m^2+n^2}$ and $\lambda_2=\sqrt{\sqrt{\lambda+4m^2n^2}-(m^2+n^2)}$, there exists  $\mu_1,\mu_2\in \R$ and $\mu_3,\mu_4\in \C$ such that
    \begin{align}\label{sol}
        Y(t)=\mu_1e^{(1+\lambda_1)t}+\mu_2e^{(1-\lambda_1)t}+\mu_3e^{(1+i\,\lambda_2)t}+\mu_4e^{(1-i\,\lambda_2)t}
    \end{align}
    we get thanks to \eqref{boundary_conditions_m_n} the system
    \begin{align}\label{complex_system}
        \begin{pmatrix}
            a^{1+\lambda_1} & a^{1-\lambda_1} & a^{1+i\,\lambda_2} & a^{1-i\,\lambda_2}\\
            b^{1+\lambda_1} & b^{1-\lambda_1} & b^{1+i\,\lambda_2} & b^{1-i\,\lambda_2}\\
            (1+\lambda_1)a^{1+\lambda_1} & (1-\lambda_1)a^{1-\lambda_1} & (1+i\,\lambda_2)a^{1+i\,\lambda_2} & (1-i\,\lambda_2)a^{1-i\,\lambda_2}\\
            (1+\lambda_1)b^{1+\lambda_1} & (1-\lambda_1)b^{1-\lambda_1} & (1+i\,\lambda_2)b^{1+i\,\lambda_2} & (1-i\,\lambda_2)b^{1-i\,\lambda_2}
        \end{pmatrix}\begin{pmatrix}
            \mu_1\\
            \mu_2\\
            \mu_3\\
            \mu_4
        \end{pmatrix}=0.
    \end{align}
    The underlying matrix $A$ falls is a specific case of the matrix studied in \textbf{Case 1}. Indeed, we have
    \begin{align*}
    \begin{vmatrix}
            X^{\alpha} & X^{\beta} & X^{\gamma} & X^{\delta}\\
            Y^{\alpha} & Y^{\beta} & Y^{\gamma} & Y^{\delta}\\
            \alpha\,X^{\alpha} & \beta\,X^{\beta} & \gamma\,X^{\gamma} & \delta\,X^{\delta}\\
            \alpha\,Y^{\alpha} & \beta\,Y^{\beta} & \gamma\,Y^{\gamma} & \delta\,Y^{\delta}
        \end{vmatrix}&=
        (\beta-\alpha)(\gamma-\delta)X^{\alpha+\beta}Y^{\gamma+\delta}+(\gamma-\alpha)(\delta-\beta)X^{\alpha+\gamma}Y^{\beta+\delta}\\
        &+(\delta-\alpha)(\beta-\gamma)X^{\alpha+\delta}Y^{\beta+\gamma}
        +(\beta-\gamma)(\delta-\alpha)X^{\beta+\gamma}Y^{\alpha+\delta}\\
        &+(\beta-\delta)(\alpha-\gamma)X^{\beta+\delta}Y^{\alpha+\gamma}+(\delta-\gamma)(\alpha-\beta)X^{\gamma+\delta}Y^{\alpha+\beta}.
    \end{align*}
    Letting
    \begin{align*}
        \left\{\begin{alignedat}{1}
            \alpha&=1+\lambda_1\\
            \beta&=1-\lambda_1\\
            \gamma&=1+i\,\lambda_2\\
            \delta&=1-i\,\lambda_2
        \end{alignedat}\right.
    \end{align*}
    and $(X,Y)=(a,b)$, we get
    \begin{align*}
        &\det(A)=-2\lambda_1(2\,i\,\lambda_2)X^{2}Y^2+(i\,\lambda_2-\lambda_1)(-i\,\lambda_2+\lambda_1)X^{2+\lambda_1+i\,\lambda_2}Y^{2-\lambda_1-i\,\lambda_2}\\
        &+(-i\,\lambda_2-\lambda_1)(-\lambda_1-i\,\lambda_2)X^{2+\lambda_1-i\,\lambda_2}Y^{2-\lambda_1+i\,\lambda_2}+(-\lambda_1-i\,\lambda_2)(-i\,\lambda_2-\lambda_1)X^{2-\lambda_1+i\,\lambda_2}Y^{2+\lambda_1-i\,\lambda_2}\\
        &+(-\lambda_1+i\,\lambda_2)(\lambda_1-i\,\lambda_2)X^{2-\lambda_1-i\,\lambda_2}Y^{2+\lambda_1+i\,\lambda_2}+(-2i\,\lambda_2)(2\lambda_1)X^2Y^2\\
        &=-8\,i\,\lambda_1\lambda_2X^2Y^2-(\lambda_1-i\,\lambda_2)^2\left(X^{2+\lambda_1+i\,\lambda_2}Y^{2-\lambda_1-i\,\lambda_2}+X^{2-\lambda_1-i\,\lambda_2}Y^{2+\lambda_1+i\,\lambda_2}\right)\\
        &+(\lambda_1+i\,\lambda_2)^2\left(X^{2+\lambda_1-i\,\lambda_2}Y^{2-\lambda_1+i\,\lambda_2}+X^{2-\lambda_1+i\,\lambda_2}Y^{2+\lambda_1-i\,\lambda_2}\right),
    \end{align*}
    and finally, if $\zeta=\lambda_1+i\,\lambda_2$
    \begin{align}\label{eq_det_complex}
        \det(A)=a^{2}b^{2}\left(\zeta^2\left(\bigg(\frac{a}{b}\bigg)^{\bar{\zeta}}+\left(\frac{b}{a}\right)^{\bar{\zeta}}\right)-\bar{\zeta}^2\left(\bigg(\frac{a}{b}\bigg)^{\zeta}+\left(\frac{b}{a}\right)^{\zeta}\right)-8\,i\,\Re(\zeta)\Im(\zeta)\right).
    \end{align}   
    We could have algebraically replaced $\lambda_2$ by $i\,\lambda_2$ in \eqref{eq_det}. Now, it is apparent that
    \begin{align*}
        \det(A)=2\,i\,a^{2}b^2\left(\Im\left(\zeta^2\left(\bigg(\frac{a}{b}\bigg)^{\bar{\zeta}}+\left(\frac{b}{a}\right)^{\bar{\zeta}}\right)\right)-4\,\Re(\zeta)\Im(\zeta)\right).
    \end{align*}
    Now, we have for all $t>0$
    \begin{align*}
        \Im\left(\zeta^2t^{\bar{\zeta}}\right)&=\Im\left((\lambda_1^2-\lambda_2^2+2\,i\,\lambda_1\lambda_2)t^{\lambda_1}\left(\cos(\lambda_2\log(t))-i\,\sin(\lambda_2\log(t))\right)\right)\\
        &=-(\lambda_1^2-\lambda_2^2)t^{\lambda_1}\sin(\lambda_2\log(t))+2\,\lambda_1\lambda_2t^{\lambda_1}\cos\left(\lambda_2\log(t)\right).
    \end{align*}
    Therefore, we deduce that $\det(A)=0$ if and only if
    \begin{align*}
        &-(\lambda_1^2-\lambda_2^2)\left(\frac{a}{b}\right)^{\lambda_1}\sin\left(\lambda_2\log\bigg(\frac{a}{b}\bigg)\right)+2\lambda_1\lambda_2\bigg(\frac{a}{b}\bigg)^{\lambda_1}\cos\left(\lambda_2\log\left(\frac{a
        }{b}\right)\right)\\
        &-(\lambda_1^2-\lambda_2^2)\left(\frac{b}{a}\right)^{\lambda_1}\sin\left(\lambda_2\log\left(\frac{b}{a}\right)\right)+2\lambda_1\lambda_2\left(\frac{b}{a}\right)^{\lambda_1}\cos\left(\log\left(\frac{b}{a}\right)\right)-4\lambda_1\lambda_2=0,
    \end{align*}
    that we finally rewrite as
    \small
    \begin{align}\label{fun_eq_gen}
        2\left(\left(\left(\frac{b}{a}\right)^{\lambda_1}+\bigg(\frac{a}{b}\bigg)^{\lambda_1}\right)\cos\left(\lambda_2\log\left(\frac{b}{a}\right)\right)-2\right)\lambda_1\lambda_2-\left(\lambda_1^2-\lambda_2^2\right)\left(\left(\frac{b}{a}\right)^{\lambda_1}+\bigg(\frac{a}{b}\bigg)^{\lambda_1}\right)\sin\left(\lambda_2\log\left(\frac{b}{a}\right)\right)=0.
    \end{align}
    \normalsize
    This is the universal equation that will keep appearing in all dimension and for all operators $\leb_m$ (where $m\geq 1$).  
    Finally,  we deduce that this equation is equivalent to
    \begin{align}\label{fund_m}
        \left(\left(1+\left(\frac{b}{a}\right)^{2\lambda_1}\right)\cos\left(\lambda_2\log\left(\frac{b}{a}\right)\right)-2\left(\frac{b}{a}\right)^{\lambda_1}\right)\lambda_1\lambda_2=(m^2+n^2)\left(\left(\frac{b}{a}\right)^{2\lambda_1}-1\right)\sin\left(\lambda_2\log\left(\frac{b}{a}\right)\right).
    \end{align}
    First, make the change of variable $R=\log\left(\dfrac{b}{a}\right)$. We have
    $
        \lambda_1=\sqrt{2m^2+2n^2+\lambda_2^2}
    $,
    which leads us to introduce the following function
    \begin{align}
        \varphi(t)&=\left(\left(1+e^{2R\sqrt{2m^2+2n^2+t^2}}\right)\cos(Rt)-2\,e^{R\sqrt{2m^2+2n^2+t^2}}\right)t\sqrt{2m^2+2n^2+t^2}\nonumber\\
        &-m^2\left(e^{2R\sqrt{2m^2+2n^2+t^2}}-1\right)\sin(Rt).
    \end{align}
    Make a second change of variable $Rt=\theta$, and define $\varphi(t)=\psi(\theta)$ so that
    \begin{align}
        \psi(\theta)&=\left(\left(1+e^{2\sqrt{(2m^2+2n^2)R^2+\theta^2}}\right)\cos(\theta)-2\,e^{\sqrt{(2m^2+2n^2)R^2+\theta^2}}\right)\frac{\theta}{R}\sqrt{(2m^2+2n^2)+\left(\frac{\theta^2}{R^2}\right)}\nonumber\\
        &-m^2\left(e^{2\sqrt{(2m^2+2n^2)R^2+\theta^2}}-1\right)\sin(\theta)\nonumber\\
        &=\frac{1}{R^2}\left(\left(1+e^{2\sqrt{(2m^2+2n^2)R^2+\theta^2}}\right)\cos(\theta)-2\,e^{\sqrt{(2m^2+2n^2)R^2+\theta^2}}\right)\theta\sqrt{(2m^2+2n^2)R^2+\theta^2}\nonumber\\
        &-(m^2+n^2)\left(e^{2\sqrt{(2m^2+2n^2)R^2+\theta^2}}-1\right)\sin(\theta).
    \end{align}
    In order to simplify notations, let
    \begin{align}\label{var_renamed}
        \left\{\begin{alignedat}{1}
        a^2&=(2m^2+2n^2)R^2\\
        b&=\frac{1}{R}\\
        c^2&=m^2+n^2
        \end{alignedat}\right.
    \end{align}
    so that
    \begin{align}\label{def_psi}
        \psi(\theta)=b^2\left(\left(1+e^{2\sqrt{a^2+\theta^2}}\right)\cos(\theta)-2\,e^{\sqrt{a^2+\theta^2}}\right)\theta\sqrt{a^2+\theta^2}-c^2\left(e^{2\sqrt{a^2+\theta^2}}-1\right)\sin(\theta).
    \end{align}
    As previously, we want to estimate the first strictly positive zero of $\psi$. Notice that 
    \begin{align}\label{second_bound}
    \left\{\begin{alignedat}{1}
        &\psi(2\pi)=b^2\left(e^{\sqrt{a^2+(2\pi)^2}}-1\right)^22\pi\sqrt{a^2+(2\pi)^2}>0\\
        &\psi(\pi)=-b^2\left(e^{\sqrt{a^2+\pi^2}}+1\right)^2\pi\sqrt{a^2+\pi^2}<0.
        \end{alignedat}\right.
    \end{align}
    Therefore, we will show that $\psi(\theta)<0$ for all $0<\theta\leq \pi$, which will imply that the first non-trivial zero $\alpha$ of $\psi$ satisfies the inequality $\pi<\alpha<2\pi$. First, notice that since $\cos(\theta)\leq 0$ for all $\dfrac{\pi}{2}\leq \theta\leq \pi$, we have
    \begin{align}
        \psi(\theta)\leq -2\,e^{\sqrt{a^2+\theta^2}}\theta\sqrt{a^2+\theta^2}<0\quad \text{for all}\;\, \frac{\pi}{2}\leq \theta\leq \pi.
    \end{align}
    Therefore, the interval $\left[\dfrac{\pi}{2},\pi\right]$ is a zero-free region of $\psi$, and we need only show that $\psi(\theta)<0$ on $\left(0,\dfrac{\pi}{2}\right)$.
    
    We have
    \begin{align}\label{der_psi}
        \psi'(\theta)&=b^2\left(\frac{2\theta}{\sqrt{a^2+\theta^2}}e^{2\sqrt{a^2+\theta^2}}\cos(\theta)-\left(1+e^{2\sqrt{a^2+\theta^2}}\right)\sin(\theta)-\frac{2\theta}{\sqrt{a^2+\theta^2}}e^{\sqrt{a^2+\theta^2}}\right)\theta\sqrt{a^2+\theta^2}\nonumber\\
        &+b^2\left(\left(1+e^{2\sqrt{a^2+\theta^2}}\right)\cos(\theta)-2e^{\sqrt{a^2+\theta^2}}\right)\left(\sqrt{a^2+\theta^2}+\frac{\theta^2}{\sqrt{a^2+\theta^2}}\right)\nonumber\\
        &-\frac{2c^2\theta}{\sqrt{a^2+\theta^2}}e^{2\sqrt{a^2+\theta^2}}\sin(\theta)-c^2\left(e^{2\sqrt{a^2+\theta^2}}-1\right)\cos(\theta)\nonumber\\
        &=\cos(\theta)\left(e^{2\sqrt{a^2+\theta^2}}\left(2b^2\theta^2+b^2\sqrt{a^2+\theta^2}+\frac{b^2}{\sqrt{a^2+\theta^2}}-c^2\right)+b^2\sqrt{a^2+\theta^2}+\frac{b^2\theta^2}{\sqrt{a^2+\theta^2}}+c^2\right)\nonumber\\
        &-\theta\sin(\theta)\left(e^{2\sqrt{a^2+\theta^2}}\left(\sqrt{a^2+\theta^2}+\frac{2c^2\theta}{\sqrt{a^2+\theta^2}}\right)+\sqrt{a^2+\theta^2}\right)\nonumber\\
        &-2b^2e^{\sqrt{a^2+\theta^2}}\left(\theta^2+\sqrt{a^2+\theta^2}+\frac{\theta^2}{\sqrt{a^2+\theta^2}}\right)\nonumber\\
        &=\cos(\theta)h_1(\theta)-\theta\sin(\theta)h_2(\theta)-h_3(\theta).
    \end{align}
    We trivially have for all $\theta\in \left[0,\dfrac{\pi}{2}\right]$ the inequality
    \begin{align*}
        \psi'(\theta)\leq \cos(\theta)h_1(\theta),
    \end{align*}
    where
    \begin{align}\label{def_h1}
        h_1(\theta)=e^{2\sqrt{a^2+\theta^2}}\left(2b^2\theta^2+b^2\sqrt{a^2+\theta^2}+\frac{b^2}{\sqrt{a^2+\theta^2}}-c^2\right)+b^2\sqrt{a^2+\theta^2}+\frac{b^2\theta^2}{\sqrt{a^2+\theta^2}}+c^2.
    \end{align}
    Now, consider the function
    \begin{align*}
        f(\theta)&=2b^2\theta^2+b^2\sqrt{a^2+\theta^2}+\frac{b^2}{\sqrt{a^2+\theta^2}}-c^2\\
        &=\frac{2\theta^2}{R^2}+\frac{1}{R^2}\sqrt{(2m^2+2n^2)R^2+\theta^2}+\frac{1}{R^2\sqrt{(2m^2+2n^2)R^2+\theta^2}}-(m^2+n^2).
    \end{align*}
    Using the elementary inequality $\sqrt{x+y}\leq \sqrt{x}+\sqrt{y}$ (valid for all $x,y\geq 0$), we deduce that for all $\theta\in \left[0,\dfrac{\pi}{2}\right]$
    \begin{align}\label{ineq_f}
        f(\theta)&\leq \frac{\pi^2}{2R^2}+\frac{1}{R^2}\left(\sqrt{2m^2+2n^2}\,R+\frac{\pi}{2}\right)+\frac{1}{R^3\sqrt{2m^2+2n^2}}-(m^2+n^2)\\
        &=\frac{1}{2R^3\sqrt{2m^2+2n^2}}\left(-(2m^2+2n^2)^{\frac{3}{2}}R^3+2(2m^2+2n^2)R^2+\pi(\pi+1)\sqrt{2m^2+2n^2}R+2\right).
    \end{align}
    Now, let $P_{m,n}$ be the following polynomial function
    \begin{align*}
        P_{m,n}(R)=-(2m^2+2n^2)^{\frac{3}{2}}R^3+2(2m^2+2n^2)R^2+\pi(\pi+1)\sqrt{2m^2+2n^2}R+2.
    \end{align*}
    Making the change of variable $s=\sqrt{2m^2+2n^2}R$, if $Q(\sqrt{2m^2+2n^2}R)=P(R)$, we get
    \begin{align*}
        Q(s)=-s^3+2s^2+\pi(\pi+1)s+2.
    \end{align*}    
    Let us estimate the roots of $Q$. 
    \begin{figure}[H]
         	 	\centering
         	 	\includegraphics[width=0.6\textwidth]{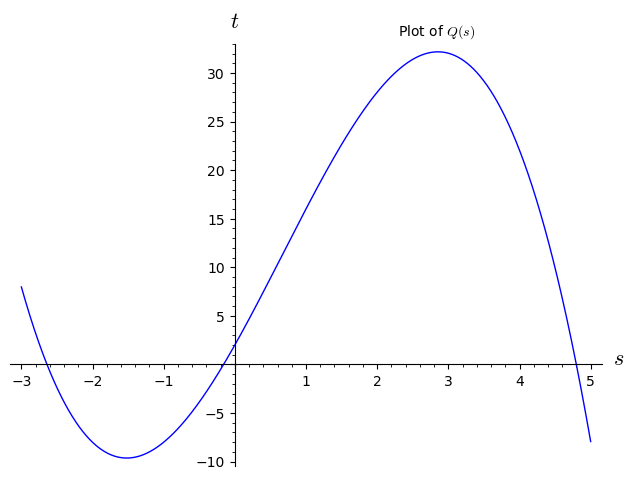}
         	 	\caption{Numerical Approximation of the Roots of $Q$.} 
         	 \end{figure}
    We have
    \begin{align*}
        \left\{\begin{alignedat}{1}
        Q(-3)&=27+18-3\pi(\pi+1)+2=47-3\pi(\pi+1)>0\\
        Q(-2)&=8+8-2\pi(\pi+1)+2=18-2\pi(\pi+1)<0\\
        Q(-1)&=5-\pi(\pi+1)<0\\
        Q(0)&=2>0\\
        Q(4)&=-64+32+4\pi(\pi+1)+2=2(-15+2\pi(\pi+1))>0\\
        Q(5)&=-125+50+5\pi(\pi+1)+2=-73+5\pi(\pi+1)<0,
        \end{alignedat}\right.
    \end{align*}
    where we only used the elementary inequality $13<\pi(\pi+1)<14$. Therefore, $Q$
    admits three real roots $r_1<r_2<r_3$ such that
    \begin{align}
       \left\{\begin{alignedat}{1}
          -3&<r_1<-2\\
          -1&<r_2<0\\
          4&<r_3<5.
       \end{alignedat}\right.
    \end{align}
    Furthermore, we have
    \begin{align*}
        Q'(s)=-3s^2+4s+\pi(\pi+1),
    \end{align*}
    and the roots of $Q'$ are given by
    \begin{align*}
        r_{\pm}=\frac{2}{3}\pm \frac{1}{3}\sqrt{4+3\pi(\pi+1)}<4,
    \end{align*}
    using again the elementary inequality $\pi(\pi+1)<14$ and the inequality $\sqrt{x+y}\leq \sqrt{x}+\sqrt{y}$. In particular, we deduce that $Q$ is strictly decreasing on $(4,\infty)$, which implies that
    \begin{align*}
        Q(s)\leq Q(5)=-73+5\pi(\pi+1)<0\quad \text{for all}\;\, s\geq 5.
    \end{align*}
    Therefore, coming back to $P_{m,n}(R)=Q(\sqrt{2m^2+2n^2}R)$, we deduce that for all $\sqrt{2m^2+2n^2}R\geq 5$, we have
    \begin{align}\label{ineq_bound_partial}
        P_{m,n}(R)\leq P_{m,n}\left(\frac{5}{\sqrt{2m^2+2n^2}}\right)=-(73-5\pi(\pi+1))\quad \text{for all}\;\, R\geq \frac{5}{\sqrt{2m^2+2n^2}}.
    \end{align}
    Therefore, provided that the following inequality holds
    \begin{align}\label{lower_bound2}
        R\geq \frac{5}{\sqrt{2m^2+2n^2}},
    \end{align}
    we deduce that for all $\theta\in \left[0,\dfrac{\pi}{2}\right]$,
    \begin{align}\label{ineq_f2}
        f(\theta)\leq -\frac{73-5\pi(\pi+1)}{2R^3\sqrt{2m^2+2n^2}}.
    \end{align}
    Therefore, by \eqref{def_h1} and \eqref{ineq_f2}, we deduce that for all $R$ satisfying the inequality \eqref{lower_bound2}, we have
    \begin{align*}
        h_1(\theta)&\leq -\frac{73-5\pi(\pi+1)}{2R^3\sqrt{2m^2+2n^2}}e^{2\sqrt{(2m^2+2n^2)R^2+\theta^2}}+\frac{1}{R^2}\sqrt{(2m^2+2n^2)R^2+\theta^2}+\frac{\theta^2}{R^2\sqrt{(2m^2+2n^2)R^2+\theta^2}}\\
        &+(m^2+n^2)
    \end{align*}
    which converges as $R\rightarrow \infty$ to $-\infty$ uniformly in $\theta\in\left[0,\dfrac{\pi}{2}\right]$. More precisely, we estimate
    \begin{align*}
        h_1(\theta)&\leq -\frac{73-5\pi(\pi+1)}{2R^3\sqrt{2m^2+2n^2}}e^{2\sqrt{2m^2+2n^2}\,R}+\frac{\sqrt{2m^2+2n^2}}{R}+\frac{\pi}{2R^2}+\frac{\pi^2}{4\sqrt{2m^2+2n^2}R^3}+(m^2+n^2)\\
        &=\frac{1}{4\sqrt{2m^2+2n^2}R^3}\left(4(m^2+n^2)^{\frac{3}{2}}R^3+4(2m^2+2n^2)R^2+2\pi\sqrt{2m^2+2n^2}R\right.\\
        &\left.-2(73-5\pi(\pi+1))e^{2\sqrt{2m^2+2n^2}R}+\pi^2\right).
    \end{align*}
    We are therefore let to study the function
    \begin{align*}
        f(R)=4(m^2+n^2)^{\frac{3}{2}}R^3+4(2m^2+2n^2)R^2+2\pi\sqrt{2m^2+2n^2}R+\pi^2-2(73-5\pi(\pi+1))e^{2\sqrt{2m^2+2n^2}R}
    \end{align*}
    Once more, making a change of variable $s=\sqrt{2m^2+2n^2}R$, we are reduced to studying the function
    \begin{align*}
        g(s)=4s^3+4s^2+2\pi s+\pi^2-2(73-5\pi(\pi+1))e^{2s}.
    \end{align*}
    Let us show that $g(s)<0$ for all $s\geq 0$. Indeed, we have for all $s\geq 0$
    \begin{align*}
        e^{2s}\geq 1+2s+\frac{(2s)^2}{2}+\frac{(2s)^3}{6}=1+2s+2s^2+\frac{4}{3}s^3,
    \end{align*}
    and we easily check that
    \begin{align*}
        \left\{\begin{alignedat}{1}
            2(73-5\pi(\pi+1))&>\pi^2\\
            4(73-5\pi(\pi+1))&>2\pi\\
            4(73-5\pi(\pi+1))&>4\\
            \frac{8}{3}(73-5\pi(\pi+1))&>4.
        \end{alignedat}\right.
    \end{align*}
    Therefore, we finally deduce that for all $R$ such that \eqref{lower_bound2} holds true we have
    \begin{align*}
        h_1(\theta)\leq 0\quad \text{for all}\;\, \theta\in \left[0,\frac{\pi}{2}\right],
    \end{align*}
    which implies by \eqref{der_psi} that $\psi$ is strictly decreasing on $[0,\frac{\pi}{2}]$,  which shows by \eqref{second_bound} that $\psi(\theta)<0$ for all $0<\theta\leq \pi$. Therefore, we finally deduce by \eqref{second_bound} that the first non-trivial zero $\theta_1>0$ of $\psi$ satisfies the estimate
    \begin{align*}
        \pi<\theta_1<2\pi,
    \end{align*}
    Coming back to the equation \eqref{fund_m}, we deduce since $R=\log\left(\frac{b}{a}\right)$, $R\lambda_2=\theta$ 
    that
    \begin{align*}
        \sqrt{\sqrt{\lambda+4m^2n^2}-(m^2+n^2)}=\lambda_2= \frac{\theta_1}{R}>\frac{\pi}{\log\left(\frac{b}{a}\right)}.
    \end{align*}
    Therefore, we get
    \begin{align*}
        \sqrt{\lambda+4m^2n^2}-(m^2+n^2)>\frac{\pi^2}{\log^2\left(\frac{b}{a}\right)},
    \end{align*}
    or
    \begin{align}\label{lower_estimate_lambda_n}
        \lambda>\left((m^2+n^2)+\frac{\pi^2}{\log^2\left(\frac{b}{a}\right)}\right)^2-4m^2n^2=\left((m-n)^2+\frac{\pi^2}{\log^2\left(\frac{b}{a}\right)}\right)\left((m+n)^2+\frac{\pi^2}{\log^2\left(\frac{b}{a}\right)}\right).
    \end{align}

    \textbf{Step 4: Lower estimate of all eigenvalue when $m$ is a strictly positive integer.} Assume from now on that $m\geq 1$ is a positive integer.

    Since
    \begin{align*}
        \left((m-n)^2+\frac{\pi^2}{\log^2\left(\frac{b}{a}\right)}\right)\left((m+n)^2+\frac{\pi^2}{\log^2\left(\frac{b}{a}\right)}\right)\conv{\frac{b}{a}\rightarrow \infty}(m^2-n^2)^2,
    \end{align*}
    we deducet that the  quantity on the right-hand side of \eqref{lower_estimate_lambda_n} is minimal when $n=m$ or $n=-m$ if the conformal class is large enough. First, notice that \eqref{lower_bound2} holds for all $n\in\Z$ if and only if
    \begin{align}\label{new_bound_conf1}
        \log\left(\frac{b}{a}\right)\geq \frac{5}{\sqrt{2m^2}}=\frac{5}{\sqrt{2}\,m}.
    \end{align}
    Therefore, assume that \eqref{new_bound_conf1} holds. 
    Indeed, let $\alpha>0$, and define
    \begin{align*}
        f(t)=\left((m-t)^2+\alpha^2\right)\left((m+t)^2+\alpha^2\right)=(m^2-t^2)^2+2(m^2+t^2)\alpha^2+\alpha^4.
    \end{align*}
    This is a strictly convex function, which therefore admits a unique minimum. We have
    \begin{align*}
        f'(t)=-4\,t\,(m^2-t^2)+4\,t\,\alpha^2=4t(t^2+\alpha^2-m^2).
    \end{align*}
    If $\alpha^2-m^2\geq 0$, or $\alpha\geq m$, then $f$ admits its minimum at $t=0$, while for $\alpha<m$, $f$ is minimal at $t=\sqrt{m^2-\alpha^2}$ and $t=-\sqrt{m^2-\alpha^2}$. Then, we have
    \begin{align*}
        &f(0)=(m^2+\alpha^2)^2=m^4+2\,\alpha^2m^2+\alpha ^4\\
        &f(m)=\alpha^2(4m^2+\alpha^2)=4\,\alpha^2m^2+\alpha^4,
    \end{align*}
    which shows that
    \begin{align}\label{step_min}
        f(0)-f(m)=m^4-2\alpha^2m^2=m^2(m^2-2\alpha^2)>0
    \end{align}
    provided that $m>\alpha\sqrt{2}$. We therefore assume from now on that
    \begin{align}\label{new_bound_conf2}
        \log\left(\frac{b}{a}\right)\geq \frac{\pi\sqrt{2}}{m}.
    \end{align}
    Furthermore, notice that $f$ is decreasing on $[0,\sqrt{m^2-\alpha^2})$ and increasing on $[\sqrt{m^2-\alpha^2},\infty[$. Therefore, we have 
    \begin{align}\label{step_min2}
        \inf_{n\in\Z}f(n)=\min\ens{f\left([\sqrt{m^2-\alpha^2}]\right),f\left([\sqrt{m^2-\alpha^2}]+1\right)}.
    \end{align}
    We have $[\sqrt{m^2-\alpha^2}]=m-1$ if and only if $\alpha^2\leq 2m-1$. Then, we get
    \begin{align*}
        &\min\ens{f\left([\sqrt{m^2-\alpha^2}]\right),f\left([\sqrt{m^2-\alpha^2}]+1\right)}\\
        &=\min\ens{f(m-1),f(m)}=\min\ens{(1+\alpha^2)((2m-1)^2+\alpha^2),(4m^2+\alpha^2)\alpha^2}.
    \end{align*}
    We have
    \begin{align*}
        g(\alpha)=(1+\alpha^2)((2m-1)^2+\alpha^2)-(4m^2+\alpha^2)\alpha^2=(2m-1)^2-2(2m-1)\alpha^2\geq 0
    \end{align*}
    if and only if
    $
        \alpha\leq \sqrt{\dfrac{2m-1}{2}},
    $
    or
    \begin{align}\label{new_bound_conf3}
        \log\left(\frac{b}{a}\right)\geq \frac{\pi\sqrt{2}}{\sqrt{2m-1}}.
    \end{align}
    Therefore, thanks to \eqref{new_bound_conf1}, \eqref{new_bound_conf2}, and \eqref{new_bound_conf3}, provided that
    \begin{align}\label{new_bound_conf4}
        \log\left(\frac{b}{a}\right)\geq \max\ens{\frac{5}{\sqrt{2}\,m},\frac{\pi\sqrt{2}}{m},\frac{\pi\sqrt{2}}{\sqrt{2m-1}}}=\frac{\pi\sqrt{2}}{\sqrt{2m-1}},
    \end{align}
    we have
    \begin{align*}
        \inf_{n\in\Z}\left((m-n)^2+\frac{\pi^2}{\log^2\left(\frac{b}{a}\right)}\right)\left((m+n)^2+\frac{\pi^2}{\log^2\left(\frac{b}{a}\right)}\right)=\left(4m^2+\frac{\pi^2}{\log^2\left(\frac{b}{a}\right)}\right)\frac{\pi^2}{\log^2\left(\frac{b}{a}\right)},
    \end{align*}
    which concludes the proof of the theorem.
    \end{proof}
    \begin{rem}
        The proof of the theorem also shows that the spectrum of $\displaystyle\Pi_{n^2}(\leb_m^{\ast}\leb_m)$ is given by 
        \begin{align*}
            \lambda=\left((m+n)^2+\frac{\theta_k^2}{\log^2\left(\frac{b}{a}\right)}\right)\left((m-n)^2+\frac{\theta_k^2}{\log^2\left(\frac{b}{a}\right)}\right)
        \end{align*}
        where $\ens{\theta_k}_{k\in\Z}\subset \Z$ is the discrete sets of zeroes of the function
        \begin{align*}
            \psi(\theta)&=\frac{1}{\log^2\left(\frac{b}{a}\right)}\left(\left(1+e^{2\sqrt{(2m^2+2n^2)\log^2\left(\frac{b}{a}\right)+\theta^2}}\right)\cos(\theta)-2\,e^{\sqrt{(2m^2+2n^2)\log^2\left(\frac{b}{a}\right)}+\theta^2}\right)\\
            &\times\theta\sqrt{(2m^2+2n^2)\log^2\left(\frac{b}{a}\right)+\theta^2}
            -(m^2+n^2)\left(e^{2\sqrt{(2m^2+2n^2)\log^2\left(\frac{b}{a}\right)+\theta^2}}-1\right)\sin(\theta).
        \end{align*}
        By parity, we have $\theta_{-k}=\theta_k$ for all $k\in\Z$.
        Indeed, looking at \eqref{der_psi}, we deduce that for $h_1(\theta)> 0$ for $\theta$ large enough. Since the dominant term of $h_1$ is $\theta^2e^{\sqrt{a^2+\theta^2}}$ which is also the dominant term of $h_2$ (while $h_3$ is negligible), we deduce that $\psi'$ changed of sign infinitely often. Since (using \eqref{var_renamed})
        \begin{align*}
            \psi(2k\pi)&=\frac{1}{\log^2\left(\frac{b}{a}\right)}\left(e^{\sqrt{(2m^2+2n^2){\log^2\left(\frac{b}{a}\right)}+(2k\pi)^2}}-1\right)^2 2k\pi\sqrt{(2m^2+2n^2){\log^2\left(\frac{b}{a}\right)}+(2k\pi)^2}>0\\
            \psi((2k+1)\pi)&=-\frac{1}{\log^2\left(\frac{b}{a}\right)}\left(e^{\sqrt{(2m^2+2n^2){\log^2\left(\frac{b}{a}\right)}+(2k+1)^2\pi^2}}+1\right)^2(2k+1)\pi\\
            &\times \sqrt{(2m^2+2n^2)
            {\log^2\left(\frac{b}{a}\right)}+(2k+1)^2\pi^2}<0,
        \end{align*}
        we deduce that $\psi$ admits a discrete set of zeroes $\ens{\theta_k}_{k\in\Z}$ such that $\theta_k\conv{k\rightarrow \pm\infty}\pm\infty$. Furthermore, the previous analysis suggests that for all $k\in\Z\setminus\ens{0}$, we have
        \begin{align}\label{conjectural_inequality}
            k\,\pi<\theta_k<(k+1)\pi.
        \end{align}
        Let us also mention that numerical experiments suggest that
        \begin{align}\label{conj_id}
            \theta_1(a,b)\conv{\frac{b}{a}\rightarrow \infty}\pi.
        \end{align}
        Indeed, taking $m=n=1$, and $R^2=10000$, a numerical estimation gives $0<\theta_1-\pi<2\cdot 10^{-88}$ and $-2\cdot 10^{-88}<\theta_2-2\pi< 0$, which also suggests more generally that for all $k\in \N$
        \begin{align}\label{conjectural_identity}
            \theta_k(a,b)\conv{\frac{b}{a}\rightarrow \infty}k\,\pi,
        \end{align}
        but that the inequality \eqref{conjectural_inequality} is no longer true (provided that $|k|\geq 2$) when the conformal class is large enough.
    \end{rem}
    \begin{theorem}\label{main_neck_lm}
        Let $m\geq 1$, and assume that 
        \begin{align}
            \log\left(\frac{b}{a}\right)\geq \frac{\pi\sqrt{2}}{\sqrt{2m-1}}.
        \end{align}
        Then, if $\lambda_m>0$ is defined by \eqref{min},  we have the estimate
        \begin{align}
            \lambda_m>\left(4m^2\frac{4\pi^2}{\log^2\left(\frac{b}{a}\right)}\right)\frac{4\pi^2}{\log^2\left(\frac{b}{a}\right)}.
        \end{align}
        Furthermore, provided that $m\geq 1$ is an integer and
        \begin{align}
            \log\left(\frac{b}{a}\right)\geq \sqrt{\frac{\pi}{2}}\frac{1}{2m-1}\sqrt{12m^2+4m-1+\sqrt{(12m^2+4m-1)^2+60(2m-1)^2}},
        \end{align}
        then
        \begin{align}
            \left(4m^2\frac{\pi^2}{\log^2\left(\frac{b}{a}\right)}\right)\frac{\pi^2}{\log^2\left(\frac{b}{a}\right)}<\lambda_m<\left(4m^2\frac{4\pi^2}{\log^2\left(\frac{b}{a}\right)}\right)\frac{4\pi^2}{\log^2\left(\frac{b}{a}\right)}.
        \end{align}
        and the minimiser is unique and given by
        \begin{align*}
            u(r,\theta)=\Re\left(u_m(r)e^{i\,m\,\theta}\right),
        \end{align*}
        where $Y(t)=u_m(\log(t))\in W^{2,2}_0([\log(a),\log(b)])$ is a non-trivial solution to the following ordinary differential equation
        \begin{align*}
            &Y''''(t)-4\,Y'''(t)-2(2m^2-3)Y''(t)+4(2m^2-1)Y'(t)-\left(\lambda_m-\left(2m^2(m^2-2)+1\right)\right)Y(t)=0,
        \end{align*}
        where $\lambda_m>0$ is as above.
    \end{theorem}
    \begin{proof}
    The hypothesis implies in particular that for all $n\in\Z$, the following inequality holds
    \begin{align*}
        \log\left(\frac{b}{a}\right)\geq \frac{5}{\sqrt{2m^2+2n^2}}.
    \end{align*}
    Indeed, we have
    \begin{align*}
        \frac{5}{\sqrt{2m^2+2n^2}}\leq \frac{5}{m\sqrt{2}}<\frac{\pi\sqrt{2}}{m}
    \end{align*}
    as $\pi>5/2$.
    Therefore, for all $n\in\Z$, thanks to Theorem \ref{theoreme_ode_m_n}, there exists a solution $Y_n$ to the linear differential equation \eqref{equa_diff_m_n} with boundary conditions \eqref{boundary_conditions_m_n}, with minimal $\lambda_{m,n}>0$ that satisfies the bound of Theorem \ref{theoreme_ode_m_n}. Therefore, if $f_n(r)=Y_n(\log(r))$, we deduce that $f_n$ solves the equation \eqref{eq_f_m_n}, and 
    \begin{align*}
        v_n(r,\theta)=\Re\left(f_n(r)e^{i\,n\,\theta}\right)
    \end{align*}
    that we scale (without changing notations) so that
    \begin{align*}
        \int_{\Omega}\frac{v_n^2}{|x|^4}dx=1,
    \end{align*}
    is an admissible minimiser of \eqref{min} with
    \begin{align}\label{v_n_energy}
        \int_{\Omega}\left(\leb_mv_n\right)^2dx=\lambda_{m,n}.
    \end{align}
    Now, let $u$ be a minimiser of of \eqref{min}. Writing
    \begin{align*}
        u(r,\theta)=\sum_{n\in\Z}u_n(r)e^{i\,n\,\theta}
    \end{align*}
    Since $\lambda_m>0$, we have $u\neq 0$, and we deduce that there exists $n\in \Z$ such that $u_n\neq 0$. In particular, we deduce by Theorem \ref{theoreme_ode_m_n} since $Y(t)=u_n(e^t)$ is a non-trivial solution of \eqref{equa_diff_m_n} that
    \begin{align*}
        \lambda_m \geq \lambda_{m,n}>\left((m-n)^2+\frac{\pi^2}{\log^2\left(\frac{b}{a}\right)}\right)\left((m+n)^2+\frac{\pi^2}{\log^2\left(\frac{b}{a}\right)}\right).
    \end{align*}
    By the proof of Theorem \ref{theoreme_ode_m_n}, we have
    \begin{align*}
        \inf_{n\in\Z}\left((m-n)^2+\frac{\pi^2}{\log^2\left(\frac{b}{a}\right)}\right)\left((m+n)^2+\frac{\pi^2}{\log^2\left(\frac{b}{a}\right)}\right)=\left(4m^2+\frac{\pi^2}{\log^2\left(\frac{b}{a}\right)}\right)\frac{\pi^2}{\log^2\left(\frac{b}{a}\right)}
    \end{align*}
    thanks to our hypothesis on $R=\log\left(\dfrac{b}{a}\right)$, the bound of the theorem is proved. Now, we need only prove that $\lambda_m=\lambda_{m,m}$ provided that the conformal class is large enough. 
    
    By contradiction, if $\lambda_m>\lambda_{m,m}$, then $v_m$ is a solution that has a strictly smaller energy by \eqref{v_n_energy}. Now, assume that there exists $n\neq m$ such that $u_n\neq 0$. Then, we get
    \begin{align*}
        \lambda_m\geq \lambda_{m,n}>\lambda_{m,m},
    \end{align*}
    provided that $R$ is big enough, thanks to the estimate
    \begin{align*}
        \lambda_{m,n}<\left((m-n)^2+\frac{4\pi^2}{\log^2\left(\frac{b}{a}\right)}\right)\left((m+n)^2+\frac{4\pi^2}{\log^2\left(\frac{b}{a}\right)}\right).
    \end{align*}
    Indeed, we have $\lambda_{m,m}<\lambda_{m,n}$ provided that 
    \begin{align}\label{ineq_explicit_minimiser}
        \left((m-n)^2+\frac{\pi^2}{\log^2\left(\frac{b}{a}\right)}\right)\left((m+n)^2+\frac{\pi^2}{\log^2\left(\frac{b}{a}\right)}\right)\geq \left(4m^2+\frac{4\pi^2}{\log^2\left(\frac{b}{a}\right)}\right)\frac{4\pi^2}{\log^2\left(\frac{b}{a}\right)}
    \end{align}
    for some integer $n\neq m$. Since the left-hand side is symmetric in $n$, letting $\alpha=\dfrac{\pi}{\log\left(\frac{b}{a}\right)}$, making the change of variable $n=m+1+t$ and $n=m-1-s$ we have
    \begin{align*}
        &\inf_{n\in\Z}\left((m-n)^2+\frac{\pi^2}{\log^2\left(\frac{b}{a}\right)}\right)\left((m+n)^2+\frac{\pi^2}{\log^2\left(\frac{b}{a}\right)}\right)\\
        &=\min\ens{\inf_{t\geq 0}\left((t+1)^2+\alpha^2\right)\left((2m+1+t)^2+\alpha^2\right),\inf_{s\geq 0}\left((s+1)^2+\alpha^2\right)\left((2m-1-s)^2+\alpha^2\right)}.
    \end{align*}
    Therefore, we introduce the following function on $\R_+$
    \begin{align*}
        f(t)&=\left((t+1)^2+\alpha^2\right)\left((2m+1+t)^2+\alpha^2\right).
    \end{align*}
    As $f$ is trivially increasing on $\R_+$, we deduce that
    \begin{align*}
        \inf_{t\geq 0}f(t)=f(0)=\left(1+\alpha^2\right)\left((2m+1)^2+\alpha^2\right).
    \end{align*}
    Likewise, let $g:\R_+\rightarrow \R$ be such that
    \begin{align*}
        g(t)=\left(t^2+\alpha^2\right)\left((2m-t)^2+\alpha^2\right)
    \end{align*}
    that we minimise on $\R_+\cap\ens{t:t\geq 1}$. We have
    \begin{align*}
        g(t)=\left(t^2+\alpha^2\right)\left(t^2-4m\,t+4m^2+\alpha^2\right)=t^4-4m\,t^3+2(2m^2+\alpha^2)t^2-4m\alpha^2t+\alpha^2(4m^2+\alpha^2).
    \end{align*}
    Since $g$ is coercive, it admits a minimum.
    We have
    \begin{align*}
        g'(t)&=4t^3-12mt^2+4(2m^2+\alpha^2)t-4m\alpha^2=4(t^3-3mt^2+(2m^2+\alpha^2)t-m\alpha^2).
    \end{align*}
    We immediately see that $t=m$ is a root of $g'$. Therefore, we have
    \begin{align*}
        g'(t)=4(t-m)(t^2-2mt+\alpha^2)=4(t-m)((t-m)^2-(m^2-\alpha^2)).
    \end{align*}    
    Assuming from now on that $\alpha^2\leq m^2$, we deduce that $(t-m)^2-(m^2-\alpha^2)\geq 0$ if and only if
    \begin{align*}
        t\leq m-\sqrt{m^2-\alpha^2}\quad \text{or}\;\, t\geq m+\sqrt{m^2-\alpha^2}.
    \end{align*}
    Therefore, we deduce that $g$ is decreasing on $]-\infty,m-\sqrt{m^2-\alpha^2}]$, increasing on $[m-\sqrt{m^2-\alpha^2},m]$, decreasing on $[m,m+\sqrt{m^2-\alpha^2}]$, and increasing on $[m+\sqrt{m^2-\alpha^2},\infty[$. In particular, we have
    \begin{align*}
        \min_{t\in \R}g(t)&=\min\ens{g(m-\sqrt{m^2-\alpha^2}),g(m+\sqrt{m^2-\alpha^2})}\\
        &=\left(\left(m-\sqrt{m^2-\alpha^2}\right)^2+\alpha^2\right)\left(\left(m+\sqrt{m^2-\alpha^2}\right)^2+\alpha^2\right).
    \end{align*}
    Now, since we look for the minimum of $g$ on $\Z\setminus\ens{m}$, notice that $m-\sqrt{m^2-\alpha^2}\leq 1$ if and only if 
    \begin{align}\label{ineq_final_conf}
        m\geq \frac{1+\alpha^2}{2}
    \end{align}
    and we finally deduce that whenever the bound \eqref{ineq_final_conf} is satisfied (notice that it is satisfied by hypothesis), we have
    \begin{align*}
        \inf_{t\geq 1}g(t)=g(1)=\left(1+\alpha^2\right)\left((2m-1)^2+\alpha^2\right).
    \end{align*}
    Therefore thanks to \eqref{ineq_explicit_minimiser}, we have $\lambda_{m,m}<\lambda_{m,n}$ for all $n\neq m$ provided that
    \begin{align*}
        \left(1+\alpha^2\right)\left((2m-1)^2+\alpha^2\right)>4\alpha^2(4m^2+4\alpha^2),
    \end{align*}
    or
    \begin{align*}
        15\alpha^4+(12m^2+4m-1)\alpha^2-(2m-1)^2\leq 0;
    \end{align*}
    which yields the condition
    \begin{align*}
        \alpha^2&\leq \frac{1}{30}\left(-(12m^2+4m-1)+\sqrt{(12m^2+4m-1)^2+60(2m-1)^2}\right)\\
        &=\frac{2(2m-1)^2}{12m^2+4m-1+\sqrt{(12m^2+4m-1)^2+60(2m-1)^2}}.
    \end{align*}
    which finally furnishes the condition given in the statement of the theorem.        
    \end{proof}

\section{Second Eigenvalue Problem}

    The first ODE lemma will allow us (together with the forthcoming weighted Poincaré estimates) to control terms of the form
    \begin{align*}
        \int_{\Omega}|A|^4u^2d\vg
    \end{align*}
    in the second derivative of the Willmore energy. However, the terms of the form
    \begin{align*}
        \int_{\Omega}|A|^2|\D u|^2dx
    \end{align*}
    require a new estimate of the form
    \begin{align*}
        \int_{\Omega}(\leb_mu)^2dx\geq \frac{\lambda_0^2}{\log^2\left(\frac{b}{a}\right)}\int_{\Omega}\frac{|\D u|^2}{|x|^2}dx\quad \forall u\in W^{2,2}(\Omega),
    \end{align*}
    where $\lambda_0>0$ is independent of $\Omega=B_b\setminus\bar{B}_a(0)$. The method to prove such an identity is now clear. 

    \subsection{Minimisation Problem}

    Fix $0<a<b<\infty$, and consider the following minimisation problem
    \begin{align}\label{min2}
        \mu_m=\inf\ens{\int_{\Omega}(\leb_mu)^2dx: u\in W^{2,2}(B_b\setminus\bar{B}_a(0))\quad \text{and}\quad \int_{\Omega}\frac{|\D u|^2}{|x|^2}dx=1}.
    \end{align}
    A minimiser $u\in W^{2,2}_0(\Omega)$ exists thanks to a similar proof, and it verifies 
    \begin{align}\label{second_eigenvalue_fundamental}
        \leb_m^{\ast}\leb_mu=-\frac{\mu_m}{|x|^2}\left(\Delta-2\frac{x}{|x|^2}\cdot \D\right)u.
    \end{align}
    Indeed, integrating by parts, for all $u,v\in W^{2,2}_0(\Omega)$, we have
    \begin{align*}
        \int_{\Omega}\frac{\D u\cdot \D v}{|x|^2}dx=-\int_{\Omega}v\dive\left(\frac{1}{|x|^2}\D u\right)dx=-\int_{\Omega}v\frac{1}{|x|^2}\left(\Delta u-2\frac{x}{|x|^2}\cdot \D u\right)dx.
    \end{align*}
    Notice that we trivially have for all $n\in\Z$
    \begin{align*}
        \Pi_{n^2}\left(\Delta -2\frac{x}{|x|^2}\cdot \D\right)=\p{r}^2-\frac{1}{r}\p{r}-\frac{n^2}{r^2}.
    \end{align*}
    Recalling formula \eqref{projection_m}, if 
    \begin{align*}
        &\Pi_{n^2}(\leb_m^{\ast}\leb_m)=\p{r}^4+\frac{2}{r}\p{r}^3-(2(m+1)(m-1)+2\,n^2+1)\frac{1}{r^2}\p{r}^2+(2(m+1)(m-1)+2\,n^2+1)\frac{1}{r^3}\p{r}\nonumber\\
        &+\left((m+1)^2(m-1)^2-2n^2(m+1)(m-1)+n^2(n^2-4)\right)\frac{1}{r^4},
    \end{align*}
    we deduce that for all $\mu\geq 0$
    \begin{align*}
        &\Pi_{n^2}\left(\leb_m^{\ast}\leb_m+\frac{\mu}{|x|^2}\left(\Delta-2\frac{x}{|x|^2}\cdot \D\right)\right)=\p{r}^4+\frac{2}{r}\p{r}^3-(2(m^2-1)+2\,n^2+1)\frac{1}{r^2}\p{r}^2\\
        &+(2(m^2-1)+2\,n^2+1)\frac{1}{r^3}\p{r}
        +\left((m^2-1)^2-2n^2(m^2-1)+n^2(n^2-4)\right)\frac{1}{r^4}\\
        &+\frac{\mu}{r^2}\left(\Delta-\frac{1}{r^2}\p{r}-\frac{n^2}{r^2}\right).
    \end{align*}
    Now, consider a smooth solution to the equation
    \begin{align}\label{second_eigenvalue}
        &f''''+\frac{2}{r}f'''-\left(2m^2+2n^2-1\right)\frac{1}{r^2}f''+\left(2m^2+2n^2-1\right)\frac{1}{r^3}f'+\left((m^2-n^2-1)^2-4n^2)\right)\frac{1}{r^4}f\nonumber\\
        &=-\frac{\mu}{r^2}\left(f''-\frac{1}{r}f'-\frac{n^2}{r^2}f\right).
    \end{align}
    Making as previously the change of variable $f(r)=Y(\log(r))$, remember that by \eqref{change_var_log} and \eqref{change_var_log2}, we have
    \begin{align}\label{change_var_log_bis}
    \left\{\begin{alignedat}{1}
        &\p{r}f(r)=\frac{1}{r}Y'(\log(r))\\
        &\p{r}^2f(r)=\frac{1}{r^2}Y''(\log(r))-\frac{1}{r^2}Y'(\log(r))\\
        &\p{r}^3f(r)=\frac{1}{r^3}Y'''(\log(r))-\frac{3}{r^3}Y''(\log(r))+\frac{2}{r^3}Y'(\log(r))\\
        &\p{r}^4f(r)=\frac{1}{r^4}Y''''(\log(r))-\frac{6}{r^4}Y'''(\log(r))+\frac{11}{r^4}Y''(\log(r))-\frac{6}{r^4}Y'(\log(r)).
        \end{alignedat}\right.
    \end{align}
    Therefore, \eqref{eq_f_m_n} and \eqref{change_var_log} show that
    \begin{align}\label{change_var_log2_bis}
        &f''''+\frac{2}{r}f'''-(2m^2+2n^2-1)\frac{1}{r^2}f''+(2m^2+2n^2-1)\frac{1}{r^3}f'\nonumber\\
        &=\frac{1}{r^4}\bigg(Y''''(\log(r))-4\,Y'''(\log(r))-2(m^2+n^2-3)Y''(\log(r))+4(m^2+n^2-1)Y'(\log(r))\bigg).
    \end{align}
    On the other hand \eqref{change_var_log_bis} shows that
    \begin{align}\label{change_var_log3}
        &\frac{1}{r^2}\left(f''-\frac{1}{r}f'-\frac{n^2}{r^2}f\right)=\frac{1}{r^4}\bigg(Y''(\log(r))-2\,Y'(\log(r))-n^2Y(\log(r))\bigg),
    \end{align}
    and taking $r=e^t$, we deduce that \eqref{second_eigenvalue} is equivalent to the equation
    \begin{align}\label{second_eigenvalue2}
        &Y''''(t)-4\,Y'''(t)-\left(2(m^2+n^2-3)-\mu\right)Y''(t)+(4(m^2+n^2-1)-\mu)Y'(t)\nonumber\\
        &+\left((m^2-n^2-1)^2-(4+\mu)n^2\right)Y(t)=0.
    \end{align}
    The characteristic polynomial of this equation is given by
    \begin{align*}
        P(X)=X^4-4\,X^3-(2(m^2+n^2-3)-\mu)X^2+2(2(m^2+n^2-1)-\mu)X+(m^2-n^2-1)^2-(4+\mu)n^2.
    \end{align*}
    First consider the polynomial
    \begin{align*}
        R(X)=X^2-2X-n^2.
    \end{align*}
    We have
    \begin{align*}
        R(Y+1)=Y^2-n^2-1.
    \end{align*}
    Therefore, the computation in \eqref{change_var_polynomial} shows that
    \begin{align*}
        Q(Y)&=P(Y+1)=Y^4-2(m^2+n^2)Y^2+(m^2-n^2)^2+\mu\left(Y^2-n^2-1\right)\\
        &=Y^4-(2(m^2+n^2)-\mu)Y^2+(m^2-n^2)^2-\mu(n^2+1)
    \end{align*}
    which is once more a biquadratic polynomial! Therefore, the roots of $P$ are given by
    \begin{align*}
        \left\{\begin{alignedat}{1}
            r_1&=1+\sqrt{\frac{2(m^2+n^2)-\mu+\sqrt{16m^2n^2-4(m^2-1)\mu+\mu^2}}{2}}\\
            r_2&=1-\sqrt{\frac{2(m^2+n^2)-\mu+\sqrt{16m^2n^2-4(m^2-1)\mu+\mu^2}}{2}}\\
            r_3&=1+\sqrt{\frac{2(m^2+n^2)-\mu-\sqrt{16m^2n^2-4(m^2-1)\mu+\mu^2}}{2}}\\
            r_4&=1-\sqrt{\frac{2(m^2+n^2)-\mu-\sqrt{16m^2n^2-4(m^2-1)\mu+\mu^2}}{2}}.
        \end{alignedat}\right.
    \end{align*}
    We will distinguish three cases provided that $n\geq 1$ or $m=1$ and $n\geq 0$. Provided that either of those conditions hold, we deduce that $16m^2n^2-4(m^2-1)\mu+\mu^2\geq 0$, and we have
    \begin{align*}
        2(m^2+n^2)-\mu-\sqrt{16m^2n^2-4(m^2-1)\mu+\mu^2}\leq 0
    \end{align*}
    if and only if
    \begin{align*}
        4(m^2+n^2)^2-4(m^2+n^2)\mu+\mu^2\leq 16m^2n^2-4(m^2-1)\mu+\mu^2,
    \end{align*}
    or
    \begin{align*}
        \mu\geq \frac{1}{n^2+1}\left((m^2+n^2)^2-4m^2n^2\right)=\frac{(m^2-n^2)^2}{1+n^2}.
    \end{align*}

    \subsection{General ODE Lemma}

    \begin{theorem}\label{second_theoreme_ode_m_n}
        Let $m\geq 1$, and $n\in\Z$. Let $0<a<b<\infty$ and distinguish the following two cases.
        \emph{\textbf{Assumption I: $n^2>\dfrac{m^2(m^2-2)}{2m^2+1+\sqrt{5m^4+2m^2+1}}$}}. Then, we assume that 
        \begin{align*}
            \log\left(\frac{b}{a}\right)\geq \pi\sqrt{\frac{25(n^2+1)+\sqrt{525(n^2+1)^2+16\pi^2(50+\pi^2)(n^4+2(2m^2+1)n^2-m^2(m^2-2))}}{2(50+\pi^2)(n^4+2(2m^2+1)n^2-m^2(m^2-2))}}.
        \end{align*}
        \emph{\textbf{Assumption II: $m\geq \sqrt{2}$ and $n^2\leq \dfrac{m^2(m^2-2)}{2m^2+1+\sqrt{5m^4+2m^2+1}}$}}. Then, we make no assumptions on $0<a<b<\infty$.

        For all $\lambda>0$, consider the following linear differential equation
        \begin{align}\label{second_equa_diff_m_n}
            &Y''''(t)-4\,Y'''(t)-\left(2(m^2+n^2-3)-\mu\right)Y''(t)+(4(m^2+n^2-1)-\mu)Y'(t)\nonumber\\
        &+\left((m^2-n^2-1)^2-(4+\mu)n^2\right)Y(t)=0
        \end{align}
        with boundary conditions
        \begin{align}\label{second_boundary}
            \left\{\begin{alignedat}{2}
                Y(\log(a))&=Y(\log(b))&&=0\\
                Y'(\log(a))&=Y'(\log(b))&&=0.
            \end{alignedat}\right.
        \end{align}
        Then, there exists $\mu>0$ such that the system \eqref{second_equa_diff_m_n}-\eqref{second_boundary} admits a non-trivial solution $Y$. Furthermore, the minimal value $\mu_{m,n}>0$ satisfies the following estimate
        \small
        \begin{align}\label{second_estimate_log}
            \frac{\left(2(m+n)^2+\dfrac{\pi^2}{\log^2\left(\frac{b}{a}\right)}\right)\left(2(m-n)^2+\dfrac{\pi^2}{\log^2\left(\frac{b}{a}\right)}\right)}{4(n^2+1)+\dfrac{2\pi^2}{\log^2\left(\frac{b}{a}\right)}}<\mu_{m,n}<\frac{\left((m+n)^2+\dfrac{2\pi^2}{\log^2\left(\frac{b}{a}\right)}\right)\left((m-n)^2+\dfrac{2\pi^2}{\log^2\left(\frac{b}{a}\right)}\right)}{n^2+1+\dfrac{2\pi^2}{\log^2\left(\frac{b}{a}\right)}}.
        \end{align}
        \normalsize
        provided that \emph{\textbf{Assumption I}} or \emph{\textbf{Assumption II}} hold. Furthermore, we have
        \begin{align*}
            \mu_{m,n}>\frac{(m^2-n^2)^2}{1+n^2}>\frac{m^4}{4}\qquad \text{if}\quad n^2\leq \frac{m^2(m^2-2)}{2m^2+1+\sqrt{5m^4+2m^2+1}}
        \end{align*}
        if \emph{\textbf{Assumption III}} holds.
        Finally, assuming that \eqref{new_conf_condition}, \eqref{new_conf_condition2} \emph{(}if $m>1$\emph{)}, and \eqref{new_conf_condition3} \emph{(}if $m\geq r_0$, where $8<r_0<9$ is the largest root of the polynomial  $X^4-8X^3-6X^2+1$\emph{)}, we have
        we have
        \begin{align}
            \inf_{n\in\Z}\mu_{m,n}>\frac{\left(4m^2+\dfrac{\pi^2}{\log^2\left(\frac{b}{a}\right)}\right)\dfrac{\pi^2}{\log^2\left(\frac{b}{a}\right)}}{4(m^2+1)+\dfrac{2\pi^2}{\log^2\left(\frac{b}{a}\right)}},
        \end{align}
        while the conditions \eqref{new_conf_condition}, \eqref{new_conf_condition2} \emph{(}if $m>1$\emph{)}, \eqref{new_conf_condition3} \emph{(}if $m\geq r_0$, where $8<r_0<9$ is the largest root of the polynomial  $X^4-8X^3-6X^2+1$\emph{)}, and \eqref{new_conf_condition4} imply that
        \begin{align}
            \inf_{n\in\Z}\mu_{m,n}=\mu_{m,m}>\frac{\left(4m^2+\dfrac{\pi^2}{\log^2\left(\frac{b}{a}\right)}\right)\dfrac{\pi^2}{\log^2\left(\frac{b}{a}\right)}}{4(m^2+1)+\dfrac{2\pi^2}{\log^2\left(\frac{b}{a}\right)}}.
        \end{align}
    \end{theorem}
    \begin{proof}
        \textbf{Assumption I: $n\geq 1$ or $m=1$.}

        Without loss of generality, since we are interested in the asymptotic behaviour of the first eigenvalue $\lambda_1(\Omega)$ as $\dfrac{b}{a}\rightarrow\infty$, which we know converges to $0$, we can assume that $\mu<2(m^2+n^2)$ (more precisely, we will show that the minimal solution is such that $\mu<2(m^2+n^2)$, which shows that this hypothesis is not restrictive)..

        Let $Y$ be a non-trivial solution of the equation \eqref{second_equa_diff_m_n} such that $\mu>0$. We distinguish three cases.

    \textbf{Case 1:} $\mu<\dfrac{(m^2-n^2)^2}{1+n^2}$.

    Then, the characteristic polynomial $P$ admits four distinct real roots, which shows that for some $(\mu_1,\mu_2,\mu_3,\mu_4)\in \R^4\setminus\ens{0}$, we have
    \begin{align*}
        Y(t)=\mu_1e^{r_1\,t}+\mu_2e^{r_2\,t}+\mu_3e^{r_3\,t}+\mu_4e^{r_4\,t},
    \end{align*}
    where
    \begin{align*}
        \left\{\begin{alignedat}{1}
        r_1=1+\lambda_1\\
        r_2=1-\lambda_1\\
        r_3=1+\lambda_2\\
        r_4=1-\lambda_2
        \end{alignedat}\right.
    \end{align*}
    and
    \begin{align*}
        \left\{\begin{alignedat}{1}
            \lambda_1&=\sqrt{\frac{2(m^2+n^2)-\mu+\sqrt{16m^2n^2-4(m^2-1)\mu+\mu^2}}{2}}\\
            \lambda_2&=\sqrt{\frac{2(m^2+n^2)-\mu-\sqrt{16m^2n^2-4(m^2-1)\mu+\mu^2}}{2}},
        \end{alignedat}\right.
    \end{align*}
    we are exactly in the situation of the first step in the proof of Theorem \ref{theoreme_ode_m_n}, and the boundary conditions imply that 
    \begin{align*}
        a^2b^2\left((\lambda_1+\lambda_2)^2\left(\bigg(\frac{a}{b}\bigg)^{\lambda_1-\lambda_2}+\left(\frac{b}{a}\right)^{\lambda_1-\lambda_2}\right)-\left(\lambda_1-\lambda_2\right)^2\left(\bigg(\frac{a}{b}\bigg)^{\lambda_1+\lambda_2}+\left(\frac{b}{a}\right)^{\lambda_1+\lambda_2}\right)-8\lambda_1\lambda_2\right)=0,
    \end{align*}
    which is impossible since $\lambda_1>\lambda_2>0$. Therefore, we get $Y=0$, a contradiction.

    \textbf{Case 2:} $\mu=\dfrac{(m^2-n^2)^2}{1+n^2}$. Then, provided that $\mu\leq 2(m^2+n^2)$ we have 
    \begin{align*}
    \left\{\begin{alignedat}{1}
        r_1&=1+\sqrt{2(m^2+n^2)-\mu}=\sqrt{2(m^2+n^2)-\frac{(m^2-n^2)^2}{1+n^2}}\\
        r_2&=1-\sqrt{2(m^2+n^2)-\mu}\\
        r_3&=r_4=1
        \end{alignedat}\right.
    \end{align*}
    Therefore, we deduce one more by the proof of Theorem \ref{theoreme_ode_m_n} (one need only replace $2m^2+2n^2$ by $2m^2+2n^2-\mu$ in formulae) that 
    \begin{align}\label{vanish_second}
        &\sqrt{2m^2+2n^2-\mu}\,(ab)^2\left(4-2\left(\frac{b}{a}\right)^{\sqrt{2m^2+2n^2-\mu}}-2\left(\frac{a}{b}\right)^{\sqrt{2m^2+2n^2-\mu}}\right.\nonumber\\
        &\left.+\sqrt{2m^2+2n^2-\mu}\left(\left(\frac{b}{a}\right)^{\sqrt{2m^2+2n^2-\mu}}-\left(\frac{a}{b}\right)^{\sqrt{2m^2+2n^2-\mu}}\right)\log\left(\frac{b}{a}\right)\right)=0.
    \end{align}
    Since $\mu<2(m^2+n^2)$ by hypothesis, we deduce that \eqref{vanish_second} is never verified, which implies that $Y=0$.

    \textbf{Case 3:} $\mu>\dfrac{(m^2-n^2)^2}{1+n^2}$. Then, the characteristic polynomial admits the following roots
    \begin{align*}
        \left\{\begin{alignedat}{1}
            r_1=1+\sqrt{\frac{\sqrt{16m^2n^2-4(m^2-1)\mu+\mu^2}+2(m^2+n^2)-\mu}{2}}\\
            r_2=1-\sqrt{\frac{\sqrt{16m^2n^2-4(m^2-1)\mu+\mu^2}+2(m^2+n^2)-\mu}{2}}\\
            r_3=1+i\sqrt{\frac{\sqrt{16m^2n^2-4(m^2-1)\mu+\mu^2}+\mu-2(m^2+n^2)}{2}}\\
            r_4=1-i\sqrt{\frac{\sqrt{16m^2n^2-4(m^2-1)\mu+\mu^2}+\mu-2(m^2+n^2)}{2}}
        \end{alignedat}\right.
    \end{align*}
    We are exactly in the situation of \textbf{Step 3} from Theorem \eqref{theoreme_ode_m_n}, and we deduce by \eqref{fun_eq_gen} also holds, namely
    \small
    \begin{align}\label{second_fun_eq_gen}
        2\left(\left(\left(\frac{b}{a}\right)^{\lambda_1}+\bigg(\frac{a}{b}\bigg)^{\lambda_1}\right)\cos\left(\lambda_2\log\left(\frac{b}{a}\right)\right)-2\right)\lambda_1\lambda_2-\left(\lambda_1^2-\lambda_2^2\right)\left(\left(\frac{b}{a}\right)^{\lambda_1}+\bigg(\frac{a}{b}\bigg)^{\lambda_1}\right)\sin\left(\lambda_2\log\left(\frac{b}{a}\right)\right)=0.
    \end{align}
    \normalsize
    where 
    \begin{align*}
        \left\{\begin{alignedat}{1}
            \lambda_1&=\sqrt{\frac{\sqrt{16m^2n^2-4(m^2-1)\mu+\mu^2}+2(m^2+n^2)-\mu}{2}}\\
            \lambda_2&=\sqrt{\frac{\sqrt{16m^2n^2-4(m^2-1)\mu+\mu^2}+\mu-2(m^2+n^2)}{2}}
        \end{alignedat}\right.
    \end{align*}
    Then, we have
    \begin{align*}
        \lambda_1^2-\lambda_2^2=2(m^2+n^2)-\mu, 
    \end{align*}
    and we get the equation
    \small
    \begin{align}\label{second_fun_eq_gen20}
        \left(\left(\left(\frac{b}{a}\right)^{\lambda_1}+\bigg(\frac{a}{b}\bigg)^{\lambda_1}\right)\cos\left(\lambda_2\log\left(\frac{b}{a}\right)\right)-2\right)\lambda_1\lambda_2=\left(m^2+n^2-\frac{\mu}{2}\right)\left(\left(\frac{b}{a}\right)^{\lambda_1}+\bigg(\frac{a}{b}\bigg)^{\lambda_1}\right)\sin\left(\lambda_2\log\left(\frac{b}{a}\right)\right),
    \end{align}
    \normalsize
    where $\lambda_1=\sqrt{2(m^2+n^2)-\mu+\lambda_2^2}$, and finally, we recover a variant of \eqref{fund_m}
    \small
    \begin{align}\label{second_fund_m}
        \left(\left(1+\left(\frac{b}{a}\right)^{2\lambda_1}\right)\cos\left(\lambda_2\log\left(\frac{b}{a}\right)\right)-2\left(\frac{b}{a}\right)^{\lambda_1}\right)\lambda_1\lambda_2=\left(m^2+n^2-\frac{\mu}{2}\right)\left(\left(\frac{b}{a}\right)^{2\lambda_1}-1\right)\sin\left(\lambda_2\log\left(\frac{b}{a}\right)\right),
    \end{align}
    \normalsize
    where $2m^2+2n^2$ is simply substituted to $2m^2+2n^2-\mu$. Therefore, using \eqref{lower_bound2}, provided that
    \begin{align}\label{condition_conformal_mu}
        \log\left(\frac{b}{a}\right)\geq \frac{5}{\sqrt{2m^2+2n^2-\mu}},
    \end{align}
    the first non-trivial solution $\lambda_2>0$ of this equation is such that
    \begin{align}\label{first_zero}
        \frac{\pi}{\log\left(\frac{b}{a}\right)}<\lambda_2<\frac{2\pi}{\log\left(\frac{b}{a}\right)}.
    \end{align}
    In particular, we have
    \begin{align*}
        \sqrt{\sqrt{16m^2n^2-4(m^2-1)\mu+\mu^2}+\mu-2(m^2+n^2)}>\frac{\pi}{\log\left(\frac{b}{a}\right)},
    \end{align*}
    which gives us
    \begin{align*}
        \sqrt{16m^2n^2-4(m^2-1)\mu+\mu^2}>2(m^2+n^2)-\mu+\frac{\pi^2}{\log^2\left(\frac{b}{a}\right)},
    \end{align*}
    and
    \begin{align*}
        16m^2n^2-4(m^2-1)\mu+\mu^2>\left(2(m^2+n^2)-\mu+\frac{\pi^2}{\log^2\left(\frac{b}{a}\right)}\right)^2,
    \end{align*}
    that we rewrite
    \begin{align}\label{second_est_last1}
        \mu^2-4(m^2-1)\mu>\left(2(m+n)^2-\mu+\frac{\pi^2}{\log^2\left(\frac{b}{a}\right)}\right)\left(2(m-n)^2-\mu+\frac{\pi^2}{\log^2\left(\frac{b}{a}\right)}\right).
    \end{align}
    In order to simplify notations, write
    \begin{align*}
        X=\frac{\pi}{\log\left(\frac{b}{a}\right)}.
    \end{align*}
    We have
    \begin{align*}
        &\left(2(m+n)^2-\mu+\frac{\pi^2}{\log^2\left(\frac{b}{a}\right)}\right)\left(2(m-n)^2-\mu+\frac{\pi^2}{\log^2\left(\frac{b}{a}\right)}\right)\\
        &=\left(2(m+n)^2+X^2\right)\left(2(m-n)^2+X^2\right)-2\left(2(m^2+n^2)+X^2\right)\mu+\mu^2
    \end{align*}
    where we used the parallelogram identity $(m^2+n^2)^2+(m^2-n^2)^2=2m^4+2n^4$. Therefore, \eqref{second_est_last1} is equivalent to 
    \begin{align*}
        2\left(2(n^2+1)+X^2\right)\mu&=\left(2(2(m^2+n^2)+X^2)-4(m^2-1)\right)\mu\\
        &>\left(2(m+n)^2+X^2\right)\left(2(m-n)^2+X^2\right),
    \end{align*}
    and we finally obtain
    \begin{align}\label{lower_estimate_mu}
        \mu>\frac{1}{2}\frac{\left(2(m+n)^2+\dfrac{\pi^2}{\log^2\left(\frac{b}{a}\right)}\right)\left(2(m-n)^2+\dfrac{\pi^2}{\log^2\left(\frac{b}{a}\right)}\right)}{2(n^2+1)+\dfrac{\pi^2}{\log^2\left(\frac{b}{a}\right)}}.
    \end{align}
    Furthermore, \eqref{first_zero} also shows that
    \begin{align}\label{upper_estimate_mu}
        \mu<\frac{\left((m+n)^2+\dfrac{2\pi^2}{\log^2\left(\frac{b}{a}\right)}\right)\left((m-n)^2+\dfrac{2\pi^2}{\log^2\left(\frac{b}{a}\right)}\right)}{n^2+1+\dfrac{2\pi^2}{\log^2\left(\frac{b}{a}\right)}}.
    \end{align}

    \textbf{Step 4: Conformal class estimate.}
    
    Now, recall that \eqref{condition_conformal_mu} depends on $\mu$, but it should only depend on $m$ and $n$. Therefore, we need to rewrite the condition so that it only depends on $m$ and $n$. If $Y^2=2X^2$, we have
    \begin{align}\label{condition_classe_conf}
        &2(m^2+n^2)-\frac{\left((m+n)^2+Y^2\right)((m-n)^2+Y^2)}{n^2+1+Y^2}\nonumber\\
        &=\frac{1}{n^2+1+Y^2}\left(2(m^2+n^2)(n^2+1+\colorcancel{Y^2}{red})-\left((m^2-n^2)^2+\colorcancel{2(m^2+n^2)Y^2}{red}+Y^4\right)\right)\nonumber\\
        &=\frac{1}{n^2+1+Y^2}\left({-Y^4+2(m^2+n^2)(n^2+1)-(m^2-n^2)^2}\right)\nonumber\\
        &=\frac{1}{n^2+1+Y^2}\left(-Y^4+n^4+2(2m^2+1)n^2-m^2(m^2-2)\right).
    \end{align}
    That gives us a non-trivial condition on $Y$ provided that $n^4+2(2m^2+1)n^2-m^2(m^2-2)>0$. The discriminant of this polynomial in $n^2$ is given by
    \begin{align*}
        4(2m^2+1)^2+4m^2(m^2-2)=4\left(5m^4+2m^2+1\right),
    \end{align*}
    which shows that its roots are given by
    \begin{align*}
        -(2m^2+1)\pm\sqrt{5m^4+2m^2+1}.
    \end{align*}
    
    \textbf{Sub-case 1: $1\leq m<  \sqrt{2}$.}
    We have
    \begin{align*}
        &2(m^2+n^2)-\frac{\left((m+n)^2+Y^2\right)\left((m-n)^2+Y^2\right)}{n^2+1+Y^2}\\
        &=\frac{1}{n^2+1+Y^2}\left(-Y^4+n^4+2(2m^2+1)n^2+(2-m^2)m^2\right)
    \end{align*}
    and we get the \emph{necessary} condition
    \begin{align*}
        Y<\sqrt[4]{n^4+2(2m^2+1)n^2+(2-m^2)m^2},
    \end{align*}
    or
    \begin{align*}
        \log\left(\frac{b}{a}\right)> \frac{\pi\sqrt{2}}{\sqrt[4]{n^4+2(2m^2+1)n^2+(2-m^2)m^2}}.
    \end{align*}
    Therefore, the condition \eqref{condition_conformal_mu} is implied by  
    \begin{align*}
        \frac{\pi\sqrt{2}}{Y}\geq 5\sqrt{\frac{n^2+1+Y^2}{-Y^4+n^4+2(2m^2+1)n^2+(2-m^2)m^2}}
    \end{align*}
    or
    \begin{align*}
        2\pi^2(-Y^4+n^4+2(2m^2+1)n^2+(2-m^2)m^2)>25Y^2\left(Y^2+1\right),
    \end{align*}
    that we finally rewrite as
    \begin{align*}
        n^4+2(2m^2+1)n^2+(2-m^2)m^2\geq \left(\frac{25}{2\pi^2}+1\right)Y^4+\frac{25}{2\pi^2}(n^2+1)Y^2.
    \end{align*}
    Therefore, we finally get the condition
    \begin{align*}
        Y^2&\leq -\frac{25}{4\pi^2}(n^2+1)+\sqrt{\frac{525}{16\pi^4}(n^2+1)^2+\left(\frac{50}{\pi^2}+4\right)(n^4+2(2m^2+1)n^2+(2-m^2)m^2)}\\
        &=-\frac{25}{4\pi^2}(n^2+1)+\frac{1}{4\pi^2}\sqrt{525(n^2+1)^2+16\pi^2(50+\pi^2)(n^4+2(2m^2+1)n^2+(2-m^2)m^2)}\\
        &=\frac{4(50+\pi^2)(n^4+2(2m^2+1)n^2+(2-m^2)m^2)}{25(n^2+1)+\sqrt{525(n^2+1)^2+16\pi^2(50+\pi^2)(n^4+2(2m^2+1)n^2+(2-m^2)m^2)}}
    \end{align*}
    which can can rewritten as
    \begin{align*}
        Y^2=2X^2=\frac{2\pi^2}{\log^2\left(\frac{b}{a}\right)}
    \end{align*}
    as
    \begin{align*}
        \log\left(\frac{b}{a}\right)&\geq \pi\sqrt{\frac{25(n^2+1)+\sqrt{525(n^2+1)^2+16\pi^2(50+\pi^2)(n^4+2(2m^2+1)n^2+(2-m^2)m^2)}}{2(50+\pi^2)(n^4+2(2m^2+1)n^2+(2-m^2)m^2)}},
    \end{align*}
    which concludes the proof of the theorem in the case $1\leq m< \sqrt{2}$. Notice that as $|n|\rightarrow \infty$, the right-hand side converges to $0$, which implies in particular that this quantity is bounded as $n\in\Z$. 

    \textbf{Sub-case 2: $m\geq \sqrt{2}$.}
    
    Now, assume that $m\geq \sqrt{2}$. Coming back to \eqref{condition_classe_conf}, we deduce that if \eqref{condition_classe_conf} is verified, then 
    \begin{align*}
        n^2>-(2m^2+1)+\sqrt{5m^4+2m^2+1}=\frac{m^2(m^2-2)}{2m^2+1+\sqrt{5m^4+2m^2+1}}.
    \end{align*}
    If 
    \begin{align*}
        n^2\leq -(2m^2+1)+\sqrt{5m^4+2m^2+1}=\frac{m^2(m^2-2)}{2m^2+1+\sqrt{5m^4+2m^2+1}}
        <\frac{m^2}{2},
    \end{align*}
    then we have
    \begin{align*}
        (m+n)^2(m-n)^2=(m^2-n^2)^2> \frac{m^4}{4}.
    \end{align*}
    In this case, using the elementary inequality
    \begin{align*}
        \mu>\frac{(m^2-n^2)^2}{1+n^2}\geq \frac{m^4}{4(1+n^2)}\geq \frac{m^4}{2(1+2m^2)},
    \end{align*}
    which suffices to our purpose since this estimate is constant as the conformal class degenerates. Therefore, we can assume that 
    \begin{align}\label{ineq_delicate}
        n^2>\frac{m^2(m^2-2)}{2m^2+1+\sqrt{5m^4+2m^2+1}},
    \end{align}
    and \eqref{condition_classe_conf} can be solved as
    \begin{align*}
        Y<\sqrt[4]{n^4+2(2m^2+1)n^2-m^2(m^2-2)},
    \end{align*}
    or
    \begin{align*}
        \log\left(\frac{b}{a}\right)>\frac{\pi\sqrt{2}}{\sqrt[4]{n^4+2(2m^2+1)n^2-m^2(m^2-2)}}.
    \end{align*}
    As $n\in\Z$, we deduce that the right-hand side is bounded as \eqref{ineq_delicate} holds. Then, coming back to \eqref{condition_conformal_mu}, we deduce that this condition holds if and only if 
    \begin{align*}
        \frac{\pi\sqrt{2}}{Y}\geq 5\sqrt{\frac{n^2+1+Y^2}{-Y^4+n^4+2(2m^2+1)n^2-m^2(m^2-2)}},
    \end{align*}
    or
    \begin{align*}
        2\pi^2\left(-Y^4+n^4+2(2m^2+1)n^2-m^2(m^2-2)\right)\geq 25Y^2(Y^2+n^2+1), 
    \end{align*}
    or
    \begin{align*}
        n^4+2(2m^2+1)n^2-m^2(m^2-2)\geq \left(\frac{25}{2\pi^2}+1\right)Y^4+\frac{25}{2\pi^2}(1+n^2)Y^2.
    \end{align*}
    Therefore, we have
    \begin{align*}
        Y^2&\leq -\frac{25}{4\pi^2}(n^2+1)+\sqrt{\frac{525}{16\pi^4}(n^2+1)^2+\left(\frac{50}{\pi^2}+4\right)\left(n^4+2(2m^2+1)n^2-m^2(m^2-2)\right)}\\
        &=\frac{4(50+\pi^2)(n^4+2(2m^2+1)n^2-m^2(m^2-2))}{25(n^2+1)+\sqrt{525(n^2+1)^2+16\pi^2(50+\pi^2)(n^4+2(2m^2+1)n^2-m^2(m^2-2))}},
    \end{align*}
    and finally
    \begin{align*}
        \log\left(\frac{b}{a}\right)\geq\pi\sqrt{\frac{25(n^2+1)+\sqrt{525(n^2+1)^2+16\pi^2(50+\pi^2)(n^4+2(2m^2+1)n^2-m^2(m^2-2))}}{2(50+\pi^2)(n^4+2(2m^2+1)n^2-m^2(m^2-2))}}.
    \end{align*}

    \textbf{Step 5: Lower Estimate for All Eigenvalues.}

    In this step, we show that if the conformal class is large enough, then
    \begin{align}\label{ineq_min}
        \inf_{n\in\Z}\mu_{m,n}>\frac{\left(2(m+n)^2+\dfrac{\pi^2}{\log^2\left(\frac{b}{a}\right)}\right)\left(2(m-n)^2+\dfrac{\pi^2}{\log^2\left(\frac{b}{a}\right)}\right)}{4(n^2+1)+\dfrac{2\pi^2}{\log^2\left(\frac{b}{a}\right)}}.
    \end{align}
    Therefore, let $\alpha=\dfrac{\pi}{\log\left(\frac{b}{a}\right)}$, and let 
    \begin{align*}
        f(t)&=\frac{(2(m+t)^2+\alpha^2)(2(m-t)^2+\alpha^2)}{2(t^2+1)+\alpha^2}=\frac{4(m^2-t^2)^2+4(m^2+t^2)\alpha^2+\alpha^4}{2(t^2+1)+\alpha^2}\\
        &=\frac{4t^4-4(2m^2-\alpha^2)t^2+4m^4+4m^2\alpha^2+\alpha^4}{2t^2+2+\alpha^2}
    \end{align*}
    As $f$ is coercive, it admits a global minimum on $\R$. Since
    \begin{align*}
        \frac{d}{dt}\left(4(m^2-t^2)^2+4(m^2+t^2)\alpha^2+\alpha^4\right)=-16t(m^2-t^2)+8\alpha^2t=8t\left(2t^2-2m^2+\alpha^2\right)
    \end{align*}
    we get
    \begin{align*}
        f'(t)&=\frac{4t}{(2t^2+2+\alpha^2)^2}\left(2(2t^2-2m^2+\alpha^2)(2t^2+2+\alpha^2)-\left(4t^4-4(2m^2-\alpha^2)t^2+4m^4+4m^2\alpha^2+\alpha^4\right)\right)\\
        &=\frac{4t}{(2t^2+2+\alpha^2)^2}\left(4t^4+4(2+\alpha^2)t^2-4m^2(m^2+2)-4(2m^2-1)\alpha^2+\alpha^4\right).
    \end{align*}
    Notice that for $\alpha=0$, we have
    \begin{align*}
        f'(t)=\frac{2t}{(t^2+1)^2}\left(t^2(t^2+2)-m^2(m^2+2)\right),
    \end{align*}
    which shows that $f'(t)\geq 0$ if and only if $t\geq m$ or $-m\leq t\leq 0$, which implies that $\displaystyle\inf_{t\in\R}f(t)=f(m)=0$. In the general case, we assume that
    \begin{align*}
        -4m^2(m^2+2)-4(2m^2-1)\alpha^2+\alpha^4\leq 0,
    \end{align*}
    which is equivalent to
    \begin{align*}
        \alpha^2\leq 2(2m^2-1)+2\sqrt{(2m^2-1)^2+m^2(m^2+2)},
    \end{align*}
    or
    \begin{align}\label{new_conf_condition}
        \log\left(\frac{b}{a}\right)\geq \frac{\pi}{\sqrt{2(2m^2-1)+2\sqrt{(2m^2-1)^2+m^2(m^2+2)}}}.
    \end{align}
    The roots of 
    \begin{align*}
        Q(X)=4X^2+4(2+\alpha^2)X-4m^2(m^2+2)-4(2m^2-1)\alpha^2+\alpha^4
    \end{align*}
    are given by
    \begin{align*}
        r_{\pm}=-\frac{2+\alpha^2}{2}\pm\frac{1}{2}\sqrt{(2+\alpha)^2+4m^2(m^2+2)+4(2m^2-1)\alpha^2-\alpha^4}.
    \end{align*}
    If
    \begin{align*}
        r&=\sqrt{-\frac{\alpha^2+2}{2}+\frac{1}{2}\sqrt{(\alpha^2+2)^2+4m^2(m^2+2)+4(2m^2-1)\alpha^2-\alpha^4}},
    \end{align*}
    we deduce that $f'(t)\geq 0$ if and only if $-r\leq t\leq 0$ or $t\geq r$. In particular, we have
    \begin{align*}
        \min_{t\in\R}f(t)=f(r)=f(-r).
    \end{align*}
    Then, notice that $r$ can be written more simply as
    \begin{align*}
        r=\sqrt{-\frac{\alpha^2+2}{2}+\sqrt{(m^2+1)^2+2m^2\alpha^2}}.
    \end{align*}
    We have
    \begin{align*}
        r\leq m
    \end{align*}
    if and only if
    \begin{align*}
        2\sqrt{(m^2+1)^2+2m^2\alpha^2}\leq 2(m^2+1)+\alpha^2,
    \end{align*}
    or
    \begin{align*}
        4(m^2+1)^2+8m^2\alpha^2\leq 4(m^2+1)^2+4(m^2+1)\alpha^2+\alpha^4,
    \end{align*}
    which is always true for $m=1$. Assuming now that $m>1$, we have $r\geq m$ if and only if
    \begin{align*}
        8m^2\alpha^2\geq 4(m^2+1)\alpha^2+\alpha^4,
    \end{align*}
    or
    \begin{align*}
        \alpha^2\leq 4(m^2-1),
    \end{align*}
    which gives the condition
    \begin{align}\label{new_conf_condition2}
        \log\left(\frac{b}{a}\right)\geq \frac{\pi}{2\sqrt{m^2-1}}.
    \end{align}
    If this condition holds, as $f$ is decreasing on $[m,r]$, we have
    \begin{align*}
        \inf_{n\in\Z}f(t)=f(m)
    \end{align*}
    provided that $r\leq m+1$ and $f(m)<f(m+1)$. The condition $r\leq m+1$ is equivalent to
    \begin{align*}
        4(m^2+1)^2+8m^2\alpha^2\leq 4((m+1)^2+1)^2+4((m+1)^2+1)\alpha^2+\alpha^4,
    \end{align*}
    which is equivalent to
    \begin{align*}
        \alpha^4-4(m^2-2m-2)\alpha^2+4(2m+1)(2m^2+2m+3)\geq 0.
    \end{align*}
    This condition is empty if $m\geq 1+\sqrt{3}$. Assuming that $m>1+\sqrt{3}$, we find the estimate (as we are interested in the limit when $\alpha\rightarrow 0$)
    \begin{align*}
        \alpha^2\leq 2(m^2-2m-2)-2\sqrt{(m^2-2m-2)^2-(2m+1)(2m^2+2m+3)},
    \end{align*}
    or
    \begin{align}\label{new_conf_condition3}
        \log\left(\frac{b}{a}\right)\geq \frac{\pi}{2(m^2-2m-2)-2\sqrt{(m^2-2m-2)^2-(2m+1)(2m^2+2m+3)}},
    \end{align}
    provided that
    \begin{align*}
        (m^2-2m-2)^2-(2m+1)(2m^2+2m+3)>0,
    \end{align*}
    or
    \begin{align*}
        m^4-8m^3-6m^2+1>0.
    \end{align*}
    Otherwise, the condition on $\alpha$ is empty. Let
    \begin{align*}
        g(t)=t^4-8t^3-6t^2+1.
    \end{align*}
    We have
    \begin{align*}
        g'(t)=4t^3-24t^2-12t=4t(t^2-6t-3)=4t((t-3)^2-12)=4t(t-3-2\sqrt{3})(t-3+2\sqrt{3}).
    \end{align*}
    Therefore, we deduce that $g'(t)<0$ for all $0<t<3+2\sqrt{3}$, and $g'(t)>0$ for all $t>3+2\sqrt{3}$. Therefore, $g$ is strictly decreasing on $[0,2\sqrt{3}]$, and increasing on $\R\cap\ens{t:t\geq 2\sqrt{3}}$. We have $2\sqrt{3}<4$, and $g(8)=-383<0$, while $g(9)=242>0$. Therefore, we only get a non-trivial condition provided that $m\geq 9$ (or $m$ larger than the largest root of $g$ if we do not consider integer values, that can be expressed with cubic roots). Finally we deduce that for $0<a<b<\infty$ such that all conditions \eqref{new_conf_condition}, \eqref{new_conf_condition2} (provided that $m>1$), and \eqref{new_conf_condition3}, we have
    \begin{align}\label{est_eigenvalue}
        \inf_{n\in\Z}f(t)=f(m)=\frac{\left(4m^2+\dfrac{\pi^2}{\log^2\left(\frac{b}{a}\right)}\right)\dfrac{\pi^2}{\log^2\left(\frac{b}{a}\right)}}{2(m^2+1)+\dfrac{\pi^2}{\log^2\left(\frac{b}{a}\right)}},
    \end{align}
    which finally shows that \eqref{ineq_min} holds true.

    \textbf{Step 6: Minimal Eigenvalue.}

    Assume that the bounds \eqref{new_conf_condition}, \eqref{new_conf_condition2} (if $m>1$), and \eqref{new_conf_condition3} from the previous step hold.
    In this step, we show that if the conformal class is large enough,
    \begin{align*}
        \inf_{n\in\Z}\mu_{m,n}=\mu_{m,m},
    \end{align*}
    which is a refinement of the previous step. Thanks to the previous analysis, it is implied by the inequality 
    \begin{align*}
        \frac{4\alpha^2(2m^2+\alpha^2)}{m^2+1+2\alpha^2}\leq \inf_{n\neq m}\frac{(2(m+n)^2+\alpha^2)(2(m-n)^2+\alpha)^2}{4(n^2+1)+2\alpha^2}.
    \end{align*}
    Since the right-hand side is symmetric in $n$, it suffices to study the previous function $f$ on $[m+1,\infty[$. Since $f$ is increasing on this interval, we have
    \begin{align*}
        \inf_{n\neq m}\frac{(2(m+n)^2+\alpha^2)(2(m-n)^2+\alpha)^2}{4(n^2+1)+2\alpha^2}=\frac{\left(2(2m+1)^2+\alpha^2\right)\left(2+\alpha^2\right)}{4(m^2+m+2)+2\alpha^2}.
    \end{align*}
    Finally, we have 
    \begin{align*}
        \frac{4\alpha^2(2m^2+\alpha^2)}{m^2+1+2\alpha^2}\leq \frac{\left(2(2m+1)^2+\alpha^2\right)\left(2+\alpha^2\right)}{4(m^2+m+2)+2\alpha^2}
    \end{align*}
    if and only if (after some elementary computations)
    \begin{align}\label{last_conf_class_cond}
        6\alpha^6+(15m^2+23)\alpha^4+4\left(6m^4+6m^3+5m^2-10m-3\right)\alpha^2\leq 4(2m+1)^2(m^2+1).
    \end{align}
    Assuming without loss of generality that $\alpha\leq 1$, we deduce that the inequality
    \eqref{last_conf_class_cond} holds provided that
    \begin{align*}
        \left(32m^4+24m^3+35m^2-40m+11\right)\alpha^2\leq 4(2m+1)^2(m^2+1),
    \end{align*}
    or
    \begin{align}\label{new_conf_condition4}
        \log\left(\frac{b}{a}\right)\geq \frac{\pi}{2(2m+1)}\sqrt{\frac{32m^4+24m^3+35m^2-40m+11}{m^2+1}},
    \end{align}
    which finally concludes the proof of the theorem if $m=1$ or $m>1$ and $n\neq 0$.

    \textbf{Assumption II: $n=0$ and $m>1$.} In this case, we have
    \begin{align*}
        Q(Y)=Y^4-(2m^2-\mu)Y^2+m^4-\mu.
    \end{align*}
    Its discriminant is given by 
    \begin{align*}
        (2m^2-\mu)^2-4(m^4-\mu)=\mu(\mu-4(m^2-1)).
    \end{align*}
    We need to distinguish three cases.

    \textbf{Case 1: $\mu>4(m^2-1)$.}
    
    If $\mu>4(m^2-1)$, then the previous proof is unchanged. 

    \textbf{Case 2: $\mu=4(m^2-1)$.} Then
    \begin{align*}
        Q(Y)&=Y^4-(2m^2-4(m^2-1))Y^2+m^4-4m^2+4=Y^4-2(m^2-2)Y^2+(m^2-2)^2\\
        &=(Y^2-(m^2-2))^2.
    \end{align*}

    \textbf{Sub-case 1: $m>\sqrt{2}$.} Then $Q$ admits two double roots $\pm \sqrt{m^2-2}$. Therefore, the solution $Y$ takes the form
    \begin{align*}
        Y(t)=\mu_1\,e^{(1+\sqrt{m^2-2})\,t}+\mu_2\,t\,e^{(1+\sqrt{m^2-2})\,t}+\mu_3\,e^{(1-\sqrt{m^2-2})\,t}+\mu_4\,t\,e^{(1-\sqrt{m^2-2})\,t}.
    \end{align*}
    Letting $\lambda_1=1+\sqrt{m^2-2}$ and $\lambda_2=1-\sqrt{m^2-2}$, the boundary conditions shows that the following matrix is not invertible:
    \begin{align*}
        A=\begin{pmatrix}
            a^{\lambda_1} & \log(a)a^{\lambda_1} & a^{\lambda_2} & \log(a)\,a^{\lambda_2}\\
            b^{\lambda_1} & \log(b)b^{\lambda_1} & b^{\lambda_2} & \log(b)\,b^{\lambda_2}\\
            \lambda_1a^{\lambda_1} & (1+\lambda_1\log(a))a^{\lambda_1} & \lambda_2a^{\lambda_1} &   (1+\lambda_2\log(a))a^{\lambda_2}\\
            \lambda_1b^{\lambda_1} & (1+\lambda_1\log(b))b^{\lambda_1} & \lambda_2b^{\lambda_2} & (1+\lambda_2\log(b))b^{\lambda_2}
        \end{pmatrix}.
    \end{align*}
    Let $X,Y\in\C\setminus\ens{0}$ and $\alpha,\beta\in\C$. We have
    \begin{align*}
        &\begin{vmatrix}
            X^{\alpha} & \log(X)X^{\alpha} & X^{\beta} & \log(X)X^{\beta}\\
            Y^{\alpha} & \log(Y)Y^{\alpha} & Y^{\beta} & \log(Y)Y^{\beta}\\
            \alpha X^{\alpha} & (1+\alpha\log(X))X^{\alpha} & \beta X^{\beta} & (1+\beta\log(X))X^{\beta}\\
            \alpha Y^{\alpha} & (1+\alpha\log(Y))Y^{\alpha} & \beta Y^{\beta} & (1+\beta\log(Y))Y^{\beta}
        \end{vmatrix}=-\log(X)X^{\beta}\begin{vmatrix}
            Y^{\alpha} & \log(Y)Y^{\alpha} & Y^{\beta}\\
            \alpha X^{\alpha} & (1+\alpha\log(X))X^{\alpha} & \beta X^{\beta}\\
            \alpha Y^{\alpha} & (1+\alpha\log(Y))Y^{\alpha} & \beta Y^{\beta}
        \end{vmatrix}\\
        &+\log(Y)Y^{\beta}\begin{vmatrix}
            X^{\alpha} & \log(X)X^{\alpha} & X^{\beta}\\
            \alpha X^{\alpha} & (1+\alpha\log(X))X^{\alpha} & \beta X^{\beta}\\
            \alpha Y^{\alpha} & (1+\alpha\log(Y))Y^{\alpha} & \beta Y^{\beta}
        \end{vmatrix}
        -(1+\beta\log(X))X^{\beta}\begin{vmatrix}
            X^{\alpha} & \log(X)X^{\alpha} & X^{\beta}\\
            Y^{\alpha} & \log(Y)Y^{\alpha} & Y^{\beta}\\
            \alpha Y^{\alpha} & (1+\alpha\log(Y))Y^{\alpha} & \beta Y^{\beta}
        \end{vmatrix}\\
        &+(1+\beta\log(Y))Y^{\beta}\begin{vmatrix}
            X^{\alpha} & \log(X)X^{\alpha} & X^{\beta}\\
            Y^{\alpha} & \log(Y)Y^{\alpha} & Y^{\beta}\\
            \alpha X^{\alpha} & (1+\alpha\log(X))X^{\alpha} & \beta X^{\beta}
        \end{vmatrix}\\
        &=-\log(X)X^{\beta}\Big(-\log(Y)Y^{\alpha}\alpha\beta(X^{\alpha}Y^{\beta}-X^{\beta}Y^{\alpha})+(1+\alpha\log(X))X^{\alpha}(\beta-\alpha)Y^{\alpha+\beta}\\
        &-(1+\alpha\log(Y))Y^{\alpha}\left(\beta X^{\beta}Y^{\alpha}-\alpha X^{\alpha}Y^{\beta}\right)\Big)+\log(Y)Y^{\beta}\Big(-\log(X)X^{\alpha}\alpha\beta\left(X^{\alpha}Y^{\beta}-X^{\beta}Y^{\alpha}\right)\\
        &+(1+\alpha\log(X))X^{\alpha}\left(\beta X^{\alpha}Y^{\beta}-\alpha X^{\beta}Y^{\alpha}\right)-(1+\alpha\log(Y))Y^{\alpha}(\beta-\alpha)X^{\alpha+\beta}\Big)\\
        &-\left(1+\beta\log(X)\right)X^{\beta}\Big(-\log(X)X^{\alpha}\left(\beta-\alpha\right)Y^{\alpha+\beta}+\log(Y)Y^{\alpha}\left(\beta X^{\alpha}Y^{\beta}-\alpha X^{\beta}Y^{\alpha}\right)\\
        &-(1+\alpha\log(Y))Y^{\alpha}\left(X^{\alpha}Y^{\beta}-X^{\beta}Y^{\alpha}\right)\Big)
        +(1+\beta\log(Y))Y^{\beta}\Big(-\log(X)X^{\alpha}\left(\beta X^{\beta}Y^{\alpha}-\alpha X^{\alpha}Y^{\beta}\right)\\
        &+\log(Y)Y^{\alpha}\left(\beta-\alpha\right)X^{\alpha+\beta}
        -(1+\alpha\log(X))X^{\alpha}\left(X^{\alpha}Y^{\beta}-X^{\beta}Y^{\alpha}\right)\Big)\\
        &=\log(X)\log(Y)\left(-2\alpha\beta\left(X^{\alpha}Y^{\beta}-X^{\beta}Y^{\alpha}\right)^2+2(\alpha X^{\alpha}Y^{\beta}-\beta X^{\beta}Y^{\alpha})(\beta X^{\alpha}Y^{\beta}-\alpha X^{\beta}Y^{\alpha})\right)\\
        &+(\beta-\alpha)^2\left(\log^2(X)+\log^2(Y)\right)X^{\alpha+\beta}Y^{\alpha+\beta}-\left(X^{\alpha}Y^{\beta}-X^{\beta}Y^{\alpha}\right)^2.
    \end{align*}
    We can rewrite this expression (naming the underlying matrix $B$) as 
    \begin{align*}
        \det(B)&=(\alpha-\beta)^2X^{\alpha+\beta}Y^{\alpha+\beta}\left(\log(X)^2-2\log(X)\log(Y)+\log^2(Y)\right)-(X^{\alpha}Y^{\beta}-X^{\beta}Y^{\alpha})^2\\
        &=(\alpha-\beta)^2X^{\alpha+\beta}Y^{\alpha+\beta}\log^2\left(\frac{Y}{X}\right)-\left(X^{\alpha}Y^{\beta}-X^{\beta}Y^{\alpha}\right)^2.
    \end{align*}
    Therefore, we get if
    \begin{align*}
    \left\{\begin{alignedat}{1}
        \alpha&=\lambda_1=1+\sqrt{m^2-2}=1+\mu\\
        \beta&=\lambda_2=1-\sqrt{m^2-2}=1-\mu,
    \end{alignedat}\right.
    \end{align*}
    we get 
    \begin{align*}
        \det(A)&=4\mu^2a^{2}b^2\log^2\left(\frac{b}{a}\right)-\left(a^{1+\mu}b^{1-\mu}-a^{1-\mu}b^{1+\mu}\right)^2\\
        &=a^4\left(4\mu^2\left(\frac{b}{a}\right)^2\log^2\left(\frac{b}{a}\right)-\left(\left(\frac{b}{a}\right)^{\mu}-\bigg(\frac{a}{b}\bigg)^{\mu}\right)^2\right)\\
        &=a^4\left(2\mu\left(\frac{b}{a}\right)\log\left(\frac{b}{a}\right)^{\mu}+\left(\left(\frac{b}{a}\right)-\bigg(\frac{a}{b}\bigg)^{\mu}\right)\right)\left(2\mu\left(\frac{b}{a}\right)\log\left(\frac{b}{a}\right)-\left(\left(\frac{b}{a}\right)^{\mu}-\bigg(\frac{a}{b}\bigg)^{\mu}\right)\right).
    \end{align*}
    Therefore, we have $\det(A)=0$ if and only if
    \begin{align*}
        2\mu\left(\frac{b}{a}\right)\log\left(\frac{b}{a}\right)-\left(\left(\frac{b}{a}\right)^{\mu}-\bigg(\frac{a}{b}\bigg)^{\mu}\right)=0.
    \end{align*}
    Therefore, we let $\alpha>0$ and we introduce on $\R_+$ the function
    \begin{align*}
        f(x)=2\alpha (1+x)^{\alpha+1}\log(1+x)-\left((1+x)^{2\alpha}-1\right).
    \end{align*}
    We have
    \begin{align*}
        f'(x)&=2\alpha(\alpha+1)(1+x)^{\alpha}\log(1+x)+2\alpha(1+x)^{\alpha}-2\alpha(1+x)^{2\alpha-1}\\
        &=2\alpha(1+x)^{\alpha}\left((\alpha+1)\log(1+x)+1-(1+x)^{\alpha-1}\right).
    \end{align*}
    If $\alpha\leq 1$, we deduce that $f$ is strictly increasing, and since $f(0)=0$, we deduce that $f(x)>0$ for all $x>0$. If $\alpha>1$ (or $m>\sqrt{3}$), let $\beta=\alpha-1>0$, and introduce the function
    \begin{align*}
        g(x)=(\beta+2)\log(1+x)+1-(1+x)^{\beta}.
    \end{align*}
    Make the change of variable $1+x=e^t$. Then, we have
    \begin{align*}
        g(t)=(\beta+2)t+1-e^{\beta t}.
    \end{align*}
    We have
    \begin{align*}
        g'(t)=\beta+2-\beta e^{\beta t}.
    \end{align*}
    Therefore, if 
    \begin{align*}
        \gamma=\frac{\log\left(\frac{\beta+2}{\beta}\right)}{\beta},
    \end{align*}
    we deduce that $g$ is strictly decreasing on $[0,\gamma]$ and strictly increasing on $[\gamma,\infty[$. We have
    \begin{align*}
        g(\gamma)=(\beta+2)\gamma+1-\frac{(\beta+2)}{\beta}=\frac{\beta+2}{\beta}\left(\beta+\log\left(1+\frac{2}{\beta}\right)-1\right)>0.
    \end{align*}
    Therefore, we deduce that there exists $\delta>\gamma$ such that $f$ is strictly increasing on $[0,\delta]$ and strictly decreasing on $[\delta,\infty[$. Since $f(0)=0$, $f$ admits a unique zero $r_0\in [e^{\gamma}-1,\infty[$. We have
    \begin{align*}
        f(t)&=2\alpha\,t\, e^{(\alpha+1)t}-(e^{2\alpha t}-1)=1+e^{(\alpha+1)t}\left(2\alpha t-e^{(\alpha-1)t}\right)\\
        &<1+e^{(\alpha+1)t}\left(2\alpha t-1-(\alpha-1)t-\frac{(\alpha-1)^2}{2}t^2\right)\\
        &=1-\frac{1}{2}e^{(\alpha+1)t}\left((\alpha-1)^2t^2-2(\alpha+1)t+2\right)<0
    \end{align*}
    provided that $(\alpha-1)^2t^2-2(\alpha+1)t+2\geq 2$, or $t\geq \frac{2(\alpha+1)}{(\alpha-1)^2}$. Therefore, we obtain the estimate 
    \begin{align*}
        \delta<e^{\frac{2(\alpha+1)}{(\alpha-1)^2}}
    \end{align*}
    Since 
    \begin{align*}
        1+x=\frac{b}{a},
    \end{align*}
    we conclude that there are no solutions provided that
    \begin{align*}
        \log\left(\frac{b}{a}\right)\geq \log\left(\frac{e^{2(\alpha+1)}}{(\alpha-1)^2}-1\right),
    \end{align*}
    which is implied by the stronger condition
    \begin{align}\label{new_conf_condition_small range}
        \log\left(\frac{b}{a}\right)\geq \frac{2\sqrt{m^2-2}}{(\sqrt{m^2-2}-1)^2}.
    \end{align}
    Notice that this condition is trivial provided that 
    \begin{align*}
        \frac{2\sqrt{m^2-2}}{(\sqrt{m^2-2}-1)^2}\leq 1,
    \end{align*}
    or
    \begin{align}
        m^4-18m^2+33\geq 0,
    \end{align}
    which is satisfied on $[\sqrt{9+4\sqrt{3}},\infty[$.

    To conclude, there are no solutions for $1<m\leq \sqrt{3}$ and $m\geq \sqrt{9+4\sqrt{3}}$, and no solution provided that \eqref{new_conf_condition_small range} if $\sqrt{3}<m<\sqrt{9+4\sqrt{3}}<4$.

    We can now move to the next sub-case.

    \textbf{Sub-case 2: $m=\sqrt{2}$}. Then $Q(Y)=Y^4$, and $P$ admits $r=1$ as quadruple root. Therefore, the solution $Y$ is given by
    \begin{align*}
        Y(t)=\mu_1e^t+\mu_2\,t\,e^t+\mu_3\,t^2\,e^t+\mu_4\,t^3\,e^t.
    \end{align*}
    The boundary conditions show that the following matrix is not invertible:
    \begin{align*}
        A=\begin{pmatrix}
            a & \log(a)a & \log^2(a)a & \log^3(a)a\\
            b & \log(b)b & \log^2(b)b & \log^3(b)b\\
            a & (1+\log(a))a & (2\log(a)+\log^2(a))a & (3\log^2(a)+\log^3(a))a\\
            b & (1+\log(b))b & (2\log(b)+\log^2(b))b & (3\log^2(b)+\log^3(b))b
        \end{pmatrix}.
    \end{align*}
    A computation shows that 
    \begin{align*}
        \det(A)&=-(ab)^2\log^4\left(\frac{b}{a}\right),
    \end{align*}
    which is clearly never zero for all $0<a<b<\infty$.

    \textbf{Sub-case 3: $1<m<\sqrt{2}$}. Then
    \begin{align*}
        Q(Y)=(Y^2+2-m^2)^2,
    \end{align*}
    which shows that $Q$ admits two double complex roots $\pm\,i\,\sqrt{2-m^2}$. If $\lambda_1=\sqrt{2-m^2}$, we deduce that the roots of $P(X)=Q(X+1)$ are given by
    \begin{align*}
        \left\{\begin{alignedat}{2}
            r_1&=r_2&&=1+i\,\lambda_1\\
            r_3&=r_4&&=1-i\,\lambda_1
        \end{alignedat}\right.
    \end{align*}
    Therefore, $Y$ is given by
    \begin{align*}
        Y(t)=\mu_1\,e^{(1+i\,\lambda_1)t}+\mu_2\,t\,e^{(1+i\,\lambda_1)}+\mu_4\,e^{(1-i\,\lambda_1)t}+\mu_4\,t\,e^{(1-i\,\lambda_1)t}.
    \end{align*}
    The boundary conditions show that the following matrix is degenerate
    \begin{align*}
        A=\begin{pmatrix}
            a^{1+i\,\lambda_1} & \log(a)a^{1+i\,\lambda_1} & a^{1-i\,\lambda_1} & \log(a)a^{1-i\,\lambda_1}\\
            b^{1+i\,\lambda_1} & \log(b)b^{1+i\,\lambda_1} & b^{1-i\,\lambda_1} & \log(b)b^{1-i\,\lambda_1}\\
            (1+i\,\lambda_1)a^{1+i\,\lambda_1} & (1+(1+i\,\lambda_1)\log(a))a^{1+i\,\lambda_1} & (1-i\,\lambda_1)a^{1-i\,\lambda_1} & \left(1+(1-i\,\lambda_1)\log(a)\right)a^{1-i\,\lambda_1}\\
            (1+i\,\lambda_1)b^{1+i\,\lambda_1} & (1+(1+i\,\lambda_1)\log(b))b^{1+i\,\lambda_1} & (1-i\,\lambda_1)b^{1-i\,\lambda_1} & \left(1+(1-i\,\lambda_1)\log(b)\right)b^{1-i\,\lambda_1}
        \end{pmatrix}.
    \end{align*}
    We have
    \begin{align*}
        \det(A)&=-\bigg(4 \, a^{2 i \, \lambda_{1}} b^{2 i \, \lambda_{1}} \lambda_{1}^{2} \log\left(a\right)^{2} - 8 \, a^{2 i \, \lambda_{1}} b^{2 i \, \lambda_{1}} \lambda_{1}^{2} \log\left(a\right) \log\left(b\right) + 4 \, a^{2 i \, \lambda_{1}} b^{2 i \, \lambda_{1}} \lambda_{1}^{2} \log\left(b\right)^{2}\\
        &- 2 \, a^{2 i \, \lambda_{1}} b^{2 i \, \lambda_{1}}
        + a^{4 i \, \lambda_{1}} + b^{4 i \, \lambda_{1}}\bigg) a^{-2 i \, \lambda_{1} + 2} b^{-2 i \, \lambda_{1} + 2}=0.
    \end{align*}
    Therefore, we have the following identity
    \begin{align*}
        &4 \, a^{2 i \, \lambda_{1}} b^{2 i \, \lambda_{1}} \lambda_{1}^{2} \log\left(a\right)^{2} - 8 \, a^{2 i \, \lambda_{1}} b^{2 i \, \lambda_{1}} \lambda_{1}^{2} \log\left(a\right) \log\left(b\right) + 4 \, a^{2 i \, \lambda_{1}} b^{2 i \, \lambda_{1}} \lambda_{1}^{2} \log\left(b\right)^{2}\\
        &- 2 \, a^{2 i \, \lambda_{1}} b^{2 i \, \lambda_{1}}
        + a^{4 i \, \lambda_{1}} + b^{4 i \, \lambda_{1}}=0,
    \end{align*}
    that is more simply rewritten as
    \begin{align*}
        4\lambda_1^2a^{2i\lambda_1}b^{2i\lambda_1}\log^2\left(\frac{b}{a}\right)+\left(b^{2i\lambda_1}-a^{2i\lambda_1}\right)^2=0,
    \end{align*}
    or
    \begin{align*}
        4\lambda_1^2\left(\frac{b}{a}\right)^{2i\lambda_1}\log^2\left(\frac{b}{a}\right)+\left(\left(\frac{b}{a}\right)^{2i\lambda_1}-1\right)^2=0.
    \end{align*}
    Therefore, we deduce that
    \begin{align}\label{impossible_system}
        \left\{\begin{alignedat}{2}
            &4\lambda_1^2\cos\left(2\lambda_1\log\left(\frac{b}{a}\right)\right)\log^2\left(\frac{b}{a}\right)+\left(\cos\left(2\lambda_1\log\left(\frac{b}{a}\right)\right)-1\right)^2-\sin^2\left(2\lambda_1\log\left(\frac{b}{a}\right)\right)=0\\
            &2\lambda_1^2\sin\left(2\lambda_1\log\left(\frac{b}{a}\right)\right)\log^2\left(\frac{b}{a}\right)+\cos\left(2\lambda_1\log\left(\frac{b}{a}\right)\right)\sin\left(2\lambda_1\log\left(\frac{b}{a}\right)\right)=0.
        \end{alignedat}\right.
    \end{align}
    If $2\lambda_1\log\left(\frac{b}{a}\right)\notin \pi\Z$, we get
    \begin{align*}
        2\lambda_1^2\log^2\left(\frac{b}{a}\right)+\cos\left(2\lambda_1\log\left(\frac{b}{a}\right)\right)=0,
    \end{align*}
    which is impossible provided that
    \begin{align*}
        2\lambda_1^2\log^2\left(\frac{b}{a}\right)>1,
    \end{align*}
    or
    \begin{align*}
        \log\left(\frac{b}{a}\right)>\frac{1}{\sqrt{2(2-m^2)}}.
    \end{align*}
    Plugging this identity in the first equality of \eqref{impossible_system} shows that 
    \begin{align*}
        &0=-8\lambda_1^4\log^4\left(\frac{b}{a}\right)+\left(2\lambda_1^2\log^2\left(\frac{b}{a}\right)-1\right)^2+4\lambda_1^4\log^4\left(\frac{b}{a}\right)-1\\
        &=-4\lambda_1^2\log^2\left(\frac{b}{a}\right),
    \end{align*}
    which is never $0$. Therefore, we deduce that either $2\lambda_1\log\left(\frac{b}{a}\right)=0$ (mod $2\pi$) or $2\lambda_1\log\left(\frac{b}{a}\right)=\pi$ (mod $2\pi$).
    In the first alternative, the first identity of \eqref{impossible_system} becomes
    \begin{align*}
        4\lambda_1^2\log^2\left(\frac{b}{a}\right)=0,
    \end{align*}
    which is never verified, while in the second alternative, it becomes
    \begin{align*}
        -4\lambda_1^2\log^2\left(\frac{b}{a}\right)+4=0,
    \end{align*}
    or
    \begin{align}\label{cond}
        \log\left(\frac{b}{a}\right)=\frac{1}{\sqrt{2-m^2}}.
    \end{align}
    However, this identity implies that $2\lambda_1\log\left(\frac{b}{a}\right)=2\neq \pi$, so we get a contradiction. 

    Finally, we can move to the third possibility.

    \textbf{Case 3: $\mu<4(m^2-1)$}.

    This is the expected case since $\mu$ must go to $0$ as the conformal class degenerates. Then, $Q(\sqrt{Y})$ admits the roots
    \begin{align*}
        \frac{1}{2}\left(2m^2-\mu\pm\,i\,\sqrt{\mu(4(m^2-1)-\mu)}\right).
     \end{align*}
    Let $a,b\in\R$ be arbitrary, and let $\alpha,\beta\in\R$ such that 
    \begin{align*}
        (\alpha+i\,\beta)^2=a+i\,b.
    \end{align*}
    We get
    \begin{align*}
        \left\{\begin{alignedat}{1}
            \alpha^2-\beta^2=a\\
            2\alpha\beta=b
        \end{alignedat}\right.
    \end{align*}
    which implies that
    \begin{align*}
        \alpha^4-a\,\alpha^2-\frac{b^2}{4}=0,
    \end{align*}
    and since $\alpha^2\geq 0$, we deduce that
    \begin{align*}
        \alpha^2=\frac{1}{2}\left(a+\sqrt{a^2+b^2}\right),
    \end{align*}
    and finally
    \begin{align*}
    \left\{\begin{alignedat}{1}
        \alpha&=\sqrt{\frac{1}{2}\left(a+\sqrt{a^2+b^2}\right)}\\
        \beta&=\mathrm{sgn}(b)\sqrt{\frac{1}{2}\left(-a+\sqrt{a^2+b^2}\right)}
        \end{alignedat}\right.
    \end{align*}
    Therefore, taking $a=\dfrac{1}{2}(2m^2-\mu)$ and $b=\dfrac{1}{2}\sqrt{\mu(4(m^2-1)-\mu)}$, we get
    \begin{align*}
        a^2+b^2=\frac{1}{4}\left(4m^4-4m^2\mu+\mu^2+4(m^2-1)\mu-\mu^2\right)=m^4-\mu.
    \end{align*}
    Therefore, we get
    \begin{align}\label{alpha_beta}
        \left\{\begin{alignedat}{1}
            \alpha&=\frac{1}{2}\sqrt{\left(2\sqrt{m^4-\mu}+2m^2-\mu\right)}\\
            \beta&=\frac{1}{2}\sqrt{\left(2\sqrt{m^4-\mu}-2m^2+\mu\right)},
        \end{alignedat}\right.
    \end{align}
    while obviously
    \begin{align*}
        \left(\alpha-i\,\beta\right)^2=2m^2-\mu-i\,\sqrt{\mu(4(m^2-1)-\mu)}.
    \end{align*}
    Finally, we deduce that the roots of $P$ are given by
    \begin{align*}
        \left\{\begin{alignedat}{1}
            r_1=1+\alpha+i\,\beta\\
            r_2=1-\alpha-i\,\beta\\
            r_3=1+\alpha-i\,\beta\\
            r_4=1-\alpha+i\,\beta
        \end{alignedat}\right.
    \end{align*}
    If $\lambda_1=\alpha+i\,\beta$ and $\lambda_2=\alpha-i\,\beta$, the boundary conditions shows that the following matrix is not invertible
    \begin{align*}
        A=\begin{pmatrix}
            a^{1+\lambda_1} & a^{1-\lambda_1} & a^{1+\lambda_2} & a^{1-\lambda_2}\\
            b^{1+\lambda_1} & b^{1-\lambda_1} & b^{1+\lambda_2} & b^{1-\lambda_2}\\
            (1+\lambda_1)a^{1+\lambda_1} & (1-\lambda_1)a^{1-\lambda_1} & (1+\lambda_2)a^{1+\lambda_2} & (1-\lambda_2)a^{1-\lambda_2}\\
            (1+\lambda_1)b^{1+\lambda_1} & (1-\lambda_1)b^{1-\lambda_1} & (1+\lambda_2)b^{1+\lambda_2} & (1-\lambda_2)b^{1-\lambda_2}
        \end{pmatrix}.
    \end{align*}
    We have
    \begin{align*}
        \det(A)&={\left(a^{\lambda_{2}} b^{\lambda_{1}} \lambda_{1} - a^{\lambda_{1}} b^{\lambda_{2}} \lambda_{1} + a^{\lambda_{2}} b^{\lambda_{1}} \lambda_{2} - a^{\lambda_{1}} b^{\lambda_{2}} \lambda_{2} + a^{\lambda_{1} + \lambda_{2}} \lambda_{1} - b^{\lambda_{1} + \lambda_{2}} \lambda_{1} - a^{\lambda_{1} + \lambda_{2}} \lambda_{2} + b^{\lambda_{1} + \lambda_{2}} \lambda_{2}\right)}\\
        &\times {\left(a^{\lambda_{2}} b^{\lambda_{1}} \lambda_{1} - a^{\lambda_{1}} b^{\lambda_{2}} \lambda_{1} + a^{\lambda_{2}} b^{\lambda_{1}} \lambda_{2} - a^{\lambda_{1}} b^{\lambda_{2}} \lambda_{2} - a^{\lambda_{1} + \lambda_{2}} \lambda_{1} + b^{\lambda_{1} + \lambda_{2}} \lambda_{1} + a^{\lambda_{1} + \lambda_{2}} \lambda_{2} - b^{\lambda_{1} + \lambda_{2}} \lambda_{2}\right)}\\
        &\times a^{-\lambda_{1} - \lambda_{2} + 2} b^{-\lambda_{1} - \lambda_{2} + 2}.
    \end{align*}
    Therefore, we either have
    \begin{align}\label{end_structure1}
        a^{\lambda_{2}} b^{\lambda_{1}} \lambda_{1} - a^{\lambda_{1}} b^{\lambda_{2}} \lambda_{1} + a^{\lambda_{2}} b^{\lambda_{1}} \lambda_{2} - a^{\lambda_{1}} b^{\lambda_{2}} \lambda_{2} + a^{\lambda_{1} + \lambda_{2}} \lambda_{1} - b^{\lambda_{1} + \lambda_{2}} \lambda_{1} - a^{\lambda_{1} + \lambda_{2}} \lambda_{2} + b^{\lambda_{1} + \lambda_{2}} \lambda_{2}=0
    \end{align}
    or
    \begin{align}\label{end_structure2}
        a^{\lambda_{2}} b^{\lambda_{1}} \lambda_{1} - a^{\lambda_{1}} b^{\lambda_{2}} \lambda_{1} + a^{\lambda_{2}} b^{\lambda_{1}} \lambda_{2} - a^{\lambda_{1}} b^{\lambda_{2}} \lambda_{2} - a^{\lambda_{1} + \lambda_{2}} \lambda_{1} + b^{\lambda_{1} + \lambda_{2}} \lambda_{1} + a^{\lambda_{1} + \lambda_{2}} \lambda_{2} - b^{\lambda_{1} + \lambda_{2}} \lambda_{2}=0.
    \end{align}
    Dividing these equation by $a^{\lambda_1+\lambda_2}$ shows that 
    \begin{align}\label{end_structure12}
        &\frac{1}{a^{\lambda_1+\lambda_2}}\left(a^{\lambda_{2}} b^{\lambda_{1}} \lambda_{1} - a^{\lambda_{1}} b^{\lambda_{2}} \lambda_{1} + a^{\lambda_{2}} b^{\lambda_{1}} \lambda_{2} - a^{\lambda_{1}} b^{\lambda_{2}} \lambda_{2} + a^{\lambda_{1} + \lambda_{2}} \lambda_{1} - b^{\lambda_{1} + \lambda_{2}} \lambda_{1} - a^{\lambda_{1} + \lambda_{2}} \lambda_{2} + b^{\lambda_{1} + \lambda_{2}} \lambda_{2}\right)\nonumber\\
        &=\lambda_1\left(\left(\frac{b}{a}\right)^{\lambda_1}-\left(\frac{b}{a}\right)^{\lambda_2}+1-\left(\frac{b}{a}\right)^{\lambda_1+\lambda_2}\right)+\lambda_2\left(\left(\frac{b}{a}\right)^{\lambda_1}-\left(\frac{b}{a}\right)^{\lambda_2}-1+\left(\frac{b}{a}\right)^{\lambda_1+\lambda_2}\right)\nonumber\\
        &=(\lambda_1+\lambda_2)\left(\left(\frac{b}{a}\right)^{\lambda_1}-\left(\frac{b}{a}\right)^{\lambda_2}\right)-\left(\lambda_1-\lambda_2\right)\left(\left(\frac{b}{a}\right)^{\lambda_1+\lambda_2}-1\right),
    \end{align}
    while the second identity becomes
    \begin{align}\label{end_structure22}
        &\frac{1}{a^{\lambda_1+\lambda_2}}\left(a^{\lambda_{2}} b^{\lambda_{1}} \lambda_{1} - a^{\lambda_{1}} b^{\lambda_{2}} \lambda_{1} + a^{\lambda_{2}} b^{\lambda_{1}} \lambda_{2} - a^{\lambda_{1}} b^{\lambda_{2}} \lambda_{2} - a^{\lambda_{1} + \lambda_{2}} \lambda_{1} + b^{\lambda_{1} + \lambda_{2}} \lambda_{1} + a^{\lambda_{1} + \lambda_{2}} \lambda_{2} - b^{\lambda_{1} + \lambda_{2}} \lambda_{2}\right)\nonumber\\
        &=\lambda_1\left(\left(\frac{b}{a}\right)^{\lambda_1}-\left(\frac{b}{a}\right)^{\lambda_2}-1+\left(\frac{b}{a}\right)^{\lambda_1+\lambda_2}\right)+\lambda_2\left(\left(\frac{b}{a}\right)^{\lambda_1}-\left(\frac{b}{a}\right)^{\lambda_2}+1-\left(\frac{b}{a}\right)^{\lambda_1+\lambda_2}\right)\nonumber\\
        &=(\lambda_1+\lambda_2)\left(\left(\frac{b}{a}\right)^{\lambda_1}-\left(\frac{b}{a}\right)^{\lambda_2}\right)+(\lambda_1-\lambda_2)\left(\left(\frac{b}{a}\right)^{\lambda_1+\lambda_2}-1\right).
    \end{align}
    Recalling that $\lambda_1=\alpha+i\,\beta$ and $\lambda_2=\alpha-i\,\beta$, we deduce that 
    \begin{align*}
        &(\lambda_1+\lambda_2)\left(\left(\frac{b}{a}\right)^{\lambda_1}-\left(\frac{b}{a}\right)^{\lambda_2}\right)-\left(\lambda_1-\lambda_2\right)\left(\left(\frac{b}{a}\right)^{\lambda_1+\lambda_2}-1\right)\\
        &=4\,i\,\beta\left(\frac{b}{a}\right)^{\alpha}\sin\left(\beta\log\left(\frac{b}{a}\right)\right)-2\,i\,\beta\left(\left(\frac{b}{a}\right)^{2\alpha}-1\right),
    \end{align*}
    and since $\beta\neq 0$, we get that \eqref{end_structure1} is equivalent to
    \begin{align}\label{end_structure13}
        2\left(\frac{b}{a}\right)^{\alpha}\sin\left(\beta\log\left(\frac{b}{a}\right)\right)=\left(\frac{b}{a}\right)^{2\alpha}-1,
    \end{align}
    while the second equation \eqref{end_structure2} is equivalent to
    \begin{align}\label{end_structure23}
        2\left(\frac{b}{a}\right)^{\alpha}\sin\left(\beta\log\left(\frac{b}{a}\right)\right)=-\left(\left(\frac{b}{a}\right)^{2\alpha}-1\right),
    \end{align}
    Recall that by \eqref{alpha_beta}, we have
    \begin{align}\label{new_alpha_beta}
        \left\{\begin{alignedat}{1}
            \alpha&=\frac{1}{2}\sqrt{\left(2\sqrt{m^4-\mu}+2m^2-\mu\right)}\\
            \beta&=\frac{1}{2}\sqrt{\left(2\sqrt{m^4-\mu}-2m^2+\mu\right)},
        \end{alignedat}\right.
    \end{align}
    Let also $R=\log\left(\dfrac{b}{a}\right)$.

    \textbf{Second Equality.} We look for the first positive zero of $f:[0,4(m^2-1)]\rightarrow\R$
    \begin{align*}
        f(t)=2\,e^{\frac{R}{2}\sqrt{\left(2\sqrt{m^4-t}+2m^2-t\right)}}\sin\left(\frac{R}{2}\sqrt{\left(2\sqrt{m^4-t}-2m^2+t\right)}\right)+\left(e^{R\sqrt{\left(2\sqrt{m^4-t}+2m^2-t\right)}}-1\right).
    \end{align*}
    We have $f(t)>0$ on $[0,\mu]$ provided that 
    \begin{align*}
        \sqrt{\left(2\sqrt{m^4-\mu}-2m^2+\mu\right)}\leq \frac{2\pi}{R},
    \end{align*}
    or
    \begin{align*}
        2\sqrt{m^4-\mu}\leq \frac{4\pi^2}{R^2}+2m^2-\mu,
    \end{align*}
    which is equivalent to
    \begin{align*}
        4(m^4-\mu)\leq 4m^4+\frac{16m^2\pi^2}{R^2}+\frac{16\pi^4}{R^4}-2\left(2m^2+\frac{4\pi^2}{R^2}\right)\mu+\mu^2,
    \end{align*}
    or
    \begin{align*}
        \mu^2-4\left((m^2-1)+\frac{2\pi^2}{R^2}\right)\mu+\frac{16m^2\pi^2}{R^2}+\frac{16\pi^4}{R^4}\leq 0.
    \end{align*}
    The discriminant of this quadratic polynomial in $\mu$ is given by 
    \begin{align*}
        16(m^2-1)^2+\frac{64(m^2-1)\pi^2}{R^2}+\frac{64\pi^4}{R^4}-\frac{64m^2\pi^2}{R^2}-\frac{64\pi^4}{R^4}=16\left((m^2-1)^2-\frac{4\pi^2}{R^2}\right).
    \end{align*}
    Assuming that 
    \begin{align*}
        \log\left(\frac{b}{a}\right)\geq \frac{2\pi}{m^2-1},
    \end{align*}
    we deduce that 
    \begin{align*}
        2\left((m^2-1)+\frac{2\pi^2}{R^2}\right)-2\sqrt{(m^2-1)^2-\frac{4\pi^2}{R^2}}< \mu<2\left((m^2-1)+\frac{2\pi^2}{R^2}\right)+2\sqrt{(m^2-1)^2-\frac{4\pi^2}{R^2}}.
    \end{align*}
    In particular, we get in this alternative
    \begin{align*}
        \mu_{m,0}&> 2\left((m^2-1)+\frac{2\pi^2}{R^2}\right)-2\sqrt{(m^2-1)^2-\frac{4\pi^2}{R^2}}\\
        &=\frac{2\left((m^2-1)+\dfrac{2\pi^2}{R^2}\right)^2-2\left((m^2-1)^2-\dfrac{4\pi^2}{R^2}\right)}{\left((m^2-1)+\dfrac{2\pi^2}{R^2}\right)+\sqrt{(m^2-1)^2-\dfrac{4\pi^2}{R^2}}}
        =\frac{\dfrac{8m^2\pi^2}{R^2}+\dfrac{8\pi^4}{R^4}}{\left((m^2-1)+\dfrac{2\pi^2}{R^2}\right)+\sqrt{(m^2-1)^2-\dfrac{4\pi^2}{R^2}}}
    \end{align*}
    or
    \begin{align}\label{alternativeI}
        \mu_{m,n}&>\frac{\left(8m^2+\dfrac{8\pi^2}{\log^2\left(\frac{b}{a}\right)}\right)\dfrac{\pi^2}{\log^2\left(\frac{b}{a}\right)}}{\left((m^2-1)+\dfrac{2\pi^2}{\log^2\left(\frac{b}{a}\right)}\right)+\sqrt{(m^2-1)^2-\dfrac{4\pi^2}{\log^2\left(\frac{b}{a}\right)}}}.
    \end{align}
 
    \textbf{First Equality.} Then, we look for the first positive zero of the function $f:[0,4(m^2-1)]\rightarrow\R$ such that for all $t\in [0,4(m^2-1)]$,
    \begin{align*}
        f(t)=2\,e^{\frac{R}{2}\sqrt{\left(2\sqrt{m^4-t}+2m^2-t\right)}}\sin\left(\frac{R}{2}\sqrt{\left(2\sqrt{m^4-t}-2m^2+t\right)}\right)-\left(e^{R\sqrt{\left(2\sqrt{m^4-t}+2m^2-t\right)}}-1\right).
    \end{align*}
    It it unclear how to directly estimate the first zero of this function, so we will have to differentiate it. Make the change of variable
    \begin{align*}
        \sqrt{m^4-t}=2s,
    \end{align*}
    or
    \begin{align*}
        t=m^4-4s^2.
    \end{align*}
    We get
    \begin{align*}
        &2\sqrt{m^4-t}+2m^2-t=4s+2m^2-(m^4-4s^2)=4s(s+1)-(m^2-1)^2+1\\
        &2\sqrt{m^4-t}-2m^2+t=4s-2m^2+m^4-4s^2=4s(1-s)+(m^2-1)^2-1,
    \end{align*}
    and finally, if $4p=(m^2-1)^2-1$, we get
    \begin{align*}
        g(s)=f(m^4-4s^2)=2e^{R\sqrt{s(s+1)-p}}\sin\left(R\sqrt{p-s(s-1)}\right)-\left(e^{2R\sqrt{s(s+1)-p}}-1\right).
    \end{align*}
    Notice that the domain of definition of $s$ is 
    \begin{align*}
        \left[\frac{1}{2}\sqrt{m^4-4(m^2-1)},\frac{1}{2}m^2\right]=\left[\frac{1}{2}|m^2-2|,\frac{1}{2}m^2\right].
    \end{align*}
    Now, we have
    \begin{align*}
        &g'(s)=\frac{R(2s+1)}{\sqrt{s(s+1)-p}}e^{R\sqrt{s(s+1)-p}}\sin\left(R\sqrt{p-s(s-1)}\right)\\
        &+\frac{R(-2s+1)}{\sqrt{p-s(s-1)}}e^{R\sqrt{s(s+1)-p}}\cos\left(R\sqrt{p-s(s-1)}\right)-\frac{2R(2s+1)}{\sqrt{s(s+1)-p}}e^{2R\sqrt{s(s+1)-p}}.
    \end{align*}
    Notice that 
    \begin{align*}
        f(0)=-\left(e^{R\sqrt{2\sqrt{m^4}+2m^2}}-1\right)=-\left(e^{2m^2R}-1\right)<0,
    \end{align*}
    and since $t=m^4-4s^2$, we need to give an \emph{upper} estimate of the first zero $|m^2-2|/2<s_1<m^2/2$ of $g$. Notice that provided that $0\leq R\sqrt{p-s(s-1)}\leq \pi/2$, and $1/2\leq s\leq m^2/2$, we have $g'(s)<0$. Recalling that $p=\dfrac{1}{4}((m^2-1)^2-1)$, we deduce that
    \begin{align*}
        0\leq R\sqrt{p-s(s-1)}\leq \frac{\pi}{2}
    \end{align*}
    if and only if
    \begin{align*}
        \frac{1}{4}((m^2-1)^2-1)-s(s-1)\leq \frac{\pi^2}{4R^2},
    \end{align*}
    or
    \begin{align*}
        s^2-s-\frac{1}{4}((m^2-1)^2-1)+\frac{\pi^2}{4R^2}\geq 0.
    \end{align*}
    The discriminant of this polynomial function is given by
    \begin{align*}
        (m^2-1)^2-\frac{\pi^2}{R^2}.
    \end{align*}
    Assuming that 
    \begin{align*}
        \log\left(\frac{b}{a}\right)\geq \frac{\pi}{m^2-1},
    \end{align*}
    we deduce that
    \begin{align*}
        s\geq \frac{1}{2}+\frac{1}{2}\sqrt{(m^2-1)^2-\frac{\pi^2}{R^2}}.
    \end{align*}
    In other words the first zero $s_1<m^2/2$ is such that
    \begin{align*}
        s_1<\frac{1}{2}+\frac{1}{2}\sqrt{(m^2-1)^2-\frac{\pi^2}{R^2}}.
    \end{align*}
    Therefore, the first positive zero $t_1>0$ of $f$ is such that
    \begin{align*}
        \sqrt{m^4-t_1}<1+\sqrt{(m^2-1)^2-\frac{\pi^2}{R^2}},
    \end{align*}
    or
    \begin{align*}
        m^4-t_1<1+(m^2-1)^2-\frac{\pi^2}{R^2}+2\sqrt{(m^2-1)^2-\frac{\pi^2}{R^2}}=m^4-2m^2+1-\frac{\pi^2}{R^2}+2\sqrt{(m^2-1)^2-\frac{\pi^2}{R^2}},
    \end{align*}
    and finally
    \begin{align*}
        t_1&>2m^2-1+\frac{\pi^2}{R^2}-2\sqrt{(m^2-1)^2-\frac{\pi^2}{R^2}}=\frac{4m^4-4m^2+1+2(2m^2-1)\dfrac{\pi^2}{R^2}+\dfrac{\pi^4}{R^4}-4(m^2-1)^2+\dfrac{4\pi^2}{R^2}}{2m^2-1+\dfrac{\pi^2}{R^2}+2\sqrt{(m^2-1)^2-\dfrac{\pi^2}{R^2}}}\\
        &=\dfrac{4m^2-3+2(2m^2+1)\dfrac{\pi^2}{R^2}+\dfrac{\pi^4}{R^4}}{2m^2-1+\dfrac{\pi^2}{R^2}+2\sqrt{(m^2-1)^2-\dfrac{\pi^2}{R^2}}}=\dfrac{4m^2-3+2(2m^2+1)\dfrac{\pi^2}{\log^2\left(\frac{b}{a}\right)}+\dfrac{\pi^4}{\log^4\left(\frac{b}{a}\right)}}{2m^2-1+\dfrac{\pi^2}{\log^2\left(\frac{b}{a}\right)}+2\sqrt{(m^2-1)^2-\dfrac{\pi^2}{\log^2\left(\frac{b}{a}\right)}}}.
    \end{align*}
    Therefore, the eigenvalue $\mu_{m,0}$ satisfies in this alternative the inequality
    \begin{align}\label{alternativeII}
        \mu_{m,0}>\dfrac{4m^2-3+2(2m^2+1)\dfrac{\pi^2}{\log^2\left(\frac{b}{a}\right)}+\dfrac{\pi^4}{\log^4\left(\frac{b}{a}\right)}}{2m^2-1+\dfrac{\pi^2}{\log^2\left(\frac{b}{a}\right)}+2\sqrt{(m^2-1)^2-\dfrac{\pi^2}{\log^2\left(\frac{b}{a}\right)}}},
    \end{align}
    provided that
    \begin{align*}
        \log\left(\frac{b}{a}\right)\geq \frac{\pi}{m^2-1}.
    \end{align*}
    Then, let us estimate a minimum on the conformal class such that 
    \begin{align*}
        \dfrac{4m^2-3+2(2m^2+1)\dfrac{\pi^2}{\log^2\left(\frac{b}{a}\right)}+\dfrac{\pi^4}{\log^4\left(\frac{b}{a}\right)}}{2m^2-1+\dfrac{\pi^2}{\log^2\left(\frac{b}{a}\right)}+2\sqrt{(m^2-1)^2-\dfrac{\pi^2}{\log^2\left(\frac{b}{a}\right)}}}\geq \frac{\left(8m^2+\dfrac{8\pi^2}{\log^2\left(\frac{b}{a}\right)}\right)\dfrac{\pi^2}{\log^2\left(\frac{b}{a}\right)}}{\left((m^2-1)+\dfrac{2\pi^2}{\log^2\left(\frac{b}{a}\right)}\right)+\sqrt{(m^2-1)^2-\dfrac{4\pi^2}{\log^2\left(\frac{b}{a}\right)}}}.
    \end{align*}
    Letting $X=\dfrac{\pi^2}{\log^2\left(\frac{b}{a}\right)}$, this inequality is equivalent to
    \begin{align*}
        2m^2-1+X^2-2\sqrt{(m^2-1)^2-X^2}\geq 2\left((m^2-1)+2X^2\right)-2\sqrt{(m^2-1)^2-4X^2},
    \end{align*}
    or
    \begin{align}\label{end_conf_cond1}
        4m-3-3X^2&\geq 2\left(\sqrt{(m^2-1)^2-X^2}-\sqrt{(m^2-1)^2-4X^2}\right)\nonumber\\
        &=\frac{6X^2}{\sqrt{(m^2-1)^2-X^2}+\sqrt{(m^2-1)^2-4X^2}}.
    \end{align}
    Assuming that 
    \begin{align}\label{end_conf_cond2}
        \log\left(\frac{b}{a}\right)\geq \frac{4\pi}{m^2-1},
    \end{align}
    we deduce that
    \begin{align*}
        \sqrt{(m^2-1)-X^2}+\sqrt{(m^2-1)^2-4X^2}\geq \dfrac{\sqrt{3}}{4}\left(\sqrt{5}+2\right)(m^2-1),
    \end{align*}
    and in particular \eqref{end_conf_cond1} is implied by the inequality
    \begin{align*}
        4m-3\geq \left(3+\frac{3\sqrt{3}}{2}\left(\sqrt{5}+2\right)\right)X^2,
    \end{align*}
    or
    \begin{align}\label{end_conf_cond3}
        \log\left(\frac{b}{a}\right)\geq \frac{2\pi}{3(2+\sqrt{3}(\sqrt{5}+2))(4m-3)}.
    \end{align}
    Then, we can move to the general lower bound on all eigenvalues. Using \eqref{est_eigenvalue} (and assuming that \eqref{new_conf_condition2}, \eqref{new_conf_condition3}, and \eqref{new_conf_condition4} hold), we look for a condition on the conformal class so that 
    \begin{align*}
        \dfrac{\left(8m^2+\dfrac{8\pi^2}{\log^2\left(\frac{b}{a}\right)}\right)\dfrac{\pi^2}{\log^2\left(\frac{b}{a}\right)}}{\left((m^2-1)+\dfrac{2\pi^2}{\log^2\left(\frac{b}{a}\right)}\right)+\sqrt{(m^2-1)^2-\dfrac{4\pi^2}{\log^2\left(\frac{b}{a}\right)}}}\geq \dfrac{\left(4m^2+\dfrac{\pi^2}{\log^2\left(\frac{b}{a}\right)}\right)\dfrac{\pi^2}{\log^2\left(\frac{b}{a}\right)}}{2(m^2+1)+\dfrac{\pi^2}{\log^2\left(\frac{b}{a}\right)}}.
    \end{align*}
    We have
    \begin{align*}
        \dfrac{\left(8m^2+\dfrac{8\pi^2}{\log^2\left(\frac{b}{a}\right)}\right)\dfrac{\pi^2}{\log^2\left(\frac{b}{a}\right)}}{\left((m^2-1)+\dfrac{2\pi^2}{\log^2\left(\frac{b}{a}\right)}\right)+\sqrt{(m^2-1)^2-\dfrac{4\pi^2}{\log^2\left(\frac{b}{a}\right)}}}\geq \dfrac{\left(8m^2+\dfrac{8\pi^2}{\log^2\left(\frac{b}{a}\right)}\right)\dfrac{\pi^2}{\log^2\left(\frac{b}{a}\right)}}{2(m^2-1)+\dfrac{2\pi^2}{\log^2\left(\dfrac{b}{a}\right)}},
    \end{align*}
    showing that the inequality is true provided that
    \begin{align*}
        2(m^2-1)+\dfrac{2\pi^2}{\log^2\left(\frac{b}{a}\right)}\leq 2(m^2+1)+\dfrac{\pi^2}{\log^2\left(\frac{b}{a}\right)},
    \end{align*}
    or
    \begin{align*}
        \log\left(\frac{b}{a}\right)\geq \frac{\pi^2}{4}.
    \end{align*}
    Finally, we have
    \begin{align*}
        \inf_{n\in\Z}\mu_{m,n}=\mu_{m,m}
    \end{align*}
    provided that 
    \begin{align*}
        \dfrac{\left(8m^2+\dfrac{8\pi^2}{\log^2\left(\frac{b}{a}\right)}\right)\dfrac{\pi^2}{\log^2\left(\frac{b}{a}\right)}}{\left((m^2-1)+\dfrac{2\pi^2}{\log^2\left(\frac{b}{a}\right)}\right)+\sqrt{(m^2-1)^2-\dfrac{4\pi^2}{\log^2\left(\frac{b}{a}\right)}}}\geq \frac{\left(8m^2+\dfrac{4\pi^2}{\log^2\left(\frac{b}{a}\right)}\right)\dfrac{\pi^2}{\log^2\left(\frac{b}{a}\right)}}{m^2+1+\dfrac{2\pi^2}{\log^2\left(\frac{b}{a}\right)}},
    \end{align*}
    which is only asymptotically true for $m\leq\sqrt{3}<2$. Since $m\geq 2$ in this last step, we cannot conclude which eigenvalue is the smallest between $\mu_{m,0}$ and $\mu_{m,m}$. This concludes the proof of the theorem.
    \end{proof}

    \subsection{Second Weighted Estimate}

    We deduce the following near-optimal weighted inequality.
    \begin{theorem}
        For all $m\geq 1$, there exists $R_m<\infty$ with the following property. Let $0<a<b<\infty$, let $\Omega=B_b\setminus\bar{B}_a(0)$. Assume that
        \begin{align*}
            \log\left(\frac{b}{a}\right)\geq R_m.
        \end{align*}
        Then, for all $u\in W^{2,2}_0(\Omega)$, we have
        \begin{align*}
            \int_{\Omega}\left(\Delta u+2(m-1)\frac{x}{|x|^2}\cdot \D u+\frac{(m-1)^2}{|x|^2}u\right)^2dx&\geq \dfrac{\left(4m^2+\dfrac{\pi^2}{\log\left(\frac{b}{a}\right)}\right)\dfrac{\pi^2}{\log^2\left(\frac{b}{a}\right)}}{4(m^2+1)+\dfrac{2\pi^2}{\log^2\left(\frac{b}{a}\right)}}\int_{\Omega}\frac{|\D u|^2}{|x|^2}dx
        \end{align*}
    \end{theorem}

    \section{Second Class of Weighted Poincaré Estimates}

    \subsection{Inequalities Associated to the Bilaplacian}

We will now generalise the last two inequalities of \cite[Lemma IV.$1$]{riviere_morse_scs}. If $\Omega=B_b\setminus\bar{B}_a(0)$, recall that for all $u\in W^{2,2}_0(\Omega)$, if 
\begin{align*}
    \leb_m=\Delta+2(m-1)\frac{x}{|x|^2}\cdot\D+\frac{(m-1)^2}{|x|^2},
\end{align*}
then
\begin{align*}
    \int_{\Omega}\left(\leb_mu\right)^2dx=\int_{\Omega}\left(\Delta u+(m^2-1)\frac{u}{|x|^2}\right)^2dx+4(m^2-1)\int_{\Omega}\left(\frac{x}{|x|^2}\cdot\D u-\frac{u}{|x|^2}\right)^2dx.
\end{align*}
Before stating the theorem, let us prove an elementary lemma that will prove crucial in all cases $m\geq 1$.

\begin{lemme}\label{lemme_ipp}
    Let $\dfrac{1}{2}<\beta<\infty$ and $u\in W^{2,2}_0(B(0,1))$. Then, we have
    \begin{align}\label{ipp1}
        \int_{B(0,1)}\frac{u^2}{|x|^{4(1-\beta)}}dx=-\frac{1}{2\beta-1}\int_{B(0,1)}\frac{u}{|x|^{2(1-\beta)}}\frac{x}{|x|^{2(1-\beta)}}\cdot \D u\,dx.
    \end{align}
    Furthermore, the following inequality holds
    \begin{align}\label{ipp2}
        \int_{B(0,1)}\frac{u^2}{|x|^{4(1-\beta)}}dx\leq \frac{1}{(2\beta-1)^2}\int_{B(0,1)}\left(\frac{x}{|x|^{2}}\cdot \D u\right)^2|x|^{4\beta}dx.
    \end{align}
\end{lemme}
\begin{proof}
    First, notice that since $u\in W^{2,2}_0(B(0,1))$, we have by the Sobolev embedding $u\in C^0(B(0,1))$, which implies that there exists a universal constant $C<\infty$ such that
    \begin{align*}
        \np{u}{\infty}{B(0,1)}\leq C\wp{u}{2,2}{B(0,1)}\leq C\np{\Delta u}{2}{\Omega},
    \end{align*}
    where we used \eqref{complete_poincaré} (that holds in the limiting case $a=0$ and $b=1$). Therefore, we have
    \begin{align*}
        \int_{B(0,1)}\frac{u^2}{|x|^{4(1-\beta)}}dx\leq C\np{\Delta u}{2}{B(0,1)}^2\int_{B(0,1)}\frac{dx}{|x|^{4(1-\beta)}}=\frac{\pi C}{2\beta-1}\np{\Delta u}{2}{\Omega}<\infty.
    \end{align*}
    Since $\beta>\dfrac{1}{2}$, we have $|x|^{4\beta-2}\Delta\log|x|=2\pi|x|^{4\beta-2}\delta_0=0$, which shows by an integration by parts that
    \begin{align}
        &\int_{B(0,1)}\frac{u}{|x|^2}\frac{x}{|x|^2}\cdot \D u\,|x|^{4\beta}dx=\frac{1}{2}\int_{B(0,1)}|x|^{4\beta-2}\D\log|x|\cdot \D (u^2)\,dx=\nonumber\\
        &-\frac{1}{2}\int_{B(0,1)}\D(|x|^{4\beta-2})\cdot \D\log|x|\,u^2dx
        =-(2\beta-1)\int_{B(0,1)}\frac{u^2}{|x|^{4(1-\beta)}}dx.
    \end{align}
    The inequality follows from Cauchy-Schwarz inequality.
\end{proof}

\begin{theorem}\label{lemmeIV.1_complement}
    Let $0<a<b<\infty$ and $\Omega=B_b\setminus\bar{B}_a(0)$. Then, for all $1/2<\beta<\infty$, there exists a universal constant $C_{\beta}'<\infty$ independent of $0<a<b<\infty$ such that for all $u\in W^{2,2}_0(\Omega)$, the following inequality holds
    \begin{align}
        \int_{\Omega}\left(\left(\frac{|x|}{b}\right)^{4\beta}+\left(\frac{a}{|x|}\right)^{4\beta}\right)\frac{u^2}{|x|^4}dx\leq C_{\beta}\int_{\Omega}(\Delta u)^2dx.
    \end{align}
    For all $\sqrt{2}-1<\beta<\infty$, there exists a universal constant $C_{\beta}'<\infty$ independent of $0<a<b<\infty$ such that for all $u\in W^{2,2}_0(\Omega)$, the following inequality holds
    \begin{align}
        \int_{\Omega}\left(\left(\frac{|x|}{b}\right)^{2\beta}+\left(\frac{a}{|x|}\right)^{2\beta}\right)\frac{|\D u|^2}{|x|^2}dx\leq C_{\beta}'\int_{\Omega}(\Delta u)^2dx.
    \end{align}
\end{theorem}
\begin{proof}
\textbf{Step 1: First integral, first part.}

Since $u\in W^{2,2}_0(\Omega)$, we can extend it—without changing notations—by $0$ as a $W^{2,2}$ function for which all $L^2$ norms (of $u$ and its derivatives) are preserved. 
Making the change of variable $y=\dfrac{x}{b}$, we get if $\bar{u}(y)=u(b\,y)$ ($y\in B(0,1)$)
\begin{align}\label{lemma41_p1}
    \int_{\Omega}\left(\frac{|x|}{b}\right)^{4\beta}\frac{u^2}{|x|^4}dx=\frac{1}{b^2}\int_{B_1\setminus\bar{B}_{b^{-1}a}(0)}|y|^{4\beta}\frac{|\bar{u}|^2}{|y|^4}dy=\frac{1}{b^2}\int_{B(0,1)}|y|^{4\beta}\frac{|\bar{u}|^2}{|y|^4}dy.
\end{align}
Thanks \cite{talenti}, the Sobolev constant in $\R^2$ for $W^{1,p}$ functions is given by 
\begin{align*}
    C(2,p)&=\frac{1}{\sqrt{\pi}}\frac{1}{2^p}\left(\frac{p-1}{2-p}\right)^{1-\frac{1}{p}}\left(\frac{1}{\Gamma\left(\frac{2}{p}\right)\Gamma\left(3-\frac{2}{p}\right)}\right)^{\frac{1}{2}}\\
    &\leq \frac{1}{\sqrt{\pi}}\frac{1}{2^p}\left(\frac{p-1}{2-p}\right)^{1-\frac{1}{p}}\frac{2}{\sqrt{\pi}}\leq \frac{1}{\pi}\frac{1}{(2-p)^{1-\frac{1}{p}}}.
\end{align*}
Therefore, for all $q>1$ such that $4(1-\beta)p'<2$, or 
\begin{align*}
    2(1-\beta)<\frac{1}{p'}=1-\frac{1}{p}\Longleftrightarrow p>\frac{1}{2\beta-1},
\end{align*}
which implies in particular that we must choose $\beta>1/2$. Then, by Hölder's inequality, we have
\begin{align}\label{lemma41_p2}
    \int_{B(0,1)}|y|^{4\beta}\frac{|\bar{u}|^2}{|y|^4}dy&\leq \left(\int_{B(0,1)}\frac{dy}{|y|^{4(1-\beta)p'}}\right)^{\frac{1}{p'}}\left(\int_{B(0,1)}|\bar{u}|^{2p}dy\right)^{\frac{1}{p}}\nonumber\\
    &\leq C\left(2,\frac{2p}{p+1}\right)^2\left(\int_{B(0,1)}\frac{dy}{|y|^{4(1-\beta)p'}}\right)^{\frac{1}{p'}}\left(\int_{B(0,1)}|\D \bar{u}|^{\frac{2p}{p+1}}dy\right)^{\frac{p+1}{p}}\nonumber\\
    &\leq C\left(2,\frac{2p}{p+1}\right)^2\left(\int_{B(0,1)}\frac{dy}{|y|^{4(1-\beta)p'}}\right)^{\frac{1}{p'}}\pi^{\frac{1}{p+1}}\int_{B(0,1)}|\D \bar{u}|^2dy
\end{align}
Then, by the standard Poincaré inequality, the parallelogram identity, and \eqref{ipp_laplacien}, we have
\begin{align}\label{lemma41_p3}
    &\int_{B(0,1)}|\D \bar{u}|^2dy=4\int_{B(0,1)}|\p{z}\bar{u}|^2dy\leq \frac{16}{\pi^2}\int_{B(0,1)}|\D\p{z}u|^2dy\nonumber\\
    &=\frac{16}{\pi^2}\int_{B(0,1)}\left(\left|(\p{z}+\p{\z})\p{z}\bar{u}\right|^2+|i(\p{z}-\p{\z})\p{z}\bar{u}|^2\right)dy
    =\frac{16}{\pi^2}\int_{B(0,1)}\left(2|\p{z}^2\bar{u}|^2+2|\p{z\z}^2\bar{u}|^2\right)dy\nonumber\\
    &=\frac{16}{\pi^2}\int_{B(0,1)}|\D^2\bar{u}|^2dy
    =\frac{16}{\pi^2}\int_{B(0,1)}\left(\Delta \bar{u}\right)^2dy.
\end{align}
Therefore, we deduce by \eqref{lemma41_p2} and \eqref{lemma41_p3} that
\begin{align}\label{lemma41_p4}
    \int_{B(0,1)}|y|^{4\beta}\frac{|\bar{u}|^2}{|y|^4}dy\leq \frac{16}{\pi^{2-\frac{1}{p+1}}}C\left(2,\frac{2p}{p+1}\right)^2\left(\frac{\pi}{1-2(1-\beta)p'}\right)^{\frac{1}{p'}}\int_{B(0,1)}(\Delta \bar{u})^2dy.
\end{align}
Since
\begin{align}\label{lemma41_p5}
    \int_{B(0,1)}(\Delta \bar{u})^2dy=b^4\int_{B(0,1)}\left(\Delta u(b\,y)\right)dy=b^2\int_{B(0,b)}(\Delta u)^2dx=b^2\int_{\Omega}(\Delta u)^2dx,
\end{align}
we finally deduce by \eqref{lemma41_p1}, \eqref{lemma41_p4}, and \eqref{lemma41_p5} that for all $\frac{1}{2}<\beta<1$ and $p>\frac{1}{2\beta-1}$
\begin{align}\label{lemmeIV.1_part2}
    \int_{\Omega}\left(\frac{|x|}{b}\right)^{4\beta}\frac{u^2}{|x|^4}dx\leq \frac{16}{\pi^{2-\frac{1}{p+1}}}C\left(2,\frac{2p}{p+1}\right)^2\left(\frac{\pi}{1-2(1-\beta)p'}\right)^{\frac{1}{p'}}\int_{\Omega}(\Delta u)^2dx.
\end{align}

\textbf{Step 2: First integral, inversion part.}

Making an inversion as in \cite[Lemma IV.$1$]{riviere_morse_scs}, let us show that
\begin{align}\label{lemmeIV.1_part3}
    \int_{\Omega}\left(\frac{a}{|x|}\right)^{4\beta}\frac{u^2}{|x|^4}dx\leq \frac{1}{((2\beta+1)^2-2)^2}\int_{\Omega}(\Delta u)^2dx.
\end{align}
Now, making the change of variable
    \begin{align*}
        a\frac{x}{|x|^2}=y
    \end{align*}
    we get
    \begin{align*}
        \frac{dx}{|x|^4}=\frac{dy}{a^2}.
    \end{align*}
    Therefore, we get
    \begin{align*}
        \int_{\Omega}\left(\frac{a}{|x|}\right)^{4\beta}\frac{u^2}{|x|^4}dx=\frac{1}{a^2}\int_{B_1\setminus\bar{B}_{b^{-1}a}(0)}|y|^{4\beta}|u(a\,\iota(y))|^2dy=\frac{1}{a^2}\int_{B(0,1)}|y|^{4\beta}|\bar{u}(y)|^2dy,
    \end{align*}
    where $\bar{u}:B(0,1)\rightarrow \R$ is the extension by $0$ of $u(a\,\iota(\,\cdot\,))$, where $\iota(x)=\dfrac{x}{|x|^2}$ is the inversion. Now, applying \eqref{ipp2} to the exponent $\beta+1>1/2$, we get
    \begin{align}\label{laplace_second_est1}
        \int_{B(0,1)}|y|^{4\beta}|\bar{u}(y)|^2dy\leq \frac{1}{(2\beta+1)^2}\int_{B(0,1)}\left(y\cdot \D \bar{u}\right)^2|y|^{4\beta}dy\leq \frac{1}{(2\beta+1)^2}\int_{B(0,1)}|y|^{4\beta+2}|\D\bar{u}|^2dy.
    \end{align}
    Integrating by parts as above, we get by \eqref{ipp2}
    \begin{align}\label{laplace_second_est2}
        \int_{B(0,1)}|y|^{4\beta+2}|\D \bar{u}|^2dy&=-\int_{B(0,1)}|y|^{4\beta+2}\bar{u}\,\Delta\bar{u}\,dy-(4\beta+2)\int_{B(0,1)}\frac{\bar{u}}{|y|^2}\frac{y}{|y|^2}\cdot \D \bar{u}\,|y|^{4\beta+4}dy\nonumber\\
        &=-\int_{B(0,1)}|y|^{4\beta+2}\bar{u}\,\Delta\bar{u}\,dy+2\int_{B(0,1)}|y|^{4\beta}|\bar{u}|^2dy.
    \end{align}
    Therefore, \eqref{laplace_second_est1} and \eqref{laplace_second_est2} show that
    \begin{align}
        \int_{B(0,1)}|y|^{4\beta}|\bar{u}|^2dy\leq -\frac{1}{(2\beta+1)^2}\int_{B(0,1)}|y|^{4\beta+2}\bar{u}\,\Delta\bar{u}\,dy+\frac{2}{(2\beta+1)^2}\int_{B(0,1)}|y|^{4\beta}|\bar{u}|^2dy.
    \end{align}
    Since $\beta>\dfrac{}{}$, we have $\dfrac{2}{(2\beta+1)^2}\leq \dfrac{1}{2}<1$, and we deduce by Cauchy-Schwarz inequality that
    \begin{align}
        \int_{B(0,1)}|y|^{4\beta}|\bar{u}|^2dy&\leq -\frac{1}{(2\beta+1)^2-2}\int_{B(0,1)}|y|^{2\beta}\bar{u}|y|^{2\beta+2}\Delta\bar{u}dy\nonumber\\
        &\leq \frac{1}{(2\beta+1)^2-2}\left(\int_{B(0,1)}|y|^{4\beta}|\bar{u}|^2dy\right)^{\frac{1}{2}}\left(\int_{B(0,1)}|y|^{4\beta+4}(\Delta \bar{u})^2dy\right)^{\frac{1}{2}}.
    \end{align}
    Therefore, we have
    \begin{align}\label{weight_laplacian1}
        \int_{B(0,1)}|y|^{4\beta}|\bar{u}|^2\leq \frac{1}{((2\beta+1)^2-2)^2}\int_{B(0,1)}|y|^{4\beta+4}(\Delta \bar{u})^2dy.
    \end{align}
    Notice that this inequality is valid for all
    \begin{align*}
        \beta>\frac{\sqrt{2}-1}{2}=0.207\cdots.
    \end{align*}
    Making the change of variable $\alpha=4\beta$, we deduce that for all $u\in W^{2,2}_0(B(0,1))$, and for all $\alpha>2(\sqrt{2}-1)=0.8282\cdots$, we have
    \begin{align}\label{weight_general}
        \int_{B(0,1)}|x|^{\alpha}u^2dx\leq \frac{16}{\left(\alpha^2+4\alpha-4\right)^2}\int_{B(0,1)}|x|^{\alpha+4}\left(\Delta u\right)^2dx.
    \end{align}
    Now, if $\bar{u}(x)=u(\iota(x))$, then in complex coordinates, we have
    \begin{align*}
        \bar{u}(z)=u\left(\frac{z}{|z|^2}\right)=u\left(\frac{\Re(z)}{|z|^2},\frac{\Im(z)}{|z|^2}\right)=u\left(\Re\left(\frac{1}{\z}\right),\Im\left(\frac{1}{\z}\right)\right)=u\left(\frac{1}{2}\left(\frac{1}{z}+\frac{1}{\z}\right),\frac{i}{2}\left(\frac{1}{z}-\frac{1}{\z}\right)\right).
    \end{align*}
    Therefore, we have
    \begin{align}\label{inverse_laplacian}
        \p{z}\bar{u}&=-\frac{1}{2z^2}\p{x_1}u(\iota(z))-\frac{i}{2z^2}\p{x_2}u(\iota(z))\nonumber\\
        \p{z\z}^2\bar{u}&=-\frac{1}{2z^2}\left(-\frac{1}{2\z^2}\p{x_1}^2 u(\iota(z))+\frac{i}{2\z^2}\p{x_1,x_2}^2u(\iota(z))\right)-\frac{i}{2z^2}\left(\frac{i}{2\z^2}\p{x_2}^2u(\iota(z))-\frac{1}{2z^2}\p{x_1,x_2}^2u(\iota(z))\right)\nonumber\\
        &=\frac{1}{4|z|^4}\Delta u(\iota(z))\nonumber\\
        \Delta \bar{u}&=\frac{1}{|z|^4}\Delta u(\iota(z)).
    \end{align}
    Then, we get
    \begin{align}\label{weight_laplacian2}
        \int_{B(0,1)}|y|^{4\beta+4}\left(\Delta \bar{u}\right)^2dy&=a^4\int_{B(0,1)}|y|^{4\beta}\left(\Delta u(a\,\iota(y))\right)^2\frac{dy}{|y|^4}
        =a^2\int_{\Omega}\left(\Delta u\right)^2\left(\frac{a}{|x|}\right)^{4\beta}dx.
    \end{align}
    Thanks to \eqref{weight_laplacian1} and \eqref{weight_laplacian2}, we deduce that
    \begin{align*}
        \int_{\Omega}\left(\frac{a}{|x|}\right)^{4\beta}\frac{u^2}{|x|^4}dx&\leq \frac{1}{((2\beta+1)^2-2)^2}\int_{\Omega}\left(\Delta u\right)^2\left(\frac{a}{|x|}\right)^{4\beta}dx\\
        &\leq \frac{1}{((2\beta+1)^2-2)^2}\int_{\Omega}\left(\Delta u\right)^2dx.
    \end{align*}
    which concludes the proof of this step.

\textbf{Step 3: Second integral, first part.}
Now, consider for all $0<\beta<1$ and for all $u\in W^{2,2}_0(\Omega)$ the following integral
\begin{align*}
    \int_{\Omega}\left(\frac{|x|}{b}\right)^{2\beta}\frac{|\D u|^2}{|x|^2}dx.
\end{align*}
As above, making a change of variable $x=b\,y$ yields for all $p>\frac{1}{\beta}$
\begin{align}\label{lemmeIV.1_part4}
    \int_{\Omega}\left(\frac{|x|}{b}\right)^{2\beta}\frac{|\D u|^2}{|x|^2}dx&=\frac{1}{b^2}\int_{B_1\setminus\bar{B}_{b^{-1}a}(0)}|y|^{2\beta}\frac{|\D \bar{u}|^2}{|y|^2}dy=\frac{1}{b^2}\int_{B(0,1)}|y|^{2\beta}\frac{|\D \bar{u}|^2}{|y|^2}dy\nonumber\\
    &\leq \frac{1}{b^2}\left(\int_{B(0,1)}\frac{dy}{|y|^{2(1-\beta)p'}}\right)^{\frac{1}{p'}}\left(\int_{B(0,1)}|\D\bar{u}|^{2p}\right)^{\frac{1}{p}}\nonumber\\
    &\leq \frac{1}{b^2}\frac{16}{\pi^{2-\frac{1}{p+1}}}C\left(2,\frac{2p}{p+1}\right)^2\left(\frac{\pi}{1-(1-\beta)p'}\right)^{\frac{1}{p'}}\int_{B(0,1)}(\Delta \bar{u})^2dy\nonumber\\
    &=\frac{16}{\pi^{2-\frac{1}{p+1}}}C\left(2,\frac{2p}{p+1}\right)^2\left(\frac{\pi}{1-(1-\beta)p'}\right)^{\frac{1}{p'}}\int_{\Omega}(\Delta {u})^2dx.
\end{align}

    \textbf{Step 4: Second integral, inversion part.}
    Reasoning as above, we deduce that for all $\beta>0$
    \begin{align}\label{final_bilaplace1}
       \int_{\Omega}\left(\frac{a}{|x|}\right)^{2\beta}\frac{|\D u|^2}{|x|^2}dx=\frac{1}{a^2}\int_{B(0,1)}|y|^{2(\beta+1)}|\D \bar{u}|^2dy,
    \end{align}
    where $\bar{u}$ is the extension by $0$ of $u(b\,\iota(\,\cdot\,)):B_1\setminus\bar{B}_{b^{-1}a}(0)\rightarrow\R$. Integrating by parts as previously, we get
    \begin{align}\label{final_bilaplace2}
        \int_{B(0,1)}|y|^{2\beta+2}|\D \bar{u}|^2dy&=-\int_{B(0,1)}|y|^{2\beta+2}\bar{u}\,\Delta\bar{u}\,dy-(2\beta+2)\int_{B(0,1)}{\bar{u}}\,({y}\cdot \D \bar{u})\,|y|^{2\beta}dy\nonumber\\
        &=-\int_{B(0,1)}|y|^{2\beta+2}\bar{u}\,\Delta\bar{u}\,dy+2\int_{B(0,1)}|y|^{2\beta}|\bar{u}|^2dy.
    \end{align}
    Using \eqref{weight_general} and assuming that $\alpha=2\beta>2(\sqrt{2}-1)$, or $\beta>\sqrt{2}-1$, we deduce that
    \begin{align}\label{final_bilaplace3}
        \int_{B(0,1)}|y|^{2\beta}|\bar{u}|^2dy&\leq \frac{1}{((\beta+1)^2-2)^2}\int_{B(0,1)}|y|^{2\beta+4}\left(\Delta\bar{u}\right)^2dy\nonumber\\
        &=\frac{a^2}{((\beta+1)^2-2)^2}\int_{\Omega}\left(\Delta u\right)^2\left(\frac{a}{|x|}\right)^{2\beta}dx,
    \end{align}
    while Cauchy-Schwarz inequality and \eqref{final_bilaplace3} show that
    \begin{align}\label{final_bilaplace4}
        -\int_{B(0,1)}|y|^{2\beta+2}\bar{u}\,\Delta\bar{u}\,dy&\leq \left(\int_{B(0,1)}|y|^{2\beta}|\bar{u}|^2dy\right)^{\frac{1}{2}}\left(\int_{B(0,1)}|y|^{2\beta+4}\left(\Delta\bar{u}\right)^2dy\right)^{\frac{1}{2}}\nonumber\\
        &\leq \frac{1}{(\beta+1)^2-2}\int_{B(0,1)}|y|^{2\beta+4}\left(\Delta \bar{u}\right)^2dx.
    \end{align}
    Therefore, we get by \eqref{final_bilaplace2}, \eqref{final_bilaplace3}, and \eqref{final_bilaplace4} for all $\beta>\sqrt{2}-1$
    \begin{align}\label{final_bilaplace5}
        \int_{B(0,1)}|y|^{2\beta+2}|\D\bar{u}|^2dy&\leq \left(\frac{1}{(\beta+1)^2-2}+\frac{2}{((\beta+1)^2-2)^2}\right)\int_{B(0,1)}|y|^{2\beta+4}(\Delta \bar{u})^2dy\nonumber\\
        &=\frac{(\beta+1)^2}{((\beta+1)^2-2)^2}\int_{B(0,1)}|y|^{2\beta+4}(\Delta \bar{u})^2dy.
    \end{align}
    Thanks to \eqref{final_bilaplace1}, the identity \eqref{final_bilaplace3}, and \eqref{final_bilaplace5}, we deduce that for all $\beta>\sqrt{2}-1$
    \begin{align*}
        \int_{\Omega}\left(\frac{a}{|x|}\right)^{2\beta}\frac{|\D u|^2}{|x|^2}dx&\leq \frac{(\beta+1)^2}{((\beta+1)^2-2)^2} \int_{\Omega}(\Delta u)^2\left(\frac{a}{|x|}\right)^{2\beta}dx\\
        &\leq \frac{(\beta+1)^2}{((\beta+1)^2-2)^2}\int_{\Omega}(\Delta u)^2dx,
    \end{align*}
    which concludes the proof of the theorem.
\end{proof}

We will also need a version of this theorem for functions that do not necessarily vanish on the boundary. 
 Thanks to the inequality from Theorem \ref{lemmeIV.1_complement}, if $0<a<b<\infty$ and $\Omega=B_b\setminus\bar{B}_a(0)$, for all $\sqrt{2}-1<\beta<1$ and for all $u\in W^{2,2}_0(\Omega)$, we have
 \begin{align}\label{poincare_nul_bord}
     \int_{\Omega}\frac{|\D u|^2}{|x|^2}\left(\left(\frac{|x|}{b}\right)^{2\beta}+\left(\frac{a}{|x|}\right)^{2\beta}\right)dx\leq C_{\beta}\int_{\Omega}(\Delta u)^2dx.
 \end{align}
 However, in one of the key lemmas (Lemma $IV.1$) used in the proof of the main theorem of \cite{riviere_morse_scs} that we adapt to the case of Willmore immersions in \cite[Lemma $4.4$]{morse_willmore_I}, we need such an estimate for functions that do not necessarily vanish on the boundary of neck regions. This new technical difficulty is due to terms in the second derivative that can only be bounded from below by
 \begin{align*}
     -\int_{\Sigma}|du|_g^2|A|^2d\vg,
 \end{align*}
 and those contributions can \emph{a priori} prevent one from bounded from below the minimal eigenvalue of a sequence of immersions, an estimate that is crucial in the proof (see \cite[Lemma IV.$4$]{riviere_morse_scs}). Thanks to the inequality \eqref{poincare_nul_bord}, the proof reduces to a delicate estimate involving biharmonic functions.

        \begin{theorem}\label{biharmonic_interpolation}
            There exists a universal constant $\Gamma<\infty$ with the following property. Let $0<a<b<\infty$ and $\Omega=B_b\setminus\bar{B}_a(0)$. For all $\dfrac{1}{2}<\beta<1$ 
            and for all $0<\gamma<1$, for all biharmonic function $\psi\in W^{2,2}(\Omega)$ \emph{(}\emph{i.e.} that satisfies $\Delta^2\psi=0$ on $\Omega$\emph{)}, assuming that 
            \small
            \begin{align}\label{conformal_class_lemma}
            \log\left(\frac{b}{a}\right)\geq \max\ens{2,\frac{1}{2\beta-1}\log\left(4\beta\right),\frac{1}{4\beta}\log\left(\frac{2}{2-\sqrt{3}}\right), \frac{1}{4(1-\beta)}\log\left(1+\frac{8\beta(1-\beta)}{(2\beta-1)^2}\right), \frac{1}{4\beta}\log(8\beta(\beta+1))},
        \end{align}
        \normalsize
        we have
        \begin{align}\label{biharmonic_inter}
            \int_{\Omega}\frac{|\D\psi|^2}{|x|^2}\left(\left(\frac{|x|}{b}\right)^{2\gamma}+\left(\frac{a}{|x|}\right)^{2\gamma}\right)dx&\leq \frac{\Gamma}{\gamma(1-\gamma)}\frac{1}{1-2\left(\frac{a}{b}\right)^{4\beta}}\int_{\Omega}\frac{\psi^2}{|x|^4}\left(\left(\frac{|x|}{b}\right)^{4\beta}+\left(\frac{a}{|x|}\right)^{4\beta}\right)dx\nonumber\\
            &+\frac{\Gamma}{\gamma(1-\gamma)}\frac{1}{1-2\left(\frac{a}{b}\right)^{4\beta}}\int_{\Omega}|\D^2\psi|^2dx.
        \end{align}
        \end{theorem}
        \begin{rem}
            The difficulty of this estimate lies in the fact that if one write $\psi=\psi_0+|x|^2\psi_1$, where $\psi_0$ and $\psi_1$ are harmonic functions, the frequencies of $\psi_0$ and $\psi_1$ have a non-trivial interaction. Furthermore, due to the introduction of the weight, some frequencies that disappear in the integration of $\D^2\psi$ need to be controlled by the weighted $L^2$ norm of $\psi$ (that captures all frequencies thanks to the shift of the radial component). Last, but not least, controlling the two logarithmic components is delicate due to the involved algebraic expressions of them in the three integrals. 
        \end{rem}
        \begin{proof}
        \textbf{Step 1:} Estimation of            \begin{align*}
                 \int_{\Omega}\frac{|\D\psi|^2}{|x|^2}\left(\frac{|x|}{b}\right)^{2\gamma}dx
            \end{align*}
            As $\psi$ is a biharmonic function, there exists harmonic functions $\psi_0,\psi_1:\Omega\rightarrow \R$ such that $\psi=\psi_0+|x|^2\psi_1$. Therefore, there exists $\alpha,\beta\in\R$ and $\ens{a_n}_{n\in\Z},\ens{b_n}_{n\in\Z}\subset \C$ such that
            \begin{align}
                \psi(z)=\alpha\log|z|+\Re\left(\sum_{n\in\Z}a_nz^n\right)+|z|^2\left(\beta\log|z|+\Re\left(\sum_{n\in\Z}b_nz^n\right)\right).
            \end{align}
            Therefore, we have
            \begin{align*}
                \p{z}\psi&=\frac{1}{2z}\left(\alpha+\sum_{n\in\Z^{\ast}}n\,a_nz^n+2|z|^2\left(\beta\log|z|+\Re\left(\sum_{n\in\Z}b_nz^n\right)\right)+|z|^2\left(\beta+\sum_{n\in\Z^{\ast}}n\,b_nz^n\right)\right)\\
                &=\frac{1}{2z}\left(\alpha+2\beta|z|^2\log|z|+(\beta+2b_0)|z|^2+\sum_{n\in\Z^{\ast}}n\,a_nz^n+|z|^2\left(\sum_{n\in\Z^{\ast}}(n+1)b_nz^n+2\sum_{n\in\Z^{\ast}}\bar{b_n}\z^n\right)\right).
            \end{align*}
            The radial component of the gradient is therefore given by 
            \begin{align*}
                &\mathrm{rad}\left(|\D\psi|^2\right)=\mathrm{rad}\left(4|\p{z}\psi|^2\right)=\frac{1}{|z|^2}\bigg\{\alpha^2+4\alpha\beta|z|^2\log|z|+2\alpha(\beta+2b_0)|z|^2+4\beta^2|z|^4\log^2|z|\\
                &+4\beta(\beta+2b_0)|z|^4\log|z|+(\beta+2b_0)^2|z|^4+\sum_{n\in\Z^{\ast}}|n|^2|a_n|^2|z|^{2n}+\sum_{n\in\Z^{\ast}}((n+1)^2+4)|b_n|^2|z|^{2(n+2)}\\
                &+2\sum_{n\in\Z^{\ast}}n(n+1)\Re\left(a_n\bar{b_n}\right)|z|^{2(n+1)}\bigg\}.
            \end{align*}    
            Therefore, we get
            \begin{align*}
                &\int_{\Omega}\frac{|\D\psi|^2}{|x|^2}\left(\frac{|x|}{b}\right)^{2\gamma}dx=2\pi\int_{a}^{b}\frac{1}{r^3}\bigg\{\alpha^2+4\alpha\beta r^2\log(r)+2\alpha(\beta+2b_0)r^2+4\beta^2r^4\log^2(r)\\
                &+4\beta(\beta+2b_0)r^4\log(r)+(\beta+2b_0)^2r^4+\sum_{n\in\Z^{\ast}}|n|^2|a_n|^2r^{2n}+\sum_{n\in\Z^{\ast}}\left((n+1)^2+4\right)|b_n|^2r^{2(n+2)}\\
                &+ 2\sum_{n\in\Z^{\ast}}n(n+1)\Re\left(a_n\bar{b_n}\right)r^{2(n+1)}\bigg\}\left(\frac{r}{b}\right)^{2\gamma}dx\\
                &=\frac{\pi\alpha^2}{1-\gamma}\left(\frac{1}{a^2}\left(\frac{a}{b}\right)^{2\gamma}-\frac{1}{b^2}\right)+\frac{4\alpha\beta}{\gamma}\left(\log(b)-\log(a)\left(\frac{a}{b}\right)^{2\gamma}\right)+\frac{2\pi\alpha(\beta+2b_0)}{\gamma}\left(1-\left(\frac{a}{b}\right)^{2\gamma}\right)\\
                &+\frac{4\pi\beta^2}{\gamma+1}\left(b^2\left(\log^2(b)-\frac{\log(b)}{\gamma+1}+\frac{1}{2(\gamma+1)^2}\right)-a^2\left(\log^2(a)-\frac{\log(a)}{\gamma+1}+\frac{1}{2(\gamma+1)^2}\right)\left(\frac{a}{b}\right)^{2\gamma}\right)\\
                &+\frac{4\pi\beta(\beta+2b_0)}{\gamma+1}\left(b^2\left(\log(b)-\frac{1}{2(\gamma+1)}\right)-a^2\left(\log(a)-\frac{1}{2(\gamma+1)}\right)\left(\frac{a}{b}\right)^{2\gamma}\right)\\
                &+\frac{\pi(\beta+2b_0)^2}{\gamma+1}\left(b^2-a^2\left(\frac{a}{b}\right)^{2\gamma}\right)+\pi\sum_{n\in\Z^{\ast}}\frac{|n|^2}{n-1+\gamma}|a_n|^2\left(b^{2(n-1)}-a^{2(n-1)}\left(\frac{a}{b}\right)^{2\gamma}\right)\\
                &+\pi\sum_{n\in\Z^{\ast}}\frac{((n+1)^2+4)}{n+1+\gamma}|b_n|^2\left(b^{2(n+1)}-a^{2(n+1)}\left(\frac{a}{b}\right)^{2\gamma}\right)\\
                &
                +2\pi\sum_{n\in\Z^{\ast}}\frac{n(n+1)}{n+\gamma}\Re\left(a_n\bar{b_n}\right)\left(b^{2n}-a^{2n}\left(\frac{a}{b}\right)^{2\gamma}\right).
            \end{align*}
            Recall that
            \begin{align*}
                \psi(z)=\alpha\log|z|+\Re\left(\sum_{n\in\Z}a_nz^n\right)+|z|^2\left(\beta\log|z|+\Re\left(\sum_{n\in\Z}b_nz^n\right)\right).
            \end{align*}
            We compute
            \begin{align*}
                \mathrm{rad}(\psi^2)&=\alpha^2\log^2|z|+2\alpha\,a_0\log|z|+2\alpha\beta|z|^2\log^2|z|+2(\alpha\,b_0+\beta\,a_0)|z|^2\log|z|+\beta^2|z|^4\log^2|z|\\
                &+2\beta\,b_0|z|^4\log|z|
                +\frac{1}{2}\sum_{n\in\Z}|a_n|^2|z|^{2n}+\frac{1}{2}\sum_{n\in\Z}|b_n|^2|z|^{2(n+2)}+\sum_{n\in\Z}\Re\left(a_n\bar{b_n}\right)|z|^{2(n+1)}.
            \end{align*}
            Therefore, we have
            \begin{align*}
                &\int_{\Omega}\frac{\psi^2}{|x|^4}\left(\frac{|x|}{b}\right)^{4\gamma}dx=2\pi\int_{a}^{b}\frac{1}{r^3}\bigg\{\alpha^2\log^2(r)+2\alpha\,a_0\log(r)+2\alpha\beta r^2\log^2(r)+2(\alpha\,b_0+\beta\,a_0)r^2\log(r)\\
                &+\beta^2r^4\log^2(r)+\frac{1}{2}\sum_{n\in\Z}|a_n|^2r^{2n}+\frac{1}{2}\sum_{n\in\Z}|b_n|^2r^{2(n+2)}+\sum_{n\in\Z}\Re\left(a_n\bar{b_n}\right)r^{2(n+1)}\bigg\}\left(\frac{r}{b}\right)^{4\gamma}dr\\
                &=\frac{\pi\alpha^2}{2\gamma-1}\left(\frac{1}{b^2}\left(\log^2(b)-\frac{\log(b)}{2\gamma-1}+\frac{1}{2(2\gamma-1)^2}\right)-\frac{1}{a^2}\left(\log^2(a)-\frac{\log(a)}{2\gamma-1}+\frac{1}{2(2\gamma-1)^2}\right)\left(\frac{a}{b}\right)^{4\gamma}\right)\\
                &+\frac{\pi\alpha\,a_0}{2\gamma-1}\left(\frac{1}{b^2}\left(\log(b)-\frac{1}{2(2\gamma-1)}\right)-\frac{1}{a^2}\left(\log(a)-\frac{1}{2(2\gamma-1)}\right)\left(\frac{a}{b}\right)^{4\gamma}\right)\\
                &+\frac{\pi\alpha\beta}{\gamma}\left(\left(\log^2(b)-\frac{\log(b)}{2\gamma}+\frac{1}{4\gamma^2}\right)-\left(\log^2(a)-\frac{\log(a)}{2\gamma^2}+\frac{1}{4\gamma}\right)\left(\frac{a}{b}\right)^{4\gamma}\right)\\
                &+\frac{\pi(\alpha\,b_0+\beta\,a_0)}{\gamma}\left(\log(b)-\frac{1}{4\gamma}-\left(\log(a)-\frac{1}{4\gamma}\right)\left(\frac{a}{b}\right)^{4\gamma}\right)\\
                &+\frac{\pi\beta^2}{2\gamma+1}\left(b^2\left(\log^2(b)-\frac{\log(b)}{2\gamma+1}+\frac{1}{2(2\gamma+1)^2}\right)-a^2\left(\log^2(a)-\frac{\log(a)}{2\gamma+1}+\frac{1}{2(2\gamma+1)^2}\right)\left(\frac{a}{b}\right)^{4\gamma}\right)\\
                &+\frac{2\pi\beta\,b_0}{2\gamma+1}\left(b^2\left(\log(b)-\frac{1}{2\gamma+1}\right)-a^2\left(\log(a)-\frac{1}{2\gamma+1}\right)\left(\frac{a}{b}\right)^{4\gamma}\right)\\
                &+\pi\sum_{n\in\Z}\frac{1}{n-1+2\gamma}|a_n|^2\left(b^{2(n-1)}-a^{2(n-1)}\left(\frac{a}{b}\right)^{4\gamma}\right)\\
                &+\pi\sum_{n\in\Z}\frac{1}{n+1+2\gamma}|b_n|^2\left(b^{2(n+1)}-a^{2(n+1)}\left(\frac{a}{b}\right)^{4\gamma}\right)\\
                &+2\pi\sum_{n\in\Z}\frac{1}{n+2\gamma}\Re\left(a_n\bar{b_n}\right)\left(b^{2n}-a^{2n}\left(\frac{a}{b}\right)^{4\gamma}\right)
            \end{align*}
            Now, we have
        \begin{align*}
            \Delta\psi&=\left(\p{r}^2+\frac{1}{r}+\frac{1}{r^2}\p{\theta}^2\right)\left(\beta\, r^2\log(r)+\Re\left(\sum_{n\in\Z}b_nr^{n+2}e^{in\theta}\right)\right)\\
            &=4\beta(\log(r)+1)+\Re\left(\sum_{n\in\Z}\left((n+2)(n+1)+(n+2)-n^2\right)b_nr^ne^{in\theta}\right)\\
            &=4\beta(\log(r)+1)+4\,\Re\left(\sum_{n\in\Z}(n+1)b_nr^ne^{in\theta}\right).
        \end{align*}
        Therefore, we have
        \begin{align}\label{laplacian_l2}
            &\int_{\Omega}(\Delta \psi)^2dx=2\pi\int_{a}^{b}\left(16r(\beta\log(r)+\beta+b_0)^2+8\sum_{n\in\Z^{\ast}}(n+1)^2|b_n|^2r^{2n+1}\right)dr\nonumber\\
            &=16\pi\beta^2\left(b^2\left(\log^2(b)-\log(b)+\frac{1}{2}\right)-a^2\left(\log^2(a)-\log(a)+\frac{1}{2}\right)\right)\nonumber\\
            &+32\pi\beta(\beta+b_0)\left(b^2\left(\log(b)-\frac{1}{2}\right)-a^2\left(\log(a)-\frac{1}{2}\right)\right)+16\pi(\beta+b_0)^2\left(b^2-a^2\right)\nonumber\\
            &+8\pi\sum_{n\in\Z\setminus\ens{-1,0}}|n+1||b_n|^2\left|b^{2(n+1)}-a^{2(n+1)}\right|.
        \end{align}
        Then, we have
        \small
        \begin{align*}
            &\p{z}^2\psi=\p{z}\left(\frac{1}{2z}\left(\alpha+2\beta|z|^2\log|z|+(\beta+2b_0)|z|^2+\sum_{n\in\Z^{\ast}}n\,a_nz^n+|z|^2\left(\sum_{n\in\Z^{\ast}}(n+1)b_nz^n+2\sum_{n\in\Z^{\ast}}\bar{b_n}\z^n\right)\right)\right)\\
            &=\frac{1}{2z^2}\left(-\alpha+\beta|z|^2+\sum_{n\in\Z^{\ast}}n(n-1)a_nz^{n}+|z|^2\sum_{n\in\Z^{\ast}}n(n+1)b_nz^n\right).
        \end{align*}
        \normalsize
        Therefore, we have
        \begin{align*}
            \mathrm{rad}\left(|\p{z}^2\psi|^2\right)&=\frac{1}{4|z|^4}\left(\alpha^2-2\alpha\beta|z|^2+\beta^2|z|^4+\sum_{n\in\Z^{\ast}}n^2(n-1)^2|a_n|^2|z|^{2n}+\sum_{n\in\Z^{\ast}}n^2(n+1)^2|b_n|^2|z|^{2(n+2)}\right.\\
            &+\left.+2\sum_{n\in\Z^{\ast}}n^2(n^2-1)\Re\left(a_n\bar{b_n}\right)|z|^{2(n+1)}\right),
        \end{align*}
        and
        \begin{align*}
            &\int_{\Omega}4|\p{z}^2\psi|^2dx=2\pi\int_{a}^b\frac{1}{r^3}\left(\alpha^2-2\alpha\beta r^2+\beta^2r^4+\sum_{n\in\Z^{\ast}}n^2(n-1)^2|a_n|^2r^{2n}+\sum_{n\in\Z^{\ast}}n^2(n+1)^2|b_n|^2r^{2(n+2)}\right.\\
            &\left.+2\sum_{n\in\Z^{\ast}}n^2(n^2-1)\Re\left(a_n\bar{b_n}\right)r^{2(n+1)}\right)dr=\pi\alpha^2\left(\frac{1}{a^2}-\frac{1}{b^2}\right)-4\pi\alpha\beta\log\left(\frac{b}{a}\right)+\pi\beta^2\left(b^2-a^2\right)\\
            &+\pi\sum_{n\in\Z\setminus\ens{0,1}}n^2|n-1||a_n|^2\left|b^{2(n-1)}-a^{2(n-1)}\right|+\pi\sum_{n\in\Z\setminus\ens{-1,0}}n^2|n+1||b_n|^2\left|b^{2(n+1)}-a^{2(n+1)}\right|\\
            &+2\pi\sum_{n\in\Z\setminus\ens{-1,0,1}}n(n^2-1)\Re\left(a_n\bar{b_n}\right)\left(b^{2n}-a^{2n}\right).
        \end{align*}
        We first estimate since $\gamma\notin \N$
        \begin{align*}
            &\sum_{n\in\Z^{\ast}}\frac{((n+1)^2+4)}{n+1+\gamma}|b_n|^2\left(b^{2(n+1)}-a^{2(n+1)}\left(\frac{a}{b}\right)^{2\gamma}\right)\leq\frac{4}{\gamma}|b_{-1}|^2\left(1-\left(\frac{a}{b}\right)^{2\gamma}\right)\\
            &+\frac{4}{\gamma}\sum_{n\in\Z\setminus\ens{0,-1}}|n+1||b_n|^2\left|b^{2(n+1)}-a^{2(n+1)}\left(\frac{a}{b}\right)^{2\gamma}\right|.
        \end{align*}
        Then, we have
        \begin{align*}
            &\sum_{n\in\Z^{\ast}}|n+1||b_n|^2\left|b^{2(n+1)}-a^{2(n+1)}\right|=\sum_{n\geq 0}|n+1||b_n|^2b^{2(n+1)}\left(1-\left(\frac{a}{b}\right)^{2(n+1)}\right)\\
            &+\sum_{n\geq 2}|n-1||b_{-n}|^2\frac{1}{a^{2(n-1)}}\left(1-\left(\frac{a}{b}\right)^{2(n-1)}\right)\\
            &\geq \left(1-\left(\frac{a}{b}\right)^2\right)\left(\sum_{n\geq 0}|n+1||b_n|^2b^{2(n+1)}+\sum_{n\geq 2}|n-1||b_{-n}|\frac{1}{a^{2(n-1)}}\right).
        \end{align*}
        On the other hand, we have
        \begin{align*}
            &\sum_{n\in\Z^{\ast}}|n+1||b_n|^2\left|b^{2(n+1)}-a^{2(n+1)}\left(\frac{a}{b}\right)^{2\gamma}\right|=\sum_{n\geq 0}|n+1||b_n|^2b^{2(n+1)}\left(1-\left(\frac{a}{b}\right)^{2(n+1+\gamma)}\right)\\
            &+\sum_{n\geq 2}|n-1||b_{-n}|^2\frac{1}{a^{2(n-1)}}\left(\left(\frac{a}{b}\right)^{2\gamma}-\left(\frac{a}{b}\right)^{2(n-1)}\right)\\
            &\leq \frac{1-\left(\frac{a}{b}\right)^{2\gamma}}{1-\left(\frac{a}{b}\right)^2}\sum_{n\geq 0}|n+1||b_n|^2b^{2(n+1)}\left(1-\left(\frac{a}{b}\right)^{2(n+1)}\right)\\
            &+\sum_{n\geq 2}|n-1||b_{-n}|^2\frac{1}{a^{2(n-1)}}\left(1-\left(\frac{a}{b}\right)^{2(n-1)}\right)\\
            &\leq \frac{1-\left(\frac{a}{b}\right)^{2\gamma}}{1-\left(\frac{a}{b}\right)^2}\sum_{n\in\Z^{\ast}}|n+1||b_n|^2\left|b^{2(n+1)}-a^{2(n+1)}\right|
        \end{align*}
        where we used the following inequality (valid for all $n\geq 0$)
        \begin{align*}
            \left(1-\left(\frac{a}{b}\right)^{2(n+1+\gamma)}\right)\leq \frac{1-\left(\frac{a}{b}\right)^{2\gamma}}{1-\left(\frac{a}{b}\right)^2}\left(1-\left(\frac{a}{b}\right)^{2(n+1)}\right).
        \end{align*}
        Indeed, we have
        \begin{align*}
            \frac{1-\left(\frac{a}{b}\right)^{2(n+1+\gamma)}}{1-\left(\frac{a}{b}\right)^{2(n+1)}}=\left(\frac{a}{b}\right)^{2\gamma}+\frac{1-\left(\frac{a}{b}\right)^{2\gamma}}{1-\left(\frac{a}{b}\right)^{2(n+1)}},
        \end{align*}
        which is minimal for $n=0$. Finally, we deduce that 
        \begin{align}\label{frequency_two}
            &\pi\sum_{n\in\Z^{\ast}}\frac{((n+1)^2+4)}{n+1+\gamma}|b_n|^2\left(b^{2(n+1)}-a^{2(n+1)}\left(\frac{a}{b}\right)^{2\gamma}\right)\leq \frac{4\pi}{\gamma}|b_{-1}|^2\left(1-\left(\frac{a}{b}\right)^{2\gamma}\right)\nonumber\\
            &+\frac{4\pi}{\gamma} \frac{1-\left(\frac{a}{b}\right)^{2\gamma}}{1-\left(\frac{a}{b}\right)^2}\sum_{n\in\Z^{\ast}}|n+1||b_n|^2\left|b^{2(n+1)}-a^{2(n+1)}\right|\nonumber\\
            &
            \leq \frac{4\pi}{\gamma}|b_{-1}|^2\left(1-\left(\frac{a}{b}\right)^{2\gamma}\right)+\frac{1}{2\gamma}\frac{1-\left(\frac{a}{b}\right)^{2\gamma}}{1-\left(\frac{a}{b}\right)^2}\int_{\Omega}(\Delta\psi)^2dx.
        \end{align}
        Other estimates are similar. First, we have
        \begin{align*}
            &2\int_{\Omega}|\D^2\psi|^2dx=\int_{\Omega}4|\p{z}^2\psi|^2dx+\frac{1}{4}\int_{\Omega}(\Delta\psi)^2dx=\pi\alpha^2\left(\frac{1}{a^2}-\frac{1}{b^2}\right)-4\pi\alpha\beta\log\left(\frac{b}{a}\right)+\pi\beta^2\left(b^2-a^2\right)\\
            &+\pi\sum_{n\in\Z\setminus\ens{0,1}}n^2|n-1||a_n|^2\left|b^{2(n-1)}-a^{2(n-1)}\right|+\pi\sum_{n\in\Z\setminus\ens{-1,0}}n^2|n+1||b_n|^2\left|b^{2(n+1)}-a^{2(n+1)}\right|\\
            &+2\pi\sum_{n\in\Z\setminus\ens{-1,0,1}}n(n^2-1)\Re\left(a_n\bar{b_n}\right)\left(b^{2n}-a^{2n}\right)\\
            &+4\pi\beta^2\left(b^2\left(\log^2(b)-\log(b)+\frac{1}{2}\right)-a^2\left(\log^2(a)-\log(a)+\frac{1}{2}\right)\right)\\
            &+8\pi\beta(\beta+b_0)\left(b^2\left(\log(b)-\frac{1}{2}\right)-a^2\left(\log(a)-\frac{1}{2}\right)\right)+4\pi(\beta+b_0)^2\left(b^2-a^2\right)\\
            &+\pi\sum_{n\in\Z\setminus\ens{-1,0}}|n+1||b_n|^2\left|b^{2(n+1)}-a^{2(n+1)}\right|\\
            &=\pi\alpha^2\left(\frac{1}{a^2}-\frac{1}{b^2}\right)-4\pi\alpha\beta\log\left(\frac{b}{a}\right)+\pi\beta^2\left(b^2-a^2\right)\\
            &+4\pi\left(b^2\left(\beta^2\log^2(b)+\beta(\beta+b_0)\log(b)+(\beta+b_0)\left(\frac{1}{2}\beta+b_0\right)\right)\right.\\
            &\left.-a^2\left(\beta^2\log^2\left(a\right)+\beta(\beta+b_0)\log\left(a\right)+\left(\beta+b_0\right)\left(\frac{1}{2}\beta+b_0\right)\right)\right)\\
            &+\pi\sum_{n\in\Z\setminus\ens{0,1}}n^2|n-1||a_n|^2\left|b^{2(n-1)}-a^{2(n-1)}\right|+\pi\sum_{n\in\Z\setminus\ens{-1,0}}(n^2+1)|n+1||b_n|^2\left|b^{2(n+1)}-a^{2(n+1)}\right|\\
            &+2\pi\sum_{n\in\Z\setminus\ens{-1,0,1}}n(n^2-1)\Re\left(a_n\bar{b_n}\right)\left(b^{2n}-a^{2n}\right).
        \end{align*}
        Now, we have to be particularly careful to show that the $L^2$ norm of $\D^2\psi$ bounds the sum in $\ens{a_n}_{n\in\Z}$. Notice that for all $n\geq 1$, we have
        \begin{align*}
            &n^2|n-1||a_n|^2b^{2(n-1)}+2n(n^2-1)\Re\left(a_n\bar{b_n}\right)b^{2n}+(n^2+1)|n+1||b_n|^2b^{2(n+1)}\\
            &\geq \left(n^2(n-1)-\frac{n^2(n+1)^2(n-1)^2}{\left(\sqrt{(n^2+1)(n+1)}\right)^2}\right)|a_n|^2b^{2(n-1)}\\
            &=n^2(n-1)\left(1-\frac{(n+1)(n-1)}{n^2+1}\right)|a_n|^2b^{2(n-1)}=\frac{2n^2(n-1)}{n^2+1}|a_n|^2b^{2(n-1)}.
        \end{align*}
        Let $n\geq 2$. Then, we have
        \begin{align*}
            &2n(n^2-1)|a_n||b_n||b^{2n}-a^{2n}|=2n(n^2-1)|a_n||b_n|b^{2n}\left(1-\left(\frac{a}{b}\right)^{2n}\right)\\
            &\leq \frac{n^2(n+1)(n-1)^2}{n^2+1}|a_n|^2b^{2n}\left(1-\left(\frac{a}{b}\right)^{2(n-1)}\right)+(n^2+1)(n+1)|b_n|^2b^{2(n+1)}\frac{\left(1-\left(\frac{a}{b}\right)^{2n}\right)^2}{1-\left(\frac{a}{b}\right)^{2(n-1)}}.
        \end{align*}
        Therefore, we get for all $n\geq 2$ (for $n=1$ the middle coefficient vanishes and the inequality is trivial)
        \begin{align*}
            &n^2|n-1||a_n|^2\left|b^{2(n-1)}-a^{2(n-1)}\right|+2n(n^2-1)\Re\left(a_n\bar{b_n}\right)\left(b^{2n}-a^{2n}\right)+(n^2+1)|n+1||b_n|^2b^{2(n+1)}\\
            &\leq \frac{2n^2|n-1|}{n^2+1}|a_n|^2b^{2(n-1)}\left(1-\left(\frac{a}{b}\right)^{2(n-1)}\right)+\frac{(n^2+1)|n+1|}{1-\left(\frac{a}{b}\right)^{2(n-1)}}|b_n|^2\left(\left(1-\left(\frac{a}{b}\right)^{2(n+1)}\right)\left(1-\left(\frac{a}{b}\right)^{2(n-1)}\right)\right.\\
            &\left.-\left(1-\left(\frac{a}{b}\right)^{2n}\right)^2\right)\\
            &=\frac{2n^2|n-1|}{n^2+1}|a_n|^2b^{2(n-1)}\left(1-\left(\frac{a}{b}\right)^{2(n-1)}\right)-(n^2+1)|n+1||b_n|^2\frac{\left(\left(\frac{a}{b}\right)^{n-1}-\left(\frac{a}{b}\right)^{n+1}\right)^2}{1-\left(\frac{a}{b}\right)^{2(n-1)}}\\
            &=\frac{2n^2|n-1|}{n^2+1}|a_n|^2b^{2(n-1)}\left(1-\left(\frac{a}{b}\right)^{2(n-1)}\right)-(n^2+1)|n+1||b_n|^2\left(\frac{a}{b}\right)^{2(n-1)}\frac{\left(1-\left(\frac{a}{b}\right)^{2}\right)^2}{1-\left(\frac{a}{b}\right)^{2(n-1)}}\\
            &\geq \frac{2n^2|n-1|}{n^2+1}|a_n|^2b^{2(n-1)}\left(1-\left(\frac{a}{b}\right)^{2(n-1)}\right)-(n^2+1)|n+1||b_n|^2\left(\frac{a}{b}\right)^{2(n-1)}\left(1-\left(\frac{a}{b}\right)^{2}\right).
        \end{align*}
        Now, we compute
        \begin{align*}
            \sum_{n=2}^{\infty}(n^2+1)\left(\frac{a}{b}\right)^{2(n-1)}=\left(\frac{a}{b}\right)^2\frac{5-5\left(\frac{a}{b}\right)^2+2\left(\frac{a}{b}\right)^4}{\left(1-\left(\frac{a}{b}\right)^2\right)^3}.
        \end{align*}
        Since
        \begin{align*}
            \int_{\Omega}(\Delta\psi)^2dx\geq 8\pi \left(1-\left(\frac{a}{b}\right)^2\right)\sum_{n\geq 2}|n+1||b_n|^2b^{2(n+1)},
        \end{align*}
        we deduce that
        \begin{align*}
            &2\int_{\Omega}|\D^2\psi|^2dx+\frac{1}{8}\left(\frac{a}{b}\right)^2\frac{5-5\left(\frac{a}{b}\right)^2+2\left(\frac{a}{b}\right)^4}{\left(1-\left(\frac{a}{b}\right)^2\right)^3}\int_{\Omega}(\Delta\psi)^2dx\\
            &\geq 2\pi\sum_{n\geq 2}\frac{n^2|n-1|}{n^2+1}|a_n|^2b^{2(n-1)}\left(1-\left(\frac{a}{b}\right)^{2(n-1)}\right).
        \end{align*}
        Reasoning similarly for negative frequencies, we deduce in fact that 
        \begin{align}\label{frequency_one}
            &2\pi\sum_{n\in \Z\setminus\ens{0,1}}\frac{n^2|n-1|}{n^2+1}|a_n|^2|b^{2(n-1)}-a^{2(n-1)}|\leq 2\int_{\Omega}|\D^2\psi|^2dx\nonumber\\
            &+\frac{1}{8}\left(\frac{a}{b}\right)^2\frac{5-5\left(\frac{a}{b}\right)^2+2\left(\frac{a}{b}\right)^4}{\left(1-\left(\frac{a}{b}\right)^2\right)^3}\int_{\Omega}(\Delta\psi)^2dx,
        \end{align}
        that we rewrite more simply as
        \begin{align*}
            2\pi\sum_{n\in\Z\setminus\ens{0,1}}\frac{n^2|n-1|}{n^2+1}|a_n|^2\left|b^{2(n-1)}-a^{2(n-1)}\right|\leq 2\left(1+\left(\frac{a}{b}\right)^2\frac{5-5\left(\frac{a}{b}\right)^2+2\left(\frac{a}{b}\right)^4}{\left(1-\left(\frac{a}{b}\right)^2\right)^3}\right)\int_{\Omega}|\D^2\psi|^2dx.
        \end{align*}
        Therefore, we get
        \begin{align*}
            &\pi\sum_{n\in \Z^{\ast}}\frac{|n|^2}{|n-1+\gamma|}|a_n|^2\left|b^{2(n-1)}-a^{2(n-1)}\left(\frac{a}{b}\right)^{2\gamma}\right|\leq \frac{\pi}{\gamma}|a_1|^2\left(1-\left(\frac{a}{b}\right)^{2\gamma}\right)\\
            &+\frac{5\pi}{\gamma}\sum_{n\in\Z\setminus\ens{0,1}}\frac{n^2|n-1|}{n^2+1}|a_n|^2\left|b^{2(n-1)}-a^{2(n-1)}\left(\frac{a}{b}\right)^{2\gamma}\right|\\
            &\leq \frac{\pi}{\gamma}|a_1|^2\left(1-\left(\frac{a}{b}\right)^{2\gamma}\right)+\frac{5\pi}{1-\gamma}\frac{1-\left(\frac{a}{b}\right)^{2\gamma}}{1-\left(\frac{a}{b}\right)^2}\sum_{n\in\Z\setminus\ens{0,1}}\frac{n^2|n-1|}{n^2+1}|a_n|^2\left|b^{2(n-1)}-a^{2(n-1)}\right|\\
            &\leq \frac{\pi}{\gamma}|a_1|^2\left(1-\left(\frac{a}{b}\right)^{2\gamma}\right)+\frac{5}{1-\gamma}\frac{1-\left(\frac{a}{b}\right)^{2\gamma}}{1-\left(\frac{a}{b}\right)^2}\left(1+\left(\frac{a}{b}\right)^2\frac{5-5\left(\frac{a}{b}\right)^2+2\left(\frac{a}{b}\right)^4}{\left(1-\left(\frac{a}{b}\right)^2\right)^3}\right)\int_{\Omega}|\D^2\psi|^2dx.
        \end{align*}
        Now, we move to the estimation of the bound the logarithm contributions.  We trivially have
        \begin{align*}
            \frac{\pi\alpha^2}{1-\gamma}\left(\frac{1}{a^2}\left(\frac{a}{b}\right)^{2\gamma}-\frac{1}{b^2}\right)\leq \frac{\pi\alpha^2}{1-\gamma}\left(\frac{1}{a^2}-\frac{1}{b^2}\right).
        \end{align*}
        Then, 
        \begin{align*}
            \frac{4\pi\beta^2}{\gamma+1}\left(b^2\left(\log^2(b)-\frac{\log(b)}{\gamma+1}+\frac{1}{2(\gamma+1)^2}\right)-a^2\left(\log^2(a)-\frac{\log(a)}{\gamma+1}+\frac{1}{2(\gamma+1)^2}\right)\left(\frac{a}{b}\right)^{2\gamma}\right)
        \end{align*}
        needs to be compared to 
        \begin{align*}
            16\pi\beta^2\left(b^2\left(\log^2(b)-\log(b)+\frac{1}{2}\right)-a^2\left(\log^2(a)-\log(a)+\frac{1}{2}\right)\right).
        \end{align*}
        Since we have assumed that $0<a<b<1$, we trivially estimate
        \begin{align*}
            b^2\left(\log^2(b)-\frac{\log(b)}{\gamma+1}+\frac{1}{2(\gamma+1)^2}\right)=b^2\left(\log^2(b)+\frac{|\log(b)|}{\gamma+1}+\frac{1}{2(\gamma+1)^2}\right)\leq b^2\left(\log^2(b)-\log(b)+\frac{1}{2}\right).
        \end{align*}
        On the other hand,
        \begin{align*}
            a^2\left(\log^2(a)-\frac{\log(a)}{\gamma+1}+\frac{1}{2(\gamma+1)^2}\right)\geq \frac{1}{(\gamma+1)^2}a^2\left(\log^2(a)-\log(a)+\frac{1}{2}\right).
        \end{align*}
        Therefore, we have
        \begin{align*}
            &\frac{4\pi\beta^2}{\gamma+1}\left(b^2\left(\log^2(b)-\frac{\log(b)}{\gamma+1}+\frac{1}{2(\gamma+1)^2}\right)-a^2\left(\log^2(a)-\frac{\log(a)}{\gamma+1}+\frac{1}{2(\gamma+1)^2}\right)\left(\frac{a}{b}\right)^{2\gamma}\right)\\
            &\leq \frac{4\pi\beta^2}{(\gamma+1)^3}\left((\gamma+1)^2b^2\left(\log^2(b)-\log(b)+\frac{1}{2}\right)-a^2\left(\log^2(a)-\log(a)+\frac{1}{2}\right)\right).
        \end{align*}
        Recalling that 
        \begin{align*}
            \mathrm{rad}\left(z\,\p{z}\psi\right)=\frac{1}{2}\left(\alpha+2\beta|z|^2\log|z|+(\beta+b_0)|z|^2\right),
        \end{align*}
        we deduce by the triangle inequality and the elementary inequality $(a+b+c)^2\leq 4(a^2+b^2+c^2)$ that
        \begin{align}\label{simplified_gradient}
            &\int_{\Omega}\frac{|\D\psi|^2}{|x|^2}\left(\frac{|x|}{b}\right)^{2\gamma}dx\leq \frac{4\pi\alpha^2}{1-\gamma}\left(\frac{1}{a^2}\left(\frac{a}{b}\right)^{2\gamma}-\frac{1}{b^2}\right)
            +\frac{16\pi\beta^2}{\gamma+1}\left(b^2\left(\log^2\left(b\right)-\frac{\log(b)}{\gamma+1}+\frac{1}{2(\gamma+1)^2}\right)\right.\nonumber\\
            &\left.-a^2\left(\log^2(a)-\frac{\log(a)}{\gamma+1}+\frac{1}{2(\gamma+1)^2}\right)\left(\frac{a}{b}\right)^{2\gamma}\right)\nonumber
            +\frac{4\pi(\beta+2b_0)^2}{\gamma+1}\left(b^2-a^2\left(\frac{a}{b}\right)^{2\gamma}\right)^{2\gamma}\nonumber\\
            &+2\pi\sum_{n\in\Z^{\ast}}\frac{|n|^2}{n-1+\gamma}|a_n|^2\left(b^{2(n-1)}-a^{2(n-1)}\left(\frac{a}{b}\right)^{2\gamma}\right)\nonumber\\
            &+2\pi\sum_{n\in\Z^{\ast}}\frac{((n+1)^2+4)}{n+1+\gamma}|b_n|^2\left(b^{2(n+1)}-a^{2(n+1)}\left(\frac{a}{b}\right)^{2\gamma}\right).
        \end{align}
        
        Then, we need to bound $a_1$ and $b_{-1}$ by the weighted $L^2$ norm of $\psi$. Notice that we have
        \begin{align}\label{pointwise_coefficient}
            &\pi\sum_{n\in\Z}\frac{1}{n-1+2\gamma}|a_n|^2\left(b^{2(n-1)}-a^{2(n-1)}\left(\frac{a}{b}\right)^{4\gamma}\right)
            +\pi\sum_{n\in\Z}\frac{1}{n+1+2\gamma}|b_n|^2\left(b^{2(n+1)}-a^{2(n+1)}\left(\frac{a}{b}\right)^{4\gamma}\right)\nonumber\\
            &+2\pi\sum_{n\in\Z}\frac{1}{n+2\gamma}\Re\left(a_n\bar{b_n}\right)\left(b^{2n}-a^{2n}\left(\frac{a}{b}\right)^{4\gamma}\right)\leq \int_{\Omega}\frac{\psi^2}{|x|^4}\left(\frac{|x|}{b}\right)^{4\gamma}dx,
        \end{align}
        and since each term is positive (since it corresponds to the integration of a squared quantity), we deduce that for all $n\in\Z$, we also have 
        \begin{align}\label{pointwise_coefficient2}
            &\frac{1}{n-1+2\gamma}|a_n|^2\left(b^{2(n-1)}-a^{2(n-1)}\left(\frac{a}{b}\right)^{4\gamma}\right)
            +\frac{1}{n+1+2\gamma}|b_n|^2\left(b^{2(n+1)}-a^{2(n+1)}\left(\frac{a}{b}\right)^{4\gamma}\right)\nonumber\\
            &+\frac{2}{n+2\gamma}\Re\left(a_n\bar{b_n}\right)\left(b^{2n}-a^{2n}\left(\frac{a}{b}\right)^{4\gamma}\right)\leq \frac{1}{\pi}\int_{\Omega}\frac{\psi^2}{|x|^4}\left(\frac{|x|}{b}\right)^{4\gamma}dx.
        \end{align}
        For $n=-1$, the inequality becomes
        \begin{align}\label{ineq_b_minus_one}
            &\frac{1}{2(1-\gamma)}|a_{-1}|^2\frac{1}{a^4}\left(\left(\frac{a}{b}\right)^{4\gamma}-\left(\frac{a}{b}\right)^4\right)+\frac{1}{2\gamma}|b_{-1}|^2\left(1-\left(\frac{a}{b}\right)^{4\gamma}\right)\nonumber\\
            &-\frac{2}{2\gamma-1}\Re\left(a_{-1}\bar{b_{-1}}\right)\frac{1}{a^2}\left(\left(\frac{a}{b}\right)^{4\gamma}-\left(\frac{a}{b}\right)^2\right)\leq \frac{1}{\pi}\int_{\Omega}\frac{\psi^2}{|x|^4}\left(\frac{x}{b}\right)^{4\gamma}dx.
        \end{align}
        Thanks to Cauchy's inequality, we have
        \begin{align*}
            &-\frac{2}{2\gamma-1}\Re\left(a_{-1}\bar{b_{-1}}\right)\frac{1}{a^2}\left(\left(\frac{a}{b}\right)^{4\gamma}-\left(\frac{a}{b}\right)^2\right)\geq -\frac{1}{2(1-\gamma)}|a_{-1}|^2\frac{1}{a^4}\left(\left(\frac{a}{b}\right)^{4\gamma}-\left(\frac{a}{b}\right)^4\right)\\
            &-\frac{2(1-\gamma)}{(2\gamma-1)^2}|b_{-1}|^2\frac{\left(\left(\frac{a}{b}\right)^{4\gamma}-\left(\frac{a}{b}\right)^2\right)^2}{\left(\frac{a}{b}\right)^{4\gamma}-\left(\frac{a}{b}\right)^4}.
        \end{align*}
        Now, we have
        \begin{align*}
            &\frac{1}{2\gamma}\left(1-\left(\frac{a}{b}\right)^{4\gamma}\right)-\frac{2(1-\gamma)}{(2\gamma-1)^2}\frac{\left(\left(\frac{a}{b}\right)^{4\gamma}-\left(\frac{a}{b}\right)^2\right)^2}{\left(\frac{a}{b}\right)^{4\gamma}-\left(\frac{a}{b}\right)^4}
            =\frac{1}{2\gamma(2\gamma-1)^2}\frac{1}{\left(\frac{a}{b}\right)^{4\gamma}-\left(\frac{a}{b}\right)^4}\\
            &\times\left((2\gamma-1)^2\left(1-\left(\frac{a}{b}\right)^{4\gamma}\right)\left(\left(\frac{a}{b}\right)^{4\gamma}-\left(\frac{a}{b}\right)^4\right)-4\gamma(1-\gamma)\left(\left(\frac{a}{b}\right)^{4\gamma}-\left(\frac{a}{b}\right)^2\right)^2\right).
        \end{align*}
        Notice that the quantity in parentheses is strictly negative for $\gamma=1/2$ and vanishes for $\gamma=1$. More precisely, as $\gamma\rightarrow 1$, we have
        \begin{align*}
            \frac{1-\gamma}{\left(\frac{a}{b}\right)^{4\gamma}-\left(\frac{a}{b}\right)^4}=\frac{1}{\left(\frac{a}{b}\right)^{4\gamma}}\frac{1-\gamma}{1-e^{-4(1-\gamma)\log\left(\frac{b}{a}\right)}}\conv{\gamma\rightarrow 1}\left(\frac{b}{a}\right)^4\frac{1}{4\log\left(\frac{b}{a}\right)}.
        \end{align*}
        Therefore, we have 
        \begin{align*}
            \frac{1}{2\gamma}\left(1-\left(\frac{a}{b}\right)^{4\gamma}\right)-\frac{2(1-\gamma)}{(2\gamma-1)^2}\frac{\left(\left(\frac{a}{b}\right)^{4\gamma}-\left(\frac{a}{b}\right)^2\right)^2}{\left(\frac{a}{b}\right)^{4\gamma}-\left(\frac{a}{b}\right)^4}&\conv{\gamma\rightarrow 1}\frac{1}{2}\left(1-\left(\frac{a}{b}\right)^4\right)-\frac{1}{2\log\left(\frac{b}{a}\right)}\left(\frac{b}{a}\right)^4\left(\left(\frac{a}{b}\right)^4-\left(\frac{a}{b}\right)^2\right)^2\\
            &=\frac{1}{2}\left(1-\left(\frac{a}{b}\right)^4\right)-\frac{1}{2\log\left(\frac{b}{a}\right)}\left(1-\left(\frac{a}{b}\right)^2\right)^2\conv{\frac{b}{a}\rightarrow\infty}\frac{1}{2}>0.
        \end{align*}
        Therefore, for $\gamma<1$ and $\log\left(\frac{b}{a}\right)>1$ large enough, this constant will be positive, which will give us a control of $b_{-1}$! Let us now get explicit estimates for which our function is positive.

        We also notice if $x=\dfrac{a}{b}$ that
        \begin{align*}
            \frac{(x^{4\gamma}-x^2)^2}{x^{4\gamma}-x^4}=x^{4(1-\gamma)}\frac{(1-x^{2(2\gamma-1)})^2}{1-x^{4(1-\gamma)}}\conv{x\rightarrow 0}0.
        \end{align*}
        Therefore, for all fixed $\dfrac{1}{2}<\gamma<1$, we have
        \begin{align*}
            \frac{1}{2\gamma}\left(1-\left(\frac{a}{b}\right)^{4\gamma}\right)-\frac{2(1-\gamma)}{(2\gamma-1)^2}\frac{\left(\left(\frac{a}{b}\right)^{4\gamma}-\left(\frac{a}{b}\right)^2\right)^2}{\left(\frac{a}{b}\right)^{4\gamma}-\left(\frac{a}{b}\right)^4}\conv{\frac{b}{a}\rightarrow \infty}\frac{1}{2\gamma}>0.
        \end{align*}
        Therefore, for all $\dfrac{1}{2}<\gamma<1$, there exists $R_{\gamma}<\infty$ such that for all $0<a<b<\infty$ such that
        \begin{align*}
            \log\left(\frac{b}{a}\right)\geq R_{\gamma}',
        \end{align*}
        we have
        \begin{align*}
            \frac{1}{2\gamma}\left(1-\left(\frac{a}{b}\right)^{4\gamma}\right)-\frac{2(1-\gamma)}{(2\gamma-1)^2}\frac{\left(\left(\frac{a}{b}\right)^{4\gamma}-\left(\frac{a}{b}\right)^2\right)^2}{\left(\frac{a}{b}\right)^{4\gamma}-\left(\frac{a}{b}\right)^4}\geq \frac{1}{4\gamma}.
        \end{align*}
        Let us estimate some $R_{\gamma}$ (that gives a similar inequality). Since
        \begin{align*}
            \frac{(x^{4\gamma}-x^2)^2}{x^{4\gamma}-x^4}=x^{4(1-\gamma)}\frac{(1-x^{2(2\gamma-1)})^2}{1-x^{4(1-\gamma)}}\leq \frac{x^{4(1-\gamma)}}{1-x^{4(1-\gamma)}},
        \end{align*}
        we notice that
        \begin{align*}
            \frac{x^{4(1-\gamma)}}{1-x^{4(1-\gamma)}}\leq \frac{1}{4\gamma}\frac{(2\gamma-1)^2}{2\gamma(1-\gamma)}
        \end{align*}
        if and only if
        \begin{align*}
            x^{4(1-\gamma)}\leq \frac{\frac{(2\gamma-1)^2}{8\gamma(1-\gamma)}}{\frac{(2\gamma-1)^2}{8\gamma(1-\gamma)}+1}=\frac{(2\gamma-1)^2}{(2\gamma-1)^2+8\gamma(1-\gamma)}
        \end{align*}
        which gives the condition
        \begin{align*}
            \frac{b}{a}\geq \left(\frac{(2\gamma-1)^2+8\gamma(1-\gamma)}{(2\gamma-1)^2}\right)^{\frac{1}{4(1-\gamma)}}
        \end{align*}
        and finally
        \begin{align}\label{conformal_class}
            \log\left(\frac{b}{a}\right)\geq \frac{1}{4(1-\gamma)}\log\left(1+\frac{8\gamma(1-\gamma)}{(2\gamma-1)^2}\right).
        \end{align}
        Therefore, we get the estimate
        \begin{align*}
            R_{\gamma}\leq \frac{1}{4(1-\gamma)}\log\left(1+\frac{8\gamma(1-\gamma)}{(2\gamma-1)^2}\right)
        \end{align*}
        which blows us as $\gamma\rightarrow \dfrac{1}{2}$ and $\gamma\rightarrow 1$. Summarising, assuming that \eqref{conformal_class} holds, we deduce that
        \begin{align*}
            \frac{1}{2\gamma}\left(1-\left(\frac{a}{b}\right)^{4\gamma}\right)-\frac{2(1-\gamma)}{(2\gamma-1)^2}\frac{\left(\left(\frac{a}{b}\right)^{4\gamma}-\left(\frac{a}{b}\right)^2\right)^2}{\left(\frac{a}{b}\right)^{4\gamma}-\left(\frac{a}{b}\right)^4}\geq \frac{1}{4\gamma}\left(1-2\left(\frac{a}{b}\right)^{4\gamma}\right)>0
        \end{align*}
        provided that
        \begin{align*}
            \log\left(\frac{b}{a}\right)\geq \frac{1}{4\gamma}\log(2)
        \end{align*}
        that is maximal at $\gamma=\dfrac{1}{2}$. An easy estimate shows that 
        \begin{align*}
            \frac{1}{4(1-\gamma)}\log\left(1+\frac{8\gamma(1-\gamma)}{(2\gamma-1)^2}\right)\geq \frac{1}{2}\log(2)\quad \text{for all}\;\, \frac{1}{2}<\gamma<1,
        \end{align*}
        so the second condition is always satisfied provided that \eqref{conformal_class} holds. Therefore, we finally deduce that \eqref{conformal_class} implies that 
        \begin{align*}
            &\frac{1}{4\gamma}|b_{-1}|^2\left(1-2\left(\frac{a}{b}\right)^{4\gamma}\right)\leq 
            \frac{1}{2(1-\gamma)}|a_{-1}|^2\frac{1}{a^4}\left(\left(\frac{a}{b}\right)^{4\gamma}-\left(\frac{a}{b}\right)^4\right)+\frac{1}{2\gamma}|b_{-1}|^2\left(1-\left(\frac{a}{b}\right)^{4\gamma}\right)\\
            &-\frac{2}{2\gamma-1}\Re\left(a_{-1}\bar{b_{-1}}\right)\frac{1}{a^2}\left(\left(\frac{a}{b}\right)^{4\gamma}-\left(\frac{a}{b}\right)^2\right)\leq \frac{1}{\pi}\int_{\Omega}\frac{\psi^2}{|x|^4}\left(\frac{|x|}{b}\right)^{4\gamma}dx,
        \end{align*}
        which shows that 
        \begin{align*}
            \frac{4\pi}{\gamma}|b_{-1}|^2\left(1-\left(\frac{a}{b}\right)^{2\gamma}\right)\leq 16\frac{1-\left(\frac{a}{b}\right)^{2\gamma}}{1-2\left(\frac{a}{b}\right)^{4\gamma}}\int_{\Omega}\frac{\psi^2}{|x|^4}\left(\frac{|x|}{b}\right)^{4\gamma}dx.
        \end{align*}
        More generally, for all $\dfrac{1}{2}<\delta<1$, if 
        \begin{align}\label{conformal_class2}
            \log\left(\frac{b}{a}\right)\geq \frac{1}{4(1-\delta)}\log\left(1+\frac{8\delta(1-\delta)}{(2\delta-1)^2}\right),
        \end{align}
        we get
        \begin{align}\label{b_minus_one_estimate}
            \frac{4\pi}{\gamma}|b_{-1}|^2\left(1-\left(\frac{a}{b}\right)^{2\gamma}\right)\leq \frac{16\,\delta}{\gamma}\frac{1-\left(\frac{a}{b}\right)^{2\gamma}}{1-2\left(\frac{a}{b}\right)^{4\delta}}\int_{\Omega}\frac{\psi^2}{|x|^4}\left(\frac{|x|}{b}\right)^{4\delta}dx.
        \end{align}
        Now, we move to the estimate of $|a_1|^2$. Taking \eqref{pointwise_coefficient2} for $n=1$ yields
        \begin{align*}
            &\frac{1}{2\gamma}|a_1|^2\left(1-\left(\frac{a}{b}\right)^{4\gamma}\right)+\frac{1}{2(\gamma+1)}|b_1|^2b^4\left(1-\left(\frac{a}{b}\right)^{4(\gamma+1)}\right)\\
            &+\frac{2}{2\gamma+1}\Re\left(a_1\bar{b_1}\right)b^2\left(1-\left(\frac{a}{b}\right)^{2(2\gamma+1)}\right)\leq \frac{1}{\pi}\int_{\Omega}\frac{\psi^2}{|x|^4}\left(\frac{|x|}{b}\right)^{4\gamma}dx.
        \end{align*}
        As previously, we get
        \begin{align}\label{step1_a1}
            &\frac{1}{2\gamma}|a_1|^2\left(1-\left(\frac{a}{b}\right)^{4\gamma}\right)+\frac{1}{2(\gamma+1)}|b_1|^2b^4\left(1-\left(\frac{a}{b}\right)^{4(\gamma+1)}\right)\\
            &+\frac{2}{2\gamma+1}\Re\left(a_1\bar{b_1}\right)b^2\left(1-\left(\frac{a}{b}\right)^{2(2\gamma+1)}\right)\nonumber\\
            &\geq |a_1|^2\left(\frac{1}{2\gamma}\left(1-\left(\frac{a}{b}\right)^{4\gamma}\right)-\frac{2(\gamma+1)}{(2\gamma+1)^2}\frac{\left(1-\left(\frac{a}{b}\right)^{2(2\gamma+1)}\right)^2}{1-\left(\frac{a}{b}\right)^{4(\gamma+1)}}\right)\nonumber\\
            &=\frac{|a_1|^2}{2\gamma(2\gamma+1)^2}\left((2\gamma+1)^2\left(1-\left(\frac{a}{b}\right)^{4\gamma}\right)\left(1-\left(\frac{a}{b}\right)^{4(\gamma+1)}\right)-4\gamma(\gamma+1)\left(1-\left(\frac{a}{b}\right)^{2(2\gamma+1)}\right)^2\right)\nonumber\\
            &=\frac{|a_1|^2}{2\gamma(2\gamma+1)^2}\left(\left(1-\left(\frac{a}{b}\right)^{4\gamma}\right)\left(1-\left(\frac{a}{b}\right)^{4(\gamma+1)}\right)-4\gamma(\gamma+1)\left(\left(\frac{a}{b}\right)^{2\gamma}-\left(\frac{a}{b}\right)^{2(\gamma+1)}\right)^2\right).
        \end{align}
        Notice that the quantity within parentheses converges to $1$ as $\dfrac{b}{a}\rightarrow\infty$.
        We trivially estimate
        \begin{align*}
            \left(x^{2\gamma}-x^{2(\gamma+1)}\right)^{2}=x^{4\gamma}\left(1-x^{2}\right)^2\leq x^{4\gamma}\leq \frac{1}{8\gamma(\gamma+1)}
        \end{align*}
        if and only if
        \begin{align*}
            \frac{b}{a}\geq \left(8\gamma(\gamma+1)\right)^{\frac{1}{4\gamma}},
        \end{align*}
        or
        \begin{align}\label{conformal_class3}
            \log\left(\frac{b}{a}\right)\geq \frac{1}{4\gamma}\log\left(8\gamma(\gamma+1)\right).
        \end{align}
        Provided that this estimate holds, we get
        \begin{align*}
            &\frac{|a_1|^2}{2\gamma(2\gamma+1)^2}\left(\left(1-\left(\frac{a}{b}\right)^{4\gamma}\right)\left(1-\left(\frac{a}{b}\right)^{4(\gamma+1)}\right)-4\gamma(\gamma+1)\left(\left(\frac{a}{b}\right)^{2\gamma}-\left(\frac{a}{b}\right)^{2(\gamma+1)}\right)^2\right)\\
            &\geq \frac{|a_1|^2}{2\gamma(2\gamma+1)^2}\left(\frac{1}{2}-\left(\frac{a}{b}\right)^{4\gamma}-\left(\frac{a}{b}\right)^{4(\gamma+1)}+\left(\frac{a}{b}\right)^{4(2\gamma+1)}\right)\\
            &\geq \frac{|a_1|^2}{2\gamma(2\gamma+1)^2}\left(1-2\left(\frac{a}{b}\right)^{4\gamma}\right)>0
        \end{align*}
        provided that
        \begin{align}\label{conformal_class4}
            \log\left(\frac{b}{a}\right)\geq \frac{1}{4\gamma}\log(2),
        \end{align}
        which is implied by \eqref{conformal_class3}. Therefore, for all $0<a<b<\infty$ satisfying \eqref{conformal_class3}, we deduce that
        \begin{align*}
            \frac{|a_1|^2}{2\gamma(2\gamma+1)^2}\left(1-2\left(\frac{a}{b}\right)^{4\gamma}\right)\leq \frac{1}{\pi}\int_{\Omega}\frac{\psi^2}{|x|^4}\left(\frac{|x|}{b}\right)^{4\gamma}dx,
        \end{align*}
        and finally
        \begin{align*}
            \frac{\pi}{\gamma}|a_1|^2\left(1-\left(\frac{a}{b}\right)^{2\gamma}\right)\leq 2(2\gamma+1)^2\frac{1-\left(\frac{a}{b}\right)^{2\gamma}}{1-2\left(\frac{a}{b}\right)^{4\gamma}}\int_{\Omega}\frac{\psi^2}{|x|^4}\left(\frac{|x|}{b}\right)^{4\gamma}dx.
        \end{align*}
        More generally, for all $\dfrac{1}{2}<\delta<1$, provided that
        \begin{align}\label{conformal_class2bis}
            \log\left(\frac{b}{a}\right)\geq \frac{1}{4\delta}\log\left(8\delta(\delta+1)\right),
        \end{align}
        we have
        \begin{align}\label{a1_estimate}
            \frac{\pi}{\gamma}|a_1|^2\left(1-\left(\frac{a}{b}\right)^{2\gamma}\right)\leq \frac{2\delta(2\delta+1)^2}{\gamma}\frac{1-\left(\frac{a}{b}\right)^{2\gamma}}{1-2\left(\frac{a}{b}\right)^{4\delta}}\int_{\Omega}\frac{\psi^2}{|x|^4}\left(\frac{|x|}{b}\right)^{4\delta}dx.
        \end{align}
       Now, coming back to the $L^2$ norm of $\p{z}^2\psi$, recall that
       \begin{align*}
           \pi\alpha^2\left(\frac{1}{a^2}-\frac{1}{b^2}\right)-4\pi\alpha\beta\log\left(\frac{b}{a}\right)+\pi\beta^2\left(b^2-a^2\right)\leq 4\int_{\Omega}|\p{z}^2\psi|^2dx.
       \end{align*}
       We trivially estimate
       \begin{align*}
           &\alpha^2\left(\frac{1}{a^2}-\frac{1}{b^2}\right)-4\alpha\beta\log\left(\frac{b}{a}\right)+\beta^2\left(b^2-a^2\right)\geq \alpha^2\left(\frac{1}{a^2}-\frac{1}{b^2}-\frac{4\log^2\left(\frac{b}{a}\right)}{b^2-a^2}\right)\\
           &=\alpha^2\frac{1}{a^2}\left(1-\left(\frac{a}{b}\right)^2-\frac{4\left(\frac{a}{b}\right)^2}{1-\left(\frac{a}{b}\right)^2}\log^2\left(\frac{b}{a}\right)\right).
       \end{align*}
       Since
       \begin{align*}
           \lim_{x\rightarrow 0}\left(1-x^2-\frac{4x^2}{1-x^2}\log^2\left(\frac{1}{x}\right)\right)=1
       \end{align*}
       for $\dfrac{b}{a}$ large enough, we get
       \begin{align*}
           \frac{\alpha^2}{2}\frac{1}{a^2}\leq \frac{4}{\pi}\int_{\Omega}|\p{z}^2\psi|^2dx.
       \end{align*}
       Let us find an explicit estimate of the conformal class for which this inequality holds. We have
       \begin{align*}
           1-x^2-\frac{4x^2}{1-x^2}\log^2(x)\geq \frac{1}{2}
       \end{align*}
       if and only if
       \begin{align*}
           2(1-x^2)^2-8x^2\log^2(x)\geq 1-x^2,
       \end{align*}
       or
       \begin{align*}
           (1-x^2)(1-2x^2)-8x^2\log^2(x)\geq 0.
       \end{align*}
       Making a change of variable $\log(x)=-t$, we are led to study the function
       \begin{align*}
           f(t)=\left(1-e^{-2t}\right)\left(1-2e^{-2t}\right)-8t^2e^{-2t}=1+2e^{-4t}-\left(3+8t^2\right)e^{-2t}.
       \end{align*}
       We have
       \begin{align*}
           &f'(t)=-8e^{-4t}-16te^{-2t}+(6+16t^2)e^{-2t}=-8e^{-4t}+(6-16t+16t^2)e^{-2t}.
       \end{align*}
       We have $f'(t)> 0$ provided that $6-16t+16t^2\geq 8$, or $8t^2-8t-1\geq 0$, which holds true for
       \begin{align*}
           t\geq \frac{1}{2}+\frac{1}{2}\sqrt{\frac{3}{2}}=1.11237\cdots
       \end{align*}
       Furthermore, we have
       \begin{align*}
           f(2)=1+2e^{-8}-35e^{-4}=e^{-8}\left(e^8+2-35e^4\right)>e^{-4}\left(e^4-35\right)>0,
       \end{align*}
       where the last inequality holds (for example) by the elementary estimate $e>2.7$. Therefore, we have $f(t)>0$ for all $t\geq 2$. In other words, for all $x\leq e^{-2}$, we have
       \begin{align*}
           1-x^2-\frac{4x^2}{1-x^2}\log^2(x)\geq \frac{1}{2}.
       \end{align*}
       Summarising, provided that
       \begin{align}\label{conforma_class5}
           \log\left(\frac{b}{a}\right)\geq 2,
       \end{align}
       we have
       \begin{align}\label{alpha_estimate_step1}
           \frac{\alpha^2}{a^2}\leq \frac{8}{\pi}\int_{\Omega}|\p{z}^2\psi|^2dx.
       \end{align}
       Finally, provided that \eqref{conforma_class5} holds, we get
       \begin{align}\label{log_alpha}
           \frac{4\pi\alpha^2}{1-\gamma}\left(\frac{1}{a^2}\left(\frac{a}{b}\right)^{2\gamma}-\frac{1}{b^2}\right)\leq \frac{32}{1-\gamma}\left(\frac{a}{b}\right)^{2\gamma}\int_{\Omega}|\p{z}^2\psi|^2dx\leq \frac{16}{1-\gamma}\left(\frac{a}{b}\right)^{2\gamma}\int_{\Omega}|\D^2\psi|^2dx.
       \end{align}
       Then, we have
       \begin{align*}
           &\alpha^2\left(\frac{1}{a^2}-\frac{1}{b^2}\right)-4\alpha\beta\log\left(\frac{b}{a}\right)-4\alpha\beta\log\left(\frac{b}{a}\right)+\beta^2(b^2-a^2)\geq \alpha^2\frac{1}{a^2}\left(1-\left(\frac{a}{b}\right)^2-\frac{8\left(\frac{a}{b}\right)^2}{1-\left(\frac{a}{b}\right)^2}\log^2\left(\frac{b}{a}\right)\right)\\
           &+\frac{1}{2}\beta^2(b^2-a^2).
        \end{align*}
        If the conformal class satisfies the inequality \eqref{conforma_class5}, we deduce in particular that
        \begin{align*}
            \frac{4x^2}{1-x^2}\log^2(x)\leq \frac{1}{2}-x^2,
        \end{align*}
        which shows that 
        \begin{align*}
             &\alpha^2\frac{1}{a^2}\left(1-\left(\frac{a}{b}\right)^2-\frac{8\left(\frac{a}{b}\right)^2}{1-\left(\frac{a}{b}\right)^2}\log^2\left(\frac{b}{a}\right)\right)
           +\frac{1}{2}\beta^2(b^2-a^2)\geq \frac{1}{2}\beta^2(b^2-a^2)+\alpha^2\frac{1}{a^2}\left(\frac{a}{b}\right)^2\geq \frac{1}{2}\beta^2(b^2-a^2).
        \end{align*}
        Therefore, we have
        \begin{align}\label{beta2_bound}
            \beta^2(b^2-a^2)\leq \frac{8}{\pi}\int_{\Omega}|\p{z}^2\psi|^2.
        \end{align}
        Now, in order to estimate the logarithm contributions, we need to bound $b_0$. Applying \eqref{pointwise_coefficient2} to $n=0$, we get
        \begin{align*}
            &\frac{1}{2\gamma-1}a_0^2\left(\frac{1}{b^2}-\frac{1}{a^2}\left(\frac{a}{b}\right)^{4\gamma}\right)+\frac{1}{2\gamma+1}b_0^2\left(b^{2}-a^2\left(\frac{a}{b}\right)^{4\gamma}\right)\\
            &+\frac{1}{2\gamma}a_0b_0\left(1-\left(\frac{a}{b}\right)^{4\gamma}\right)\leq \frac{1}{\pi}\int_{\Omega}\frac{\psi^2}{|x|^4}\left(\frac{|x|}{b}\right)^{4\gamma}dx.
        \end{align*}
        We have
        \begin{align*}
            &\frac{1}{2\gamma-1}a_0^2\left(\frac{1}{b^2}-\frac{1}{a^2}\left(\frac{a}{b}\right)^{4\gamma}\right)+\frac{1}{2\gamma+1}b_0^2\left(b^{2}-a^2\left(\frac{a}{b}\right)^{4\gamma}\right)
            +\frac{1}{\gamma}a_0b_0\left(1-\left(\frac{a}{b}\right)^{4\gamma}\right)\\
            &=\frac{1}{2\gamma-1}a_0^2\frac{1}{b^2}\left(1-\left(\frac{a}{b}\right)^{2(2\gamma-1)}\right)+\frac{1}{2\gamma+1}b_0^2b^2\left(1-\left(\frac{a}{b}\right)^{2(2\gamma+1)}\right)+\frac{1}{\gamma}a_0b_0\left(1-\left(\frac{a}{b}\right)^{4\gamma}\right)\\
            &\geq b_0^2b^2\left(\frac{1}{2\gamma+1}\left(1-\left(\frac{a}{b}\right)^{2(2\gamma+1)}\right)-\frac{2\gamma-1}{4\gamma^2}\frac{\left(1-\left(\frac{a}{b}\right)^{4\gamma}\right)^2}{1-\left(\frac{a}{b}\right)^{2(2\gamma-1)}}\right).
        \end{align*}
        Notice that
        \begin{align*}
            \frac{1}{2\gamma+1}\left(1-\left(\frac{a}{b}\right)^{2(2\gamma+1)}\right)-\frac{2\gamma-1}{4\gamma^2}\frac{\left(1-\left(\frac{a}{b}\right)^{4\gamma}\right)^2}{1-\left(\frac{a}{b}\right)^{2(2\gamma-1)}}\conv{\frac{b}{a}\rightarrow \infty}\frac{1}{2\gamma+1}-\frac{2\gamma-1}{4\gamma^2}=\frac{1}{4\gamma^2(2\gamma+1)}>0.
        \end{align*}
        Let us now find an estimate on the conformal class under which this function is strictly positive. We have
        \begin{align}\label{temp_conf_class}
            \frac{1}{2\gamma+1}\left(1-x^{2(2\gamma+1)}\right)-\frac{2\gamma-1}{4\gamma^2}\frac{\left(1-x^{4\gamma}\right)^2}{1-x^{2(2\gamma-1)}}\geq \frac{1}{8\gamma^2(2\gamma+1)}
        \end{align}
        if and only if
        \begin{align*}
            8\gamma^2(1-x^{2(2\gamma+1)}-x^{2(2\gamma-1)}+x^{8\gamma})-2(4\gamma^2-1)\left(1-2x^{4\gamma}+x^{8\gamma}\right)\geq 1-x^{2(2\gamma-1)}.
        \end{align*}
        Since
        \begin{align*}
            &8\gamma^2(1-x^{2(2\gamma+1)}-x^{2(2\gamma-1)}+x^{8\gamma})-2(4\gamma^2-1)\left(1-2x^{4\gamma}+x^{8\gamma}\right)\\
            &=-8\gamma^2\left(x^{2\gamma-1}-x^{2\gamma+1}\right)^2+2\left(1-x^{4\gamma}\right)^2
            =2\left(\left(1-x^{4\gamma}\right)^2-4\gamma^2x^{2(2\gamma-1)}\left(1-x^2\right)^2\right)\\
            &\geq 2\left(\left(1-x^{4\gamma}\right)^2-4\gamma^2x^{2(2\gamma-1)}\right)
        \end{align*}
        we deduce that the inequality
        \begin{align*}
            \left(1-x^{4\gamma}\right)^2-4\gamma^2x^{2(2\gamma-1)}\geq \frac{1}{2}
        \end{align*}
        implies \eqref{temp_conf_class}. For all
        \begin{align*}
            x\leq \left(\frac{1}{16\gamma^2}\right)^{\frac{1}{2(2\gamma-1)}}=\left(\frac{1}{4\gamma}\right)^{\frac{1}{2\gamma-1}},
        \end{align*}
        we have
        \begin{align*}
            \left(1-x^{4\gamma}\right)^2-4\gamma^2x^{2(2\gamma-1)}\geq \left(1-x^{4\gamma}\right)^2-\frac{1}{4}\geq \frac{1}{2}\Longleftrightarrow x\leq \left(1-\frac{\sqrt{3}}{2}\right)^{\frac{1}{4\gamma}}.
        \end{align*}
        Therefore, provided that 
        \begin{align}\label{conformal_class6}
            \log\left(\frac{b}{a}\right)\geq \max\ens{\frac{1}{2\gamma-1}\log\left(4\gamma\right),\frac{1}{4\gamma}\log\left(\frac{2}{2-\sqrt{3}}\right)},
        \end{align}
        the inequality \eqref{temp_conf_class} holds and we deduce in particular that
        \begin{align}\label{b0_est}
            \frac{1}{8\gamma^2(2\gamma+1)}b_0^2\,b^2\leq \frac{1}{\pi}\int_{\Omega}\frac{\psi^2}{|x|^4}\left(\frac{|x|}{b}\right)^{4\gamma}dx.
        \end{align}
        \textbf{Finally}, 
        \begin{align*}
            &16\pi\beta^2\left(b^2\left(\log^2(b)-\log(b)+\frac{1}{2}\right)-a^2\left(\log^2(a)-\log(a)+\frac{1}{2}\right)\right)\\
            &+32\pi\beta(\beta+b_0)\left(b^2\left(\log(b)-\frac{1}{2}\right)-a^2\left(\log(a)-\frac{1}{2}\right)\right)+16\pi(\beta+b_0)^2\left(b^2-a^2\right)\\
            &= 16\pi\beta^2\left(b^2\left(\log^2(b)+\log(b)-\frac{1}{2}\right)-a^2\left(\log^2(a)+\log(a)-\frac{1}{2}\right)\right)\\
            &+32\pi\beta b_0\left(b^2\log(b)-a^2\log(a)\right)+16\pi((\beta+b_0)^2-\beta b_0)(b^2-a^2)\leq \int_{\Omega}(\Delta\psi)^2dx.
        \end{align*}
        Notice that we control $(\beta^2+b_0)^2b^2$, which allows us to estimate
        \begin{align*}
            &16\pi\beta^2\left(b^2\left(\log^2(b)+\log(b)-\frac{1}{2}\right)-a^2\left(\log^2(a)+\log(a)-\frac{1}{2}\right)\right)\\
            &+32\pi\beta b_0\left(b^2\log(b)-a^2\log(a)\right)+16\pi((\beta+b_0)^2-\beta b_0)(b^2-a^2)\\
            &\geq 16\pi\beta^2\left(b^2\left(\log^2(b)+\log(b)\right)-a^2\left(\log^2(a)+\log(a)\right)\right)-8\pi\beta^2\left(\frac{(b^2\log(b)-a^2\log(a))^2}{b^2-a^2}\right)\\
            &+16\pi\left((\beta+b_0)^2-\beta\, b_0-2b_0^2-\frac{1}{2}\beta^2\right)(b^2-a^2)
        \end{align*}
        Now, we have
        \begin{align*}
            &2\left(b^2\left(\log^2(b)+\log(b)\right)-a^2\left(\log^2(a)+\log(a)\right)\right)-\frac{(b^2\log(b)-a^2\log(a))^2}{b^2-a^2}\\
            &=b^2\left(\left(2-\frac{1}{1-\left(\frac{a}{b}\right)^2}\right)\log^2(b)+2\log(b)-\left(\frac{a}{b}\right)^2\left(2+\frac{1}{1-\left(\frac{a}{b}\right)^2}\right)\log^2(a)-2\left(\frac{a}{b}\right)^2\log(a)\right.\\
            &\left.+2\log(a)\log(b)\frac{\left(\frac{a}{b}\right)^2}{1-\left(\frac{a}{b}\right)^2}\right),
        \end{align*}
        which is positive is the conformal class is large enough. Indeed, if $b$ is fixed, we have
        \begin{align*}
            &\left(2-\frac{1}{1-\left(\frac{a}{b}\right)^2}\right)\log^2(b)+2\log(b)-\left(\frac{a}{b}\right)^2\left(2+\frac{1}{1-\left(\frac{a}{b}\right)^2}\right)\log^2(a)-2\left(\frac{a}{b}\right)^2\log(a)\\
            &+2\log(a)\log(b)\frac{\left(\frac{a}{b}\right)^2}{1-\left(\frac{a}{b}\right)^2}\conv{a\rightarrow 0}\log^2(b)+2\log(b)>0
        \end{align*}
        provided that $0<b<e^{-2}$ (or $b>1$). Furthermore, for all $0<b<e^{-4}$, we have $\log^2(b)+2\log(b)>\dfrac{1}{2}\log^2(b)$, which implies that as $0<b<e^{-4}$ is fixed and $0<a<b$ is small enough, we have
        \begin{align*}
            &16\pi\beta^2\left(b^2\left(\log^2(b)-\log(b)+\frac{1}{2}\right)-a^2\left(\log^2(a)-\log(a)+\frac{1}{2}\right)\right)\\
            &+32\pi\beta(\beta+b_0)\left(b^2\left(\log(b)-\frac{1}{2}\right)-a^2\left(\log(a)-\frac{1}{2}\right)\right)+16\pi(\beta+b_0)^2\left(b^2-a^2\right)\\
            &\geq 8\pi\beta^2\log^2(b)+16\pi\left((\beta+b_0)^2-\beta\,b_0-2b_0^2-\frac{1}{2}\beta^2\right)(b^2-a^2)\\
            &=4\pi\beta^2\log^2(b)+16\pi\left((\beta+b_0)^2-\beta\,b_0-2b_0^2-\frac{1}{2}\beta^2\right)(b^2-a^2)\\
            &=4\pi\beta^2b^2\log^2(b)+16\pi\left(\frac{1}{2}\beta^2+\beta\,b_0-b_0^2\right)(b^2-a^2)\geq 8\pi\beta^2\log^2(b)-8\pi\,b_0^2(b^2-a^2),
        \end{align*}
        which implies by \eqref{b0_est} that for all $\dfrac{1}{2}<\delta<1$ such that 
        \begin{align}\label{conformal_class_7}
            \log\left(\frac{b}{a}\right)\geq \max\ens{\frac{1}{2\delta-1}\log\left(4\delta\right),\frac{1}{4\delta}\log\left(\frac{2}{2-\sqrt{3}}\right)},    
        \end{align}
        we have (for $0<b<e^{-4}$ fixed and $a$ small enough)
        \begin{align}\label{beta2_b2}
            4\pi b^2\beta^2\log^2(b)\leq 8\pi\,b_0^2(b^2-a^2)+\int_{\Omega}(\Delta\psi)^2dx\leq 64\delta^2(2\delta+1)\int_{\Omega}\frac{\psi^2}{|x|^4}\left(\frac{|x|}{b}\right)^{4\delta}dx+\int_{\Omega}(\Delta\psi)^2dx.
        \end{align}
        Now, let us estimate $0<a<b<e^{-2}$ such that \eqref{beta2_b2} holds. First, assuming that
        \begin{align}\label{conformal_class8}
            \log\left(\frac{b}{a}\right)\geq \frac{1}{2}\log(5),
        \end{align}
        we have
        \begin{align*}
            \frac{1}{1-\left(\frac{a}{b}\right)^2}\leq \frac{5}{4}
        \end{align*}
        and since $0<a<b<1$, we have $\log(a)\log(b)>0$, so that
        \begin{align*}
            &\left(2-\frac{1}{1-\left(\frac{a}{b}\right)^2}\right)\log^2(b)+2\log(b)-\left(\frac{a}{b}\right)^2\left(2+\frac{1}{1-\left(\frac{a}{b}\right)^2}\right)\log^2(a)-2\left(\frac{a}{b}\right)^2\log(a)\\
            &+2\log(a)\log(b)\frac{\left(\frac{a}{b}\right)^2}{1-\left(\frac{a}{b}\right)^2}\\
            &\geq \frac{3}{4}\log^2(b)+2\log(b)-\left(\frac{a}{b}\right)^2\left(\frac{13}{4}\log^2(a)-2\left(1+|\log(b)|\right)|\log(a)|\right)\geq \frac{1}{2}\log^2(b)+2\log(b)
        \end{align*}
        if and only if
        \begin{align*}
            \left(\frac{a}{b}\right)^2\left(\frac{13}{4}\log^2(a)-2\left(1+|\log(b)|\right)\log|a|\right)\leq \frac{1}{4}\log^2(b),
        \end{align*}
        which is verified if
        \begin{align*}
            13\,a^2\log^2(a)\leq b^2\log^2(b),
        \end{align*}
        which is verified if and only if
        \begin{align*}
            0<a\leq e^{W_{-1}\left(\frac{b\log(b)}{\sqrt{13}}\right)},
        \end{align*}
        where $W_{-1}$ is the negative branch of the Lambert function (the inverse of the function $x\mapsto xe^x$ on $[e^{-1},0[$). Recalling that the Lambert function admits the following asymptotic expansion as $x\rightarrow 0^{+}$
        \begin{align*}
            W_{-1}(x)=\log|x|+\log|\log|x||+O\left(\frac{\log|\log|x||}{\log|x|}\right),
        \end{align*}
        we get
        \begin{align*}
            W_ {-1}\left(\frac{b\log(b)}{\sqrt{13}}\right)&=\log\left(b|\log(b)|\right)-\frac{1}{2}\log(13)+O(\log|\log(b|\log(b)|)|)\\
            &=\log(b)+\log|\log(b)|-\frac{1}{2}\log(13)+O(\log|\log(b|\log(b)|)|)\\
            &=-\frac{1}{2}\log(13)+\log(b)\left(1+O\left(\frac{\log|\log(b)|}{|\log(b)|}\right)\right),
        \end{align*}
        which gives us
        \begin{align*}
            e^{W_{-1}\left(\frac{b\log(b)}{\sqrt{13}}\right)}=\frac{1}{\sqrt{13}}b\,\exp\left(1+O\left(\frac{\log|\log(b)|}{|\log(b)|}\right)\right)\geq \lambda_{\ast}\,b,
        \end{align*}
        where $0<\lambda_{\ast}<\infty$ is a constant depending only on the Lambert function, by taking the supremum over all values $0<b<e^{-8}$. Let us estimate this constant explicitly. Let $-e^{-1}\leq x<0$. Then, we have by definition
        \begin{align}\label{lambert}
            W_{-1}(x)e^{W_{-1}(x)}=x
        \end{align}
        Notice that $W_{-1}(x)\leq -1<0$ for all $-e^{-1}\leq x<0$. Takng the absolute value and the logarithm of \eqref{lambert}, we get
        \begin{align*}
            \log|W_{-1}(x)|+W_{-1}(x)=\log|x|,
        \end{align*}
        or
        \begin{align}\label{elementary_lambert}
            W_{-1}(x)=\log|x|-\log|W_{-1}(x)|\leq \log|x|
        \end{align}
        since $|W_{-1}(x)|\geq 1$. Now, thanks to \eqref{lambert}, the Lambert function satisfies the following differential equation
        \begin{align*}
            W_{-1}'(x)=\frac{W_{-1}(x)}{x(1+W_{-1}(x))}=\frac{1}{x}-\frac{1}{x(1+W_{-1}(x))}.
        \end{align*}
        Now, let $\alpha>0$ be a real number (to be determined later), and introduce the function
        \begin{align*}
            f(x)=W_{-1}(-x)-\log(x)+\alpha\log|\log(x)|
        \end{align*}
        on $]0,e^{-1}[$. We have
        \begin{align*}
            f'(x)=-\left(\frac{1}{(-x)}-\frac{1}{(-x)(1+W_{-1}(-x))}\right)-\frac{1}{x}+\frac{\alpha}{x\log(x)}=\frac{\alpha(1+W_{-1}(x))-\log(x)}{x\log(x)(1+W_{-1}(-x))}.
        \end{align*}
        Taking $\alpha=2$ yields
        \begin{align*}
            f'(x)=\frac{|W_{-1}(-x)|-|\log(x)|+|W_{-1}(-x)|-2}{x|\log(x)|(1+W_{-1}(-x))}\geq 0
        \end{align*}
        provided that $0<x<e^{-2}$. Indeed, \eqref{elementary_lambert} shows that $|W_{-1}(x)|=-W_{-1}(x)\geq |\log|x||$ for all $-e^{-1}\leq x<0$, which shows that $W_{-1}(-x)-\log|x|\geq 0$ for all $0<x<e^{-1}$. In particular, $|W_{-1}(-x)|-2\geq |\log(x)|-2\geq 0$ for all $0<x<e^{-2}$. Therefore, we deduce that for all $0<x<e^{-2}$, we have
        \begin{align*}
            W_{-1}(-x)\geq \log(x)-2\log|\log(x)|,
        \end{align*}
        and we finally have
        \begin{align*}
            W_{-1}(x)\geq \log|x|-2\log|\log|x||\quad \text{for all}\;\, -e^{-2}\leq x<0.
        \end{align*}
        Therefore, we get
        \begin{align*}
            W_{-1}\left(\frac{b\log(b)}{\sqrt{13}}\right)=\log(b)+\log|\log(b)|-\frac{1}{2}\log(13)-2\log\left|\log(b)+\log|\log(b)|-\frac{1}{2}\log(13)\right|,
        \end{align*}
        and
        \begin{align*}
            e^{W_{-1}\left(\frac{b\log(b)}{\sqrt{13}}\right)}\geq \frac{1}{\sqrt{13}}b\,e^{1+\frac{\log|\log(b)|}{\log(b)}-\frac{2\log\left|\log(b)+\log|\log(b)|-\frac{1}{2}\log(13)\right|}{\log(b)}}\geq \frac{1}{\sqrt{13}}b\,e^{1-\frac{\log|\log(b)|}{|\log(b)|}}\geq \frac{e^{1-\frac{3\log(2)}{4}}}{\sqrt{13}}b
        \end{align*}
        for all $0<b<e^{-8}$.
        In particular, our condition is verified provided that
        \begin{align}\label{conformal_class9}
            \log\left(\frac{b}{a}\right)\geq \frac{1}{2}\log\left(13\right)-\left(1-\frac{3\log(2)}{4}\right)=0.8023\cdots,
        \end{align}
        in which case 
        \begin{align*}
            &\left(2-\frac{1}{1-\left(\frac{a}{b}\right)^2}\right)\log^2(b)+2\log(b)-\left(\frac{a}{b}\right)^2\left(2+\frac{1}{1-\left(\frac{a}{b}\right)^2}\right)\log^2(a)-2\left(\frac{a}{b}\right)^2\log(a)\\
            &+2\log(a)\log(b)\frac{\left(\frac{a}{b}\right)^2}{1-\left(\frac{a}{b}\right)^2}\geq \frac{1}{2}\log^2(b)+2\log(b)\geq \frac{1}{4}\log^2(b)
        \end{align*}
        provided that $0<b\leq e^{-8}$. Summarising, for all $0<b<e^{-8}$, provided that \eqref{conformal_class9} holds, we deduce that 
        \begin{align}\label{beta2_b2_bis}
            2\pi\beta^2\log^2(b)\leq 8\pi\,b_0^2(b^2-a^2)+\int_{\Omega}(\Delta\psi)^2dx\leq 64\delta^2(\delta+1)\int_{\Omega}\frac{\psi^2}{|x|^4}\left(\frac{|x|}{b}\right)^{4\delta}dx+\int_{\Omega}(\Delta\psi)^2dx.
        \end{align}
        Finally, we have by \eqref{beta2_bound} and \eqref{beta2_b2_bis}
        \begin{align}\label{log_beta}
            &\frac{16\pi\beta^2}{\gamma+1}\left(b^2\left(\log^2\left(b\right)-\frac{\log(b)}{\gamma+1}+\frac{1}{2(\gamma+1)^2}\right)-a^2\left(\log^2(a)-\frac{\log(a)}{\gamma+1}+\frac{1}{2(\gamma+1)^2}\right)\left(\frac{a}{b}\right)^{2\gamma}\right)\nonumber\\
            &\leq \frac{16\pi\beta^2}{\gamma+1}b^2\left(\log^2\left(b\right)-\frac{\log(b)}{\gamma+1}+\frac{1}{2(\gamma+1)^2}\right)\leq \frac{16\pi\beta^2}{\gamma+1}b^2\left(\frac{3}{2}\log^2(b)+\frac{1}{(\gamma+1)^2}\right)\nonumber\\
            &\leq \frac{12}{\gamma+1}\left(64\delta^2(2\delta+1)\int_{\Omega}\frac{\psi^2}{|x|^4}\left(\frac{|x|}{b}\right)^{4\delta}dx+\int_{\Omega}(\Delta\psi)^2dx\right)+\frac{16}{(\gamma+1)^3}\left(\frac{1}{1-\left(\frac{a}{b}\right)^2}\right)\left(8\int_{\Omega}|\p{z}\psi|^2dx\right)\nonumber\\
            &=\frac{768\,\delta^2(2\delta+1)}{\gamma+1}\int_{\Omega}\frac{\psi^2}{|x|^4}\left(\frac{|x|}{b}\right)^{4\delta}dx+\frac{12}{\gamma+1}\int_{\Omega}(\Delta\psi)^2dx+\frac{128}{(\gamma+1)^3}\left(\frac{1}{1-\left(\frac{a}{b}\right)^2}\right)\int_{\Omega}|\p{z}^2\psi|^2dx\nonumber\\
            &\leq \frac{768\,\delta(2\delta+1)}{\gamma+1}\int_{\Omega}\frac{\psi^2}{|x|^4}\left(\frac{|x|}{b}\right)^{4\delta}dx+\frac{32}{(\gamma+1)^3}\left(3(\gamma+1)^2+\frac{2}{1-\left(\frac{a}{b}\right)^2}\right)\int_{\Omega}|\D^2\psi|^2dx,
        \end{align}
        where we used the identity
        \begin{align*}
            \int_{\Omega}|\D^2\psi|^2dx=2\int_{\Omega}|\p{z}^2\psi|^2dx+2\int_{\Omega}|\p{z\z}^2\psi|^2dx=2\int_{\Omega}|\p{z}\psi|^2dx+\frac{1}{8}\int_{\Omega}(\Delta\psi)^2dx.
        \end{align*}
        In order to obtain the final inequality, we proceed into three steps to emphasise the control we need with respect to the conformal class. First, thanks to \eqref{frequency_two} and, provided that \eqref{conformal_class2} holds, we have
        \begin{align}\label{final_biharmonic1}
            &\pi\sum_{n\in\Z^{\ast}}\frac{((n+1)^2+4)}{n+1+\gamma}|b_n|^2\left(b^{2(n+1)}-a^{2(n+1)}\left(\frac{a}{b}\right)^{2\gamma}\right)\leq \frac{16\,\delta}{\gamma}\frac{1-\left(\frac{a}{b}\right)^{2\gamma}}{1-2\left(\frac{a}{b}\right)^{4\delta}}\int_{\Omega}\frac{\psi^2}{|x|^4}\left(\frac{|x|}{b}\right)^{4\delta}dx\nonumber\\
            &+\frac{1}{2\gamma}\frac{1-\left(\frac{a}{b}\right)^{2\gamma}}{1-\left(\frac{a}{b}\right)^2}\int_{\Omega}(\Delta\psi)^2dx.
        \end{align}
        Then, \eqref{frequency_one} and \eqref{a1_estimate}, provided that \eqref{conformal_class2bis} holds, we get
        \begin{align}\label{final_biharmonic2}
            &\pi\sum_{n\in\Z^{\ast}}\frac{|n|^2}{|n-1+\gamma|}|a_n|^2\left|b^{2(n-1)}-a^{2(n-1)}\left(\frac{a}{b}\right)^{2\gamma}\right|\leq \frac{\pi}{\gamma}|a_1|^2\nonumber\\
            &+\frac{5\pi}{\gamma}\frac{1-\left(\frac{a}{b}\right)^{2\gamma}}{1-\left(\frac{a}{b}\right)^2}\sum_{n\in\Z\setminus\ens{0,1}}\frac{n^2|n-1|}{n^2+1}|a_n|^2\left|b^{2(n-1)}-a^{2(n-1)}\right|\nonumber\\
            &\leq \frac{2\delta(2\delta+1)^2}{\gamma}\frac{1-\left(\frac{a}{b}\right)^{2\gamma}}{1-2\left(\frac{a}{b}\right)^{4\delta}}\int_{\Omega}\frac{\psi^2}{|x|^4}\left(\frac{|x|}{b}\right)^{4\delta}dx\nonumber\\
            &+\frac{5}{\gamma}\frac{1-\left(\frac{a}{b}\right)^{2\gamma}}{1-\left(\frac{a}{b}\right)^2}\left(\int_{\Omega}|\D^2\psi|^2dx+\frac{1}{16}\left(\frac{a}{b}\right)^2\frac{5-5\left(\frac{a}{b}\right)^2+2\left(\frac{a}{b}\right)^4}{\left(1-\left(\frac{a}{b}\right)^2\right)^3}\int_{\Omega}(\Delta\psi)^2dx\right)
        \end{align}
        Therefore, for all $\dfrac{1}{2}<\gamma,\delta<1$, thanks to the inequality \eqref{simplified_gradient}, using \eqref{log_alpha}, \eqref{log_beta}, \eqref{final_biharmonic1}, \eqref{final_biharmonic2}, and assuming that \eqref{conforma_class5}, \eqref{conformal_class6} \eqref{conformal_class9}, \eqref{conformal_class2}, and \eqref{conformal_class2bis} hold, namely 
        \small
        \begin{align}\label{conformal_class_part1}
            \log\left(\frac{b}{a}\right)\geq \max\ens{2,\frac{1}{2\delta-1}\log\left(4\delta\right),\frac{1}{4\delta}\log\left(\frac{2}{2-\sqrt{3}}\right), \frac{1}{4(1-\delta)}\log\left(1+\frac{8\delta(1-\delta)}{(2\delta-1)^2}\right), \frac{1}{4\delta}\log(8\delta(\delta+1))}
        \end{align}
        \normalsize
        and that $0<b\leq e^{-8}$, we get
        \begin{align}\label{final_biharmonic_part1}
            &\int_{\Omega}\frac{|\D\psi|^2}{|x|^2}\left(\frac{|x|}{b}\right)^{2\gamma}dx\leq \frac{32}{1-\gamma}\left(\frac{a}{b}\right)^{2\gamma}\int_{\Omega}|\p{z}^2\psi|^2dx+\frac{768\,\delta(2\delta+1)}{\gamma+1}\int_{\Omega}\frac{\psi^2}{|x|^4}\left(\frac{|x|}{b}\right)^{4\delta}dx\\
            &+\frac{32}{(\gamma+1)^3}\left(3(\gamma+1)^2+\frac{2}{1-\left(\frac{a}{b}\right)^2}\right)\int_{\Omega}|\D^2\psi|^2dx\nonumber\\
            &+\frac{32\,\delta}{\gamma}\frac{1-\left(\frac{a}{b}\right)^{2\gamma}}{1-2\left(\frac{a}{b}\right)^{4\delta}}\int_{\Omega}\frac{\psi^2}{|x|^4}\left(\frac{|x|}{\delta}\right)^{4\delta}dx+\frac{1}{\gamma}\frac{1-\left(\frac{a}{b}\right)^{2\gamma}}{1-\left(\frac{a}{b}\right)^2}\int_{\Omega}(\Delta\psi)^2dx\nonumber\\
            &+\frac{4\delta(2\delta+1)^2}{\gamma}\frac{1-\left(\frac{a}{b}\right)^{2\gamma}}{1-2\left(\frac{a}{b}\right)^{4\delta}}\int_{\Omega}\frac{\psi^2}{|x|^4}\left(\frac{|x|}{b}\right)^{4\delta}dx\nonumber\\
            &+\frac{10}{\gamma}\frac{1-\left(\frac{a}{b}\right)^{2\gamma}}{1-\left(\frac{a}{b}\right)^2}\left(\int_{\Omega}|\D^2\psi|^2dx+\frac{1}{16}\left(\frac{a}{b}\right)^2\frac{5-5\left(\frac{a}{b}\right)^2+2\left(\frac{a}{b}\right)^4}{\left(1-\left(\frac{a}{b}\right)^2\right)^3}\int_{\Omega}(\Delta\psi)^2dx\right)\nonumber\\
            &\leq \left(\frac{768\,\delta(2\delta+1)}{\gamma+1}+\frac{4\delta(8+(2\delta+1)^2)}{\gamma}\frac{1-\left(\frac{a}{b}\right)^{2\gamma}}{1-2\left(\frac{a}{b}\right)^{4\delta}}\right)\int_{\Omega}\frac{\psi^2}{|x|^4}\left(\frac{|x|}{b}\right)^{4\delta}dx \nonumber\\
            &+\left(32\left(\frac{1}{(\gamma+1)^3}\left(3(\gamma+1)^2+\frac{2}{1-\left(\frac{a}{b}\right)^2}\right)+\frac{1}{1-\gamma}\left(\frac{a}{b}\right)^{2\gamma}\right)\right.\nonumber\\
            &\left.+\frac{1}{\gamma}\frac{1-\left(\frac{a}{b}\right)^{2\gamma}}{1-\left(\frac{a}{b}\right)^2}\left(18+5\left(\frac{a}{b}\right)^2\frac{5-5\left(\frac{a}{b}\right)^2+2\left(\frac{a}{b}\right)^4}{\left(1-\left(\frac{a}{b}\right)^2\right)^3}\right)\right)\int_{\Omega}|\D^2\psi|^2dx.
        \end{align}
        Notice that the hypothesis $0<a<b\leq e^{-8}$ is not restrictive since all inequalities enjoy the same scaling, and yields half of our inequality; namely,
        \begin{align}\label{shorten_poincare1}
            \int_{\Omega}\frac{|\D\psi|^2}{|x|^2}\left(\frac{|x|}{b}\right)^{2\gamma}dx\leq C_{\gamma,\delta}\left(\int_{\Omega}\frac{\psi^2}{|x|^4}\left(\frac{|x|}{b}\right)^{4\delta}dx+\int_{\Omega}|\D^2\psi|^2dx\right).
        \end{align}
        
        \textbf{Step 3:} Estimation of
        \begin{align*}
            \int_{\Omega}\frac{|\D\psi|^2}{|x|^2}\left(\frac{a}{|x|}\right)^{2\gamma}dx.
        \end{align*}
        Recall that 
        \begin{align*}
            &\mathrm{rad}\left(|\D\psi|^2\right)=\frac{1}{|z|^2}\bigg\{\left(\alpha+2\beta|z|^2\log|z|+(\beta+2b_0)|z|^2\right)^2+\sum_{n\in\Z^{\ast}}|n|^2|a_n|^2|z|^{2n}\\
            &+\sum_{n\in\Z^{\ast}}((n+1)^2+4)|b_n|^2|z|^{2(n+2)}
            +2\sum_{n\in\Z^{\ast}}n(n+1)\Re\left(a_n\bar{b_n}|z|^{2(n+1)}\right)\bigg\}\\
            &\leq \frac{1}{|z|^2}\bigg\{4\left(\alpha^2+4\beta^2|z|^4\log^2|z|+(\beta+2b_0)^2|z|^4\right)+2\sum_{n\in\Z^{\ast}}|n|^2|a_n|^2|z|^{2n}\\
            &+2\sum_{n\in\Z^{\ast}}((n+1)^2+2)|b_n|^2|z|^{2(n+2)}\bigg\}.
        \end{align*}
        Therefore, we get
        \begin{align}\label{part2_est_psi}
            &\int_{\Omega}\frac{|\D\psi|^2}{|x|^2}\left(\frac{a}{|x|}\right)^{2\gamma}dx\leq 4\pi\int_{a}^b\frac{1}{r^3}\bigg\{2\alpha^2+8\beta^2r^4\log^2(r)+2(\beta+2b_0)^2r^4+\sum_{n\in\Z^{\ast}}|n|^2|a_n|^2r^{2n}\nonumber\\
            &+\sum_{n\in\Z^{\ast}}((n+1)^2+2)|b_n|^2r^{2(n+2)}\bigg\}\left(\frac{a}{r}\right)^{2\gamma}dr=\frac{4\pi\alpha^2}{\gamma+1}\left(\frac{1}{a^2}-\frac{1}{b^2}\left(\frac{a}{b}\right)^{2\gamma}\right)\nonumber\\
            &+\frac{16\pi\beta^2}{1-\gamma}\left(b^{2}\left(\log^2(b)-\frac{1}{1-\gamma}\log(b)+\frac{1}{2(1-\gamma)^2}\right)\left(\frac{a}{b}\right)^{2\gamma}\right.\nonumber\\
            &\left.-a^{2}\left(\log^2(a)-\frac{1}{1-\gamma}\log(a)+\frac{1}{2(1-\gamma)^2}\right)\right)
            +\frac{4\pi(\beta+2b_0)^2}{1-\gamma}\left(b^2\left(\frac{a}{b}\right)^{2\gamma}-a^2\right)\nonumber\\
            &+2\pi\sum_{n\in\Z^{\ast}}\frac{|n|^2}{n-1-\gamma}|a_n|^2\left(b^{2(n-1)}\left(\frac{a}{b}\right)^{2\gamma}-a^{2(n-1)}\right)\nonumber\\
            &+2\pi\sum_{n\in\Z^{\ast}}\frac{(n+1)^2+2}{n+1+\gamma}|b_n|^2\left(b^{2(n+1)}\left(\frac{a}{b}\right)^{2\gamma}-a^{2(n+1)}\right).
        \end{align}
        We start by estimating thanks to \eqref{laplacian_l2}
        \small
        \begin{align}\label{step2_frequency2}
            &\sum_{n\in\Z\setminus\ens{0,-1}}\frac{(n+1)^2+2}{n+1+\gamma}|b_n|^2\left(b^{2(n+1)}\left(\frac{a}{b}\right)^{2\gamma}-a^{2(n+1)}\right)\leq \frac{3}{\gamma}\sum_{n\in\Z\setminus\ens{0,-1}}|n+1||b_n|^2\left|b^{2(n+1)}-a^{2(n+1)}\left(\frac{a}{b}\right)^{2\gamma}\right|\nonumber\\
            &=\frac{3}{\gamma}\left(\sum_{n\geq 0}|n+1||b_n|^2b^{2(n+1)}\left(\left(\frac{a}{b}\right)^{2\gamma}-\left(\frac{a}{b}\right)^{2(n+1)}\right)
            +\sum_{n\geq 2}|n-1||b_{-n}|^2\frac{1}{a^{2(n-1)}}\left(1-\left(\frac{a}{b}\right)^{2(n+1+\gamma}\right)\right)\nonumber\\
            &\leq \frac{3}{\gamma}\left(\sum_{n\geq 0}|n+1||b_n|^2b^{2(n+1)}\left(1-\left(\frac{a}{b}\right)^{2(n+1)}\right)
            +\frac{1-\left(\frac{a}{b}\right)^{2\gamma}}{1-\left(\frac{a}{b}\right)^2}\sum_{n\geq 2}|n-1||b_{-n}|^2\frac{1}{a^{2(n-1)}}\left(1-\left(\frac{a}{b}\right)^{2(n+1)}\right)\right)\nonumber\\
            &\leq \frac{3}{8\pi\,\gamma}\frac{1-\left(\frac{a}{b}\right)^{2\gamma}}{1-\left(\frac{a}{b}\right)^2}\int_{\Omega}(\Delta\psi)^2dx.
        \end{align}
        \normalsize
        Now, recalling \eqref{frequency_one}
        \small
        \begin{align*}
            &2\pi\sum_{n\in \Z\setminus\ens{0,1}}\frac{n^2|n-1|}{n^2+1}|a_n|^2|b^{2(n-1)}-a^{2(n-1)}|\leq 2\int_{\Omega}|\D^2\psi|^2dx+\frac{1}{8}\left(\frac{a}{b}\right)^2\frac{5-5\left(\frac{a}{b}\right)^2+2\left(\frac{a}{b}\right)^4}{\left(1-\left(\frac{a}{b}\right)^2\right)^3}\int_{\Omega}(\Delta\psi)^2dx,
        \end{align*}
        \normalsize
        we get
        \begin{align}\label{step2_frequency1}
            &\sum_{n\in\Z\setminus\ens{0,1}}\frac{|n|^2}{n-1-\gamma}|a_n|^2\left(b^{2(n-1)}\left(\frac{a}{b}\right)^{2\gamma}-a^{2(n-1)}\right)\nonumber\\
            &\leq \frac{5}{1-\gamma}\sum_{n\in\Z\setminus\ens{0,1}}\frac{n^2|n-1|}{n^2+1}|a_n|^2\left|b^{2(n-1)}\left(\frac{a}{b}\right)^{2\gamma}-a^{2(n-1)}\right|\nonumber\\
            &\leq \frac{5}{\pi(1-\gamma)}\frac{1-\left(\frac{a}{b}\right)^{2\gamma}}{1-\left(\frac{a}{b}\right)^2}\left(2\int_{\Omega}|\D^2\psi|^2dx+\frac{1}{8}\left(\frac{a}{b}\right)^2\frac{5-5\left(\frac{a}{b}\right)^2+2\left(\frac{a}{b}\right)^4}{\left(1-\left(\frac{a}{b}\right)^2\right)^3}\int_{\Omega}(\Delta\psi)^2dx\right).
        \end{align}
        Then, recalling that
        \begin{align*}
            \mathrm{rad}(\psi^2)&=\alpha^2\log^2|z|+2\alpha\,a_0\log|z|+2\alpha\beta|z|^2\log^2|z|+2(\alpha\,b_0+\beta\,a_0)|z|^2\log|z|+\beta^2|z|^4\log^2|z|\\
            &+2\beta\,b_0|z|^4\log|z|
            +\frac{1}{2}\sum_{n\in\Z}|a_n|^2|z|^{2n}+\frac{1}{2}\sum_{n\in\Z}|b_n|^2|z|^{2(n+2)}+\sum_{n\in\Z}\Re\left(a_n\bar{b_n}\right)|z|^{2(n+1)},
        \end{align*}
        we get
        \begin{align*}
            &\int_{\Omega}\frac{\psi^2}{|x|^4}\left(\frac{a}{|x|}\right)^{4\gamma}dx=2\pi\int_{a}^b\frac{1}{r^3}\bigg\{\alpha^2\log^2(r)+2\alpha\,a_0\log(r)+2\alpha\beta r^2\log^2(r)+2\left(\alpha\,b_0+\beta\,a_0\right)r^2\log(r)\\
            &+\beta^2r^4\log^2(r)+2\beta\,b_0 r^4\log(r)+\frac{1}{2}\sum_{n\in\Z}|a_n|^2|z|^{2n}+\frac{1}{2}\sum_{n\in\Z}|b_n|^2r^{2(n+2)}+\sum_{n\in\Z}\Re\left(a_n\bar{b_n}\right)r^{2(n+1)}\bigg\}\left(\frac{a}{r}\right)^{4\gamma}dr\\
            &=\frac{\pi\alpha^2}{2\gamma+1}\left(\frac{1}{a^2}\left(\log^2(a)-\frac{\log(a)}{2\gamma+1}+\frac{1}{2(2\gamma+1)^2}\right)-\frac{1}{b^2}\left(\log^2(b)-\frac{\log(b)}{2\gamma+1}+\frac{1}{2(2\gamma+1)^2}\right)\left(\frac{a}{b}\right)^{4\gamma}\right)\\
            &+\frac{\pi\alpha\,a_0}{2\gamma+1}\left(\frac{1}{a^2}\left(\log(a)-\frac{1}{2(2\gamma+1)}\right)-\frac{1}{b^2}\left(\log(b)-\frac{1}{2(2\gamma+1)}\right)\left(\frac{a}{b}\right)^{4\gamma}\right)\\
            &+\frac{\pi\alpha\beta}{\gamma}\left(\left(\log^2(a)-\frac{\log(a)}{2\gamma^2}+\frac{1}{4\gamma^2}\right)-\left(\log^2(b)-\frac{\log(b)}{2\gamma^2}+\frac{1}{4\gamma^2}\right)\left(\frac{a}{b}\right)^{4\gamma}\right)\\
            &+\frac{\pi\left(\alpha\,b_0+\beta\,a_0\right)}{\gamma}\left(\log(a)-\frac{1}{4\gamma}-\left(\log(b)-\frac{1}{4\gamma}\right)\left(\frac{a}{b}\right)^{4\gamma}\right)\\
            &+\frac{\pi\beta^2}{2\gamma-1}\left(a^2\left(\log^2(a)-\frac{\log(a)}{2\gamma-1}+\frac{1}{2(2\gamma-1)^2}\right)-b^2\left(\log^2(b)-\frac{\log(b)}{2\gamma-1}+\frac{1}{2(2\gamma-1)^2}\right)\left(\frac{a}{b}\right)^{4\gamma}\right)\\
            &+\frac{2\pi\beta\,b_0}{2\gamma-1}\left(a^2\left(\log(a)-\frac{1}{2\gamma-1}\right)-b^2\left(\log(b)-\frac{1}{2\gamma-1}\right)\left(\frac{a}{b}\right)^{4\gamma}\right)\\
            &+\pi\sum_{n\in\Z}\frac{1}{n-1-2\gamma}|a_n|^2\left(b^{2(n-1)}\left(\frac{a}{b}\right)^{4\gamma}-a^{2(n-1)}\right)
            +\pi\sum_{n\in\Z}\frac{1}{n+1-2\gamma}|b_n|^2\left(b^{2(n+1)}\left(\frac{a}{b}\right)^{4\gamma}-a^{2(n+1)}\right)\\
            &+2\pi\sum_{n\in\Z}\frac{1}{n-2\gamma}\Re\left(a_n\bar{b_n}\right)\left(b^{2n}\left(\frac{a}{b}\right)^{4\gamma}-a^{2n}\right).
        \end{align*}
        As previously, we need to bound $b_{-1}$ and $a_1$, since those frequencies are not detected by the second derivative. For all $n\in\Z$, we have
        \begin{align}\label{pointwise_frequency}
            &\frac{1}{n-1-2\gamma}|a_n|^2\left(b^{2(n-1)}\left(\frac{a}{b}\right)^{4\gamma}-a^{2(n-1)}\right)+\frac{1}{n+1-2\gamma}|b_n|^2\left(b^{2(n+1)}\left(\frac{a}{b}\right)^{4\gamma}-a^{2(n+1)}\right)\nonumber\\
            &+\frac{2}{n-2\gamma}\Re\left(a_{n}\bar{b_n}\right)\left(b^{2n}\left(\frac{a}{b}\right)^{4\gamma}-a^{2n}\right)\leq \frac{1}{\pi}\int_{\Omega}\frac{\psi^2}{|x|^4}\left(\frac{a}{|x|}\right)^{4\gamma}dx.
        \end{align}
        Applying \eqref{pointwise_frequency} to $n=-1$, we deduce that 
        \begin{align*}
            &\frac{1}{2(\gamma+1)}|a_{-1}|^2\frac{1}{a^4}\left(1-\left(\frac{a}{b}\right)^{4(\gamma+1)}\right)+\frac{1}{2\gamma}|b_{-1}|^2\left(1-\left(\frac{a}{b}\right)^{4\gamma}\right)\\
            &+\frac{2}{2\gamma+1}\Re\left(a_{-1}\bar{b_{-1}}\right)\frac{1}{a^2}\left(1-\left(\frac{a}{b}\right)^{2(2\gamma+1)}\right)\leq \frac{1}{\pi}\int_{\Omega}\frac{\psi^2}{|x|^4}\left(\frac{a}{|x|}\right)^{4\gamma}dx.
        \end{align*}
        Now, we estimate
        \begin{align*}
            &\frac{1}{2(\gamma+1)}|a_{-1}|^2\frac{1}{a^4}\left(1-\left(\frac{a}{b}\right)^{4(\gamma+1)}\right)+\frac{1}{2\gamma}|b_{-1}|^2\left(1-\left(\frac{a}{b}\right)^{4\gamma}\right)\\
            &+\frac{2}{2\gamma+1}\Re\left(a_{-1}\bar{b_{-1}}\right)\frac{1}{a^2}\left(1-\left(\frac{a}{b}\right)^{2(2\gamma+1)}\right)\\
            &\geq |b_{-1}|^2\left(\frac{1}{2\gamma}\left(1-\left(\frac{a}{b}\right)^{4\gamma}\right)-\frac{2(\gamma+1)}{(2\gamma+1)^2}\frac{\left(1-\left(\frac{a}{b}\right)^{2(2\gamma+1)}\right)^2}{1-\left(\frac{a}{b}\right)^{4(\gamma+1)}}\right)
        \end{align*}
        This is exactly the quantity appearing in \eqref{step1_a1}, and we deduce by \eqref{a1_estimate} that for all $\dfrac{1}{2}<\delta<1$, provided that \eqref{conformal_class3} holds, we deduce that
        \begin{align*}
            &\frac{1}{2(\gamma+1)}|a_{-1}|^2\frac{1}{a^4}\left(1-\left(\frac{a}{b}\right)^{4(\gamma+1)}\right)+\frac{1}{2\gamma}|b_{-1}|^2\left(1-\left(\frac{a}{b}\right)^{4\gamma}\right)\\
            &+\frac{2}{2\gamma+1}\Re\left(a_{-1}\bar{b_{-1}}\right)\frac{1}{a^2}\left(1-\left(\frac{a}{b}\right)^{2(2\gamma+1)}\right)\geq \frac{|b_{-1}|^2}{2\gamma(2\gamma+1)^2}\left(1-2\left(\frac{a}{b}\right)^{4\gamma}\right),
        \end{align*}
        which implies that 
        \begin{align*}
            \pi|b_{-1}|^2\leq \frac{2\gamma(2\gamma+1)^2}{1-2\left(\frac{a}{b}\right)^{4\gamma}}\int_{\Omega}\frac{\psi^2}{|x|^4}\left(\frac{a}{|x|}\right)^{4\gamma}dx
        \end{align*}
        and more generally, for all $\dfrac{1}{2}<\delta<1$, if \eqref{conformal_class2bis} holds, we have
        \begin{align}\label{step2_b_minus_one}
            \frac{\pi}{\gamma}|b_{-1}|^2\left(1-\left(\frac{a}{b}\right)^{2\gamma}\right)\leq \frac{2\delta(2\delta+1)^2}{\gamma}\frac{1-\left(\frac{a}{b}\right)^{2\gamma}}{1-2\left(\frac{a}{b}\right)^{4\delta}}\int_{\Omega}\frac{\psi^2}{|x|^4}\left(\frac{a}{|x|}\right)^{4\delta}dx.
        \end{align}
        Likewise, taking $n=1$ yields
        \begin{align*}
            &\frac{1}{2\gamma}|a_1|^2\left(1-\left(\frac{a}{b}\right)^{4\gamma}\right)+\frac{1}{2(1-\gamma)}|b_1|^2b^4\left(\left(\frac{a}{b}\right)^{4\gamma}-\left(\frac{a}{b}\right)^4\right)-\frac{2}{2\gamma-1}\Re\left(a_1\bar{b_1}\right)b^2\left(\left(\frac{a}{b}\right)^{4\gamma}-\left(\frac{a}{b}\right)^{2}\right)\\
            &\leq \frac{1}{\pi}\int_{\Omega}\frac{\psi^2}{|x|^4}\left(\frac{a}{|x|}\right)^{4\gamma}dx.
        \end{align*}
        Once more, we recognise (up to changing $b_{-1}$ by $a_1$ and $a_{-1}$ by $b_1$ and the weight in the integral) \eqref{ineq_b_minus_one}. Therefore, for all $\dfrac{1}{2}<\delta<1$, if \eqref{conformal_class2} holds, we get by \eqref{b_minus_one_estimate}
        \begin{align}\label{step2_a_one}
            \frac{4\pi}{\gamma}|a_1|^2\left(1-\left(\frac{a}{b}\right)^{2\gamma}\right)\leq \frac{16\,\delta}{\gamma}\frac{1-\left(\frac{a}{b}\right)^{2\gamma}}{1-2\left(\frac{a}{b}\right)^{4\delta}}\int_{\Omega}\frac{\psi^2}{|x|^4}\left(\frac{a}{|x|}\right)^{4\delta}dx.
        \end{align}
        We now move to the estimate of logarithm components. For this, we will need to bound $b_0$. We have by \eqref{pointwise_frequency}
        \begin{align*}
            &\frac{1}{2\gamma+1}a_0^2\frac{1}{b^2}\left(1-\left(\frac{a}{b}\right)^{2}\right)+\frac{1}{2\gamma-1}b_0^2a^2\left(1-\left(\frac{a}{b}\right)^{4(\gamma+1)}\right)\\
            &\frac{1}{\gamma}a_0b_0\left(1-\left(\frac{a}{b}\right)^{4\gamma}\right)\leq \frac{1}{\pi}\int_{\Omega}\frac{\psi^2}{|x|^4}\left(\frac{a}{|x|}\right)^{4\gamma}dx.
        \end{align*}
        We have
        \begin{align*}
            &\frac{1}{2\gamma+1}a_0^2\frac{1}{a^2}\left(1-\left(\frac{a}{b}\right)^{2(2\gamma+1)}\right)+\frac{1}{2\gamma-1}b_0^2a^2\left(1-\left(\frac{a}{b}\right)^{2(2\gamma-1)}\right)\\
            &\frac{1}{\gamma}a_0b_0\left(1-\left(\frac{a}{b}\right)^{4\gamma}\right)
            \geq b_0^2a^2\left(\frac{1}{2\gamma-1}\left(1-\left(\frac{a}{b}\right)^{2(2\gamma-1)}\right)-\frac{2\gamma+1}{4\gamma^2}\frac{\left(1-\left(\frac{a}{b}\right)^{4\gamma}\right)^2}{1-\left(\frac{a}{b}\right)^{2(2\gamma+1)}}\right).
        \end{align*}
        As previously, observer that
        \begin{align*}
            \frac{1}{2\gamma-1}\left(1-\left(\frac{a}{b}\right)^{2(2\gamma-1)}\right)-\frac{2\gamma+1}{4\gamma^2}\frac{\left(1-\left(\frac{a}{b}\right)^{4\gamma}\right)^2}{1-\left(\frac{a}{b}\right)^{2(2\gamma+1)}}\conv{\frac{b}{a}\rightarrow\infty}\frac{1}{2\gamma-1}-\frac{2\gamma+1}{4\gamma^2}=\frac{1}{4\gamma^2(2\gamma-1)}>0.
        \end{align*}
        Now, we have
        \begin{align*}
            \frac{1}{2\gamma-1}\left(1-x^{2(2\gamma-1)}\right)-\frac{2\gamma+1}{4\gamma^2}\frac{\left(1-x^{4\gamma}\right)^2}{1-x^{2(2\gamma+1)}}\geq \frac{1}{8\gamma^2(2\gamma-1)}>0
        \end{align*}
        if and only if
        \begin{align*}
            8\gamma^2\left(1-x^{2(2\gamma-1)}-x^{2(2\gamma+1)}+x^{8\gamma}\right)-2(4\gamma^2-1)\left(1-2x^{4\gamma}+x^{8\gamma}\right)\geq 1-x^{2(2\gamma+1)}.
        \end{align*}
        Now, notice that the condition \eqref{conformal_class6} implies that 
        \begin{align*}
            8\gamma^2\left(1-x^{2(2\gamma-1)}-x^{2(2\gamma+1)}+x^{8\gamma}\right)-2(4\gamma^2-1)\left(1-2x^{4\gamma}+x^{8\gamma}\right)\geq 1.
        \end{align*}
        In particular, if \eqref{conformal_class6} holds, we deduce that 
        \begin{align}\label{step2_b0}
            \frac{1}{8\gamma^2(2\gamma-1)}b_0^2a^2\leq \frac{1}{\pi}\int_{\Omega}\frac{\psi^2}{|x|^4}\left(\frac{a}{|x|}\right)^{4\gamma}dx.
        \end{align}
        However, this estimate does not suffice since we need to bound $b_0^2b^2$. Therefore, we will use the estimate from \textbf{Step 1}.
        Now, recall that for all $0<a<e^2\,b<\infty$, we have by \eqref{alpha_estimate_step1}
        \begin{align*}
            \frac{\alpha^2}{a^2}\leq \frac{8}{\pi}\int_{\Omega}|\p{z}^2\psi|^2dx.
        \end{align*}
        Therefore, we get
        \begin{align}\label{step2_log_alpha}
            \frac{4\pi\alpha^2}{\gamma+1}\left(\frac{1}{a^2}-\frac{1}{b^2}\left(\frac{a}{b}\right)^{2\gamma}\right)\leq \frac{32}{\gamma+1}\left(1-\left(\frac{a}{b}\right)^{2(\gamma+1)}\right)\int_{\Omega}|\p{z}^2\psi|^2dx.
        \end{align}
        Using the previous estimates \eqref{beta2_b2_bis}, \eqref{beta2_bound}, and \eqref{b0_est}, if $0<b<e^{-8}$ and 
        \begin{align*}
            \log\left(\frac{b}{a}\right)\geq \frac{1}{2}\log(13)+\frac{3}{4}\log(2)-1,
        \end{align*}
        and \eqref{conformal_class_7} holds and $0<a<e^2\,b<e^{-8}$, then 
        \begin{align}\label{previous_estimates}
        \left\{\begin{alignedat}{1}
            &2\pi\beta^2\log^2(b)\leq 64\delta^2(2\delta+1)\int_{\Omega}\frac{\psi^2}{|x|^4}\left(\frac{|x|}{b}\right)^{4\delta}dx+\int_{\Omega}(\Delta\psi)^2dx\\
            &\pi\beta^2b^2\leq \frac{8}{1-\left(\frac{a}{b}\right)^2}\int_{\Omega}|\p{z}^2\psi|^2\\
            &\pi b_0^2b^2\leq 8\delta^2(2\delta+1)\int_{\Omega}\frac{\psi^2}{|x|^4}\left(\frac{|x|}{b}\right)^{4\gamma}dx
            \end{alignedat}\right.
        \end{align}
        Therefore, we estimate
        \begin{align}\label{step2_log_beta}
            &\frac{16\pi\beta^2}{1-\gamma}\left(b^2\left(\log^2(b)-\frac{1}{1-\gamma}\log(b)+\frac{1}{2(1-\gamma)^2}\right)\left(\frac{a}{b}\right)^{2\gamma}-a^2\left(\log^2(a)-\frac{1}{1-\gamma}\log(a)+\frac{1}{2(1-\gamma)^2}\right)\right)\nonumber\\
            &\leq \frac{24\pi\beta^2}{1-\gamma}b^2\log^2(b)\left(\frac{a}{b}\right)^{2\gamma}+\frac{16\pi\beta^2}{(1-\gamma)}b^2\left(\frac{a}{b}\right)^{2\gamma}\leq \frac{768\,\delta^2(2\delta+1)}{1-\gamma}\left(\frac{a}{b}\right)^{2\gamma}\int_{\Omega}\frac{\psi^2}{|x|^4}\left(\frac{|x|}{b}\right)^{4\delta}dx\nonumber\\
            &+\frac{12}{1-\gamma}\left(\frac{a}{b}\right)^{2\gamma}\int_{\Omega}(\Delta\psi)^2dx+\frac{128}{1-\gamma}\frac{\left(\frac{a}{b}\right)^{2\gamma}}{1-\left(\frac{a}{b}\right)^2}\int_{\Omega}|\p{z}^2\psi|^2dx\nonumber\\
            &\leq \frac{768\,\delta^2(2\delta+1)}{1-\gamma}\left(\frac{a}{b}\right)^{2\gamma}\int_{\Omega}\frac{\psi^2}{|x|^4}\left(\frac{|x|}{b}\right)^{4\delta}dx+\frac{32}{1-\gamma}\left(3+\frac{2}{1-\left(\frac{a}{b}\right)^{2}}\right)\left(\frac{a}{b}\right)^{2\gamma}\int_{\Omega}|\D^2\psi|^2dx.
        \end{align}
        Finally, we get by \eqref{previous_estimates}
        \begin{align}\label{step2_last_log}
            &\frac{4\pi(\beta+2b_0)^2}{1-\gamma}\left(b^2\left(\frac{a}{b}\right)^{2\gamma}-a^2\right)\leq \frac{8\pi\beta^2}{1-\gamma}b^2\left(\frac{a}{b}\right)^{2\gamma}\left(1-\left(\frac{a}{b}\right)^{2(1-\gamma)}\right)+\frac{32b_0^2}{1-\gamma}b^2\left(\frac{a}{b}\right)^{2\gamma}\left(1-\left(\frac{a}{b}\right)^{2(1-\gamma)}\right)\nonumber\\
            &\leq \frac{32}{1-\gamma}\left(\frac{a}{b}\right)^{2\gamma}\frac{1-\left(\frac{a}{b}\right)^{2(1-\gamma)}}{1-\left(\frac{a}{b}\right)^{2}}\int_{\Omega}|\p{z}^2\psi|^2dx+\frac{256\delta^2(2\delta+1)}{1-\gamma}\frac{1-\left(\frac{a}{b}\right)^{2(1-\gamma)}}{1-\left(\frac{a}{b}\right)^{2}}\int_{\Omega}\frac{\psi^2}{|x|^4}\left(\frac{|x|}{b}\right)^{4\gamma}dx.
        \end{align}
        Therefore, for all $\dfrac{1}{2}<\gamma,\delta<1$ if \eqref{conformal_class_part1} holds, we deduce by \eqref{part2_est_psi} and the estimates \eqref{step2_frequency2}, \eqref{step2_frequency1}, \eqref{step2_b_minus_one}, \eqref{step2_a_one},\eqref{step2_log_alpha}, \eqref{step2_log_beta}, and \eqref{step2_last_log} that
        \begin{align}\label{final_biharmonic_part2}
            &\int_{\Omega}\frac{|\D\psi|^2}{|x|^{2}}\left(\frac{a}{|x|}\right)^{2\gamma}dx\leq \frac{16}{\gamma+1}\left(1-\left(\frac{a}{b}\right)^{2(\gamma+1)}\right)\int_{\Omega}|\D^2\psi|^2dx\nonumber\\
            &+\frac{768\,\delta^2(2\delta+1)}{1-\gamma}\left(\frac{a}{b}\right)^{2\gamma}\int_{\Omega}\frac{\psi^2}{|x|^4}\left(\frac{|x|}{b}\right)^{4\delta}dx+\frac{32}{1-\gamma}\left(3+\frac{2}{1-\left(\frac{a}{b}\right)^{2}}\right)\left(\frac{a}{b}\right)^{2\gamma}\int_{\Omega}|\D^2\psi|^2dx\nonumber\\
            &+\frac{256\,\delta^2(2\delta+1)}{1-\gamma}\frac{1-\left(\frac{a}{b}\right)^{2(1-\gamma)}}{1-\left(\frac{a}{b}\right)^{2}}\int_{\Omega}\frac{\psi^2}{|x|^4}\left(\frac{|x|}{b}\right)^{4\gamma}dx+\frac{16}{1-\gamma}\left(\frac{a}{b}\right)^{2\gamma}\frac{1-\left(\frac{a}{b}\right)^{2(1-\gamma)}}{1-\left(\frac{a}{b}\right)^{2}}\int_{\Omega}|\D^2\psi|^2dx\nonumber\\
            &+\frac{6}{\gamma}\frac{1-\left(\frac{a}{b}\right)^{2\gamma}}{1-\left(\frac{a}{b}\right)^2}\int_{\Omega}|\D^2\psi|^2dx+\frac{8\delta(2\delta+1)^2}{\gamma}\left(\frac{a}{b}\right)^{2\gamma}\frac{1-\left(\frac{a}{b}\right)^{2(1-\gamma)}}{1-2\left(\frac{a}{b}\right)^{4\delta}}\int_{\Omega}\frac{\psi^2}{|x|^4}\left(\frac{a}{|x|}\right)^{4\delta}dx\nonumber\\
            &+\frac{10}{1-\gamma}\frac{1-\left(\frac{a}{b}\right)^{2\gamma}}{1-\left(\frac{a}{b}\right)^2}\left(2+\left(\frac{a}{b}\right)^2\frac{5-5\left(\frac{a}{b}\right)^2+2\left(\frac{a}{b}\right)^4}{\left(1-\left(\frac{a}{b}\right)^2\right)^3}\right)\int_{\Omega}|\D^2\psi|^2dx\nonumber\nonumber\\
            &+\frac{8\,\delta}{\gamma}\left(\frac{a}{b}\right)^{2\gamma}\frac{1-\left(\frac{a}{b}\right)^{2(1-\gamma)}}{1-2\left(\frac{a}{b}\right)^{4\delta}}\int_{\Omega}\frac{\psi^2}{|x|^4}\left(\frac{a}{|x|}\right)^{4\delta}dx\nonumber\\
            &=\frac{256\,\delta^2}{1-\gamma}\left(3\left(\frac{a}{b}\right)^{2\gamma}+\frac{1-\left(\frac{a}{b}\right)^{2(1-\gamma)}}{1-\left(\frac{a}{b}\right)^2}\right)\int_{\Omega}\frac{\psi^2}{|x|^4}\left(\frac{|x|}{b}\right)^{4\delta}dx\nonumber\\
            &+\frac{8\,\delta((2\delta+1)^2+1)}{\gamma}\left(\frac{a}{b}\right)^{2\gamma}\frac{1-\left(\frac{a}{b}\right)^{2(1-\gamma)}}{1-2\left(\frac{a}{b}\right)^{4\delta}}\int_{\Omega}\frac{\psi^2}{|x|^4}\left(\frac{a}{|x|}\right)^{4\delta}dx\nonumber\\
            &+\left(\frac{16}{\gamma+1}\left(1-\left(\frac{a}{b}\right)^{2(\gamma+1)}\right)+\frac{6}{\gamma}\frac{1-\left(\frac{a}{b}\right)^{2\gamma}}{1-\left(\frac{a}{b}\right)^2}+\frac{16}{1-\gamma}\left(\frac{a}{b}\right)^{2\gamma}\left(6+\frac{5-\left(\frac{a}{b}\right)^{2(1-\gamma)}}{1-\left(\frac{a}{b}\right)^2}\right)\right.\nonumber\\
            &\left.+\frac{10}{1-\gamma}\frac{1-\left(\frac{a}{b}\right)^{2\gamma}}{1-\left(\frac{a}{b}\right)^2}\left(2+\left(\frac{a}{b}\right)^2\frac{5-5\left(\frac{a}{b}\right)^2+2\left(\frac{a}{b}\right)^4}{\left(1-\left(\frac{a}{b}\right)^2\right)^3}\right)\right)\int_{\Omega}|\D^2\psi|^2dx,
        \end{align}
        where we used the identity
        \begin{align*}
            \int_{\Omega}|\D^2\psi|^2dx=2\int_{\Omega}|\p{z}^2\psi|^2dx+\frac{1}{8}\int_{\Omega}(\Delta\psi)^2dx.
        \end{align*}
        Finally, using \eqref{final_biharmonic_part1} and \eqref{final_biharmonic_part2}, we deduce that for all $\dfrac{1}{2}<\gamma,\delta<1$, if
        \small
        \begin{align}\label{conformal_class_part2}
            \log\left(\frac{b}{a}\right)\geq \max\ens{2,\frac{1}{2\delta-1}\log\left(4\delta\right),\frac{1}{4\delta}\log\left(\frac{2}{2-\sqrt{3}}\right), \frac{1}{4(1-\delta)}\log\left(1+\frac{8\delta(1-\delta)}{(2\delta-1)^2}\right), \frac{1}{4\delta}\log(8\delta(\delta+1))},
        \end{align}
        \normalsize
        then we have
        \begin{align*}
            &\int_{\Omega}\frac{|\D\psi|^2}{|x|^2}\left(\left(\frac{|x|}{b}\right)^{2\gamma}+\left(\frac{a}{|x|}\right)^{2\gamma}\right)dx\leq C_{\gamma,\delta}\left(\int_{\Omega}\frac{\psi^2}{|x|^4}\left(\left(\frac{|x|}{b}\right)^{4\delta}+\left(\frac{a}{|x|}\right)^{4\delta}\right)dx+\int_{\Omega}|\D^2\psi|^2dx\right)
        \end{align*}
        Notice that in fact, the proof of this step works for all $0<\gamma<1$ (provided that integrals involving the $L^2$ norm of $\psi$ are taken with respect to the $\delta$ weight). Furthermore, the hypothesis $0<a<b\leq e^{-8}$ is not restrictive by an immediate scaling argument. Therefore, the proof is complete.        
        \end{proof}

             \begin{theorem}\label{interpolation_weighted_poincare}
            Let $0<a<b<\infty$ and $\Omega=B_b\setminus\bar{B}_a(0)$. For all $\dfrac{1}{2}<\beta<1$ and $\sqrt{2}-1<\gamma<1$ there exists a universal constant $C_{\beta,\gamma}<\infty$ such that
            for all $u\in W^{2,2}(\Omega)$, provided that
            \small
            \begin{align}\label{conformal_class_cor}
            \log\left(\frac{b}{a}\right)\geq \max\ens{2,\frac{1}{2\beta-1}\log\left(4\beta\right),\frac{1}{4\beta}\log\left(\frac{2}{2-\sqrt{3}}\right), \frac{1}{4(1-\beta)}\log\left(1+\frac{8\beta(1-\beta)}{(2\beta-1)^2}\right), \frac{1}{4\beta}\log(8\beta(\beta+1))},
            \end{align}
            \normalsize
            we have
            \begin{align}\label{interpolation_weight}
                \int_{\Omega}\frac{|\D u|^2}{|x|^2}\left(\left(\frac{|x|}{b}\right)^{2\gamma}+\left(\frac{a}{|x|}\right)^{2\gamma}\right)dx\leq C_{\beta,\gamma}\left(\int_{\Omega}\frac{u^2}{|x|^4}\left(\left(\frac{|x|}{b}\right)^{4\beta}+\left(\frac{a}{|x|}\right)^{4\beta}\right)dx+\int_{\Omega}|\D^2u|^2dx\right).
            \end{align}
        \end{theorem}
        \begin{proof}
            On $\Omega$, make the decomposition $u=\varphi+\psi$, where
            \begin{align*}
            \left\{\begin{alignedat}{2}
                \Delta^2\varphi&=\Delta^2u\qquad&&\text{in}\;\, \Omega\\
                \varphi&=0\qquad&&\text{in}\;\, \partial \Omega\\
                \partial_{\nu}\varphi&=0\qquad &&\text{in}\;\, \partial \Omega,
                \end{alignedat}\right.
            \end{align*}
            and
            \begin{align*}
                \left\{\begin{alignedat}{2}
                    \Delta^2\psi&=0\qquad&& \text{in}\;\, \Omega\\
                    \psi&=u\qquad&& \text{on}\;\, \partial \Omega\\
                    \partial_{\nu}\psi&=\partial_{\nu}u\qquad&&\text{on}\;\, \partial \Omega.
                \end{alignedat}\right.
            \end{align*}

            \textbf{Step 2:} Estimate of
            \begin{align*}
                \int_{\Omega}\frac{|\D\varphi|^2}{|x|^2}\left(\left(\frac{a}{b}\right)^{2\gamma}+\left(\frac{a}{|x|}\right)^{2\gamma}\right)dx.
            \end{align*}
            
            Thanks to the estimate of Theorem \ref{lemmeIV.1_complement}, we deduce that for all $\gamma>\sqrt{2}-1$,
            \begin{align}\label{interpolation1}
                \int_{\Omega}\frac{|\D \varphi|^2}{|x|^2}\left(\left(\frac{|x|}{b}\right)^{2\gamma}+\left(\frac{a}{|x|}\right)^{2\gamma}\right)dx\leq C_{\gamma}\int_{\Omega}(\Delta \varphi)^2dx.
            \end{align}
            Thanks to Cauchy-Schwarz inequality, we have
            \begin{align}
                \int_{\Omega}(\Delta\varphi)^2dx=\int_{\Omega}\varphi\,\Delta^2\varphi\,dx=\int_{\Omega}\varphi\,\Delta^2 u\,dx=\int_{\Omega}\Delta\varphi\,\Delta u\,dx\leq \left(\int_{\Omega}(\Delta\varphi)^2dx\right)^{\frac{1}{2}}\left(\int_{\Omega}(\Delta u)^2dx\right)^{\frac{1}{2}}.
            \end{align}
            Therefore, we have
            \begin{align}\label{interpolation2}
                \int_{\Omega}(\Delta\varphi)^2dx\leq \int_{\Omega}(\Delta u)^2dx,
            \end{align}
            and \eqref{interpolation1} and \eqref{interpolation2} imply that
            \begin{align*}
                \int_{\Omega}\frac{|\D \varphi|^2}{|x|^2}\left(\left(\frac{|x|}{b}\right)^{2\gamma}+\left(\frac{a}{|x|}\right)^{2\gamma}\right)dx\leq C_{\gamma}'\int_{\Omega}(\Delta u)^2dx.
            \end{align*}
            Furthermore, the inequality from Theorem \ref{lemme} shows that
            \begin{align*}
                \int_{\Omega}\frac{\varphi^2}{|x|^4}\left(\left(\frac{|x|}{b}\right)^{4\beta}+\left(\frac{a}{|x|}\right)^{4\beta}\right)dx\leq C_{\beta}\int_{\Omega}(\Delta \varphi)^2dx\leq C_{\beta}\int_{\Omega}(\Delta u)^2dx.
            \end{align*}
            On the other hand, using Theorem \ref{biharmonic_interpolation} for the biharmonic function $\psi$, we deduce by \eqref{interpolation2} that 
            \begin{align*}
                &\int_{\Omega}\frac{|\D\psi|^2}{|x|^2}\left(\left(\frac{|x|}{b}\right)^{2\gamma}+\left(\frac{a}{|x|}\right)^{2\beta}\right)dx\leq C_{\beta,\gamma}\int_{\Omega}\frac{\psi^2}{|x|^4}\left(\left(\frac{|x|}{b}\right)^{4\beta}+\left(\frac{a}{|x|}\right)^{4\beta}\right)dx+C_{\beta,\gamma}\int_{\Omega}|\D^2\psi|^2dx\\
                &=2\,C_{\beta,\gamma}\int_{\Omega}\frac{\varphi^2}{|x|^4}\left(\left(\frac{|x|}{b}\right)^{4\beta}+\left(\frac{a}{|x|}\right)^{4\beta}\right)dx+2\,C_{\beta,\gamma}\int_{\Omega}\frac{u^2}{|x|^4}\left(\left(\frac{|x|}{b}\right)^{4\beta}+\left(\frac{a}{|x|}\right)^{4\beta}\right)dx\\
                &+2\,C_{\beta,\gamma}\int_{\Omega}|\D^2\varphi|^2dx+2\,C_{\beta,\gamma}\int_{\Omega}|\D^2u|^2dx\\
                &\leq 2\,C_{\beta,\gamma}\int_{\Omega}\frac{u^2}{|x|^4}\left(\left(\frac{|x|}{b}\right)^{4\beta}+\left(\frac{a}{|x|}\right)^{4\beta}\right)dx+2\,C_{\beta,\gamma}\left(1+C_{\beta}\right)\int_{\Omega}(\Delta u)^2dx+2\,C_{\beta,\gamma}\int_{\Omega}|\D^2u|^2dx
            \end{align*}
            where we used the fact that for all $\chi\in W^{2,2}_0(\Omega)$,
            \begin{align*}
                \int_{\Omega}|\D^2\chi|^2dx=\int_{\Omega}(\Delta\chi)^2dx.
            \end{align*}
            Therefore, we finally get
            \begin{align*}
                &\int_{\Omega}\frac{|\D u|^2}{|x|^2}\left(\left(\frac{|x|}{b}\right)^{2\gamma}+\left(\frac{a}{|x|}\right)^{2\gamma}\right)dx\leq 2\int_{\Omega}\frac{|\D\varphi|^2}{|x|^2}\left(\left(\frac{|x|}{b}\right)^{2\gamma}+\left(\frac{a}{|x|}\right)^{2\gamma}\right)dx\\
                &+2\int_{\Omega}\frac{|\D\psi|^2}{|x|^2}\left(\left(\frac{|x|}{b}\right)^{2\gamma}+\left(\frac{a}{|x|}\right)^{2\gamma}\right)dx\\
                &\leq 4\,C_{\beta,\gamma}\int_{\Omega}\frac{u^2}{|x|^4}\left(\left(\frac{|x|}{b}\right)^{4\beta}+\left(\frac{a}{|x|}\right)^{4\beta}\right)dx+2\left(C_{\gamma}'+2\,C_{\beta,\gamma}\left(1+C_{\beta}\right)\right)\int_{\Omega}(\Delta u)^2dx+4\,C_{\beta,\gamma}\int_{\Omega}|\D^2u|^2dx
            \end{align*}
            which concludes the proof of the corollary.
            \end{proof}

\subsection{General Case}

For $m>1$, the proof is a bit more involved, but the inequality is more precise. 

\begin{theorem}\label{poincare_weight_m}
    Let $m>1$, and let
    \begin{align*}
        \leb_m=\Delta+2(m-1)\cdot \frac{x}{|x|^2}\cdot \D+\frac{(m-1)^2}{|x|^2}.
    \end{align*}
    Let $0<a<b<\infty$ and $\Omega=B_b\setminus\bar{B}_a(0)$. Then, for all $0<\alpha<\min\ens{\dfrac{4\,m}{3},2}$, there exists a universal constant $C_{m,\alpha}<\infty$ such that for all $u\in W^{2,2}_0(\Omega)$, we have
    \begin{align}
    \left\{\begin{alignedat}{2}
        &\int_{\Omega}\frac{u^2}{|x|^4}\left(\left(\frac{|x|}{b}\right)^{2\alpha}+\left(\frac{a}{|x|}\right)^{2\alpha}\right)dx&&\leq C_{m,\alpha}\int_{\Omega}\left(\leb_mu\right)^2\left(\left(\frac{|x|}{b}\right)^{2\alpha}+\left(\frac{a}{|x|}\right)^{2\alpha}\right)dx\\
        &\int_{\Omega}\frac{|\D u|^2}{|x|^2}\left(\left(\frac{|x|}{b}\right)^{2\alpha}+\left(\frac{a}{|x|}\right)^{2\alpha}\right)dx&&\leq C_{m,\alpha}\int_{\Omega}\left(\leb_mu\right)^2\left(\left(\frac{|x|}{b}\right)^{2\alpha}+\left(\frac{a}{|x|}\right)^{2\alpha}\right)dx
        \end{alignedat}\right.
    \end{align}
\end{theorem}
\begin{rem}
    In particular, for all $\alpha>0$, we have
    \begin{align*}
        \int_{\Omega}\left(\frac{u^2}{|x|^4}+\frac{|\D u|^2}{|x|^2}\right)\left(\left(\frac{|x|}{b}\right)^{2\alpha}+\left(\frac{a}{|x|}\right)^{2\alpha}\right)dx\leq C_{\alpha}\int_{\Omega}\left(\Delta u+2(m-1)\frac{x}{|x|^2}\cdot \D u+\frac{(m-1)^2}{|x|^2}u\right)^2dx.
    \end{align*}
\end{rem}
\begin{proof}
Recall that
\begin{align}\label{fourth_order_m}
    \leb_m^{\ast}\leb_m=\Delta^2+2(m^2-1)\frac{1}{|x|^2}\Delta-4(m^2-1)\left(\frac{x}{|x|^2}\right)^t\cdot  \D^2(\,\cdot\,)\cdot\left(\frac{x}{|x|^2}\right)+\frac{(m^2-1)^2}{|x|^4}.
\end{align}
We have
\begin{align}\label{ipp_m_alpha1}
    &\int_{\Omega}|x|^{2\alpha}\left(\leb_mu\right)^2dx=\int_{\Omega}|x|^{2\alpha}\left(\Delta u+2(m-1)\frac{x}{|x|^2}\cdot \D u+\frac{(m-1)^2}{|x|^2}u\right)\leb_mu\,dx\nonumber\\
    &=\int_{\Omega}u\,\Delta\left(|x|^{2\alpha}\leb_mu\right)dx-2(m-1)\int_{\Omega}u\,\frac{x}{|x|^2}\cdot \D\left(|x|^{2\alpha}\leb_mu\right)dx+\int_{\Omega}|x|^{2\alpha}u\,\frac{(m-1)^2}{|x|^2}\leb_mu\,dx\nonumber\\
    &=\int_{\Omega}|x|^{2\alpha}u\,\leb_m^{\ast}\leb_mu\,dx+\int_{\Omega}\Delta\left(|x|^{2\alpha}\right)u\,\leb_mu\,dx+2\int_{\Omega}u\,\D\left(|x|^{2\alpha}\right)\cdot \D \left(\leb_mu\right)dx\nonumber\\
    &-2(m-1)\int_{\Omega}u\frac{x}{|x|^2}\cdot \left(|x|^{2\alpha}\right)\leb_mu\,dx\nonumber\\
    &=\int_{\Omega}|x|^{2\alpha}u\,\leb_m^{\ast}\leb_mu\,dx+4\alpha^2\int_{\Omega}|x|^{2\alpha}\frac{u}{|x|^2}\,\leb_mu\,dx+2\int_{\Omega}u\,\D\left(|x|^{2\alpha}\right)\cdot \D \left(\leb_mu\right)dx\nonumber\\
    &-4(m-1)\alpha\int_{\Omega}|x|^{2\alpha}\frac{u}{|x|^2}\leb_mu\,dx\nonumber\\
    &=\int_{\Omega}|x|^{2\alpha}\leb_m^{\ast}\leb_mu\,dx+4\alpha\int_{\Omega}|x|^{2\alpha}u\frac{x}{|x|^2}\cdot\D\left(\leb_mu\right)dx+\left(4\alpha^2-4(m-1)\alpha\right)\int_{\Omega}|x|^{2\alpha}\frac{u}{|x|^2}\leb_mu\,dx,
\end{align}
where we used the identities
\begin{align*}
    \left\{\begin{alignedat}{1}
        \D |x|^{2\alpha}=2\alpha\,x|x|^{2\alpha-2}\\
        \Delta |x|^{2\alpha}=4\alpha^2|x|^{2\alpha-2}.
    \end{alignedat}\right.
\end{align*}
By \eqref{fourth_order_m}
\begin{align}\label{ipp_m_alpha2}
    \int_{\Omega}|x|^{2\alpha}u\,\leb_m^{\ast}\leb_mu\,dx&=\int_{\Omega}|x|^{2\alpha}\left(u\,\Delta^2u\,+2(m^2-1)\frac{1}{|x|^2}u\,\Delta u-4(m^2-1)u\left(\frac{x}{|x|^2}\right)^t\cdot\D^2u\cdot\left(\frac{x}{|x|^2}\right)\right.\nonumber\\
    &\left.+\frac{(m^2-1)^2}{|x|^4}u^2\right)dx.
\end{align}
Then, 
\begin{align}\label{ipp_m_alpha3}
    &\int_{\Omega}|x|^{2\alpha}u\,\Delta^2u\,dx=\int_{\Omega}|x|^{2\alpha}\left(\Delta u\right)^2dx+\int_{\Omega}u\,\Delta u\Delta\left(|x|^{2\alpha}\right)dx+2\int_{\Omega}\D\left(|x|^{2\alpha}\right)\cdot \D u\,\Delta u\,dx\nonumber\\
    &=\int_{\Omega}|x|^{2\alpha}(\Delta u)^2dx+4\alpha^2\int_{\Omega}|x|^{2\alpha}\frac{u}{|x|^2}\Delta u\,dx+4\alpha\int_{\Omega}|x|^{2\alpha}\left(\frac{x}{|x|^2}\cdot \D u\right)\,\Delta u\,dx.
\end{align}
Therefore, we have by \eqref{ipp_m_alpha2} and \eqref{ipp_m_alpha3}
\begin{align}\label{ipp_m_alpha4}
    &\int_{\Omega}|x|^{2\alpha}u\,\leb_m^{\ast}\leb_mu\,dx=\int_{\Omega}|x|^{2\alpha}\left((\Delta u)^2+2(m^2-1)\frac{u}{|x|^2}\Delta u+\frac{(m^2-1)^2}{|x|^4}u^2\right)dx\nonumber\\
    &+4\alpha^2\int_{\Omega}|x|^{2\alpha}\frac{u}{|x|^2}\Delta u\,dx+4\alpha\int_{\Omega}|x|^{2\alpha}\left(\frac{x}{|x|^2}\cdot \D u\right)\Delta u\,dx\nonumber\\
    &-4(m^2-1)\int_{\Omega}|x|^{2\alpha}u\left(\frac{x}{|x|^2}\right)^t\cdot \D^2u\cdot\left(\frac{x}{|x|^2}\right)dx\nonumber\\
    &=\int_{\Omega}|x|^{2\alpha}\left(\Delta u+\frac{(m^2-1)}{|x|^2}u\right)^2dx
    +4\alpha^2\int_{\Omega}|x|^{2\alpha}\frac{u}{|x|^2}\Delta u\,dx+4\alpha\int_{\Omega}|x|^{2\alpha}\left(\frac{x}{|x|^2}\cdot \D u\right)\Delta u\,dx\nonumber\\
    &-4(m^2-1)\int_{\Omega}|x|^{2\alpha}u\left(\frac{x}{|x|^2}\right)^t\cdot \D^2u\cdot\left(\frac{x}{|x|^2}\right)dx.
\end{align}
Therefore, we have by \eqref{ipp_m_alpha1} and \eqref{ipp_m_alpha4}
\begin{align}\label{ipp_m_alpha5}
    &\int_{\Omega}|x|^{2\alpha}\left(\leb_mu\right)^2dx=\int_{\Omega}|x|^{2\alpha}\left(\Delta u+\frac{(m^2-1)}{|x|^2}u\right)^2dx+4\alpha\int_{\Omega}|x|^{2\alpha}u\frac{x}{|x|^2}\cdot\D\left(\leb_mu\right)dx\nonumber\\
    &+\left(4\alpha^2-4(m-1)\alpha\right)\int_{\Omega}|x|^{2\alpha}\frac{u}{|x|^2}\leb_mu\,dx
    +4\alpha^2\int_{\Omega}|x|^{2\alpha}\frac{u}{|x|^2}\Delta u\,dx+4\alpha\int_{\Omega}|x|^{2\alpha}\left(\frac{x}{|x|^2}\cdot \D u\right)\Delta u\,dx\nonumber\\
    &-4(m^2-1)\int_{\Omega}|x|^{2\alpha}u\left(\frac{x}{|x|^2}\right)^t\cdot \D^2u\cdot\left(\frac{x}{|x|^2}\right)dx
\end{align}
Then, integrating by parts, we get
\begin{align}\label{ipp_m_alpha6}
    &4\alpha\int_{\Omega}|x|^{2\alpha}u\frac{x}{|x|^2}\cdot \D\left(\leb_mu\right)dx=-8\alpha^2\int_{\Omega}|x|^{2\alpha}\frac{u}{|x|^2}\leb_mu\,dx-4\alpha\int_{\Omega}|x|^{2\alpha}\left(\frac{x}{|x|^2}\cdot \D u\right)\leb_mu\,dx\nonumber\\
    &=-8\alpha^2\int_{\Omega}|x|^{2\alpha}\frac{u}{|x|^2}\Delta u\,dx-16(m-1)\alpha^2\int_{\Omega}|x|^{2\alpha}\frac{u}{|x|^2}\left(\frac{x}{|x|^2}\cdot \D u\right)\,dx-8(m-1)^2\alpha^2\int_{\Omega}|x|^{2\alpha}\frac{u^2}{|x|^4}dx\nonumber\\
    &-4\alpha\int_{\Omega}|x|^{2\alpha}\left(\frac{x}{|x|^2}\cdot \D u\right)\Delta u\,dx-8(m-1)\alpha\int_{\Omega}|x|^{2\alpha}\left(\frac{x}{|x|^2}\cdot \D u\right)^2dx\nonumber\\
    &-4(m-1)^2\alpha\int_{\Omega}|x|^{2\alpha}\left(\frac{x}{|x|^2}\cdot \D u\right)\frac{u}{|x|^2}dx\nonumber\\
    &=-8\alpha^2\int_{\Omega}|x|^{2\alpha}\frac{u}{|x|^2}\Delta u\,dx-4\alpha\int_{\Omega}|x|^{2\alpha}\left(\frac{x}{|x|^2}\cdot \D u\right)\Delta u\,dx-8(m-1)\alpha\int_{\Omega}|x|^{2\alpha}\left(\frac{x}{|x|^2}\cdot \D u\right)^2dx\nonumber\\
    &-\left(16(m-1)\alpha^2+4(m-1)^2\alpha\right)\int_{\Omega}|x|^{2\alpha}\left(\frac{x}{|x|^2}\cdot \D u\right)\frac{u}{|x|^2}dx-8(m-1)^2\alpha^2\int_{\Omega}|x|^{2\alpha}\frac{u^2}{|x|^4}dx
\end{align}
while
\begin{align}\label{ipp_m_alpha7}
    &\left(4\alpha^2-4(m-1)\alpha\right)\int_{\Omega}|x|^{2\alpha}\frac{u}{|x|^2}\leb_mu\,dx
    +4\alpha^2\int_{\Omega}|x|^{2\alpha}\frac{u}{|x|^2}\Delta u\,dx+4\alpha\int_{\Omega}|x|^{2\alpha}\left(\frac{x}{|x|^2}\cdot \D u\right)\Delta u\,dx\nonumber\\
    &=\left(8\alpha^2-4(m-1)\alpha\right)\int_{\Omega}|x|^{2\alpha}\frac{u}{|x|^2}\Delta u\,dx+4\alpha\int_{\Omega}|x|^{2\alpha}\left(\frac{x}{|x|^2}\cdot \D u\right)\Delta u\,dx\nonumber\\
    &+2(m-1)(4\alpha^2-4(m-1)\alpha)\int_{\Omega}|x|^{2\alpha}\left(\frac{x}{|x|^2}\cdot \D u\right)\frac{u}{|x|^2}dx+(m-1)^2(4\alpha^2-4(m-1)\alpha)\int_{\Omega}|x|^{2\alpha}\frac{u^2}{|x|^4}dx.
\end{align}
Therefore, we have by \eqref{ipp_m_alpha6} and \eqref{ipp_m_alpha7}
\begin{align}\label{ipp_m_alpha8}
    &4\alpha\int_{\Omega}|x|^{2\alpha}u\frac{x}{|x|^2}\cdot \D\left(\leb_mu\right)dx+\left(4\alpha^2-4(m-1)\alpha\right)\int_{\Omega}|x|^{2\alpha}\frac{u}{|x|^2}\leb_mu\,dx
    +4\alpha^2\int_{\Omega}|x|^{2\alpha}\frac{u}{|x|^2}\Delta u\,dx\nonumber\\
    &+4\alpha\int_{\Omega}|x|^{2\alpha}\left(\frac{x}{|x|^2}\cdot \D u\right)\Delta u\,dx=-\colorcancel{8\alpha^2\int_{\Omega}|x|^{2\alpha}\frac{u}{|x|^2}\Delta u\,dx}{red}-\colorcancel{4\alpha\int_{\Omega}|x|^{2\alpha}\left(\frac{x}{|x|^2}\cdot \D u\right)\Delta u\,dx}{blue}\nonumber\\
    &-8(m-1)\alpha\int_{\Omega}|x|^{2\alpha}\left(\frac{x}{|x|^2}\cdot \D u\right)^2dx
    -\left(16(m-1)\alpha^2+4(m-1)^2\alpha\right)\int_{\Omega}|x|^{2\alpha}\left(\frac{x}{|x|^2}\cdot \D u\right)\frac{u}{|x|^2}dx\nonumber\\
    &-8(m-1)^2\alpha^2\int_{\Omega}|x|^{2\alpha}\frac{u^2}{|x|^4}dx+\left(\colorcancel{8\alpha^2}{red}-4(m-1)\alpha\right)\int_{\Omega}|x|^{2\alpha}\frac{u}{|x|^2}\Delta u\,dx\nonumber\\
    &+\colorcancel{4\alpha\int_{\Omega}|x|^{2\alpha}\left(\frac{x}{|x|^2}\cdot \D u\right)\Delta u\,dx}{blue}+(m-1)^2(4\alpha^2-4(m-1)\alpha)\int_{\Omega}|x|^{2\alpha}\frac{u^2}{|x|^4}dx
    \nonumber\\
    &+2(m-1)(4\alpha^2-4(m-1)\alpha)\int_{\Omega}|x|^{2\alpha}\left(\frac{x}{|x|^2}\cdot \D u\right)\frac{u}{|x|^2}dx\nonumber\\
    &=-8(m-1)\alpha\int_{\Omega}|x|^{2\alpha}\left(\frac{x}{|x|^2}\cdot \D u\right)^2dx-4(m-1)\alpha\int_{\Omega}|x|^{2\alpha}\frac{u}{|x|^2}\Delta u\,dx\nonumber\\
    &+\left(2(m-1)(4\alpha^2-4(m-1)\alpha)-\left(16(m-1)\alpha^2+4(m-1)^2\alpha\right)\right)\int_{\Omega}|x|^{2\alpha}\left(\frac{x}{|x|^2}\cdot\D u\right)\frac{u}{|x|^2}dx\nonumber\\
    &-(m-1)^2\left(4\alpha^2+4(m-1)\alpha\right)\int_{\Omega}|x|^{2\alpha}\frac{u^2}{|x|^4}dx.
\end{align}
Then, we have
\begin{align}\label{ipp_m_alpha9}
    \int_{\Omega}|x|^{2\alpha}\frac{u}{|x|^2}\Delta u\,dx&=-\int_{\Omega}\D\left(|x|^{2\alpha-2}u\right)\cdot \D u\,dx=-2(\alpha-1)\int_{\Omega}|x|^{2\alpha}\left(\frac{x}{|x|^2}\cdot \D u\right)\frac{u}{|x|^2}dx\nonumber\\
    &-\int_{\Omega}|x|^{2\alpha-2}|\D u|^2dx.
\end{align}
Therefore, we have by \eqref{ipp_m_alpha8} and \eqref{ipp_m_alpha9}
\begin{align}\label{ipp_m_alpha10}
    &4\alpha\int_{\Omega}|x|^{2\alpha}u\frac{x}{|x|^2}\cdot \D\left(\leb_mu\right)dx+\left(4\alpha^2-4(m-1)\alpha\right)\int_{\Omega}|x|^{2\alpha}\frac{u}{|x|^2}\leb_mu\,dx
    +4\alpha^2\int_{\Omega}|x|^{2\alpha}\frac{u}{|x|^2}\Delta u\,dx\nonumber\\
    &+4\alpha\int_{\Omega}|x|^{2\alpha}\left(\frac{x}{|x|^2}\cdot \D u\right)\Delta u\,dx=-8(m-1)\alpha\int_{\Omega}|x|^{2\alpha}\left(\frac{x}{|x|^2}\cdot \D u\right)^2dx+4(m-1)\alpha\int_{\Omega}|x|^{2\alpha}\frac{|\D u|^2}{|x|^2}dx\nonumber\\
    &+\left(2(m-1)(4\alpha^2-4(m-1)\alpha)-\left(16(m-1)\alpha^2+4(m-1)^2\alpha\right)-8(m-1)\alpha(1-\alpha)\right)\nonumber\\
    &\times\int_{\Omega}|x|^{2\alpha}\left(\frac{x}{|x|^2}\cdot\D u\right)\frac{u}{|x|^2}dx-(m-1)^2\left(4\alpha^2+4(m-1)\alpha\right)\int_{\Omega}|x|^{2\alpha}\frac{u^2}{|x|^4}dx.
\end{align}
Now, recall the following computation
\begin{align}\label{ipp_m_alpha11}
    \int_{\Omega}|x|^{2\alpha}\left(\frac{x}{|x|^2}\cdot \D u\right)\frac{u}{|x|^2}dx&=\frac{1}{2}\int_{\Omega}|x|^{2\alpha-2}\frac{x}{|x|^2}\cdot \D\left(u^2\right)\,dx=\frac{1}{2}\int_{\Omega}|x|^{2\alpha-2}\dive\left(\frac{x}{|x|^2}\,u^2\right)dx\nonumber\\
    &=(1-\alpha)\int_{\Omega}|x|^{2\alpha}\frac{u^2}{|x|^4}dx,
\end{align}
where we used that
\begin{align*}
    \dive\left(\frac{x}{|x|^2}\right)=\Delta\log|x|=0\qquad \text{in}\;\, \mathscr{D}'(\R^2\setminus\ens{0}).
\end{align*}
Therefore, we have 
\begin{align}\label{ipp_m_alpha12}
    &4\alpha\int_{\Omega}|x|^{2\alpha}u\frac{x}{|x|^2}\cdot \D\left(\leb_mu\right)dx+\left(4\alpha^2-4(m-1)\alpha\right)\int_{\Omega}|x|^{2\alpha}\frac{u}{|x|^2}\leb_mu\,dx
    +4\alpha^2\int_{\Omega}|x|^{2\alpha}\frac{u}{|x|^2}\Delta u\,dx\nonumber\\
    &+4\alpha\int_{\Omega}|x|^{2\alpha}\left(\frac{x}{|x|^2}\cdot \D u\right)\Delta u\,dx=-8(m-1)\alpha\int_{\Omega}|x|^{2\alpha}\left(\frac{x}{|x|^2}\cdot \D u\right)^2dx+4(m-1)\alpha\int_{\Omega}|x|^{2\alpha}\frac{|\D u|^2}{|x|^2}dx\nonumber\\
    &+\bigg\{\Big(2(m-1)(4\alpha^2-4(m-1)\alpha)-\left(16(m-1)\alpha^2+4(m-1)^2\alpha\right)-8(m-1)\alpha(1-\alpha)\Big)(1-\alpha)\nonumber\\
    &-(m-1)^2(4\alpha^2+4(m-1)\alpha)\bigg\}\int_{\Omega}|x|^{2\alpha}\frac{u^2}{|x|^4}dx.
\end{align}
Now, thanks to \eqref{operator3}, we have
    \begin{align*}
        \left(\frac{x}{|x|^2}\right)^t\cdot \D^2u\cdot\left(\frac{x}{|x|^2}\right)=\frac{1}{r^2}\p{r}^2u.
    \end{align*}
Therefore, using polar coordinates, we get
    \begin{align}\label{ipp_m_alpha13}
        &\int_{\Omega}|x|^{2\alpha}u\,\left(\frac{x}{|x|^2}\right)^t\cdot \D^2u\cdot\left(\frac{x}{|x|^2}\right)dx=\int_{S^1}\left(\int_{a}^br^{2\alpha-1}u(r,\theta)\p{r}^2u(r,\theta)dr\right)d\theta\nonumber\\
        &=-\int_{S^1}\left(\int_{a}^br^{2\alpha-1}(\p{r}u(r,\theta))^2dr\right)d\theta-(2\alpha-1)\int_{S^1}\left(\int_{a}^br^{2\alpha-2}u(r,\theta)\p{r}u(r,\theta)dr\right)d\theta\nonumber\\
        &=-\int_{S^1}\left(\int_{a}^br^{2\alpha-1}(\p{r}u(r,\theta))^2dr\right)d\theta-\frac{2\alpha-1}{2}\int_{S^1}\left(\int_{a}^br^{2\alpha-2}\p{r}(u(r,\theta))^2dr\right)d\theta\nonumber\\
        &=-\int_{S^1}\left(\int_{a}^br^{2\alpha-1}(\p{r}u(r,\theta))^2dr\right)d\theta+(\alpha-1)(2\alpha-1)\int_{S^1}\left(\int_{a}^br^{2\alpha-3}(u(r,\theta))^2dr\right)d\theta\nonumber\\
        &=-\int_{\Omega}|x|^{2\alpha-2}|\p{r}u|^2dx+(\alpha-1)(2\alpha-1)\int_{\Omega}|x|^{2}\frac{u^2}{|x|^4}dx\nonumber\\
        &=-\int_{\Omega}|x|^{2\alpha}\left(\frac{x}{|x|^2}\cdot \D u\right)^2dx+(\alpha-1)(2\alpha-1)\int_{\Omega}|x|^{2\alpha}\frac{u^2}{|x|^4}dx,
    \end{align}
Finally, by \eqref{ipp_m_alpha5}, \eqref{ipp_m_alpha12}, and \eqref{ipp_m_alpha13}, we deduce that
\begin{align}\label{ipp_m_alpha14}
    &\int_{\Omega}|x|^{2\alpha}\left(\leb_mu\right)^2dx=\int_{\Omega}|x|^{2\alpha}\left(\Delta u+\frac{(m^2-1)}{|x|^2}u\right)^2dx-8(m-1)\alpha\int_{\Omega}|x|^{2\alpha}\left(\frac{x}{|x|^2}\cdot \D u\right)^2dx\nonumber\\
    &+4(m-1)\alpha\int_{\Omega}|x|^{2\alpha}\frac{|\D u|^2}{|x|^2}dx\nonumber\\
    &+\bigg\{\Big(2(m-1)(4\alpha^2-4(m-1)\alpha)-\left(16(m-1)\alpha^2+4(m-1)^2\alpha\right)-8(m-1)\alpha(1-\alpha)\Big)(1-\alpha)\nonumber\\
    &-(m-1)^2(4\alpha^2+4(m-1)\alpha)\bigg\}\int_{\Omega}|x|^{2\alpha}\frac{u^2}{|x|^4}dx\nonumber\\
    &+4(m^2-1)\int_{\Omega}|x|^{2\alpha}\left(\frac{x}{|x|^2}\cdot \D u\right)^2dx-4(m^2-1)(\alpha-1)(2\alpha-1)\int_{\Omega}|x|^{2\alpha}\frac{u^2}{|x|^4}dx\nonumber\\
    &=\int_{\Omega}|x|^{2\alpha}\left(\Delta u+\frac{(m^2-1)}{|x|^2}u\right)^2dx+4(m-1)(m+1-2\alpha)\int_{\Omega}|x|^{2\alpha}\left(\frac{x}{|x|^2}\cdot \D u\right)^2dx\nonumber\\
    &+4(m-1)\alpha\int_{\Omega}|x|^{2\alpha}\frac{|\D u|^2}{|x|^2}dx
    -4(m-1)(2\alpha^2+(m^2-2m-3)\alpha+(m+1))\int_{\Omega}|x|^{2\alpha}\frac{u^2}{|x|^4}dx.
\end{align}
If $2\alpha^2+(m^2-2m-3)\alpha+(m+1)\leq 0$, we are done, so assume that $2\alpha^2+(m^2-2m-3)\alpha+(m+1)>0$. First, notice that 
\begin{align}\label{ipp_m_alpha15}
    \int_{\Omega}|x|^{2\alpha}\frac{|\D u|^2}{|x|^2}dx=\int_{\Omega}|x|^{2\alpha}\left(\frac{x}{|x|^2}\cdot \D u\right)^2dx+\int_{\Omega}|x|^{2\alpha}\frac{(\p{\theta}u)^2}{|x|^4}dx.
\end{align}
Then, the elementary identity \eqref{ipp_m_alpha11} and Cauchy-Schwarz inequality imply that for all $\alpha<1$
\begin{align}\label{ipp_m_alpha16}
    \int_{\Omega}|x|^{2\alpha}\frac{u^2}{|x|^4}dx\leq \frac{1}{(1-\alpha)^2}\int_{\Omega}|x|^{2\alpha}\left(\frac{x}{|x|^2}\cdot \D u\right)^2dx.
\end{align}
Therefore, we have by \eqref{ipp_m_alpha14}, \eqref{ipp_m_alpha15}, and \eqref{ipp_m_alpha16}
\begin{align*}
    &\int_{\Omega}|x|^{2\alpha}\left(\leb_mu\right)^2dx\geq \int_{\Omega}|x|^{2\alpha}\left(\frac{x}{|x|^2}\cdot \D u\right)^2dx+4(m-1)\alpha\int_{\Omega}|x|^{2\alpha}\frac{(\p{\theta}u)^2}{|x|^4}dx\\
    &+4(m-1)\left(m+1-2\alpha+\alpha-\frac{1}{(1-\alpha)^2}\left(2\alpha^2+(m^2-2m-3)\alpha+(m+1)\right)\right)\int_{\Omega}|x|^{2\alpha}\left(\frac{x}{|x|^2}\cdot \D u\right)^2dx\\
    &=\int_{\Omega}|x|^{2\alpha}\left(\frac{x}{|x|^2}\cdot \D u\right)^2dx+4(m-1)\alpha\int_{\Omega}|x|^{2\alpha}\frac{(\p{\theta}u)^2}{|x|^4}dx\\
    &-\frac{4(m-1)\alpha}{(1-\alpha)^2}\left(\alpha^2-(m+1)\alpha+m^2\right)\int_{\Omega}|x|^{2\alpha}\left(\frac{x}{|x|^2}\cdot \D u\right)^2dx.
\end{align*}
Since
\begin{align*}
    (m+1)^2-4m^2=-3m^2+2m+1<0\quad \text{for all}\;\, m>1, 
\end{align*}
we deduce that for all $m>1$ and $\alpha>0$, we have 
\begin{align*}
    &\int_{\Omega}\frac{|\D u|^2}{|x|^2}\left(\frac{a}{|x|}\right)^{2\alpha}dx\leq \frac{(1+\alpha)^2}{4(m-1)\alpha(\alpha^2+(m+1)\alpha+m^2)}\int_{\Omega}\left(\leb_mu\right)^2\left(\frac{a}{|x|}\right)^{2\alpha}dx\\
    &\leq \frac{(1+\alpha)^2}{4(m-1)\alpha(\alpha^2+(m+1)\alpha+m^2)}\int_{\Omega}\left(\Delta u+2(m-1)\frac{x}{|x|^2}\cdot \D u+\frac{(m-1)^2}{|x|^2}u\right)^2dx.
\end{align*}
Furthermore, the inequality \eqref{ipp_m_alpha16} shows that for all $m>1$ and for all $\alpha>0$, we have
\begin{align*}
    \int_{\Omega}\frac{u^2}{|x|^4}\left(\frac{a}{|x|}\right)^{2\alpha}dx&\leq \frac{1+\alpha}{4(m-1)\alpha(\alpha^2+(m+1)\alpha+m^2)}\int_{\Omega}\left(\leb_mu\right)^2\left(\frac{a}{|x|}\right)^{2\alpha}dx\\
    &\leq \frac{1+\alpha}{4(m-1)\alpha(\alpha^2+(m+1)\alpha+m^2)}\int_{\Omega}\left(\Delta u+2(m-1)\frac{x}{|x|^2}\cdot \D u+\frac{(m-1)^2}{|x|^2}u\right)^2dx.
\end{align*}
For the second inequality with the weight $\left(\dfrac{|x|}{b}\right)^{2\alpha}$, we will take advantage of the quasi conformal invariance of our problem and of the almost self-adjointness properties of $\leb_m^{\ast}\leb_m$. 

   Now, making the change of variable
    \begin{align*}
        \frac{x}{|x|^2}=y
    \end{align*}
    we get
    \begin{align*}
        \frac{dx}{|x|^4}=dy.
    \end{align*}
    Therefore, we get
    \begin{align}\label{2lm1}
        \int_{\Omega}|x|^{-2\alpha}\frac{u^2}{|x|^4}dx=\int_{\Omega'}|y|^{2\alpha}|u(i(y))|^2dy=\int_{\Omega')}|y|^{2\alpha}|\bar{u}(y)|^2dy,
    \end{align}
    where $\Omega'=B_{a^{-1}}\setminus\bar{B}_{b^{-1}}(0)$, and 
    where $\bar{u}=u\circ \iota:\Omega'\rightarrow \R$. Then, thanks to \eqref{inverse_laplacian}, if $\bar{u}=u(\iota(x))$, then 
    \begin{align*}
        \Delta \bar{u}=\frac{1}{|x|^4}\Delta u(\iota(x)).
    \end{align*}
    On the other hand, we have
    \begin{align*}
        x\cdot \D\bar{u}&=x_1\left(\frac{x_2^2-x_1^2}{(x_1^2+x_2^2)^2}\p{x_1}u-\frac{2x_1x_2}{(x_1^2+x_2^2)^2}\p{x_2}u\right)+x_2\left(-\frac{2x_1x_2}{(x_1^2+x_2^2)}\p{x_1}u\right)\\
        &=-\frac{x}{|x|^2}\cdot \D u(\iota(x)).
    \end{align*}
    Therefore, we have
    \begin{align*}
        \Delta \bar{u}-2(m-1)\frac{x}{|x|^2}\cdot \D \bar{u}+\frac{(m-1)^2}{|x|^2}\bar{u}=\frac{1}{|x|^4}\Delta u\circ \iota+2(m-1)\frac{x}{|x|^4}\cdot \D u\circ\iota+\frac{(m-1)^2}{|x|^2}u\circ \iota,
    \end{align*}
    or more compactly
    \begin{align}\label{inversion_lm1}
        \leb_m^{\ast}\bar{u}=\frac{1}{|x|^4}\left(\leb_mu\right)\circ\iota
    \end{align}
    which implies that for all $b>1$
    \begin{align}\label{inversion_lm2}
        &\int_{\Omega'}|y|^{2\alpha}\left(\Delta \bar{u}-2(m-1)\frac{y}{|y|^2}\cdot \D \bar{u}+\frac{(m-1)^2}{|y|^2}\bar{u}\right)^2dy\nonumber\\
        &=\int_{\Omega}\frac{1}{|x|^{2\alpha}}\left(|x|^4\Delta u+2(m-1)|x|^2\, x\cdot \D u+(m-1)^2|x|^2u\right)^2\frac{dx}{|x|^4}\nonumber\\
        &=\int_{\Omega}{|x|^{4-2\alpha}}\left(\Delta u+2(m-1)\frac{x}{|x|^2}\cdot \D u+\frac{(m-1)^2}{|x|^2}u\right)^2dx=\int_{\Omega}|x|^{4-2\alpha}\left(\leb_mu\right)^2dx.
    \end{align}
    Now, recall that by \eqref{operator1}, we have
    \begin{align*}
        \leb_m^{\ast}\leb_m=\Delta^2+2(m^2-1)\frac{1}{|x|^2}\Delta-4(m^2-1)\left(\frac{x}{|x|^2}\right)^{t}\cdot \D^2\left(\,\cdot\,\right)\cdot\left(\frac{x}{|x|^2}\right)+\frac{(m^2-1)^2}{|x|^4},
    \end{align*}
    while \eqref{operator2} implies that 
    \begin{align*}
        &\leb_m\leb_m^{\ast}=\widetilde{\leb_{m-2}}^{\ast}\widetilde{\leb_{m-2}}=\leb_{m-2}^{\ast}\leb_{m-2}\\
        &=\Delta^2+2((m-2)^2-1)\frac{1}{|x|^2}\Delta-4((m-2)^2-1)-1)\left(\frac{x}{|x|^2}\right)^t\cdot\D^2\left(\,\cdot\,\right)\cdot\left(\frac{x}{|x|^2}\right)+\frac{((m-2)^2-1)^2}{|x|^4}.
    \end{align*}
    Therefore, we compute as previously
    \begin{align}\label{new_m_alpha1}
        &\int_{\Omega}|x|^{2\alpha}\left(\leb_{m}^{\ast}u\right)^2dx=\int_{\Omega}|x|^{2\alpha}\left(\Delta u-2(m-1)\frac{x}{|x|^2}\cdot \D u+\frac{(m-1)^2}{|x|^2}u\right)\leb_m^{\ast}u\,dx\nonumber\\
        &=\int_{\Omega}u\,\Delta\left(|x|^{2\alpha}\leb_m^{\ast}u\right)dx+2(m-1)\int_{\Omega}u\frac{x}{|x|^2}\cdot \D\left(|x|^{2\alpha}\leb_m^{\ast}u\right)dx+\int_{\Omega}|x|^{2\alpha}u\frac{(m-1)^2}{|x|^2}\leb_m^{\ast}u\,dx\nonumber\\
        &=\int_{\Omega}|x|^{2\alpha}u\,\leb_m\leb_m^{\ast}u\,dx+4\alpha\int_{\Omega}|x|^{2\alpha}u\frac{x}{|x|^2}\cdot \D\left(\leb_m^{\ast}u\right)dx+\left(4\alpha^2+4(m-1)\alpha\right)\int_{\Omega}|x|^{2\alpha}\frac{u}{|x|^2}\leb_m^{\ast}u\,dx.
    \end{align}
    Since $\leb_m\leb_m^{\ast}=\leb_{m-2}^{\ast}\leb_{m-2}$, we have by \eqref{ipp_m_alpha4}
    \begin{align}\label{new_m_alpha2}
        &\int_{\Omega}|x|^{2\alpha}u\,\leb_m\leb_m^{\ast}u\,dx=\int_{\Omega}|x|^{2\alpha}\left(\Delta u+\frac{((m-2)^2-1)}{|x|^2}u\right)^2dx+4\alpha^2\int_{\Omega}|x|^{2\alpha}\frac{u}{|x|^2}\Delta u\,dx\nonumber\\
        &+4\alpha\int_{\Omega}|x|^{2\alpha}\left(\frac{x}{|x|^2}\cdot \D u\right)\Delta u\,dx
        -4((m-2)^2-1)\int_{\Omega}|x|^{2\alpha}u\left(\frac{x}{|x|^2}\right)^t\cdot \D^2u\cdot\left(\frac{x}{|x|^2}\right)dx.
    \end{align}
    Then, we have (compare with \eqref{ipp_m_alpha6})
    \begin{align}\label{new_m_alpha3}
        &4\alpha\int_{\Omega}|x|^{2\alpha}u\frac{x}{|x|^2}\cdot\D\left(\leb_m^{\ast}u\right)dx=-8\alpha^2\int_{\Omega}|x|^{2\alpha}\frac{u}{|x|^2}\leb_m^{\ast}u\,dx-4\alpha\int_{\Omega}|x|^{2\alpha}\left(\frac{x}{|x|^2}\cdot \D u\right)\leb_m^{\ast}u\,dx\nonumber\\
        &=-8\alpha^2\int_{\Omega}|x|^{2\alpha}\frac{u}{|x|^2}\Delta u\,dx -4\alpha\int_{\Omega}|x|^{2\alpha}\left(\frac{x}{|x|^2}\cdot \D u\right)\Delta u\,dx+8(m-1)\alpha\int_{\Omega}|x|^{2\alpha}\left(\frac{x}{|x|^2}\cdot \D u\right)^2dx\nonumber\\
        &+\left(16(m-1)\alpha^2-4(m-1)^2\alpha\right)\int_{\Omega}|x|^{2\alpha}\left(\frac{x}{|x|^2}\cdot \D u\right)\frac{u}{|x|^2}dx-8(m-1)^2\alpha^2\int_{\Omega}|x|^{2\alpha}\frac{u^2}{|x|^4}dx.
    \end{align}
    Then, we simply expand as in \eqref{ipp_m_alpha7}
    \begin{align}\label{new_m_alpha4}
        &\left(4\alpha^2+4(m-1)\alpha\right)\int_{\Omega}|x|^{2\alpha}\frac{u}{|x|^2}\leb_m^{\ast}u\,dx+4\alpha^2\int_{\Omega}|x|^{2\alpha}\frac{u}{|x|^2}\Delta u\,dx+4\alpha\int_{\Omega}|x|^{2\alpha}\left(\frac{x}{|x|^2}\cdot \D u\right)\Delta u\,dx\nonumber\\
        &=\left(8\alpha^2+4(m-1)\alpha\right)\int_{\Omega}|x|^{2\alpha}\frac{u}{|x|^2}\Delta u\,dx+4\alpha\int_{\Omega}|x|^{2\alpha}\left(\frac{x}{|x|^2}\cdot \D u\right)\Delta u\,dx\nonumber\\
        &-2(m-1)(4\alpha^2+4(m-1)\alpha)\int_{\Omega}|x|^{2\alpha}\left(\frac{x}{|x|^2}\cdot \D u\right)\frac{u}{|x|^2}dx+(m-1)^2\left(4\alpha^2+4(m-1)\alpha\right)\int_{\Omega}|x|^{2\alpha}\frac{u^2}{|x|^4}dx.
    \end{align}
    Therefore, we have by \eqref{new_m_alpha3} and \eqref{new_m_alpha4}
    \begin{align}\label{new_m_alpha5}
        &4\alpha\int_{\Omega}|x|^{2\alpha}u\frac{x}{|x|^2}\cdot\D\left(\leb_m^{\ast}u\right)dx+\left(4\alpha^2+4(m-1)\alpha\right)\int_{\Omega}|x|^{2\alpha}\frac{u}{|x|^2}\leb_m^{\ast}u\,dx+4\alpha^2\int_{\Omega}|x|^{2\alpha}\frac{u}{|x|^2}\Delta u\,dx\nonumber\\
        &+4\alpha\int_{\Omega}|x|^{2\alpha}\left(\frac{x}{|x|^2}\cdot \D u\right)\Delta u\,dx
        =-\colorcancel{8\alpha^2\int_{\Omega}|x|^{2\alpha}\frac{u}{|x|^2}\Delta u\,dx}{red} -\colorcancel{4\alpha\int_{\Omega}|x|^{2\alpha}\left(\frac{x}{|x|^2}\cdot \D u\right)\Delta u\,dx}{blue}\nonumber\\
        &+8(m-1)\alpha\int_{\Omega}|x|^{2\alpha}\left(\frac{x}{|x|^2}\cdot \D u\right)^2dx
        +\left(16(m-1)\alpha^2-4(m-1)^2\alpha\right)\int_{\Omega}|x|^{2\alpha}\left(\frac{x}{|x|^2}\cdot \D u\right)\frac{u}{|x|^2}dx\nonumber\\
        &-8(m-1)^2\alpha^2\int_{\Omega}|x|^{2\alpha}\frac{u^2}{|x|^4}dx
        +\left(\colorcancel{8\alpha^2}{red}+4(m-1)\alpha\right)\int_{\Omega}|x|^{2\alpha}\frac{u}{|x|^2}\Delta u\,dx+\colorcancel{4\alpha\int_{\Omega}|x|^{2\alpha}\left(\frac{x}{|x|^2}\cdot \D u\right)\Delta u\,dx}{blue}\nonumber\\
        &-2(m-1)(4\alpha^2+4(m-1)\alpha)\int_{\Omega}|x|^{2\alpha}\left(\frac{x}{|x|^2}\cdot \D u\right)\frac{u}{|x|^2}dx+(m-1)^2\left(4\alpha^2+4(m-1)\alpha\right)\int_{\Omega}|x|^{2\alpha}\frac{u^2}{|x|^4}dx\nonumber\\
        &=8(m-1)\alpha\int_{\Omega}|x|^{2\alpha}\left(\frac{x}{|x|^2}\cdot\D u\right)^2dx+4(m-1)\alpha\int_{\Omega}|x|^{2\alpha}\frac{u}{|x|^2}\Delta u\,dx\nonumber\\
        &+\left(16(m-1)\alpha^2-4(m-1)^2\alpha-2(m-1)\left(4\alpha^2+4(m-1)\alpha\right)\right)\int_{\Omega}|x|^{2\alpha}\left(\frac{x}{|x|^2}\cdot \D u\right)\frac{u}{|x|^2}dx\nonumber\\
        &+(m-1)^2\left(4(m-1)\alpha-4\alpha^2\right)\int_{\Omega}|x|^{2\alpha}\frac{u^2}{|x|^4}dx.
    \end{align}
    By \eqref{new_m_alpha1}, \eqref{new_m_alpha2}, and \eqref{new_m_alpha5}, we deduce that 
    \begin{align}\label{new_m_alpha6}
        &\int_{\Omega}|x|^{2\alpha}\left(\leb_m^{\ast}u\right)^2dx=\int_{\Omega}|x|^{2\alpha}\left(\Delta u+\frac{((m-2)^2-1)}{|x|^2}u\right)^2dx+8(m-1)\alpha\int_{\Omega}|x|^{2\alpha}\left(\frac{x}{|x|^2}\cdot\D u\right)^2dx\nonumber\\
        &-4\left((m-2)^2-1\right)\int_{\Omega}|x|^{2\alpha}u\left(\frac{x}{|x|^2}\right)\cdot \D^2u\cdot\left(\frac{x}{|x|^2}\right)dx+4(m-1)\alpha\int_{\Omega}|x|^{2\alpha}\frac{u}{|x|^2}\Delta u\,dx\nonumber\\
        &+\left(16(m-1)\alpha^2-4(m-1)^2\alpha-2(m-1)\left(4\alpha^2+4(m-1)\alpha\right)\right)\int_{\Omega}|x|^{2\alpha}\left(\frac{x}{|x|^2}\cdot \D u\right)\frac{u}{|x|^2}dx\nonumber\\
        &+(m-1)^2\left(4(m-1)\alpha-4\alpha^2\right)\int_{\Omega}|x|^{2\alpha}\frac{u^2}{|x|^4}dx.
    \end{align}
    Now, we have
    \begin{align}\label{new_m_alpha7}
        &\int_{\Omega}|x|^{2\alpha-2}u\Delta u\,dx=-\int_{\Omega}|x|^{2\alpha}\frac{|\D u|^2}{|x|^2}dx-2(\alpha-1)\int_{\Omega}|x|^{2\alpha}\left(\frac{x}{|x|^2}\cdot \D u\right)\frac{u}{|x|^2}dx\nonumber\\
        &=-\int_{\Omega}|x|^{2\alpha}\left(\frac{x}{|x|^2}\cdot \D u\right)^2dx-\int_{\Omega}|x|^{2\alpha}\frac{(\p{\theta}u)^2}{|x|^4}dx-2(\alpha-1)\int_{\Omega}|x|^{2\alpha}\left(\frac{x}{|x|^2}\cdot \D u\right)\frac{u}{|x|^2}dx
    \end{align}
    while \eqref{ipp_m_alpha13} implies that
    \begin{align}\label{new_m_alpha8}
        \int_{\Omega}|x|^{2\alpha}u\left(\frac{x}{|x|^2}\right)^t\cdot \D^2u\cdot\left(\frac{x}{|x|^2}\right)dx=-\int_{\Omega}|x|^{2\alpha}\left(\frac{x}{|x|^2}\cdot\D u\right)^2dx+(\alpha-1)(2\alpha-1)\int_{\Omega}|x|^{2\alpha}\frac{u^2}{|x|^4}dx.
    \end{align}
    Therefore, we have by \eqref{new_m_alpha6}, \eqref{new_m_alpha7}, and \eqref{new_m_alpha8}
    \begin{align}\label{new_m_alpha9}
        &\int_{\Omega}|x|^{2\alpha}\left(\leb_m^{\ast}u\right)^2dx=\int_{\Omega}|x|^{2\alpha}\left(\Delta u+\frac{((m-2)^2-1)}{|x|^2}u\right)^2dx\nonumber\\
        &+\left(4\left((m-2)^2-1\right)+4(m-1)\alpha\right)\int_{\Omega}|x|^{2\alpha}\left(\frac{x}{|x|^2}\cdot \D u\right)^2dx-4(m-1)\alpha\int_{\Omega}|x|^{2\alpha}\frac{(\p{\theta}u)^2}{|x|^4}dx\nonumber\\
        &+\bigg\{16(m-1)\alpha^2-4(m-1)^2\alpha-2(m-1)\left(4\alpha^2+4(m-1)\alpha\right)
        -8(m-1)\alpha(\alpha-1)\bigg\}\nonumber\\
        &\times\int_{\Omega}|x|^{2\alpha}\left(\frac{x}{|x|^2}\cdot \D u\right)\frac{u}{|x|^2}dx\nonumber\\
        &+\bigg\{(m-1)^2\left(4(m-1)\alpha-4\alpha^2\right)-4\left((m-2)^2-1\right)(\alpha-1)(2\alpha-1)\bigg\}\int_{\Omega}|x|^{2\alpha}\frac{u^2}{|x|^4}dx.
    \end{align}
    Finally, using \eqref{ipp_m_alpha11}, since
    \begin{align*}
        &(m-1)^2\left(4(m-1)\alpha-4\alpha^2\right)-4\left((m-2)^2-1\right)(\alpha-1)(2\alpha-1)\\
        &+(1-\alpha)\bigg\{16(m-1)\alpha^2-4(m-1)^2\alpha-2(m-1)\left(4\alpha^2+4(m-1)\alpha\right)
        -8(m-1)\alpha(\alpha-1)\bigg\}\\
        &=4(m-1)\left(2\alpha^2+(m^2-2m-3)\alpha-m+3\right),
    \end{align*}
    we get by \eqref{new_m_alpha9}
    \begin{align}\label{new_m_alpha10}
        &\int_{\Omega}|x|^{2\alpha}\left(\leb_m^{\ast}u\right)^2dx=\int_{\Omega}|x|^{2\alpha}\left(\Delta u+\frac{((m-2)^2-1)}{|x|^2}u\right)^2dx\nonumber\\
        &+\left(4\left((m-2)^2-1\right)+4(m-1)\alpha\right)\int_{\Omega}|x|^{2\alpha}\left(\frac{x}{|x|^2}\cdot \D u\right)^2dx-4(m-1)\alpha\int_{\Omega}|x|^{2\alpha}\frac{(\p{\theta}u)^2}{|x|^4}dx\nonumber\\
        &+4(m-1)\left(2\alpha^2+(m^2-2m-3)\alpha-m+3\right)\int_{\Omega}|x|^{2\alpha}\frac{u^2}{|x|^4}dx.
    \end{align}
    We would be finished it it were not for this tangential part of the gradient. Therefore, using identity \eqref{new_m_alpha7}, we get
    \begin{align}\label{new_m_alpha11}
        &-4(m-1)\alpha\int_{\Omega}|x|^{2\alpha}\frac{(\partial_{\theta}u)^2}{|x|^4}dx=4(m-1)\alpha\int_{\Omega}|x|^{2\alpha-2}u\,\Delta u\,dx+4(m-1)\alpha\int_{\Omega}|x|^{2\alpha}\left(\frac{x}{|x|^{2}}\cdot \D u\right)^2dx\nonumber\\
        &+8(m-1)\alpha(\alpha-1)\int_{\Omega}|x|^{2\alpha}\left(\frac{x}{|x|^2}\cdot \D u\right)\frac{u}{|x|^2}dx
        =4(m-1)\alpha\int_{\Omega}|x|^{2\alpha-2}u\,\Delta u\,dx\nonumber\\
        &+4(m-1)\alpha\int_{\Omega}|x|^{2\alpha}\left(\frac{x}{|x|^{2}}\cdot \D u\right)^2dx-8(m-1)\alpha(\alpha-1)^2\int_{\Omega}|x|^{2\alpha}\frac{u^2}{|x|^4}dx,
    \end{align}
    and \eqref{new_m_alpha10} becomes
    \begin{align}\label{new_m_alpha12}
        &\int_{\Omega}|x|^{2\alpha}\left(\leb_m^{\ast}u\right)^2dx=\int_{\Omega}|x|^{2\alpha}\left(\Delta u+\frac{((m-2)^2-1)}{|x|^2}u\right)^2dx\nonumber\\
        &+\left(4\left((m-2)^2-1\right)+8(m-1)\alpha\right)\int_{\Omega}|x|^{2\alpha}\left(\frac{x}{|x|^2}\cdot \D u\right)^2dx-4(m-1)\alpha\int_{\Omega}|x|^{2\alpha}u\,\Delta u\,dx\nonumber\\
        &+4(m-1)\left(-2\alpha^3+6\alpha^2+(m^2-2m-5)\alpha-m+3\right)\int_{\Omega}|x|^{2\alpha}\frac{u^2}{|x|^4}dx.
    \end{align}
    Then, we have
    \begin{align}\label{new_m_alpha13}
        &-4(m-1)\alpha\int_{\Omega}|x|^{2\alpha}u\,\Delta u\,dx=-4(m-1)\alpha\int_{\Omega}|x|^{2\alpha}u\left(\Delta u+\frac{((m-2)^2-1)}{|x|^2}u\right)dx\nonumber\\
        &+4(m-1)^2(m-3)\alpha\int_{\Omega}|x|^{2\alpha}\frac{u^2}{|x|^4}dx
        \geq -4(m-1)^2\alpha^2\int_{\Omega}|x|^{2\alpha}\frac{u^2}{|x|^4}dx\nonumber\\
        &-\int_{\Omega}|x|^{2\alpha}\left(\Delta +\frac{((m-2)^2-1)}{|x|^2}u\right)^2dx
        +4(m-1)^3(m-3)\alpha\int_{\Omega}|x|^{2\alpha}\frac{u^2}{|x|^4}dx.
    \end{align}
    On the other hand, we have (provided that $(m-2)^2-1+2(m-1)\alpha>0$ (which is always true for $m\geq 2$ and $\alpha>1$) by \eqref{ipp_m_alpha16}
    \small
    \begin{align}\label{new_m_alpha14}
        \left(4((m-2)^2-1)+8(m-1)\alpha\right)\int_{\Omega}|x|^{2\alpha}\left(\frac{x}{|x|^2}\cdot \D u\right)^2dx\geq 4(m-1)\left(m-3+2\alpha\right)(\alpha-1)^2\int_{\Omega}|x|^{2\alpha}\frac{u^2}{|x|^4}dx.
    \end{align}
    \normalsize
    Since 
    \begin{align}\label{new_m_alpha15}
        &4(m-1)\left(m-3+2\alpha\right)(\alpha-1)^2+4(m-1)\left(-2\alpha^3+6\alpha^2+(m^2-2m-5)\alpha-m+3\right)\nonumber\\
        &-4(m-1)^2\alpha^2+4(m-1)^3(m-3)\alpha=4\,m(m-1)^2(m-3)\alpha,
    \end{align}
    we deduce by \eqref{new_m_alpha12}, \eqref{new_m_alpha12}, \eqref{new_m_alpha12}, and \eqref{new_m_alpha12} that
    \begin{align}\label{m_large_ineq}
        4\,m(m-1)^2(m-3)\alpha\int_{\Omega}|x|^{2\alpha}\frac{u^2}{|x|^4}dx\leq \int_{\Omega}|x|^{2\alpha}\left(\leb_m^{\ast}u\right)^2dx.
    \end{align}
    Therefore, for all $m>3$ and for all $\alpha>0$ we deduce that
    \begin{align}\label{m_large_ineq2}
        \int_{\Omega}|x|^{2\alpha}\frac{u^2}{|x|^4}dx\leq \frac{1}{4\,m(m-1)^2(m-3)\alpha}\int_{\Omega}|x|^{2\alpha}\left(\Delta u-2(m-1)\frac{x}{|x|^2}\cdot \D u+\frac{(m-1)^2}{|x|^4}\right)^2dx.
    \end{align}
    We also deduce that for all $m>3$ and for all $\alpha>0$
    \begin{align}\label{m_large_ineq3}
        &4(m-1)(m-3+2\alpha)\int_{\Omega}|x|^{2\alpha}\left(\frac{x}{|x|^2}\cdot \D u\right)^2dx\leq \int_{\Omega}|x|^{2\alpha}\left(\leb_m^{\ast}u\right)^2dx\nonumber\\
        &+\left(4(m-1)(m-3+2\alpha)(\alpha-1)^2-4m(m-1)^2(m-3)\alpha\right)\int_{\Omega}|x|^{2\alpha}\frac{u^2}{|x|^4}dx\nonumber\\
        &\leq \left(1+\left(\frac{(m-3+2\alpha)(\alpha-1)^2}{m(m-1)(m-3)\alpha}-1\right)_+\right)\int_{\Omega}|x|^{2\alpha}\left(\leb_m^{\ast}u\right)^2dx,
    \end{align}
    which shows that we have in particular that for all $m>3$ and $\alpha>0$
    \begin{align}\label{m_large_ineq4}
        \int_{\Omega}|x|^{2\alpha}\left(\frac{x}{|x|^2}\cdot \D u\right)^2dx\leq \left(\frac{1}{4(m-1)(m-3+2\alpha)}+\frac{(\alpha-1)^2}{4m(m-1)^2(m-3)\alpha}\right)\int_{\Omega}|x|^{2\alpha}\left(\leb_m^{\ast}u\right)^2dx.
    \end{align}
    Furthermore, notice the two inequalities \eqref{m_large_ineq2} and \eqref{m_large_ineq4} show that for all $m>3$ and $\alpha>0$, there exists $C_{m,\alpha}<\infty$ such that
    \begin{align}
        \int_{\Omega}|x|^{2\alpha}\left(\Delta u+\frac{((m-2)^2-1)}{|x|^2}u\right)^2dx\leq C_{m,\alpha}\int_{\Omega}\left(\Delta u-2(m-1)\frac{x}{|x|^2}\cdot \D u+\frac{(m-1)^2}{|x|^2}u\right)^2dx.
    \end{align}
    Finally, an integration by parts and Cauchy-Schwarz inequality show that 
    \begin{align}\label{ipp_finale_m_alpha}
        &\int_{\Omega}|x|^{2\alpha}\frac{|\D u|^2}{|x|^2}dx=-\int_{\Omega}|x|^{2\alpha}\frac{u}{|x|^2}\Delta u\,dx+2(1-\alpha)\int_{\Omega}|x|^{2\alpha}\left(\frac{x}{|x|^2}\cdot \D u\right)\frac{u}{|x|^2}dx\nonumber\\
        &=-\int_{\Omega}|x|^{2\alpha}\frac{u}{|x|^2}\leb_m^{\ast}u\,dx-2(m-1)\int_{\Omega}|x|^{2\alpha}\left(\frac{x}{|x|^2}\cdot \D u\right)\frac{u}{|x|^2}dx+2(1-\alpha)\int_{\Omega}|x|^{2\alpha}\left(\frac{x}{|x|^2}\cdot \D u\right)\frac{u}{|x|^4}dx\nonumber\\
        &+(m-1)^2\int_{\Omega}|x|^{2\alpha}\frac{u^2}{|x|^4}dx\nonumber\\
        &=-\int_{\Omega}|x|^{2\alpha}\frac{u}{|x|^2}\leb_m^{\ast}u\,dx+\left((m-1)^2+(2m-(\alpha+4))(\alpha-1)\right)\int_{\Omega}|x|^{2\alpha}\frac{u^2}{|x|^4}dx\\
        &\leq \int_{\Omega}|x|^{2\alpha}\left(\leb_m^{\ast}u\right)^2dx+\left(\frac{1}{4}-(\alpha^2-(2m-3)\alpha-m^2+4m-5)\right)\int_{\Omega}|x|^{2\alpha}\frac{u^2}{|x|^4}dx\nonumber\\
        &\leq \left(1+\left(\frac{1}{4(m-1)(m-3+2\alpha)}+\frac{(\alpha-1)^2}{4m(m-1)^2(m-3)\alpha}\right)\left(\frac{1}{4}-(\alpha^2-(2m-3)\alpha-m^2+4m-5)\right)_+\right)\nonumber\\
        &\times \int_{\Omega}|x|^{2\alpha}\left(\leb_m^{\ast}u\right)^2dx.\label{grad_ineq}
    \end{align}
    Now, using \eqref{inversion_lm2}, we deduce by \eqref{m_large_ineq2} and \eqref{grad_ineq}
    that for all $m>3$ and for all $\alpha>0$, we have
    \begin{align*}
        \int_{\Omega}|x|^{4-2\alpha}\frac{|\D u|^2}{|x|^2}dx+\int_{\Omega}|x|^{4-2\alpha}\frac{u^2}{|x|^4}dx\leq C_{m,\alpha}\int_{\Omega}|x|^{4-2\alpha}\left(\Delta u+2(m-1)\frac{x}{|x|^2}\cdot \D u+\frac{(m-1)^2}{|x|^2}u\right)^2dx.
    \end{align*}
    Since the inequality holds true for all $\alpha>0$, we deduce in particular that for all $0<\alpha<2$, we have
    \begin{align}
        \int_{\Omega}\frac{|\D u|^2}{|x|^2}\left(\frac{|x|}{b}\right)^{2\alpha}dx+\int_{\Omega}\frac{u^2}{|x|^4}\left(\frac{|x|}{b}\right)^{2\alpha}dx\leq C_{m,2-\alpha}\int_{\Omega}\left(\leb_mu\right)^2\left(\frac{|x|}{b}\right)^{2\alpha}dx.
    \end{align}
    Now, the case $1<m\leq 3$ remains to be treated. In reality, we are mostly interested in the cases $m=2$ and $m=3$. Since $(m-2)^2-1\leq 0$ for $1\leq m\leq 3$, it is expected that the previous computation does not lead to the because additional terms mostly contribute negative to the coefficient in front of the weight $L^2$ norm of $u$. Starting from \eqref{new_m_alpha12}, we rewrite thanks to \eqref{ipp_m_alpha11}
    \begin{align}\label{extra_m_alpha1}
        &-4(m-1)\alpha\int_{\Omega}|x|^{2\alpha}u\,\Delta u\,dx=-4(m-1)\alpha\int_{\Omega}|x|^{2\alpha}u\,\leb_m^{\ast}u\,dx-8(m-1)^2\alpha\int_{\Omega}|x|^{2\alpha}\left(\frac{x}{|x|^2}\cdot u\right)dx\nonumber\\
        &+4(m-1)^3\alpha\int_{\Omega}|x|^{2\alpha}\frac{u^2}{|x|^4}dx\nonumber\\
        &=-4(m-1)\alpha\int_{\Omega}|x|^{2\alpha}u\,\leb_m^{\ast}u\,dx+\left(8(m-1)^2\alpha(\alpha-1)+4(m-1)^3\alpha\right)\int_{\Omega}|x|^{2\alpha}\frac{u^2}{|x|^4}dx\nonumber\\
        &\geq \left(8(m-1)^2\alpha(\alpha-1)+4(m-1)^3\alpha-\delta\right)\int_{\Omega}|x|^{2\alpha}\frac{u^2}{|x|^4}dx-\frac{4(m-1)\alpha^2}{\delta}\int_{\Omega}|x|^{2\alpha}\left(\leb_m^{\ast}u\right)^2dx
    \end{align}
    for all $\delta>0$. Therefore, we deduce by \eqref{new_m_alpha12}, \eqref{ipp_m_alpha16}, and \eqref{extra_m_alpha1} that for all $\delta>0$,
    \begin{align}\label{extra_m_alpha2}
        &\left(1+\frac{4(m-1)\alpha^2}{\delta}\right)\int_{\Omega}|x|^{2\alpha}\left(\leb_m^{\ast}u\right)^2dx
        \geq\int_{\Omega}\left(\Delta u+\frac{((m-2)^2-1)}{|x|^2}u\right)^2dx\nonumber\\
        &+\bigg\{\left(4((m-2)^2-1)+8(m-1)\alpha\right)(\alpha-1)^2+8(m-1)^2\alpha(\alpha-1)+4(m-1)^3\alpha-\delta\nonumber\\
        &+4(m-1)\left(-2\alpha^3+6\alpha^2+(m^2-2m-5)\alpha-m+3\right)\bigg\}\int_{\Omega}|x|^{2\alpha}\frac{u^2}{|x|^4}dx\nonumber\\
        &=\int_{\Omega}\left(\Delta u+\frac{((m-2)^2-1)}{|x|^2}u\right)^2dx+\left(4(m-1)^2\alpha(3\alpha+2(m-3))-\delta\right)\int_{\Omega}|x|^{2\alpha}\frac{u^2}{|x|^4}dx.
    \end{align}
    Therefore, for all $\alpha>2-\dfrac{2}{3}m$, taking
    \begin{align*}
        \delta=2(m-1)^2\alpha(3\alpha+2(m-3))
    \end{align*}
    we deduce that
    \begin{align}\label{extra_m_alpha3}
        \int_{\Omega}|x|^{2\alpha}\frac{u^2}{|x|^4}dx&\leq \left(\frac{1}{2(m-1)^2\alpha(3\alpha+2(m-3))}+\frac{4\alpha^2}{(m-1)\alpha(3\alpha+2(m-3))^2}\right)\int_{\Omega}|x|^{2\alpha}\left(\leb_m^{\ast}u\right)^2dx\nonumber\\
        &=C_{2,m,\alpha}\int_{\Omega}|x|^{2\alpha}\left(\Delta u-2(m-1)\frac{x}{|x|^2}\cdot \D u+\frac{(m-1)^2}{|x|^2}\right)^2dx
    \end{align}
    Then, we have
    \begin{align*}
        &4(m-1)^2\alpha(3\alpha+2(m-3))-\left(4((m-2)^2-1)+8(m-1)\alpha\right)\\
        &=-4(m-1)(2\alpha^3-2(m+2)\alpha^2+2(-m^2+3m+1)\alpha+m-3)
    \end{align*}
    an
    \begin{align*}
        4\left((m-2)^2-1\right)+8(m-1)\alpha=4(m-1)\left(2\alpha+m-3\right),
    \end{align*}
    which implies that
    \begin{align}\label{extra_m_alpha4}
        &4(m-1)\left(2\alpha+m-3\right)\int_{\Omega}|x|^{2\alpha}\left(\frac{x}{|x|^2}\cdot\D u\right)^2dx\leq \left(1+\frac{4(m-1)\alpha^2}{\delta}\right)\int_{\Omega}|x|^{2\alpha}\left(\leb_m^{\ast}u\right)^2dx\nonumber\\
        &+4(m-1)(2\alpha^3-2(m+2)\alpha^2+2(-m^2+3m+1)\alpha+m-3+\delta)\int_{\Omega}|x|^{2\alpha}\frac{u^2}{|x|^4}dx.
    \end{align}
    Taking $\delta=1$, we deduce that (taking $\delta=1$) that for all $\alpha>\dfrac{3-m}{2}$
    \begin{align}\label{extra_m_alpha5}
        &\int_{\Omega}|x|^{2\alpha}\left(\frac{x}{|x|^2}\cdot \D u\right)^2dx\leq \left(\frac{1+4(m-1)\alpha^2}{4(m-1)(2\alpha+m-3)}\right.\nonumber\\
        &\left.+\frac{(2\alpha^3-2(m+2)\alpha^2+2(-m^2+3m+1)\alpha+m-2)_+}{2\alpha+m-3}C_{2,m,\alpha}\right)\int_{\Omega}|x|^{2\alpha}\left(\leb_m^{\ast}u\right)^2dx.
    \end{align}
    Since $1<m\leq 3$, we have
    \begin{align*}
        2-\frac{2}{3}m>\frac{3-m}{2}.
    \end{align*}
    Therefore, for all $\alpha>2-\dfrac{2}{3}m$, \eqref{ipp_finale_m_alpha}, \eqref{extra_m_alpha3}, and \eqref{extra_m_alpha5} imply that
    \begin{align}\label{extra_m_alpha5bis}
        \int_{\Omega}|x|^{2\alpha}\frac{|\D u|^2}{|x|^2}dx\leq \left(1+\left(\frac{1}{4}-(\alpha^2-(2m-3)\alpha-m^2+4m-5)\right)_+C_{2,m,\alpha}\right)\int_{\Omega}|x|^{2\alpha}\left(\leb_m^{\ast}u\right)^2dx.
    \end{align}
    Therefore, using the inversion formula \eqref{inversion_lm2} and the inequalities  \eqref{extra_m_alpha3} and \eqref{extra_m_alpha5bis}, we deduce that for all $\alpha>2-\dfrac{2}{3}m$, we have
    \begin{align*}
        \int_{\Omega}|x|^{4-2\alpha}\frac{|\D u|^2}{|x|^2}dx+\int_{\Omega}|x|^{4-2\alpha}\frac{u^2}{|x|^4}dx\leq C_{m,\alpha}'\int_{\Omega}|x|^{4-2\alpha}\left(\Delta u+2(m-1)\frac{x}{|x|^2}\cdot\D u+\frac{(m-1)^2}{|x|^2}u\right)^2dx
    \end{align*}
    In particular, we deduce that for all $1<m\leq 3$ and for all $0<\alpha<\dfrac{4}{3}m$, we have
    \begin{align}
        \int_{\Omega}\frac{|\D u|^2}{|x|^2}\left(\frac{|x|}{b}\right)^{2\alpha}dx+\int_{\Omega}\frac{u^2}{|x|^4}\left(\frac{|x|}{b}\right)^{2\alpha}dx\leq C_{m,2-\alpha}\int_{\Omega}\left(\leb_mu\right)^2\left(\frac{|x|}{b}\right)^{2\alpha}dx,
    \end{align}
    which concludes the proof of the theorem.
\end{proof}

\section{First Minimisation Problem in Dimension 4}

    \subsection{Statement of the Main Theorem}

    Let $0<a<b<\infty$, let $\Omega=B_b\setminus\bar{B}_a(0)\subset \R^4$, and consider the following minimisation problem:
    \begin{align}\label{biharmonique_dim4}
        \lambda=\inf\ens{\int_{\Omega}(\Delta u)^2dx:\;\, u\in W^{2,2}_0(\Omega)\quad \text{and}\quad \int_{\Omega}\frac{u^2}{|x|^4}dx=1}.
    \end{align}
    Thanks to the Poincaré inequality, there exists a constant $C(\Omega)<\infty$ such that for all $u\in W^{2,2}_0(\Omega)$, we have
    \begin{align*}
        &\int_{\Omega}u^2dx\leq C(\Omega)\int_{\Omega}|\D u|^2dx\\
        &\int_{\Omega}|\D u|^2dx\leq C(\Omega)\int_{\Omega}|\D^2u|^2dx.
    \end{align*}
    Now, notice that
    \begin{align}\label{d2=Delta}
        \int_{\Omega}|\D^2u|^2dx&=\sum_{i,j=1}^4\int_{\Omega}(\p{x_i,x_j}^2u)^2dx=-\sum_{i,j=1}^4\int_{\Omega}\p{x_i}(\p{x_i,x_j}^2u)\p{x_j}u\,dx=-\sum_{i,j=1}^4\int_{\Omega}\p{x_j}(\p{x_i}^2u)\p{x_j}u\,dx\nonumber\\
        &=\sum_{i,j=1}^4\int_{\Omega}\p{x_i}^2u\,\p{x_j}^2u\,dx=\int_{\Omega}(\Delta u)^2dx.
    \end{align}
    Therefore, we finally deduce that for all $u\in W^{2,2}_0(\Omega)$, the holds
    \begin{align}\label{poincare_bilaplace}
        \int_{\Omega}\left(|\D^2u|^2+|\D u|^2+u^2\right)dx\leq C(\Omega)\int_{\Omega}(\Delta u)^2dx.
    \end{align}
    Therefore, taking a minimising sequence $\ens{u_k}_{k\in\N}\subset W^{2,2}_0(\Omega)$ such that
    \begin{align*}
        &\int_{\Omega}\frac{u_k^2}{|x|^4}dx=1\\
        &\int_{\Omega}\left(\Delta u\right)^2dx\conv{k\rightarrow \infty}\lambda,
    \end{align*}
    we deduce by \eqref{poincare_bilaplace} that $\ens{u_k}_{k\in\N}$ is bounded in $W^{2,2}_0(\Omega)$, which shows that there exists $u\in W^{2,2}_0(\Omega)$ such that up to a subsequence, $u_k\hookrightarrow u$ weakly as $k\rightarrow\infty$. Thanks to the continuous Sobolev embedding $W^{2,2}(\Omega)\hookrightarrow W^{1,4}(\Omega)$, and the compact embedding $W^{1,4}(\Omega)\hookrightarrow L^p(\Omega)$ for all $p<\infty$, we deduce in particular that
    \begin{align*}
        \lim_{k\rightarrow\infty}\wp{u_k-u}{1,2}{\Omega}=0.
    \end{align*}
    In particular, 
    \begin{align}\label{constraint1}
        1=\int_{\Omega}\frac{u_k^2}{|x|^4}dx\conv{k\rightarrow\infty}\int_{\Omega}\frac{u^2}{|x|^4}dx.
    \end{align}
    Finally, Fatou's Lemma shows that 
    \begin{align}\label{constraint2}
        \int_{\Omega}(\Delta u)^2dx\leq\liminf_{k\rightarrow\infty}\int_{\Omega}(\Delta u_k)^2dx=\lambda.
    \end{align}
    Therefore, $u\in W^{2,2}_0(\Omega)$ satisfies the constraint \eqref{constraint1}, and \eqref{constraint2} shows that $u$ is a minimiser for our problem \eqref{biharmonique_dim4}. Therefore, $u$ is a solution of the system of equations:
    \begin{align}\label{eigenvalue_bilaplacian}
        \left\{\begin{alignedat}{2}
            \Delta^2u&=\frac{\lambda}{|x|^4}u\qquad&&\text{in}\;\,\Omega\\
            u&=0\qquad&&\text{on}\;\,\partial\Omega\\
            \partial_{\nu}u&=0\qquad&&\text{on}\;\,\partial\Omega.
        \end{alignedat}\right.
    \end{align}
    Integrating by parts, by deduce that
    \begin{align*}
        \int_{\Omega}(\Delta u)^2dx=\int_{\Omega}u\,\Delta^2u\,dx=\int_{\Omega}\frac{u^2}{|x|^4}dx=\lambda,
    \end{align*}
    and \eqref{constraint} shows that $\lambda>0$.

    \begin{theorem}\label{neck_estimate_dim4}
        Let $u\in W^{2,2}_0(\Omega)$ a minimiser of \eqref{biharmonique_dim4}, and assume that
        \begin{align}\label{conformal_bound_dim4}
            \log\left(\frac{b}{a}\right)\geq \frac{15\sqrt{4+3\pi(\pi+1)}}{2}.
        \end{align}
        Then, we have
        \begin{align}\label{bound_eigenvalue_dim4}
            \lambda>\frac{4\pi^2}{\log^2\left(\frac{b}{a}\right)}.
        \end{align}
        Furthermore, the minimiser is a radial function $u(r)$, where $Y(r)=u(e^r)$ is a non-trivial solution of the ordinary differential equation
        \begin{align*}
            Y''''(t)-4\,Y''(t)=\lambda\,Y(t)
        \end{align*}
        such that $Y=Y'=0$ on $\partial[\log(a),\log(b)]$.
    \end{theorem}
    \begin{proof}
        The bound \eqref{bound_eigenvalue_dim4} holds thanks to the forthcoming Theorem \ref{theoreme_ode_dimension4} from next section. Furthermore, notice that for all $n\in\N$ the bound
        \begin{align*}
            \left(n^2+\frac{\pi^2}{\log^2\left(\frac{b}{a}\right)}\right)\left((n+2)^2+\frac{\pi^2}{\log^2\left(\frac{b}{a}\right)}\right)<\lambda_n<\left(n^2+\frac{4\pi^2}{\log^2\left(\frac{b}{a}\right)}\right)\left((n+2)^2+\frac{4\pi^2}{\log^2\left(\frac{b}{a}\right)}\right)
        \end{align*}
        shows that $\lambda_0<\lambda_n$ for all $n\geq 1$ provided that
        \begin{align*}
            \left(1+\frac{\pi^2}{\log^2\left(\frac{b}{a}\right)}\right)\left(9+\frac{\pi^2}{\log^2\left(\frac{b}{a}\right)}\right)>\frac{4\pi^2}{\log^2\left(\frac{b}{a}\right)}\left(4+\frac{4\pi^2}{\log\left(\frac{b}{a}\right)}\right).
        \end{align*}
        Making the change of variable $X=\dfrac{\pi^2}{\log^2\left(\frac{b}{a}\right)}$, we deduce that this inequality is equivalent to 
        \begin{align*}
            (1+X)(9+X)>4X(4+4X)=16X(X+1),
        \end{align*}
        or $-1<X<\dfrac{3}{5}$, which finally gives the bound
        \begin{align*}
            \log\left(\frac{b}{a}\right)>\pi\sqrt{\frac{5}{3}},
        \end{align*}
        which is satisfied by assumption. Therefore, we conclude as in the proof of Theorem  \ref{main_neck_lm} that the solution is radial and given as previously.
    \end{proof}
    \begin{rem}
        The argument implies that the lower bound
        \begin{align*}
            \frac{b}{a}\geq e^{\pi\sqrt{\frac{5}{3}}}=57.73\cdots
        \end{align*}
        may be optimal to have a radial solution. It also suggests that \eqref{conformal_bound_dim4} is far from optimal. It would be interesting to investigate whether our minimisation problem admits a non-radial minimiser for small enough conformal class.
    \end{rem}
    We deduce the following corollary that we state as theorem.
    \begin{theorem}\label{bound_biharmonic}
        Let $0<a<b<\infty$, let $\Omega=B_b\setminus\bar{B}_a(0)\subset\R^4$, and assume that
        \begin{align}\label{large_conf_class}
            \log\left(\frac{b}{a}\right)\geq \frac{15\sqrt{4+3\pi(\pi+1)}}{2}.
        \end{align}
        Then, for all $u\in W^{2,2}_0(\Omega)$, we have
        \begin{align}\label{eigenvalue_bilaplacian2}
            \int_{\Omega}(\Delta u)^2dx\geq \frac{4\pi^2}{\log^2\left(\frac{b}{a}\right)}\int_{\Omega}\frac{u^2}{|x|^4}dx.
        \end{align}
    \end{theorem}

    \subsection{Eigenvalue Estimates for the Associated Ordinary Differential Equation}

    Recall the the Laplacian on $\R^4$ admits the following expansion in polar coordinates
    \begin{align*}
        \Delta=\p{r}^2+\frac{3}{r}\p{r}+\frac{1}{r^2}\Delta_{S^{3}}.
    \end{align*}
    Furthermore, all  function $u\in L^2(S^3)$ admits an expansion
    \begin{align*}
        u(\theta)=\sum_{n=0}^{\infty}a_nY_n(\theta),
    \end{align*}
    for some $\ens{a_n}_{n\in\N}\subset \R$, 
    where $Y_n$ is the restriction of a homogeneous harmonic polynomial on $\R^4$ of degree $n$ (and the $\ens{Y_n}_{n\in\N}$ are orthogonal in $L^2(S^3)$ equipped with a fixed weight). Furthermore, $Y_n$ has eigenvalue $n(n+2)$ for $-\Delta_{S^3}$. Therefore, if $u\in C^{\infty}(\R^4)$, using in polar coordinates $(r,\theta)\in (0,\infty)\times S^3$, we have an expansion
    \begin{align*}
        u(r,\theta)=\sum_{n=0}^{\infty}a_n(r)Y_n(\theta),
    \end{align*}
    and
    \begin{align*}
        \Delta u=\sum_{n=0}^{\infty}\left(a_n''(r)+\frac{3}{r}a_n'(r)-\frac{n(n+2)}{r^2}a_n(r)\right)Y_n(\theta).
    \end{align*}
    Then, we have
    \begin{align*}
        &\p{r}\left(a_n''(r)+\frac{3}{r}a_n'(r)-\frac{n(n+2)}{r^2}a_n(r)\right)=a_n'''(r)+\frac{3}{r}a_n''(r)-\frac{3+n(n+2)}{r^2}a_n'(r)+\frac{2n(n+2)}{r^3}a_n(r)\\
        &\p{r}^2\left(a_n''(r)+\frac{3}{r}a_n'(r)-\frac{n(n+2)}{r^2}a_n(r)\right)=a_n''''(r)+\frac{3}{r}a_n'''(r)-\frac{6+n(n+2)}{r^2}a_n''(r)\\
        &+\frac{2(3+n(n+2))+2n(n+2)}{r^3}a_n'(r)-\frac{6n(n+2)}{r^4}a_n(r)\\
        &=a_n''''(r)+\frac{3}{r}a_n'''(r)-\frac{6+n(n+2)}{r^2}a_n''(r)+\frac{6+4n(n+2)}{r^3}a_n'(r)-\frac{6n(n+2)}{r^4}a_n(r).
    \end{align*}
    Therefore, we get
    \begin{align*}
        \Delta^2u&=\sum_{n=0}^{\infty}\bigg(a_n''''(r)+\frac{3}{r}a_n'''(r)-\frac{6+n(n+2)}{r^2}a_n''(r)+\frac{6+4n(n+2)}{r^3}a_n'(r)-\frac{6n(n+2)}{r^4}a_n(r)\\
        &+\frac{3}{r}\left(a_n'''(r)+\frac{3}{r}a_n''(r)-\frac{3+n(n+2)}{r^2}a_n'(r)+\frac{2n(n+2)}{r^3}a_n(r)\right)\\
        &-\frac{n(n+2)}{r^2}\left(a_n''(r)+\frac{3}{r}a_n'(r)-\frac{n(n+2)}{r^2}a_n(r)\right)\bigg)Y_n(\theta)\\
        &\sum_{n=0}^{\infty}\left(a_n''''(r)+\frac{6}{r}a_n'''(r)+\frac{3-2n(n+2)}{r^2}a_n''(r)-\frac{3+2n(n+2)}{r^3}a_n'(r)+\frac{n^2(n+2)^2}{r^4}a_n(r)\right)Y_n(\theta).
    \end{align*}
    In other words, we have
    \begin{align*}
        \Pi_{n(n+2)}(\Delta^2)=\p{r}^4+\frac{6}{r}\p{r}^3+\frac{3-2n(n+2)}{r^2}\p{r}^2-\frac{3+2n(n+2)}{r^3}\p{r}+\frac{n^2(n+2)^2}{r^4}.
    \end{align*}
    Now, consider the following eigenvalue problem on some interval $[a,b]$, where $0<a<b<\infty$
    \begin{align}\label{bilap_eigenvalue}
        \left\{\begin{alignedat}{1}
        &f''''+\frac{6}{r}f'''+\frac{3-2n(n+2)}{r^2}f''-\frac{3+2n(n+2)}{r^3}f'+\frac{n^2(n+2)^2}{r^4}f=\frac{\lambda}{r^4}f\\
        &f(a)=f(b)=f'(a)=f'(b)=0,
        \end{alignedat}\right.
    \end{align}
    where $\lambda\geq 0$.
    Making the change of variable $f(r)=Y(\log(r))$, we get
    \begin{align*}
        &\p{r}f(r)=\frac{1}{r}Y'(\log(r))\\
        &\p{r}^2f(r)=\frac{1}{r^2}Y''(\log(r))-\frac{1}{r^2}Y'(\log(r))\\
        &\p{r}^3f(r)=\frac{1}{r^3}Y'''(\log(r))-\frac{3}{r^3}Y''(\log(r))+\frac{2}{r^3}Y'(\log(r))\\
        &\p{r}^4f(r)=\frac{1}{r^4}Y''''(\log(r))-\frac{6}{r^4}Y'''(\log(r))+\frac{11}{r^4}Y''(\log(r))-\frac{6}{r^4}Y'(\log(r)).
    \end{align*}
    Therefore, we deduce that
    \begin{align*}
        &r^4\left(f''''+\frac{6}{r}f'''+\frac{3-2n(n+2)}{r^2}f''-\frac{3+2n(n+2)}{r^3}f'+\frac{n^2(n+2)^2}{r^4}f\right)\\
        &=Y''''\colorcancel{-6\,Y'''}{red}+11\,Y''-6\,Y'+6\left(\colorcancel{Y'''}{red}-3\,Y''+2\,Y'\right)+(3-2n(n+2))\left(Y''-Y'\right)\\
        &-(3+2n(n+2))Y'+n^2(n+2)^2Y\\
        &=Y''''-(4+2n(n+2))Y''+n^2(n+2)^2Y,
    \end{align*}
    and \eqref{bilap_eigenvalue} is equivalent to 
    \begin{align}\label{bilap_eigenvalue2}
    \left\{\begin{alignedat}{2}
        &Y''''-(4+2n(n+2))Y''+n^2(n+2)^2Y=\lambda\, Y\qquad&& \text{in}\;\, (\log(a),\log(b))\\
        &Y(\log(a))=Y(\log(b))=Y'(\log(a))=Y'(\log(b)).
        \end{alignedat}\right.
    \end{align}
    The characteristic polynomial of the equation is 
    \begin{align*}
        P(X)=X^4-(4+2n(n+2))X^2+n^2(n+2)^2-\lambda,
    \end{align*}
    which is a biquadratic polynomial! The discriminant of $P(\sqrt{X})$ is given by
    \begin{align*}
        4(2+n(n+2))^2-4n^2(n+2)^2+4\lambda=16n(n+2)+16+4\lambda=4(4(n+1)^2+\lambda),
    \end{align*}
    which shows that the roots of $P$ are given by 
    \begin{align}\label{root4}
    \left\{\begin{alignedat}{2}
        r_1&=&&\sqrt{(n+1)^2+1+\sqrt{\lambda+4(n+1)^2}}\\
        r_2&=-&&\sqrt{(n+1)^2+1+\sqrt{\lambda+4(n+1)^2}}\\
        r_3&=&&\sqrt{(n+1)^2+1-\sqrt{\lambda+4(n+1)^2}}\\
        r_4&=-&&\sqrt{(n+1)^2+1-\sqrt{\lambda+4(n+1)^2}}.
        \end{alignedat}\right.
    \end{align}
    Since $\lambda\geq 0$ (by assumption), we have
    \begin{align*}
        (n+1)^2+1-\sqrt{\lambda+4(n+1)^2}\leq 0
    \end{align*}
    if and only if
    \begin{align*}
        (n+1)^4+2(n+1)^2+1\leq \lambda+4(n+1)^2,
    \end{align*}
    or
    \begin{align*}
        \lambda\geq ((n+1)^2-1)^2.
    \end{align*}
    \begin{theorem}\label{theoreme_ode_dimension4}
        Let $n\in\N$. Let $0<a<b<\infty$ be such that
        \begin{align}\label{log_bound1_0}
            \log\left(\frac{b}{a}\right)>\max\ens{\frac{15\sqrt{4+3\pi(\pi+1)}}{\sqrt{2(n+1)^2+2}},\frac{\log(1+\sqrt{2(n+1)^2+2})}{\sqrt{2(n+1)^2+2}}}
        \end{align}        
        For all $\lambda>0$, consider the following linear differential equation on $I=(\log(a),\log(b))$:
        \begin{align}\label{bilap_eigenvalue2bis}
    \left\{\begin{alignedat}{2}
        Y''''(t)-(4+2n(n+2))Y''(t)+n^2(n+2)^2Y(t)&=\lambda_nY(t)\qquad&& \text{in}\;\, I\\
        Y=Y'&=0\qquad&&\text{on}\;\,\partial I.
        \end{alignedat}\right.
    \end{align}
        There exists $\lambda>0$ such that the system \eqref{bilap_eigenvalue2} admits a non-trivial solution $Y$. Furthermore, the minimal value $\lambda_{n}>0$ satisfies the following estimate 
        \begin{align}\label{estimate_log_lambda_4}
            \lambda_n>\left(n^2+\frac{\pi^2}{\log^2\left(\frac{b}{a}\right)}\right)\left((n+2)^2+\frac{\pi^2}{\log^2\left(\frac{b}{a}\right)}\right).
        \end{align}
        In particular, we have
        \begin{align}
            \inf_{n\in \N}\lambda_{n}>\frac{4\pi^2}{\log^2\left(\frac{b}{a}\right)}.
        \end{align}
    \end{theorem}
    \begin{rem}
        To obtain a more symmetric result as in Theorem \ref{theoreme_ode_m_n}, we can rewrite the estimate as
        \begin{align*}
            \lambda_n>\left(((n+1)-1)^2+\frac{\pi^2}{\log^2\left(\frac{b}{a}\right)}\right)\left(((n+1)+1)^2+\frac{\pi^2}{\log^2\left(\frac{b}{a}\right)}\right).
        \end{align*}
        Furthermore, the proof shows that the bound
        \begin{align}\label{precise}
            \lambda_n<\left(n^2+\frac{4\pi^2}{\log^2\left(\frac{b}{a}\right)}\right)\left((n+2)^2+\frac{4\pi^2}{\log^2\left(\frac{b}{a}\right)}\right)
        \end{align}
        holds for all $n\in\N$.
    \end{rem}
    \begin{proof}
        Let $Y$ be a non-trivial solution of the system. We need to distinguish three cases. 

        \textbf{Case 1:} $\lambda<((n+1)^2-1)^2$.

        Notice that by assumption this case cannot happen for $n=0$, but our proof also works for $n=0$.
        
        Then, the characteristic polynomial $P(X)=X^4-(4+2n(n+2))X^2+n^2(n+2)^2-\lambda$ of the ordinary differential equation has four distinct real roots, and 
        \begin{align}
            Y(t)=\mu_1\,e^{r_1\,t}+\mu_2\,e^{r_2\,t}+\mu_3\,e^{r_3\,t}+\mu_4\,e^{r_4\,t},
        \end{align}
        where $r_1,r_2,r_3,r_4$ are given in \eqref{root4}. Also write
        \begin{align}
            \left\{\begin{alignedat}{2}
                r_1&=&&\lambda_1\\
                r_2&=-&&\lambda_1\\
                r_3&=&&\lambda_2\\
                r_4&=-&&\lambda_2
            \end{alignedat}\right.,
        \end{align}
        where $\lambda_1=\sqrt{((n+1)^2+1+\sqrt{\lambda+4(n+1)^2}}$ and $\lambda_2=\sqrt{(n+1)^2+1-\sqrt{\lambda+4(n+1)^2}}$. The boundary conditions are equivalent to
         \begin{align*}
        A\,\mu=\begin{pmatrix}
            a^{r_1} & a^{r_2} & a^{r_3} & a^{r_4}\\
            b^{r_1} & b^{r_2} & b^{r_3} & b^{r_4}\\
            r_1\,a^{r_1} & r_2\,a^{r_2} & r_3\,a^{r_3} & r_4\,a^{r_4}\\
            r_1\,b^{r_2} & r_2\,b^{r_2} & r_3\,b^{r_3} & r_4\,b^{r_4}
        \end{pmatrix}\begin{pmatrix}
            \mu_1\\
            \mu_2\\
            \mu_3\\
            \mu_4
        \end{pmatrix}=0. 
    \end{align*}
    We recover the same algebraic expression as in \textbf{Step 1} of Theorem \ref{theoreme_ode_m_n}. Therefore, we have 
    \begin{align*}
        \det(A)&=(\lambda_1+\lambda_2)^2\left(\left(\frac{a}{b}\right)^{\lambda_1-\lambda_2}+\left(\frac{b}{a}\right)^{\lambda_1-\lambda_2}\right)-\left(\lambda_1-\lambda_2\right)^2\left(\left(\frac{a}{b}\right)^{\lambda_1+\lambda_2}+\left(\frac{b}{a}\right)^{\lambda_1+\lambda_2}\right)-8\lambda_1\lambda_2,
    \end{align*}
    and since $\lambda_1\lambda_2\neq 0$, and $\lambda_1-\lambda_2\neq 0$, the rest of the proof carries on and shows that $Y=0$.

    \textbf{Case 2:} $\lambda=((n+1)^2-1)^2$.

    Then, the roots are $r_3=r_4=0$, and $r_1=\sqrt{2(n+1)^2+2}$, $r_2=-\sqrt{2(n+1)^2+2}$. Therefore, $Y$ is given by 
    \begin{align*}
        Y(t)=\mu_1\, e^{\sqrt{2(n+1)^2+2}}+\mu_2e^{-\sqrt{2(n+1)^2+2}}+\mu_3+\mu_4\,t,
    \end{align*}
    and the boundary conditions are equivalent to 
    \begin{align}\label{boundary4}
        \begin{pmatrix}
            a^{\sqrt{2(n+1)^2+2}} & a^{-\sqrt{2(n+1)^2+2}} & 1 & \log(a)\\
            b^{\sqrt{2(n+1)^2+2}} & b^{-\sqrt{2(n+1)^2+2}} & 1 & \log(b)\\
            \sqrt{2(n+1)^2+2}\,a^{\sqrt{2(n+1)^2+2}} & -\sqrt{2(n+1)^2+2}\,a^{-\sqrt{2(n+1)^2+2}} & 0 & 1\\
            \sqrt{2(n+1)^2+2}\,b^{\sqrt{2(n+1)^2+2}} & -\sqrt{2(n+1)^2+2}\,b^{-\sqrt{2(n+1)^2+2}} & 0 & 1
        \end{pmatrix}\begin{pmatrix}
            \mu_1\\
            \mu_2\\
            \mu_3\\
            \mu_4
        \end{pmatrix}=0.
    \end{align}
    We recognise the determinant of \textbf{Step 2} of Theorem \ref{theoreme_ode_m_n} given in \eqref{determinant_equality_case1}.   
    Therefore, the proof of Theorem \eqref{theoreme_ode_m_n} applies and shows that $\det(A)\neq 0$, which finally implies that $Y=0$.

    \textbf{Case 3:} $\lambda>((n+1)^2-1)^2$.

    Then, we have by \eqref{root4}
    \begin{align}\label{root4_bis}
    \left\{\begin{alignedat}{2}
        r_1&=&&\sqrt{\sqrt{\lambda+4(n+1)^2}+(n+1)^2+1}\\
        r_2&=-&&\sqrt{\sqrt{\lambda+4(n+1)^2}+(n+1)^2+1}\\
        r_3&=&&i\sqrt{\sqrt{\lambda+4(n+1)^2}-((n+1)^2+1)}\\
        r_4&=-&&i\sqrt{\sqrt{\lambda+4(n+1)^2}-((n+1)^2+1)}.
        \end{alignedat}\right.
    \end{align}
    Provided that $\lambda_1=\sqrt{\sqrt{\lambda+4(n+1)^2}+(n+1)^2+1}$ and $\lambda_2=\sqrt{\sqrt{\lambda+4(n+1)^2}-((n+1)^2+1)}$, we deduce that there exists $\mu_1,\mu_2\in \R$ and $\mu_3,\mu_4\in \C$ such that
    \begin{align}\label{sol4}
        Y(t)=\mu_1e^{\lambda_1\,t}+\mu_2e^{-\lambda_1\,t}+\mu_3e^{i\,\lambda_2\,t}+\mu_4e^{-i\,\lambda_2)t}
    \end{align}
    we get thanks to \eqref{boundary_conditions_m_n} the system
    \begin{align}\label{complex_system4}
        \begin{pmatrix}
            a^{\lambda_1} & a^{-\lambda_1} & a^{i\,\lambda_2} & a^{-i\,\lambda_2}\\
            b^{\lambda_1} & b^{-\lambda_1} & b^{i\,\lambda_2} & b^{-i\,\lambda_2}\\
            \lambda_1\,a^{\lambda_1} & -\lambda_1\,a^{-\lambda_1} & i\,\lambda_2\,a^{i\,\lambda_2} & -i\,\lambda_2\,a^{-i\,\lambda_2}\\
            \lambda_1\,b^{\lambda_1} & -\lambda_1\,b^{-\lambda_1} & i\,\lambda_2\,b^{i\,\lambda_2} & -i\,\lambda_2\,b^{-i\,\lambda_2}
        \end{pmatrix}\begin{pmatrix}
            \mu_1\\
            \mu_2\\
            \mu_3\\
            \mu_4
        \end{pmatrix}=0.
    \end{align}
    We recognise the determinant of \textbf{Step 3} of Theorem \ref{theoreme_ode_m_n}, and following the exact same steps, we deduce by \eqref{fund_m} that
    \small
    \begin{align*}
        \left(\left(1+\left(\frac{b}{a}\right)^{2\lambda_1}\right)\cos\left(\lambda_2\log\left(\frac{b}{a}\right)\right)-2\left(\frac{b}{a}\right)^{\lambda_1}\right)\lambda_1\lambda_2=((n+1)^2+1)\left(\left(\frac{b}{a}\right)^{2\lambda_1}-1\right)\sin\left(\lambda_2\log\left(\frac{b}{a}\right)\right).
    \end{align*}
    \normalsize
    Since $\lambda_1=\sqrt{2((n+1)^2+1)+\lambda_2^2}$, we deduce that our equation is equivalent to \eqref{fund_m} (with $n$ replaced by $k$ so that we do not use twice the same notations), provide that $m\geq 1$ and $k\geq 0$ are chosen such that
    \begin{align*}
        m^2+k^2=(n+1)^2+1
    \end{align*}
    We can therefore choose $m=1$, and $k=n+1$. Therefore, the proof of Theorem \ref{theoreme_ode_m_n} shows that provided that 
    \small
    \begin{align*}
        \log\left(\frac{b}{a}\right)\geq R_{1,n+1}=\max\ens{\frac{15\sqrt{4+3\pi(\pi+1)}}{\sqrt{2(n+1)^2+2}},\frac{\log(1+\sqrt{2(n+1)^2+2})}{\sqrt{2(n+1)^2+2}},\frac{\log(1+\log(1+\sqrt{2(n+1)^2+2}))}{\sqrt{2(n+1)^2+2}}},
    \end{align*}
    \normalsize
    (which is a decreasing function of $n$), we have $\lambda_2\log\left(\frac{b}{a}\right)>\pi$, or
    \begin{align*}
        \sqrt{\sqrt{\lambda+4(n+1)^2}-((n+1)^2+1)}>\frac{\pi}{\log\left(\frac{b}{a}\right)}.
    \end{align*}
    Therefore, we get
    \begin{align*}
        \sqrt{\lambda+4(n+1)^2}-((n+1)^2+1)>\frac{\pi^2}{\log^2\left(\frac{b}{a}\right)},
    \end{align*}
    and finally, 
    \begin{align*}
        \lambda >\left((n+1)^2+1+\frac{\pi^2}{\log^2\left(\frac{b}{a}\right)}\right)^2-4(n+1)^2=\left(n^2+\frac{\pi^2}{\log^2\left(\frac{b}{a}\right)}\right)\left((n+2)^2+\frac{\pi^2}{\log^2\left(\frac{b}{a}\right)}\right),
    \end{align*}
    which concludes the proof of the theorem.
    \end{proof}

    \subsection{Second Estimate in Neck Regions}

    Before stating the main result, let us generalise \ref{lemme_ipp} to all dimension.
    \begin{lemme}\label{lemme_ipp_general}
        Let $2\leq n<\infty$, let $\Omega\subset \R^n$ be an open set, and consider $u\in W^{1,n}_0(\Omega)$. Then, for all $\beta>0$, we have
        \begin{align}\label{identity_ipp_general1}
            \beta\int_{\Omega}|x|^{\beta}\frac{u^2}{|x|^{n}}dx=-2\int_{\Omega}|x|^{\beta}u\frac{x}{|x|^{n}}\cdot \D u\,dx,
        \end{align}
        and
        \begin{align}\label{identity_ipp_general2}
            \int_{\Omega}|x|^{\beta}\frac{u^2}{|x|^n}dx\leq \frac{4}{\beta^2}\int_{\Omega}|x|^{\beta}\left(\frac{x}{|x|^{\frac{n}{2}}}\cdot \D u\right)^2dx
        \end{align}
    \end{lemme}
    \begin{proof}
        Since we have already proved this statement for $n=2$ (notice that the proof works for all domain), we can assume that $n\geq 3$. Thanks to the Sobolev embedding $W^{1,n}(\R^n)\hookrightarrow L^p(\R^4)$ for all $p<\infty$, we deduce that all integrals are finite. Since the Green's function for the Laplacian in $\R^n$ is given (up to a constant) by $x\mapsto |x|^{2-n}$, 
        we deduce that
        \begin{align*}
            \dive\left(\frac{x}{|x|^n}\right)=0\qquad\text{in}\;\, \mathscr{D}'(\R^n\setminus\ens{0}).
        \end{align*}
        Therefore, since $\beta>0$, we have $|x|^{\beta}\Delta|x|^{2-n}=0$ in $\mathscr{D}'(\R^n)$, and
        \begin{align*}
            \int_{\Omega}|x|^{\beta}\frac{u}{|x|^{\frac{n}{2}}}\frac{x}{|x|^{\frac{n}{2}}}\cdot \D u\,dx=\frac{1}{2}\int_{\Omega}|x|^{\beta}\frac{x}{|x|^n}\cdot \D (u^2)\,dx=\frac{1}{2}\int_{\Omega}|x|^{\beta}\dive\left(\frac{x}{|x|^n}\,u^2\right)dx=-\frac{\beta}{2}\int_{\Omega}|x|^{\beta}\frac{u^2}{|x|^n}dx,
        \end{align*}
        and the inequality follows from Cauchy-Schwarz inequality.
    \end{proof}

    \begin{theorem}
        Let $0<a<b<\infty$, let $\Omega=B_b\setminus\bar{B}_a(0)\subset\R^4$, and let $0<\beta<\infty$. There exists $C_{\beta}<\infty$ such that for all $u\in W^{2,2}_0(\Omega)$, the following estimate hold:
        \begin{align}\label{bilap_second_est1}
            \int_{\Omega}\frac{u^2}{|x|^4}\left(\left(\frac{|x|}{b}\right)^{\beta}+\left(\frac{a}{|x|}\right)^{\beta}\right)dx\leq C_{\beta}\int_{\Omega}(\Delta u)^2dx
        \end{align}
        and
        \begin{align}\label{bilap_second_est2}
            \int_{\Omega}\frac{|\D u|^2}{|x|^2}\left(\left(\frac{|x|}{b}\right)^{\beta}+\left(\frac{a}{|x|}\right)^{\beta}\right)dx\leq C_{\beta}\int_{\Omega}(\Delta u)^2dx.
        \end{align}
    \end{theorem}    
    \begin{proof}
        \textbf{Part 1: The Scalar Change of Variable.}
        Let $u\in W^{2,2}_0(B(0,1))$. Then, thanks to the Sobolev embeddings $W^{2,2}(\R^4)\hookrightarrow W^{1,4}_0(\R^4)$, and $W^{1,4}(\R^4)\hookrightarrow L^p(\R^4)$ for all $p<\infty$, we deduce that for all $\beta>0$, 
        \begin{align*}
            \int_{B(0,1)}|x|^{\beta}\frac{u^2}{|x|^4}dx+\int_{B(0,1)}|x|^{\beta}\frac{|\D u|^2}{|x|^2}<\infty.
        \end{align*}
        Assume that $0<\beta<2$, since the inequality is trivial for $\beta \geq 2$ thanks to \eqref{poincare_bilaplace}. Integrating by parts, we deduce by Lemma \ref{lemme_ipp_general} that
        \begin{align*}
            \int_{B(0,1)}|x|^{\beta-2}|\D u|^2dx&=-\int_{B(0,1)}|x|^{\beta-2}u\,\Delta u\,dx+(2-\beta)\int_{B(0,1)}|x|^{\beta}\frac{u}{|x|^2}\frac{x}{|x|^2}\cdot \D u\,dx\\
            &=-\int_{B(0,1)}|x|^{\beta-2}u\,\Delta u\,dx-\frac{2(2-\beta)}{\beta}\int_{B(0,1)}|x|^{\beta}\frac{u^2}{|x|^4}dx.
        \end{align*}
        Using Cauchy's inequality, we deduce that
        \begin{align*}
            -\int_{B(0,1)}|x|^{\beta-2}u\,\Delta u\,dx\leq \frac{2-\beta}{\beta}\int_{B(0,1)}|x|^{\beta}\frac{u^2}{|x|^4}dx+\frac{4\beta}{2-\beta}\int_{B(0,1)}|x|^{\beta}(\Delta u)^2dx,
        \end{align*}
        and finally
        \begin{align*}
            \frac{2-\beta}{\beta}\int_{B(0,1)}|x|^{\beta}\frac{u^2}{|x|^4}dx+\int_{B(0,1)}|x|^{\beta}\frac{|\D u|^2}{|x|^2}dx\leq \frac{4\beta}{2-\beta}\int_{B(0,1)}|x|^{\beta}(\Delta u)^2dx.
        \end{align*}
        Less precisely, for all $\beta>0$, since $\D u\in L^4(B(0,1))$ by Sobolev embedding and $u\in W^{2,2}_0(B(0,1))$ we have
        \begin{align}\label{gen_est_du}
            \int_{B(0,1)}|x|^{\beta-2}|\D u|^2dx&\leq\left(\int_{B(0,1)}\frac{dx}{|x|^{4-2\beta}}\right)^{\frac{1}{2}}\left(\int_{B(0,1)}|\D u|^4dx\right)^{\frac{1}{2}}\nonumber\\
            &=\frac{\pi}{\sqrt{\beta}}\left(\int_{B(0,1)}|\D u|^4dx\right)^{\frac{1}{2}}
            \leq C\left(\int_{B(0,1)}(\Delta u)^2dx\right)^{\frac{1}{2}}dx,
        \end{align}
        where we used the Poincaré inequality and \eqref{d2=Delta}
        \begin{align*}
            \int_{B(0,1)}|\D u|^4dx\leq C\int_{B(0,1)}|\D^2u|^2dx=C\int_{B(0,1)}(\Delta u)^2dx.
        \end{align*}
        Therefore, Lemma \ref{lemme_ipp_general} shows that for all $\beta>0$, there holds:
        \begin{align}\label{gen_est_u}
            \int_{B(0,1)}|x|^{\beta}\frac{u^2}{|x|^4}dx\leq C\int_{B(0,1)}(\Delta u)^2dx.
        \end{align}
        Now, if $u\in W^{2,2}_0(\Omega)$, making a change of variable $x=b\,y$, we get
        \begin{align*}
            \int_{\Omega}\frac{u^2}{|x|^4}\left(\frac{|x|}{b}\right)^{\beta}dx=\int_{B_1\setminus\bar{B}_{b^{-1}a}(0)}|y|^{\beta}\frac{|u(by)|^2}{|y|^4}dy=\int_{B(0,1)}|y|^{\beta}\frac{|\bar{u}|^2}{|y|^4}dy,
        \end{align*}
        where $\bar{u}\in W^{2,2}_0(B(0,1))$ is the extension by $0$ of $u(b\,\cdot\,):B_1\setminus\bar{B}_{b^{-1}a}(0)$. Using \eqref{gen_est_u}, we deduce that
        \begin{align*}
            \int_{B(0,1)}|y|^{\beta}\frac{|\bar{u}|^2}{|y|^4}dy\leq C\int_{B(0,1)}(\Delta \bar{u})^2dy=C\int_{B(0,1)}(\Delta u(by))^2b^4\,dy=C\int_{\Omega}(\Delta u)^2dx,
        \end{align*}
        and making the same change of variable for $\D u$, the result follows for the first component of the integral \eqref{bilap_second_est2}.

        \textbf{Part 2: The Inversion.} Making an inversion $x=a\frac{y}{|y|^2}$, we get
        \begin{align*}
            \int_{\Omega}\frac{u^2}{|x|^4}\left(\frac{a}{|x|}\right)^{\beta}dx=\int_{B(0,1)}|\bar{u}|^2\frac{|y|^4}{a^4}|y|^{\beta}\frac{a^2dy}{|y|^8}=\int_{B(0,1)}|y|^{\beta}\frac{|\bar{u}|^2}{|y|^4}dy,
        \end{align*}
        and by conformal invariance of the Bilaplacian in dimension $4$, the two estimates follow from the computations of \textbf{Part 2}.
    \end{proof}
    \begin{rem}
        For $0<\beta<2$, we have the more precise estimates
        \begin{align*}
            &\int_{\Omega}\frac{u^2}{|x|^4}\left(\frac{|x|}{b}\right)^{\beta}dx\leq \frac{4\beta^2}{(2-\beta)^2}\int_{\Omega}(\Delta u)^2\left(\frac{|x|}{b}\right)^{\beta}dx\\
            &\int_{\Omega}\frac{u^2}{|x|^4}\left(\frac{a}{|x|}\right)^{\beta}dx\leq \frac{4\beta^2}{(2-\beta)^2}\int_{\Omega}(\Delta u)^2\left(\frac{a}{|x|}\right)^{\beta}dx,
        \end{align*}
        and
        \begin{align*}
            &\int_{B(0,1)}\frac{|\D u|^2}{|x|^2}\left(\frac{|x|}{b}\right)^{\beta}dx\leq \frac{4\beta}{2-\beta}\int_{\Omega}(\Delta u)^2\left(\frac{|x|}{b}\right)^{\beta}dx\\
            &\int_{B(0,1)}\frac{|\D u|^2}{|x|^2}\left(\frac{a}{|x|}\right)^{\beta}dx\leq \frac{4\beta}{2-\beta}\int_{\Omega}(\Delta u)^2\left(\frac{a}{|x|}\right)^{\beta}dx.
        \end{align*}
    \end{rem}
    Finally, we deduce the following result.
    \begin{theorem}\label{weight_final_bilap}
        Let $0<a<b<\infty$, let $\Omega=B_b\setminus\bar{B}_a(0)\subset \R^4$, and assume that
        \begin{align*}
            \log\left(\frac{b}{a}\right)\geq \frac{15\sqrt{4+3\pi(\pi+1)}}{2}.
        \end{align*}
        Then, for all $\beta>0$, there exists a constant $C_{\beta}>0$ such that for all $u\in W^{2,2}_0(\Omega)$, we have
        \begin{align}
            \int_{\Omega}(\Delta u)^2dx\geq C_{\beta}\int_{\Omega}\left(\frac{u^2}{|x|^4}+\frac{|\D u|^2}{|x|^2}\right)\left(\left(\frac{|x|}{b}\right)^{\beta}+\left(\frac{a}{|x|}\right)^{\beta}+\frac{1}{\log^2\left(\frac{b}{a}\right)}\right)dx.
        \end{align}
    \end{theorem}

    \section{First Minimisation Problem in All Dimension}

    Recall that the Laplacian on $\R^m$ admits the following expansion in polar coordinates
    \begin{align}
        \Delta=\p{r}^2+\frac{m-1}{r}\p{r}+\frac{1}{r^2}\Delta_{S^{m-1}},
    \end{align}
    where $\Delta_{S^{m-1}}$ is the Laplacian on $S^{m-1}$. Thanks to classical results, all function $u\in L^{2}(S^{m-1})$ admits an expansion 
    \begin{align*}
        u(\theta)=\sum_{n=0}^{\infty}a_nY_n(\theta)
    \end{align*}
    for some sequence $\ens{a_n}_{n\in\N}\subset \R$, where $Y_n$ is the restriction to $S^{m-1}$ of a homogeneous harmonic polynomial of $\R^m$ (and the $\ens{Y_n}_{n\in\N}$ are orthogonal in $L^2(S^{m-1})$. Furthermore, $Y_n$ has eigenvalue $n(n+m-2)$. As we have already treated the case $m=2$, we can assume that $m\geq 3$. Now, if $u\in C^{\infty}(\R^m)$, using polar coordinates $(r,\theta)\in (0,\infty)\times S^{m-1}$, we obtain an expansion
    \begin{align*}
        u(r,\theta)=\sum_{n=0}^{\infty}a_n(r)Y_n(\theta),
    \end{align*}
    and
    \begin{align*}
        \Delta u=\sum_{n=0}^{\infty}\left(a_n''(r)+\frac{m-1}{r}a_n'(r)-\frac{n(n+m-2)}{r^2}a_n(r)\right)Y_n(\theta).
    \end{align*}
    Then, we compute
    \begin{align*}
        &\p{r}\left(a_n''(r)+\frac{m-1}{r}a_n'(r)-\frac{n(n+m-2)}{r^2}a_n(r)\right)=a_n'''(r)+\frac{m-1}{r}a_n''(r)-\frac{m-1+n(n+m-2)}{r^2}a_n'(r)\\
        &+\frac{2n(n+m-2)}{r^3}a_n(r),\\
        &\p{r}^2\left(a_n''(r)+\frac{m-1}{r}a_n'(r)-\frac{n(n+m-2)}{r^2}a_n(r)\right)=a_n''''(r)+\frac{m-1}{r}a_n'''(r)\\
        &-\frac{2(m-1)+n(n+m-2)}{r^2}a_n''(r)
        +\frac{2(m-1)+4n(n+m-2)}{r^3}a_n'(r)-\frac{6n(n+m-2)}{r^4}a_n(r).
    \end{align*}
    Therefore, we get
    \begin{align}\label{bilaplacian_all_dimension}
        &\Delta^2u=\sum_{n=0}^{\infty}\bigg(a_n''''(r)+\frac{m-1}{r}a_n'''(r)-\frac{2(m-1)+n(n+m-2)}{r^2}a_n''(r)
        +\frac{2(m-1)+4n(n+m-2)}{r^3}a_n'(r)\nonumber\\
        &-\frac{6n(n+m-2)}{r^4}a_n(r)\nonumber\\
        &+\frac{m-1}{r}\left(a_n'''(r)+\frac{m-1}{r}a_n''(r)-\frac{m-1+n(n+m-2)}{r^2}a_n'(r)
        +\frac{2n(n+m-2)}{r^3}a_n(r)\right)\nonumber\\
        &-\frac{n(n+m-2)}{r^2}\left(a_n''(r)+\frac{m-1}{r}a_n'(r)-\frac{n(n+m-2)}{r^2}a_n(r)\right)\bigg)Y_n(\theta)\nonumber\\
        &=\sum_{n=0}^{\infty}\left(a_n''''(r)+\frac{2(m-1)}{r}a_n'''(r)+\frac{(m-1)(m-3)-2n(n+m-2)}{r^2}a_n''(r)\right.\nonumber\\
        &\left.-\frac{(m-1)(m-3)+2(m-3)n(n+m-2)}{r^3}a_n'(r)+\frac{n^2(n+m-2)^2+2(m-4)n(n+m-2)}{r^4}\right)Y_n(\theta).
    \end{align}
    In other words, we have
    \begin{align}\label{projection_bilaplacian}
        &\Pi_{n(n+m-2)}(\Delta^2)=\p{r}^4+\frac{2(m-1)}{r}\p{r}^3+\frac{(m-1)(m-3)-2n(n+m-2)}{r^2}\p{r}^2\nonumber\\
        &-\frac{(m-1)(m-3)+2(m-3)n(n+m-2)}{r^3}\p{r}+\frac{n^2(n+m-2)^2+2(m-4)n(n+m-2)}{r^4}.
    \end{align}
    Now, consider the following eigenvalue problem on some interval $[a,b]$, where $0<a<b<\infty$
    \begin{align}\label{bilap_eigenvalue_gen}
        \left\{\begin{alignedat}{1}
        &f''''+\frac{2(m-1)}{r}f'''+\frac{(m-1)(m-3)-2n(n+m-2)}{r^2}f''\\
        &-\frac{(m-1)(m-3)+2(m-3)n(n+m-2)}{r^3}f'\\
        &+\frac{n^2(n+m-2)^2+2(m-4)n(n+m-2)}{r^4}f=\frac{\lambda}{r^4}f\\
        &f(a)=f(b)=f'(a)=f'(b)=0,
        \end{alignedat}\right.
    \end{align}
    where $\lambda\geq 0$.
    Making the change of variable $f(r)=Y(\log(r))$, we get
    \begin{align}\label{change_var_log_gen}
    \left\{\begin{alignedat}{1}
        &\p{r}f(r)=\frac{1}{r}Y'(\log(r))\\
        &\p{r}^2f(r)=\frac{1}{r^2}Y''(\log(r))-\frac{1}{r^2}Y'(\log(r))\\
        &\p{r}^3f(r)=\frac{1}{r^3}Y'''(\log(r))-\frac{3}{r^3}Y''(\log(r))+\frac{2}{r^3}Y'(\log(r))\\
        &\p{r}^4f(r)=\frac{1}{r^4}Y''''(\log(r))-\frac{6}{r^4}Y'''(\log(r))+\frac{11}{r^4}Y''(\log(r))-\frac{6}{r^4}Y'(\log(r)).
        \end{alignedat}\right.
    \end{align}
    Therefore, we have
    \begin{align}\label{change_var_log_gen2}
        &r^4\left(f''''+\frac{2(m-1)}{r}f'''+\frac{(m-1)(m-3)-2n(n+m-2)}{r^2}f''\right.\nonumber\\
        &\left.-\frac{(m-1)(m-3)+2(m-3)(n+m-2)}{r^3}f'\right.\nonumber\\
        &\left.+\frac{n^2(n+m-2)^2+2(m-4)n(n+m-2)}{r^4}f\right)\nonumber\\
        &=Y''''-6\,Y'''+11\,Y''-6\,Y'
        +2(m-1)\left(Y'''-3\,Y''+2\,Y'\right)\nonumber\\
        &+\left((m-1)(m-3)-2n(n+m-2)\right)\left(Y''-Y'\right)
        -\left((m-1)(m-3)+2(m-3)n(n+m-2)\right)Y'\nonumber\\
        &+\left(n^2(n+m-2)^2+2(m-4)n(n+m-2)\right)Y\nonumber\\
        &=Y''''+2(m-4)Y'''+((m-1)(m-9)+11-2n(n+m-2))Y''\nonumber\\
        &-2((m-1)(m-5)+3+(m-4)n(n+m-2))Y'+(n^2(n+m-2)^2+2(m-4)n(m+n-2))Y.
    \end{align}
    Therefore, \eqref{bilap_eigenvalue_gen} is equivalent to 
    \begin{align}\label{bilap_eingenvalue_gen_adjoint}
    \left\{\begin{alignedat}{1}
        &Y''''+2(m-4)Y'''+((m-1)(m-9)+11-2n(n+m-2))Y''\\
        &-2((m-1)(m-5)+3+(m-4)n(n+m-2))Y'\\
        &+(n^2(n+m-2)^2+2(m-4)n(m+n-2))Y=\lambda\,Y\\
        &Y(\log(a))=Y(\log(b))=Y'(\log(a))=Y'(\log(b))=0.
        \end{alignedat}\right.
    \end{align}
    The characteristic polynomial of this equation is given by 
    \begin{align*}
        P(X)&=X^4+2(m-4)X^3+((m-1)(m-9)+11-2n(n+m-2))X^2\\
        &-2((m-1)(m-5)+3+(m-4)n(n+m-2))X+n^2(n+m-2)^2+2(m-4)n(n+m-2)-\lambda.
    \end{align*}
    Make the change of variable $X=Y-(m-4)/2$. Then we get
    \begin{align*}
        &X^2=Y^2-(m-4)Y+\frac{(m-4)^2}{4}\\
        &X^3=Y^3-\frac{3}{2}(m-4)Y^2+\frac{3}{4}(m-4)^2Y-\frac{1}{8}(m-4)^3\\
        &X^4=Y^4-2(m-4)Y^3+\frac{3}{2}(m-4)^2Y^2-\frac{1}{2}(m-4)^3Y+\frac{(m-4)^4}{16}.
    \end{align*}
    Also, let us rewrite the coefficients
    \begin{align*}
        &(m-1)(m-9)+11=m^2-10m+20=(m-5)^2-5\\
        &(m-1)(m-5)+3=m^2-6m+8=(m-2)(m-4).
    \end{align*}
    We therefore rewrite
    \begin{align*}
        P(X)&=X^4+2(m-4)X^3+(m^2-10m+20-2n(n+m-2))X^2-2(m-4)(m-2+n(n+m-2))X\\
        &+n^2(n+m-2)^2+2(m-4)n(n+m-2)-\lambda.
    \end{align*}
    Therefore, the coefficient in $Y$ is
    \begin{align*}
        &-\frac{1}{2}(m-4)^3+\frac{3}{2}(m-4)^3-(m-4)(m^2-10\,m+20-\colorcancel{2n(n+m-2)}{red})\\
        &-2(m-4)(m-2+\colorcancel{n(n+m-2)}{red})\\
        &=(m-4)\left((m-4)^2-m^2+10\,m-20\right)-2(m-2)(m-4)=(m-4)(2m-4)-2(m-4)(m-2)=0.
    \end{align*}
    Therefore, we get
    \begin{align}\label{change_var_bilap_poly}
        &Q(Y)=P(Y-(m-4)/2)=Y^4+\frac{3}{2}(m-4)^2Y^2+\frac{(m-4)^4}{16}\nonumber\\
        &+2(m-4)\left(-\frac{3}{2}(m-4)Y^2-\frac{1}{8}(m-4)^3\right)\nonumber\\
        &+(m^2-10m+20-2n(n+m-2))\left(Y^2+\frac{(m-4)^2}{4}\right)\nonumber\\
        &-2(m-4)(m-2+n(n+m-2))\left(-\frac{m-4}{2}\right)\nonumber\\
        &+n^2(n+m-2)^2+2(m-4)n(n+m-2)-\lambda\nonumber\\
        &=Y^4+(m^2-10m+20-\frac{3}{2}(m-4)^2-2n(n+m-2))Y^2-\frac{3}{16}(m-4)^4\nonumber\\
        &+\frac{(m-4)^2}{4}\left(m^2-10m+20-2n(n+m-2)\right)+(m-4)^2(m-2+n(n+m-2))+n^2(n+m-2)^2\nonumber\\
        &+2(m-4)n(n+m-2)-\lambda\nonumber\\
        &=Y^4-\left(\frac{1}{2}m^2-2m+4+2n(n+m-2)\right)Y^2+\left(\frac{m(m-4)}{4}+n(n+m-2)\right)^2-\lambda\nonumber\\
        &=Y^4-\left(\frac{1}{2}((m-2)^2+4)+2n(n+m-2)\right)Y^2+\left(\frac{m(m-4)}{4}+n(n+m-2)\right)^2-\lambda.
    \end{align}
    Indeed, we have
    \begin{align*}
        &-\frac{3}{16}(m-4)^4+\frac{(m-4)^2}{4}(m^2-10m+20-2n(n+m-2))+(m-4)^4 (m-2+n(n+m-2))\\
        &=\frac{(m-4)^2}{16}\left(-3m^2+\colorcancel{24m}{red}-\colorcancel{48}{blue}+4m^2-\colorcancel{40m}{red}+\colorcancel{80}{blue}-8n(n+m-2)+\colorcancel{16m}{red}-\colorcancel{32}{blue}+16n(n+m-2)\right)\\
        &=\frac{(m-4)^2}{16}\left(m^2+8n(n+m-2)\right),
    \end{align*}
    which shows that the constant term is
    \begin{align*}
        &\frac{m^2(m-4)^2}{16}+\frac{1}{2}(m-4)^2n(n+m-2)+n^2(n+m-2)+2(m-4)n(n+m-2)\\
        &=\frac{m^2(m-4)^2}{16}+\frac{1}{2}(m-4)n(n+m-2)((m-4)+4)+n^2(n+m-2)^2\\
        &=\frac{m^2(m-4)^2}{16}+\frac{1}{2}m(m-4)n(n+m-2)+n^2(n+m-2)^2\\
        &=\left(\frac{m(m-4)}{4}+n(n+m-2)\right)^2.
    \end{align*}
    Once more, we find a biquadratic polynomial! The discriminant of $Q(\sqrt{Y})$ is given by 
    \begin{align}\label{discriminant_bilap_gen}
        D&=\left(\frac{1}{2}((m-2)^2+4)+2n(n+m-2)\right)^2-4\left(\frac{1}{4}m(m-4)+n(n+m-2)\right)^2+4\lambda\nonumber\\
        &=\left(\frac{1}{2}(m^2-4m+8)+2n(n+m-2)\right)^2-\left(\frac{1}{2}(m^2-4m)+2n(n+m-2)\right)^2+4\lambda\nonumber\\
        &=4((m^2-4m+4)+4n(n+m-2))+4\lambda\nonumber\\
        &=4((m-2)^2+4n^2+4n(m-2))+4\lambda=4((2n+m-2)^2+\lambda).
    \end{align}
    For $m=4$, we recover the formula $4(4(n+1)^2+\lambda)$. Therefore, the roots of $Q(\sqrt{Y})$ are given by 
    \begin{align*}
        \frac{1}{4}((m-2)^2+4)+n(n+m-2)\pm \sqrt{(2n+m-2)^2+\lambda},
    \end{align*}
    and finally, the roots of $Q$ are given by
    \begin{align}\label{root_Q_gen}
        \left\{\begin{alignedat}{1}
            s_1&=\sqrt{\frac{1}{4}((m-2)^2+4)+n(n+m-2)+\sqrt{(2n+m-2)^2+\lambda}}\\
            s_2&=-\sqrt{\frac{1}{4}((m-2)^2+4)+n(n+m-2)+\sqrt{(2n+m-2)^2+\lambda}}\\
            s_3&=\sqrt{\frac{1}{4}((m-2)^2+4)+n(n+m-2)-\sqrt{(2n+m-2)^2+\lambda}}\\
            s_4&=-\sqrt{\frac{1}{4}((m-2)^2+4)+n(n+m-2)-\sqrt{(2n+m-2)^2+\lambda}}\\
        \end{alignedat}\right.
    \end{align}
    Finally, the roots of $P$ are given by 
    \begin{align}\label{root_P_gen}
        \left\{\begin{alignedat}{1}
            r_1&=-\frac{m-4}{2}+\sqrt{\frac{1}{4}((m-2)^2+4)+n(n+m-2)+\sqrt{(2n+m-2)^2+\lambda}}\\
            r_2&=-\frac{m-4}{2}-\sqrt{\frac{1}{4}((m-2)^2+4)+n(n+m-2)+\sqrt{(2n+m-2)^2+\lambda}}\\
            r_3&=-\frac{m-4}{2}+\sqrt{\frac{1}{4}((m-2)^2+4)+n(n+m-2)-\sqrt{(2n+m-2)^2+\lambda}}\\
            r_4&=-\frac{m-4}{2}-\sqrt{\frac{1}{4}((m-2)^2+4)+n(n+m-2)-\sqrt{(2n+m-2)^2+\lambda}}\\
        \end{alignedat}\right.
    \end{align}
    Taking $m=2$, we recover \eqref{racines_m_n} (where $m=1$ in this identity), whilst for $m=4$, we recover \eqref{root4}.
    We can also rewrite the roots as 
    \begin{align}\label{root_P_gen2}
        \left\{\begin{alignedat}{1}
            r_1&=-\frac{m-4}{2}+\sqrt{\frac{1}{4}(2n+m-2)^2+1+\sqrt{(2n+m-2)^2+\lambda}}\\
            r_2&=-\frac{m-4}{2}-\sqrt{\frac{1}{4}(2n+m-2)^2+1+\sqrt{(2n+m-2)^2+\lambda}}\\
            r_3&=-\frac{m-4}{2}+\sqrt{\frac{1}{4}(2n+m-2)^2+1-\sqrt{(2n+m-2)^2+\lambda}}\\
            r_4&=-\frac{m-4}{2}-\sqrt{\frac{1}{4}(2n+m-2)^2+1-\sqrt{(2n+m-2)^2+\lambda}}\\
        \end{alignedat}\right.
    \end{align}
    Notice that there are complex roots if and only if
    \begin{align*}
        \lambda>\frac{1}{16}\left((2n+m-2)^2-4\right)^2=\frac{1}{16}(2n+m)^2(2n+m-4)^2,
    \end{align*}
    For all $m\geq 3$, this quantity is minimal for $n=0$, and we will prove the following theorem.
    \begin{theorem}\label{ode_gen_dimension}
        Let $m>2$, $n\in\N$, $0<a<b<\infty$, and assume that
        \begin{align*}
            \log\left(\frac{b}{a}\right)>15\sqrt{2+\frac{3\pi}{2}(\pi+1)}.
        \end{align*}
        Let $Y$ be a non-trivial solution of the system
        \begin{align}\label{bilap_eingenvalue_gen_adjoint2}
    \left\{\begin{alignedat}{1}
        &Y''''+2(m-4)Y'''+((m-1)(m-9)+11-2n(n+m-2))Y''\\
        &-2((m-1)(m-5)+3+(m-4)n(n+m-2))Y'\\
        &+(n^2(n+m-2)^2+2(m-4)n(m+n-2))Y=\lambda_n\,Y\\
        &Y(\log(a))=Y(\log(b))=Y'(\log(a))=Y'(\log(b)),
        \end{alignedat}\right.
    \end{align}
    where $\lambda_n\geq 0$. Then, we have
    \begin{align*}
        \lambda_n>\frac{1}{16}(2n+m)^2(2n+m-4)^2+\frac{c^2}{\log^2\left(\frac{b}{a}\right)}.
    \end{align*}
    In particular, we have
    \begin{align*}
        \inf_{n\in\N}\lambda_n>\left(\frac{m^2}{4}+\frac{\pi^2}{\log^2\left(\frac{b}{a}\right)}\right)\left(\frac{(m-4)^2}{4}+\frac{\pi^2}{\log^2\left(\frac{b}{a}\right)}\right).
    \end{align*}
    Finally, provided that 
    \begin{align*}
        \log\left(\frac{b}{a}\right)>\pi\sqrt{\frac{3m^2-16m+28+\sqrt{(3m^2-16m+28)^2+60(2m-1)((m-1)^2-2)}}{(2m-1)((m-1)^2-2)}},
    \end{align*}
    we have
    \begin{align*}
        \inf_{n\in\N}\lambda_n=\lambda_0.
    \end{align*}
    \end{theorem}
    \begin{proof}
        Let $Y$ be a non-trivial solution of the system. As previously, we distinguish three cases.

        \textbf{Case 1: $\lambda<\dfrac{1}{16}((2n+m-2)^2-4)^2$}. 
        Then, the characteristic polynomial 
        \begin{align*}
            &P(X)=X^4+2(m-4)X^3+((m-1)(m-9)+11-2n(n+m-2))X^2
        \\
        &-2((m-1)(m-5)+3+(m-4)n(n+m-2))X+n^2(n+m-2)^2+2(m-4)n(n+m-2)-\lambda\\
        &=\left(X+\frac{m-4}{2}\right)^4-\left(\frac{1}{2}((m-2)^2+4)+2n(n+m-2)\right)\left(X+\frac{m-4}{2}\right)^2\\
        &+\left(\frac{m(m-4)}{4}+n(n+m-2)\right)^2-\lambda
        \end{align*} of the ordinary differential equation has four distinct real roots, and 
        \begin{align}
            Y(t)=\mu_1\,e^{r_1\,t}+\mu_2\,e^{r_2\,t}+\mu_3\,e^{r_3\,t}+\mu_4\,e^{r_4\,t},
        \end{align}
        where $r_1,r_2,r_3,r_4$ are given in \eqref{root4}. Also write
        \begin{align}
            \left\{\begin{alignedat}{1}
                r_1&=-\frac{m-4}{2}+\lambda_1\\
                r_2&=-\frac{m-4}{2}-\lambda_1\\
                r_3&=-\frac{m-4}{2}+\lambda_2\\
                r_4&=-\frac{m-4}{2}-\lambda_2
            \end{alignedat}\right.,
        \end{align}
        where $\lambda_1=\frac{1}{4}(2n+m-2)^2+1+\sqrt{(2n+m-2)^2+\lambda}$ and $\lambda_2=\frac{1}{4}(2n+m-2)^2+1-\sqrt{(2n+m-2)^2+\lambda}$. The boundary conditions are equivalent to
         \begin{align*}
        A\,\mu=\begin{pmatrix}
            a^{r_1} & a^{r_2} & a^{r_3} & a^{r_4}\\
            b^{r_1} & b^{r_2} & b^{r_3} & b^{r_4}\\
            r_1\,a^{r_1} & r_2\,a^{r_2} & r_3\,a^{r_3} & r_4\,a^{r_4}\\
            r_1\,b^{r_2} & r_2\,b^{r_2} & r_3\,b^{r_3} & r_4\,b^{r_4}
        \end{pmatrix}\begin{pmatrix}
            \mu_1\\
            \mu_2\\
            \mu_3\\
            \mu_4
        \end{pmatrix}=0. 
    \end{align*}
    We recover up to a $a^{-m+2}b^{-m+2}$ factor the determinant appearing in the proof of Theorem \ref{theoreme_ode_m_n}, and since $\lambda_1\lambda_2\neq 0$, $(\lambda_1+\lambda_2)(\lambda_1-\lambda_2)\neq 0$, the rest of the proof applies identically and we deduce that $Y=0$.

    \textbf{Case 2: $\lambda=\frac{1}{16}((2n+m-2)^2-4)^2$.}
    Then, the roots are given by 
    \begin{align*}
        \left\{\begin{alignedat}{1}
            r_1&=-\frac{m-4}{2}+\sqrt{\frac{1}{2}(2n+m-2)^2+2}\\
            r_2&=-\frac{m-4}{2}-\sqrt{\frac{1}{2}(2n+m-2)^2+2}\\
            r_3&=r_4=-\frac{m-4}{2}.
        \end{alignedat}\right.
    \end{align*}
    Write $\mu=\sqrt{\frac{1}{2}(2n+m-2)^2+2}$ to simplify notations. Therefore, $Y$ is given by 
    \begin{align*}
        Y(t)=\mu_1e^{\mu\,t}+\mu_1e^{-\mu\,t}+\mu_3e^{-\frac{m-4}{2}t}+\mu_4\,t\,e^{-\frac{m-4}{2}t}.
    \end{align*}
    The boundary conditions are equivalent to
    \begin{align*}
        \begin{pmatrix}
            a^{-\frac{m-4}{2}+\mu} & a^{-\frac{m-4}{2}-\mu} & a^{-\frac{m-4}{2}} & \log(a)a^{-\frac{m-4}{2}}\\
            b^{-\frac{m-4}{2}+\mu} & b^{-\frac{m-4}{2}-\mu} & b^{-\frac{m-4}{2}} & \log(b)b^{-\frac{m-4}{2}}\\
            \left(-\frac{m-4}{2}+\mu\right)a^{-\frac{m-4}{2}+\mu} & \left(-\frac{m-4}{2}-\mu\right)a^{-\frac{m-4}{2}-\mu}&-\frac{m-4}{2}a^{-\frac{m-4}{2}}&\left(1-\frac{m-4}{2}\log(a)\right)a^{-\frac{m-4}{2}}\\
            \left(-\frac{m-4}{2}+\mu\right)b^{-\frac{m-4}{2}+\mu} & \left(-\frac{m-4}{2}-\mu\right)b^{-\frac{m-4}{2}-\mu}&-\frac{m-4}{2}b^{-\frac{m-4}{2}}&\left(1-\frac{m-4}{2}\log(b)\right)b^{-\frac{m-4}{2}}
        \end{pmatrix}
        \begin{pmatrix}
            \mu_1\\
            \mu_2\\
            \mu_3\\
            \mu_4
        \end{pmatrix}=0.
    \end{align*}
    If $A$ is the matrix above, we get
    \begin{align*}
        \det(A)&={\left(a^{\mu} \mu \log\left(a\right) + b^{\mu} \mu \log\left(a\right) - a^{\mu} \mu \log\left(b\right) - b^{\mu} \mu \log\left(b\right) - 2 \, a^{\mu} + 2 \, b^{\mu}\right)} a^{-m - \mu + 4} {\left(a^{\mu} - b^{\mu}\right)} b^{-m - \mu + 4} \mu\\
        &=\left(2\,b^{\mu}\left(\left(\frac{b}{a}\right)^{\mu}-1\right)-\mu\,b^{\mu}\left(1+\left(\frac{b}{a}\right)^{\mu}\right)\log\left(\frac{b}{a}\right)\right) a^{-m - \mu + 4} {\left(a^{\mu} - b^{\mu}\right)} b^{-m - \mu + 4} \mu\\
        &= a^{-m - \mu + 4} {\left(a^{\mu} - b^{\mu}\right)} b^{-m + 4} \mu\left(2\left(\left(\frac{b}{a}\right)^{\mu}-1\right)-\mu\left(1+\left(\frac{b}{a}\right)^{\mu}\right)^{\mu}\log\left(\frac{b}{a}\right)\right).
    \end{align*}
    Since $a\neq b$, we have $\det(A)=0$ if and only if
    \begin{align*}
        2\left((1+x)^{\mu}-1\right)-\mu\left(1+(1+x)^{\mu}\right)\log(1+x)=0.
    \end{align*}
    where $x=\dfrac{b}{a}-1$. Thanks to the proof of Theorem \ref{theoreme_ode_m_n}, for all $\alpha\geq \sqrt{2}$ and $x>0$, we have
    \begin{align*}
        2\left((1+x)^{\alpha}-1\right)-\alpha\left(1+(1+x)^{\alpha}\right)\log(1+x)<0,
    \end{align*}
    Since $\mu=\sqrt{\frac{1}{2}((2n+m-2)^2+2}\geq \sqrt{2}$, we deduce that $\det(A)>0$, which implies that $Y=0$. Therefore, only one possibility remains.

    \textbf{Case 3: $\lambda>\frac{1}{16}((2n+m-2)^2-4)^2$.} Then, the roots of $P$ are given by 
    \begin{align*}
        \left\{\begin{alignedat}{1}
            r_1&=-\frac{m-4}{2}+\sqrt{\sqrt{(2n+m-2)^2+\lambda}+\frac{1}{4}(2n+m-2)^2+1}\\
            r_2&=-\frac{m-4}{2}-\sqrt{\sqrt{(2n+m-2)^2+\lambda}+\frac{1}{4}(2n+m-2)^2+1}\\
            r_3&=-\frac{m-4}{2}+i\sqrt{\sqrt{(2n+m-2)^2+\lambda}-\frac{1}{4}(2n+m-2)^2-1}\\
            r_4&=-\frac{m-4}{2}-i\sqrt{\sqrt{(2n+m-2)^2+\lambda}-\frac{1}{4}(2n+m-2)^2-1}.
        \end{alignedat}\right.
    \end{align*}
    If 
    \begin{align}
        \lambda_1=\sqrt{\sqrt{(2n+m-2)^2+\lambda}+\frac{1}{4}(2n+m-2)^2+1}
    \end{align}
    and 
    \begin{align}
        \lambda_2=\sqrt{\sqrt{(2n+m-2)^2+\lambda}-\frac{1}{4}(2n+m-2)^2-1}.
    \end{align}
    Up to a $a^{-m+2}b^{-m+2}$ factor, we recover the quantity from the proof of Theorem \ref{theoreme_ode_m_n}. Therefore, following the exact same steps,  we deduce that this identity is equivalent to 
    \begin{align}
        &\left(\left(1+\left(\frac{b}{a}\right)^{2\lambda_1}\right)\cos\left(\lambda_2\log\left(\frac{b}{a}\right)\right)-2\left(\frac{b}{a}\right)^{\lambda_1}\right)\lambda_1\lambda_2\\
        &=\left(1+\left(n+\frac{m-2}{2}\right)^2\right)\left(\left(\frac{b}{a}\right)^{2\lambda_1}-1\right)\sin\left(\lambda_2\log\left(\frac{b}{a}\right)\right).\nonumber
    \end{align}
    This equation correspond to \eqref{fund_m} for $m=1$ and $n=n+\frac{m-2}{2}$. Notice that we need not assume $m$ or $n$ integer in the proof of Theorem \ref{theoreme_ode_m_n}. In particular, the proof applies and shows that if 
    \begin{align*}
        \log\left(\frac{b}{a}\right)\geq \max\ens{\frac{15\sqrt{4+3\pi(\pi+1)}}{\sqrt{2+2\left(n+\frac{m-2}{2}\right)^2}},\frac{2\log\left(1+\sqrt{2+2\left(n+\frac{m-2}{2}\right)^2}\right)}{\sqrt{2+2\left(n+\frac{m-2}{2}\right)^2}}},
    \end{align*}
    which is in particular satisfied provided that 
    \begin{align*}
        \log\left(\frac{b}{a}\right)\geq 15\sqrt{2+\frac{3\pi}{2}(\pi+1)},
    \end{align*}
    then the first non-trivial zero $\theta_0$ of $\psi$ is such that
    \begin{align*}
        \pi<\theta_0<2\pi,
    \end{align*}
    which is equivalent to 
    \begin{align*}
        \sqrt{\sqrt{(2n+m-2)^2+\lambda_n}-\frac{1}{4}(2n+m-2)^2-1}=\lambda_2>\frac{\pi}{\log\left(\frac{b}{a}\right)},
    \end{align*}
    or
    \begin{align*}
        \sqrt{(2n+m-2)^2+\lambda_n}>\left(n+\frac{m-2}{2}\right)^2+1+\frac{\pi^2}{\log^2\left(\frac{b}{a}\right)},
    \end{align*}
    and finally
    \begin{align*}
        \lambda_n&>\left(\left(n+\frac{m-2}{2}\right)^2+1+\frac{\pi^2}{\log^2\left(\frac{b}{a}\right)}\right)^2-4\left(n+\frac{m-2}{2}\right)^2\\
        &=\left(\left(n+\frac{m}{2}\right)^2+\frac{\pi^2}{\log^2\left(\frac{b}{a}\right)}\right)\left(\left(n+\frac{m-4}{2}\right)^2+\frac{\pi^2}{\log^2\left(\frac{b}{a}\right)}\right).
    \end{align*}
    Finally, the other inequality shows that 
    \begin{align*}
        \lambda_n<\left(\left(n+\frac{m}{2}\right)^2+\frac{4\pi^2}{\log^2\left(\frac{b}{a}\right)}\right)\left(\left(n+\frac{m-4}{2}\right)^2+\frac{4\pi^2}{\log^2\left(\frac{b}{a}\right)}\right).
    \end{align*}
    For all $m\geq 3$, we see that 
    \begin{align}
        \inf_{n\in\N}\lambda_n>\left(\frac{m^2}{4}+\frac{\pi^2}{\log^2\left(\frac{b}{a}\right)}\right)\left(\frac{(m-4)^2}{4}+\frac{\pi^2}{\log^2\left(\frac{b}{a}\right)}\right),
    \end{align}
    and—provided that $m\neq 2$—we have $\displaystyle\inf_{n\in\N}\lambda_n=\lambda_0$ provided that 
    \begin{align*}
        \left(\frac{m^2}{4}+\frac{4\pi^2}{\log^2\left(\frac{b}{a}\right)}\right)\left(\frac{(m-4)^2}{4}+\frac{4\pi^2}{\log^2\left(\frac{b}{a}\right)}\right)<\left(\left(1+\frac{m}{2}\right)^2+\frac{\pi^2}{\log^2\left(\frac{b}{a}\right)}\right)\left(\frac{(m-2)^2}{4}+\frac{\pi^2}{\log^2\left(\frac{b}{a}\right)}\right).
    \end{align*}
    Letting $X=\dfrac{\pi}{\log\left(\frac{b}{a}\right)}$, we deduce that this equation is equivalent to 
    \begin{align*}
        \left(\frac{m^2}{4}+4X^2\right)\left(\frac{(m-4)^2}{4}+4X^2\right)<\left(\frac{(m+2)^2}{4}+X^2\right)\left(\frac{(m-2)^2}{4}+X^2\right),
    \end{align*}
    or
    \begin{align*}
        15X^4+\left(m^2+(m-4)^4-\frac{(m+2)^2+(m-2)^2}{4}\right)X^2+\frac{m^2(m-4)^2-(m+2)^2(m-2)^2}{16}<0,
    \end{align*}
    that can also be rewritten as
    \begin{align*}
        15X^4+\left(m^2+(m-4)^4-\frac{(m+2)^2+(m-2)^2}{4}\right)X^2-\frac{(2m-1)((m-1)^2-2)}{4}<0.
    \end{align*}
    Therefore, this inequality is satisfied if and only if
    \begin{align*}
        X^2&<\frac{1}{30}\left(-\left(m^2+(m-4)^2-\frac{(m+2)^2+(m-2)^2}{4}\right)\right.\\
        &\left.+\sqrt{\left(m^2+(m-4)^2-\frac{(m+2)^2+(m-2)^2}{4}\right)^2+15(2m-1)((m-1)^2-2)}\right)\\
        &=\frac{1}{2}\frac{(2m-1)((m-1)^2-2)}{\frac{3}{2}m^2-8m+14+\sqrt{\left(\frac{3}{2}m^2-8m+14\right)^2+15(2m-1)((m-1)^2-2)}},
    \end{align*}
    and finally, we get the condition
    \begin{align*}
        \frac{\pi}{\log\left(\frac{b}{a}\right)}<\sqrt{\frac{(2m-1)((m-1)^2-2)}{3m^2-16m+28}+\sqrt{(3m^2-16m+28)^2+60(2m-1)((m-1)^2-2)}},
    \end{align*}
    and
    \begin{align*}
        \log\left(\frac{b}{a}\right)>\pi\sqrt{\frac{3m^2-16m+28+\sqrt{(3m^2-16m+28)^2+60(2m-1)((m-1)^2-2)}}{(2m-1)((m-1)^2-2)}},
    \end{align*}
    which finally concludes the proof of the theorem.
    \end{proof}

    Let $0<a<b<\infty$, let $\Omega=B_b\setminus\bar{B}_a(0)\subset \R^m$, and consider the following minimisation problem:
    \begin{align}\label{min_gen_dim}
        \lambda=\min\ens{\int_{\Omega}(\Delta u)^2dx:\int_{\Omega}\frac{u^2}{|x|^4}dx=1}.
    \end{align}
    Since \eqref{poincare_bilaplace} is true for all $m\geq 2$, using the Sobolev embedding $W^{2,2}(\Omega)\hookrightarrow W^{1,\frac{2m}{m-2}}(\Omega)$, if $m=3$, we get $W^{2,2}(\Omega)\hookrightarrow W^{1,6}(\Omega)\hookrightarrow C^0(\Omega)$ (where the last embedding), for $m=4$, we have a compact embedding $W^{2,2}(\Omega)\hookrightarrow L^p(\Omega)$ for all $p<\infty$, and for $m\geq 5$, we have a compact embedding $W^{2,2}(\Omega)\hookrightarrow L^{\frac{2m}{m-4}}(\Omega)$, and since $\dfrac{2m}{m-4}>2$, we deduce that the condition 
    \begin{align*}
        \int_{\Omega}\frac{u_k^2}{|x|^4}dx=1
    \end{align*}
    passes to the limit as $k\rightarrow \infty$, which implies that there exists a minimiser of \eqref{min_gen_dim} for all $m\geq 2$. It satisfies the following equation
    \begin{align*}
        \left\{\begin{alignedat}{2}
        \Delta^2u&=\frac{\lambda}{|x|^4}u\qquad&&\text{in}\;\,\Omega\\
        u&=0\qquad&&\text{on}\;\, \partial\Omega\\
        \partial_{\nu}u&=0\qquad &&\text{on}\;\,\partial\Omega.
        \end{alignedat}\right.
    \end{align*}

    Now, the previous Theorem \ref{ode_gen_dimension} implies the following result.
    \begin{theorem}\label{eigenvalue_gen}
        Let $m\geq 3$ and $u\in W^{2,2}_0(\Omega)$ be a minimiser of \eqref{min_gen_dim}, and assume that
        \begin{align*}
            \log\left(\frac{b}{a}\right)\geq 15\sqrt{2+\frac{3\pi}{2}(\pi+1)}.
        \end{align*}
        Then, we have
        \begin{align*}
            \lambda>\left(\frac{m^2}{4}+\frac{\pi^2}{\log^2\left(\frac{b}{a}\right)}\right)\left(\frac{(m-4)^2}{4}+\frac{\pi^2}{\log^2\left(\frac{b}{a}\right)}\right).
        \end{align*}
        Furthermore, provided that 
        \begin{align*}
        \log\left(\frac{b}{a}\right)>\pi\sqrt{\frac{3m^2-16m+28+\sqrt{(3m^2-16m+28)^2+60(2m-1)((m-1)^2-2)}}{(2m-1)((m-1)^2-2)}},
        \end{align*}
        the minimiser is a radial function $u(r)$, where $Y(r)=u(e^r)$ is a non-trivial solution of the ordinary differential equation
        \begin{align*}
            Y''''+2(m-4)Y'''+((m-5)^2-5)Y''-2(m-2)(m-4)Y'=\lambda\,Y
        \end{align*}
        such that $Y=Y'=0$ on $\partial[\log(a),\log(b)]$.
    \end{theorem}

    Finally, we deduce the following inequality for dimension $m\geq 3$.

    \begin{theorem}\label{poincare_gen}
        Let $m\geq 3$, let $0<a<b<\infty$, and let $\Omega=B_b\setminus\bar{B}_a(0)\subset \R^m$. Then, for all  $u\in W^{2,2}_0(\Omega)$, we have
        \begin{align}\label{poincare_dim_gen}
            \int_{\Omega}(\Delta u)^2dx\geq \left(\frac{m^2}{4}+\frac{\pi^2}{\log^2\left(\frac{b}{a}\right)}\right)\left(\frac{(m-4)^2}{4}+\frac{\pi^2}{\log^2\left(\frac{b}{a}\right)}\right)\int_{\Omega}\frac{u^2}{|x|^4}dx.
        \end{align}
    \end{theorem}

    \section{Second Eigenvalue Problem in All Dimension}

    Assume that $m\geq 3$, and let $0<a<b<\infty$, and $\Omega=B_b\setminus\bar{B}_a(0)\subset \R^m$. We also need to obtain an estimate of the form
    \begin{align*}
        \int_{\Omega}(\Delta u)^2dx\geq \frac{\lambda_0^2}{\log^2\left(\frac{b}{a}\right)}\int_{\Omega}\frac{|\D u|^2}{|x|^2}dx\quad \forall u\in W^{2,2}(\Omega),
    \end{align*}
    where $\lambda_0>0$ is independent of $\Omega=B_b\setminus\bar{B}_a(0)$. Notice however that since the weight is non-critical, we should get a constant uniformly bounded from below as the conformal class diverges to $\infty$. Indeed, the following result holds.
     \begin{theorem}[\cite{morse_biharmonic}]\label{poincare_gen0}
        Let $d\geq 3$, let $0<a<b<\infty$, and let $\Omega=B_b\setminus\bar{B}_a(0)\subset \R^d$.Then, for all  $u\in W^{2,2}_0(\Omega)$, we have
        \begin{align}\label{poincare_dim_gen0}
            \int_{\Omega}(\Delta u)^2dx\geq \left(\frac{d^2}{4}+\frac{\pi^2}{\log^2\left(\frac{b}{a}\right)}\right)\left(\frac{(d-4)^2}{4}+\frac{\pi^2}{\log^2\left(\frac{b}{a}\right)}\right)\int_{\Omega}\frac{u^2}{|x|^4}dx.
        \end{align}
        Furthermore, for all $d\geq 3$ and for all $u\in W^{2,2}_0(\Omega)\subset W^{2,2}(\R^d)$, we have
        \begin{align}\label{poincare_dim_gen1}
            \int_{\Omega}(\Delta u)^2dx\geq \mu_1(\Omega)\int_{\Omega}\frac{|\D u|^2}{|x|^2}dx,
        \end{align}
        where
        \begin{align*}
            \mu_1(\Omega)>\left\{\begin{alignedat}{2}
                &\frac{25+\dfrac{104\pi^2}{\log^2\left(\frac{b}{a}\right)}+\dfrac{16\pi^4}{\log^2\left(\frac{b}{a}\right)}}{36+\dfrac{16\pi^2}{\log^2\left(\frac{b}{a}\right)}}\qquad&& \text{for}\;\,d=3\\
                &\frac{9+\dfrac{10\pi^2}{\log^2\left(\frac{b}{a}\right)}+\dfrac{\pi^4}{\log^2\left(\frac{b}{a}\right)}}{3+\dfrac{\pi^2}{\log^2\left(\frac{b}{a}\right)}}\qquad&&\text{for}\;\, d=4\\
                &\frac{d^2(d-4)^2+((d-2)^2+4)\dfrac{8\pi^2}{\log^2\left(\frac{b}{a}\right)}+\dfrac{16\pi^4}{\log^4\left(\frac{b}{a}\right)}}{4(d-4)^2+\dfrac{16\pi^2}{\log^2\left(\frac{b}{a}\right)}}\qquad &&\text{for}\;\, d\geq 5.
            \end{alignedat}\right.
        \end{align*}
    \end{theorem}    
    \begin{proof}
    Consider the following minimisation problem
    \begin{align*}
        \mu=\inf\ens{\int_{\Omega}(\Delta u)^2dx:u\in W^{2,2}(\Omega)\quad \text{and}\quad \int_{\Omega}\frac{|\D u|^2}{|x|^2}dx=1}.
    \end{align*}
    A similar proof shows that a minimiser $u\in W^{2,2}_0(\Omega)$ exists, and verifies
    \begin{align*}
        \Delta^2u=-\frac{\mu}{|x|^2}\left(\Delta-2\frac{x}{|x|^2}\cdot \D\right)u.
    \end{align*}
    Recall that the Laplacian on $\R^m$ is given in polar coordinates as  
    \begin{align*}
        \Delta=\p{r}^2+\frac{m-1}{r}\p{r}+\frac{1}{r^2}\Delta_{S^{m-1}}.
    \end{align*}
    Make an expansion of $u\in C^{\infty}(\R^m)$ in polar coordinates $(r,\theta)\in (0,\infty)\times S^{n-1}$ as
    \begin{align*}
        u(r,\theta)=\sum_{n=0}^{\infty}a_n(r)Y_n(\theta),
    \end{align*}
    where ${Y_n:S^{m-1}\rightarrow \R}_{n\in\N}$ is an orthonormal basis of $L^2(S^{m-1})$ made of restrictions to $S^{m-1}$ of homogeneous harmonic polynomials. Furthermore, we have for all $n\in\N$
    \begin{align*}
        -\Delta_{S^{m-1}}Y_n=n(n+m-2)Y_n.
    \end{align*}
    In particular, we have
    \begin{align*}
        \Delta u=\sum_{n=0}^{\infty}\left(a_n''(r)+\frac{1}{r}a_n'(r)-\frac{n(n+m-2)}{r^2}\right)Y_n(\theta).
    \end{align*}
    and by \eqref{bilaplacian_all_dimension}, we have
    \begin{align*}
        \Delta^2u&=\sum_{n=0}^{\infty}\left(a_n''''(r)+\frac{2(m-1)}{r}a_n'''(r)+\frac{(m-1)(m-3)-2n(n+m-2)}{r^2}a_n''(r)\right.\\
        &\left.-\frac{(m-1)(m-3)+2(m-3)n(n+m-2)}{r^3}a_n'(r)+\frac{n^2(n+m-2)^2+2(m-4)n(n+m-2)}{r^4}\right)Y_n(\theta).
    \end{align*}
    Finally, we will also need the formula \eqref{projection_bilaplacian}
    \begin{align*}
        &\Pi_{n(n+m-2)}(\Delta^2)=\p{r}^4+\frac{2(m-1)}{r}\p{r}^3+\frac{(m-1)(m-3)-2n(n+m-2)}{r^2}\p{r}^2\\
        &-\frac{(m-1)(m-3)+2(m-3)n(n+m-2)}{r^3}\p{r}+\frac{n^2(n+m-2)^2+2(m-4)n(n+m-2)}{r^4}.
    \end{align*}
    We also have
    \begin{align*}
        \Pi_{n(n+m-2)}\left(\Delta-2\dfrac{x}{|x|^2}\cdot \D\right)=\p{r}^2+\frac{m-3}{r}\p{r}-\frac{n(n+m-2)}{r^2}.
    \end{align*}
    Therefore, we are led to consider the following eigenvalue problem on the compact interval $[a,b]$
    \begin{align}\label{second_problem_gen_dim}
    \left\{\begin{alignedat}{1}
        &f''''+\frac{2(m-1)}{r}f'''+\frac{(m-1)(m-3)-2n(n+m-2)+\mu}{r^2}f''\\
        &-\frac{(m-1)(m-3)+2(m-3)n(n+m-2)-(m-3)\mu}{r^3}f'\\
        &+\frac{n^2(n+m-2)^2+(2(m-4)-\mu)n(n+m-2)}{r^4}f=0\\
        &f(a)=f(b)=f'(a)=f'(b)=0,
        \end{alignedat}\right.
    \end{align}
    where $\mu\geq 0$. Making the change of variable $f(r)=Y(\log(r))$, we get by \eqref{change_var_log_gen}
    \begin{align*}
        r^2\left(f''+\frac{m-3}{r}f'-\frac{n(n+m-2)}{r^2}f\right)=Y''(\log(r))+(m-4)Y'(\log(r))-n(n+m-2)Y(\log(r)),
    \end{align*}
    while \eqref{change_var_log_gen2} shows that the eigenvalue problem \eqref{second_problem_gen_dim} is equivalent to 
    \begin{align}\label{bilap_eigenvalue_gen02}
        \left\{\begin{alignedat}{1}
            &Y''''+2(m-4)Y'''+(m^2-10m+20-2n(n+m-2)+\mu)Y''\\
            &-(m-4)\left(2(m-2)+2n(n+m-2)-\mu\right)Y'\\
            &+(n^2(n+m-2)^2+(2(m-4)-\mu)n(n+m-2))Y=0\\
            &Y(\log(a))=Y(\log(b))=Y'(\log(a))=Y'(\log(b)).
        \end{alignedat}\right.
    \end{align}
    The character polynomial of this equation is given by 
    \begin{align*}
        P(X)&=X^4+2(m-4)X^3+(m^2-10m+20-2n(n+m-2)+\mu)X^2\\
        &-(m-4)\left(2(m-2)+2n(n+m-2)-\mu\right)X+n^2(n+m-2)+(2(m-4)-\mu)n(n+m-2),
    \end{align*}
    or more simply
    \begin{align*}
        P(X)&=X^4+2(m-4)X^3+(m^2-10m+20-2\lambda_n+\mu)X^2-(m-4)(2(m-2)+2\lambda_n-\mu)X\\
        &+\lambda_n(\lambda_n+2(m-4)-\mu).
    \end{align*}
    In particular, we see that for the critical dimension $m=4$, we get a biquadratic equation. Consider the polynomial
    \begin{align*}
        R(X)=X^2+(m-4)X-n(n+m-2).
    \end{align*}
    If $X=Y-\dfrac{(m-4)}{2}$, we get
    \begin{align*}
        R(Y)&=\left(Y^2-(m-4)Y+\frac{(m-4)^2}{4}\right)+(m-4)\left(Y-\frac{m-4}{2}\right)-n(n+m-2)\\
        &=Y^2-\frac{(m-4)^2}{4}-n(n+m-2)
    \end{align*}
    Therefore, if $Q(Y)=P(X-(m-4)/2)$, we deduce by \eqref{change_var_bilap_poly} that 
    \begin{align*}
        Q(Y)&=Y^4-\left(\frac{1}{2}((m-2)^2+4)+2n(n+m-2)\right)Y^2+\left(\frac{m(m-4)}{4}+n(n+m-2)\right)^2\\
        &+\mu\left(Y^2-\frac{(m-4)^2}{4}-n(n+m-2)\right)\\
        &=Y^4-\left(\frac{1}{2}((m-2)^2+4)+2n(n+m-2)-\mu\right)Y^2+\left(\frac{m(m-4)}{4}+n(n+m-2)\right)^2\\
        &-\mu\left(\frac{(m-4)^2}{4}+n(n+m-2)\right).
    \end{align*}
    Thanks to \eqref{discriminant_bilap_gen}, we have
    \begin{align*}
        \left(\frac{1}{2}((m-2)^2+4)+2n(n+m-2)\right)^2-4\left(\frac{m(m-4)}{4}+n(n+m-2)\right)^2=4(2n+m-2)^2
    \end{align*}
    Therefore, we deduce that the discriminant of $Q$ is given by 
    \begin{align*}
        D_m&=\left(\frac{1}{2}(m^2-4m+8)+2n(n+m-2)-\mu\right)^2\\
        &-4\left(\frac{m(m-4)}{4}+n(n+m-2)\right)^2+\left((m-4)^2+4n(n+m-2)\right)\mu\\
        &=4(2n+m-2)^2-((m^2-4m+8)+4n(n+m-2))\mu+\mu^2+((m^2-8m+16)+4n(n+m-2))\mu\\
        &=4(2n+m-2)^2-4(m-2)\mu+\mu^2.
    \end{align*}
    The discriminant of this quadratic polynomial in $\mu$ is given by 
    \begin{align*}
        16(m-2)^2-16(2n+m-2)^2\leq 0
    \end{align*}
    for all $m\geq 3$ and $n\geq 0$. Therefore, this polynomial is always positive, which can also be seen directly since
    \begin{align*}
           &4(2n+m-2)^2-4(m-2)\mu+\mu^2=(\mu-2(m-2))^2+4(2n+m-2)^2-4(m-2)^2\\
           &=(\mu-2(m-2))^2+4(2n)(2n+2m-4)
           =(\mu-2(m-2))^2+16n(n+m-2).
    \end{align*}
    We deduce that the roots of $Q(\sqrt{Y})$ are given by 
    \begin{align*}
    \left\{\begin{alignedat}{1}
        s_1&=\frac{1}{2}\left(\frac{1}{2}((m-2)^2+4)+2n(n+m-2)-\mu+\sqrt{(\mu-2(m-2))^2+16n(n+m-2)}\right)\\
        s_2&=\frac{1}{2}\left(\frac{1}{2}((m-2)^2+4)+2n(n+m-2)-\mu-\sqrt{(\mu-2(m-2))^2+16n(n+m-2)}\right),
        \end{alignedat}\right.
    \end{align*}
    and finally, the roots of $P$ are given by 
    \begin{align}\label{second_root_gen}
        \left\{\begin{alignedat}{1}
            r_1&=-\frac{m-4}{2}+\sqrt{\frac{1}{4}(m-2)^2+1+n(n+m-2)-\frac{\mu}{2}+\frac{1}{2}\sqrt{(\mu-2(m-2))^2+16n(n+m-2)}}\\
            r_1&=-\frac{m-4}{2}-\sqrt{\frac{1}{4}(m-2)^2+1+n(n+m-2)-\frac{\mu}{2}+\frac{1}{2}\sqrt{(\mu-2(m-2))^2+16n(n+m-2)}}\\
            r_1&=-\frac{m-4}{2}+\sqrt{\frac{1}{4}(m-2)^2+1+n(n+m-2)-\frac{\mu}{2}-\frac{1}{2}\sqrt{(\mu-2(m-2))^2+16n(n+m-2)}}\\
            r_1&=-\frac{m-4}{2}-\sqrt{\frac{1}{4}(m-2)^2+1+n(n+m-2)-\frac{\mu}{2}-\frac{1}{2}\sqrt{(\mu-2(m-2))^2+16n(n+m-2)}} 
        \end{alignedat}\right.
    \end{align}
    Then, notice that 
    \begin{align*}
        \sqrt{(\mu-2(m-2))^2+16n(n+m-2)}>\frac{1}{2}(m-2)^2+2+2n(n+m-2)-\mu
    \end{align*}
    if and only if
    \begin{align*}
        &(\mu-2(m-2))^2+16n(n+m-2)\geq \frac{1}{4}(m-2)^4+2(m-2)^2+4+4n^2(n+m-2)^2\\
        &+2((m-2)^2+4)n(n+m-2)-((m-2)^2+4+4n(n+m-2))\mu+\mu^2,
    \end{align*}
    or
    \begin{align*}
        &((m-2)^2-4(m-2)+4+4n(n+m-2))\mu\\
        &\geq \frac{1}{4}(m-2)^2-2(m-2)^2+4+4n^2(n+m-2)^2+2((m-2)^2-4)n(n+m-2)
    \end{align*}
    which is equivalent to 
    \begin{align*}
        \left((m-4)^2+4n(n+m-2)\right)\mu &\geq \left(\frac{1}{2}(m-2)^2-2\right)^2+4n^2(n+m-2)^2+2((m-2)^2-4)n(n+m-2)\\
        &=\left(\frac{1}{2}(m-2)^2-2+2n(n+m-2)\right)^2,
    \end{align*}
    and finally (unless $m=4$ and $n=0$)
    \begin{align*}
        \mu \geq \frac{1}{4}\frac{\left((m-2)^2-4+4n(n+m-2)\right)^2}{(m-4)^2+4n(n+m-2)}=\frac{1}{4}\frac{\left(m(m-4)+4n(n+m-2)\right)^2}{(m-4)^2+4n(n+m-2)}.
    \end{align*}
    Therefore, assume that $(m,n)\neq (4,0)$ in the following. Then we need to distinguish three cases as previously.

    \textbf{Case 1: $\mu<\dfrac{1}{4}\dfrac{\left(m(m-4)+4n(n+m-2)\right)^2}{(m-4)^2+4n(n+m-2)}$.} Then, the characteristic polynomial admits four distinct real roots, and the solution $Y$ to \eqref{bilap_eigenvalue_gen02} is given by
    \begin{align*}
        Y(t)=\mu_1e^{r_1t}+\mu_2e^{r_2t}+\mu_3e^{r_3t}+\mu_4e^{r_4t}
    \end{align*}
    for some $\mu_i\in \R$. Writing 
    \begin{align*}
        \left\{\begin{alignedat}{1}
            r_1=-\frac{m-4}{2}+\lambda_1\\
            r_2=-\frac{m-4}{2}-\lambda_1\\
            r_3=-\frac{m-4}{2}+\lambda_2\\
            r_4=-\frac{m-4}{2}-\lambda_2
        \end{alignedat}\right.,
    \end{align*}
    if $A$ is the boundary conditions matrix, we have by \eqref{eq_det} 
    \small
    \begin{align*}
        \det(A)=(ab)^{-(m-4)}\left((\lambda_1+\lambda_2)^2\left(\left(\frac{a}{b}\right)^{\lambda_1-\lambda_2}+\left(\frac{b}{a}\right)^{\lambda_1-\lambda_2}\right)-\left(\lambda_1-\lambda_2\right)^2\left(\left(\frac{a}{b}\right)^{\lambda_1+\lambda_2}+\left(\frac{b}{a}\right)^{\lambda_1+\lambda_2}\right)-8\lambda_1\lambda_2\right)
    \end{align*}
    \normalsize
    which implies that $\det(A)\neq 0$, a contradiction.

    \textbf{Case 2: $\mu=\dfrac{1}{4}\dfrac{\left(m(m-4)+4n(n+m-2)\right)^2}{(m-4)^2+4n(n+m-2)}$.} In this case, we have
    \begin{align*}
        \frac{1}{2}\sqrt{(\mu-2(m-2))^2+16n(n+m-2)}=\frac{1}{4}(m-2)^2+1+n(n+m-2)-\mu,
    \end{align*}
    which implies that
    \begin{align*}
        \left\{\begin{alignedat}{1}
            r_1&=-\frac{m-4}{2}+\sqrt{\frac{1}{2}(m-2)^2+2+2n(n+m-2)-\mu}=-\frac{m-4}{2}+\lambda_1\\
            r_2&=-\frac{m-4}{2}-\sqrt{\frac{1}{2}(m-2)^2+2+2n(n+m-2)-\mu}=-\frac{m-4}{2}-\lambda_1\\
            r_3&=r_4=-\frac{m-4}{2}.
        \end{alignedat}\right.
    \end{align*}
    Recall that by \eqref{determinant_equality_case1}, we have
    \begin{align}\label{determinant_equality_case12}
            &\det\begin{pmatrix}
                X^{\alpha} & X^{\beta} & X^{\gamma} & \log(X)X^{\gamma}\\
                Y^{\alpha} & Y^{\beta} & Y^{\gamma} & \log(Y)Y^{\gamma}\\
                \alpha X^{\alpha} & \beta X^{\beta} & \gamma X^{\gamma} & (1+\gamma\log(X))X^{\gamma}\\
                \alpha Y^{\alpha} & \beta Y^{\beta} & \gamma Y^{\gamma} & (1+\gamma\log(Y))X^{\gamma}
            \end{pmatrix}\nonumber\\
            &=(\alpha-\beta)\left(X^{\alpha+\beta}Y^{2\gamma}+X^{2\gamma}Y^{\alpha+\beta}\right)+(\beta-\alpha)\left(X^{\alpha+\gamma}Y^{\beta+\gamma}+X^{\beta+\gamma}Y^{\alpha+\gamma}\right)\nonumber\\
            &+\left(\alpha(\beta-\gamma)+\gamma^2\right)X^{\alpha+\gamma}Y^{\beta+\gamma}\log\left(\frac{Y}{X}\right)-\left(\beta(\alpha-\gamma)+\gamma^2\right) X^{\beta+\gamma}Y^{\alpha+\gamma}\log\left(\frac{Y}{X}\right)\nonumber\\
            &+\gamma\left(\alpha X^{\beta+\gamma}Y^{\alpha+\gamma}-\beta X^{\alpha+\gamma}Y^{\beta+\gamma}\right)\log\left(\frac{Y}{X}\right).
        \end{align}
        Taking now
        \begin{align}\label{determinant_equality_case22}
        \left\{\begin{alignedat}{1}
            &\alpha=-\frac{m-4}{2}+\sqrt{\frac{1}{2}(m-2)^2+2+2n(n+m-2)}=-\frac{m-4}{2}+\lambda_1\\
            &\beta=-\frac{m-4}{2}-\sqrt{\frac{1}{2}(m-2)^2+2+2n(n+m-2)}=-\frac{m-4}{2}-\lambda_1\\
            &\gamma=r_3=-\frac{m-4}{2}
            \end{alignedat}\right.
        \end{align}
        if $A$ is the boundary conditions matrix, we get (if $X=a$ and $Y=b$)
    \begin{align}
        &\det(A)=2\lambda_1\left(X^{4-m}Y^{4-m}+X^{4-m}Y^{4-m}\right)-2\lambda_1\left(X^{4-m+\lambda_1}Y^{4-m-\lambda_1}+X^{4-m-\lambda_1}Y^{4-m+\lambda_1}\right)\nonumber\\
        &+\left(\left(-\colorcancel{\frac{m-4}{2}}{blue}+\lambda_1\right)(-\lambda_1)+\colorcancel{\frac{(m-4)^2}{4}}{red}\right)X^{4-m+\lambda_1}Y^{4-m-\lambda_1}\log\left(\frac{Y}{X}\right)\nonumber\\
        &-\left(\left(-\colorcancel{\frac{m-4}{2}}{blue}-\lambda_1\right)\lambda_1+\colorcancel{\frac{(m-4)^2}{4}}{red}\right)\log\left(\frac{Y}{X}\right)X^{4-m-\lambda_1}Y^{4-m+\lambda_1}\log\left(\frac{Y}{X}\right)\nonumber\\
        &-\frac{m-4}{2}\left(\left(-\colorcancel{\frac{m-4}{2}}{red}+\colorcancel{\lambda_1}{blue}\right)X^{4-m-\lambda_1}Y^{4-m+\lambda_1}-\left(-\colorcancel{\frac{m-4}{2}}{red}-\colorcancel{\lambda_1}{blue}\right)X^{4-m+\lambda_1}Y^{4-m-\lambda_1}\right)\log\left(\frac{Y}{X}\right)\nonumber\\
        &=4\lambda_1X^{4-m}Y^{4-m}-2\lambda_1\left(X^{4-m+\lambda_1}Y^{4-m-\lambda_1}+X^{4-m-\lambda_1}Y^{4-m+\lambda_1}\right)\nonumber\\
        &+\lambda_1^2\left(X^{4-m+\lambda_1}Y^{4-m-\lambda_1}-X^{4-m-\lambda_1}Y^{4-m+\lambda_1}\right)\log\left(\frac{Y}{X}\right)\nonumber\\
        &=\lambda_1(ab)^{4-m}\left(4-2\left(\frac{b}{a}\right)^{\lambda_1}-2\bigg(\frac{a}{b}\bigg)^{\lambda_1}+\lambda_1\left(\left(\frac{b}{a}\right)^{\lambda_1}-\bigg(\frac{a}{b}\bigg)^{\lambda_1}\right)\log\left(\frac{b}{a}\right)\right)
        <0,
    \end{align}
    where we used the inequality proven after \eqref{determinant_equality_case4}.
    Therefore, this case is also excluded, and we can finally move to the analysis of the last case.

    \textbf{Case 3: $\mu>\dfrac{1}{4}\dfrac{\left(m(m-4)+4n(n+m-2)\right)^2}{(m-4)^2+4n(n+m-2)}$.}

    Then, the roots of the characteristic polynomial $P$ are given by 
    \begin{align*}
        \left\{\begin{alignedat}{1}
            r_1&=-\frac{m-4}{2}+\sqrt{\frac{1}{2}\sqrt{(\mu-2(m-2))^2+16n(n+m-2)}+\frac{1}{4}(m-2)^2+1+n(n+m-2)-\frac{\mu}{2}}\\
            r_2&=-\frac{m-4}{2}+\sqrt{\frac{1}{2}\sqrt{(\mu-2(m-2))^2+16n(n+m-2)}+\frac{1}{4}(m-2)^2+1+n(n+m-2)-\frac{\mu}{2}}\\
            r_3&=-\frac{m-4}{2}+i\sqrt{\frac{1}{2}\sqrt{(\mu-2(m-2))^2+16n(n+m-2)}-\frac{1}{4}(m-2)^2-1-n(n+m-2)+\frac{\mu}{2}}\\
            r_4&=-\frac{m-4}{2}+i\sqrt{\frac{1}{2}\sqrt{(\mu-2(m-2))^2+16n(n+m-2)}-\frac{1}{4}(m-2)^2-1-n(n+m-2)+\frac{\mu}{2}}.
        \end{alignedat}\right.
    \end{align*}
    Recalling that
    \begin{align*}
    \begin{vmatrix}
            X^{\alpha} & X^{\beta} & X^{\gamma} & X^{\delta}\\
            Y^{\alpha} & Y^{\beta} & Y^{\gamma} & Y^{\delta}\\
            \alpha\,X^{\alpha} & \beta\,X^{\beta} & \gamma\,X^{\gamma} & \delta\,X^{\delta}\\
            \alpha\,Y^{\alpha} & \beta\,Y^{\beta} & \gamma\,Y^{\gamma} & \delta\,Y^{\delta}
        \end{vmatrix}&=
        (\beta-\alpha)(\gamma-\delta)X^{\alpha+\beta}Y^{\gamma+\delta}+(\gamma-\alpha)(\delta-\beta)X^{\alpha+\gamma}Y^{\beta+\delta}\\
        &+(\delta-\alpha)(\beta-\gamma)X^{\alpha+\delta}Y^{\beta+\gamma}
        +(\beta-\gamma)(\delta-\alpha)X^{\beta+\gamma}Y^{\alpha+\delta}\\
        &+(\beta-\delta)(\alpha-\gamma)X^{\beta+\delta}Y^{\alpha+\gamma}+(\delta-\gamma)(\alpha-\beta)X^{\gamma+\delta}Y^{\alpha+\beta},
    \end{align*}
    we deduce by the computations preceding \eqref{eq_det_complex} that the boundary conditions matrix satisfies the equation
    \begin{align*}
        \det(A)=2i\,(ab)^{4-m}\left(\Im\left(\zeta^2\left(\left(\frac{a}{b}\right)^{\bar{\zeta}}+\left(\frac{b}{a}\right)^{\bar{\zeta}}\right)\right)-4\,\Re(\zeta)\Im(\zeta)\right)
    \end{align*}
    where $\zeta=\lambda_1+i\,\lambda_2$, and
    \begin{align*}
    \left\{\begin{alignedat}{1}
        \lambda_1&=\sqrt{\frac{1}{2}\sqrt{(\mu-2(m-2))^2+16n(n+m-2)}+\frac{1}{4}(m-2)^2+1+n(n+m-2)-\frac{\mu}{2}}\\
        \lambda_2&=\sqrt{\frac{1}{2}\sqrt{(\mu-2(m-2))^2+16n(n+m-2)}-\frac{1}{4}(m-2)^2-1-n(n+m-2)+\frac{\mu}{2}}
        \end{alignedat}\right.
    \end{align*}
    Therefore, following the exact same steps before \eqref{fun_eq_gen}, we deduce that 
    \small
    \begin{align}\label{second_fun_eq_gen2}
        2\left(\left(\left(\frac{b}{a}\right)^{\lambda_1}+\bigg(\frac{a}{b}\bigg)^{\lambda_1}\right)\cos\left(\lambda_2\log\left(\frac{b}{a}\right)\right)-2\right)\lambda_1\lambda_2-\left(\lambda_1^2-\lambda_2^2\right)\left(\left(\frac{b}{a}\right)^{\lambda_1}+\bigg(\frac{a}{b}\bigg)^{\lambda_1}\right)\sin\left(\lambda_2\log\left(\frac{b}{a}\right)\right)=0.
    \end{align}
    \normalsize
    We have
    \begin{align*}
        \lambda_1^2-\lambda_2^2&=\frac{1}{2}(m-2)^2+2+2n(n+m-2)-\mu\\
        (\lambda_1\lambda_2)^2&=\frac{1}{4}\left((\mu-2(m-2))^2+16n(n+m-2)\right)-\left(\frac{1}{2}(m-2)^2+1+n(n+m-2)-\frac{\mu}{2}\right)^2\\
        &=\frac{1}{4}\left((m-4)^2+4n(n+m-2)\right)\mu-\frac{1}{16}\left(m(m-4)+4n(n+m-2)\right)^2.
    \end{align*}
    Therefore, the equation becomes if $\lambda_1=\sqrt{\frac{1}{2}(m-2)^2+2+2n(n+m-2)+\lambda_2^2}$, noticing that $\lambda_2\geq 0$ as $\mu\geq \dfrac{1}{4}\dfrac{\left(m(m-4)+4n(n+m-2)\right)^2}{(m-4)^2+4n(n+m-2)}$, provided that 
    \begin{align*}
        2p^2+2q^2=\frac{1}{2}(m-2)^2+2+2n(n+m-2)-\mu,
    \end{align*}
    assuming without loss of generality that 
    \begin{align*}
        \mu\leq \frac{1}{2}(m-2)^2+2n(n+m-2),
    \end{align*}
    we can take $p=1$ and
    \begin{align*}
        q=\sqrt{\frac{1}{4}(m-2)^2+n(n+m-2)-\frac{\mu}{2}}
    \end{align*}
    since $n$ ($p$ in the current notations) need not be integer in this part of the proof of Theorem \ref{ode_gen_dimension}, we deduce that provided that 
    \begin{align}\label{bound_conformal_class_mu}
        \log\left(\frac{b}{a}\right)\geq \frac{5\pi}{\sqrt{\frac{1}{2}(m-2)^2+2+2n(n+m-2)-\mu}},
    \end{align}
    the first positive zero $\lambda_2^{\ast}>0$ of the equation  \eqref{second_fun_eq_gen2} satisfies the inequality
    \begin{align}\label{est_first_zero}
        \pi<\frac{\lambda_2^{\ast}}{\log\left(\frac{b}{a}\right)}<2\pi.
    \end{align}
    In particular, we get
    \begin{align*}
        \sqrt{(\mu-2(m-2))^2+16n(n+m-2)}-\frac{1}{2}(m-2)^2-2-2n(n+m-2)+\mu> \frac{2\pi^2}{\log^2\left(\frac{b}{a}\right)},
    \end{align*}
    or
    \begin{align*}
        (\mu-2(m-2))^2+16n(n+m-2)&>\frac{4\pi^4}{\log^4\left(\frac{b}{a}\right)}+\left(\frac{1}{2}(m-2)^2+2+2n(n+m-2)-\mu\right)\frac{4\pi^2}{\log^2\left(\frac{b}{a}\right)}\\
        &+\left(\frac{1}{2}(m-2)^2+2+2n(n+m-2)-\mu\right)^2,
    \end{align*}
    that we rewrite
    \begin{align*}
        &\left((m-4)^2+4n(n+m-2)+\frac{4\pi^2}{\log^2\left(\frac{b}{a}\right)}\right)\mu> \frac{1}{4}\left(m(m-4)+4n(n+m-2)\right)^2\\
        &+\frac{1}{2}\left((m-2)^2+4+4n(n+m-2)\right)\frac{4\pi^2}{\log^2\left(\frac{b}{a}\right)}+\frac{4\pi^4}{\log^4\left(\frac{b}{a}\right)},
    \end{align*}
    and finally
    \begin{align}\label{lower_bound_mu}
        \mu>\frac{\left(m(m-4)+4n(n+m-2)\right)^2+\left((m-2)^2+4+4n(n+m-2)\right)\dfrac{8\pi^2}{\log^2\left(\frac{b}{a}\right)}+\dfrac{16\pi^4}{\log^4\left(\frac{b}{a}\right)}}{4(m-4)^2+16n(n+m-2)+\dfrac{16\pi^2}{\log^2\left(\frac{b}{a}\right)}}.
    \end{align}
    Now, as previously, we need to express the bound \eqref{bound_conformal_class_mu} into a condition that only depends on $m$ and $n$. On the other hand, the estimate \eqref{est_first_zero} shows that 
    \begin{align*}
        \sqrt{(\mu-2(m-2))^2+16n(n+m-2)}-\frac{1}{2}(m-2)^2-2-2n(n+m-2)+\mu<\frac{8\pi^2}{\log^2\left(\frac{b}{a}\right)},
    \end{align*}
    which implies that 
    \begin{align*}
        (\mu-2(m-2))^2+16n(n+m-2)&<\frac{64\pi^4}{\log^4\left(\frac{b}{a}\right)}+\left(\frac{1}{2}(m-2)^2+2+2n(n+m-2)-\mu\right)\frac{16\pi^2}{\log^2\left(\frac{b}{a}\right)}\\
        &+\left(\frac{1}{2}(m-2)^2+2+2n(n+m-2)-\mu\right)^2,
    \end{align*}
    which is equivalent to
    \begin{align*}
        \left((m-4)^2+4n(n+m-2)+\frac{16\pi^2}{\log^2\left(\frac{b}{a}\right)}\right)&<\frac{1}{4}\left(m(m-4)+4n(n+m-2)\right)^2\\
        &+\frac{1}{2}\left((m-2)^2+4+4n(n+m-2)\right)\frac{16\pi^2}{\log^2\left(\frac{b}{a}\right)}+\frac{64\pi^4}{\log^4\left(\frac{b}{a}\right)},
    \end{align*}
    and finally
    \begin{align}
        \mu<\frac{(m(m-4)+4n(n+m-2))^2+\left((m-2)^2+4+4n(n+m-2)\right)\dfrac{64\pi^2}{\log^2\left(\frac{b}{a}\right)}+\dfrac{256\pi^4}{\log^4\left(\frac{b}{a}\right)}}{4(m-4)^2+16n(n+m-2)+\dfrac{64\pi^2}{\log^2\left(\frac{b}{a}\right)}}.
    \end{align}
    Let $X=\dfrac{4\pi}{\log^2\left(\frac{b}{a}\right)}$, and let us estimate
    \begin{align*}
        &c(X)=\frac{1}{2}(m-2)^2+2+2n(n+m-2)\\
        &-\frac{(m(m-4)+4n(n+m-2))^2+4\left((m-2)^2+4+4n(n+m-2)\right)X^2+X^4}{4(m-4)^2+16n(n+m-2)+4X^2}\\
        &=\frac{1}{4(m-4)^2+16n(n+m-2)+4X^2}\\
        &\times\left(2\left((m-2)^2+4+4n(n+m-2))\left((m-4)^2+4n(n+m-2)+X^2)\right)\right)\right.\\
        &\left.-(m(m-4)+4n(n+m-2))^2-4\left((m-2)^2+4+4n(n+m-2)\right)X^2-X^4\right)\\
        &=\frac{1}{4(m-4)^2+16n(n+m-2)+4X^2}\\
        &\times\bigg(2((m-2)^2+4)(m-4)^2+8\left((m-2)^2+(m-4)^2+4\right)n(n+m-2)\\
        &+32n^2(n+m-2)^2-2((m-2)^2+4+4n(n+m-2))X^2-m^2(m-4)^2-8m(m-4)n(n+m-2)\\
        &-16n^2(n+m-2)^2-X^4\bigg)\\
        &=\frac{1}{4(m-4)^2+16n(n+m-2)+4X^2}\bigg((m-4)^4+8((m-4)^2+4)n(n+m-2)\\
        &+16n^2(n+m-2)^2-2((m-2)^2+4+4n(n+m-2))X^2-X^4\bigg),
    \end{align*}
    where we used
    \begin{align*}
        8((m-2)^2+(m-4)^2+4)-8m(m-4)&=8((2m^2-12m+20)-m^2+4m)=8(m^2-8m+24)\\
        &=8((m-4)^2+8)\\
        2((m-2)^2+4)-m^2&=m^2-8m+16=(m-4)^2.
    \end{align*}
   At $X=0$, we have $C(X)\geq 0$, and $C(X)=0$ if and only if $m=4$ and $n=0$, but this case is excluded from our analysis. Therefore, we get the following condition on $X$
   \begin{align*}
       X<\sqrt{(m-2)^2+4+4\lambda_n+\sqrt{((m-2)^2+4+4\lambda_n)^2+(m-4)^4+8((m-4)^2+4)\lambda_n+16\lambda_n^2}},
   \end{align*}
   where we wrote for simplicity $\lambda_n=\lambda_n$. Finally, we get the condition 
   or
   \begin{align}\label{bound_conformal_class_mu2}
       \log\left(\frac{b}{a}\right)> \frac{\pi}{\sqrt{(m-2)^2+4+4\lambda_n+\sqrt{((m-2)^2+4+4\lambda_n)^2+(m-4)^4+8((m-4)^2+4)\lambda_n+16\lambda_n^2}}}.
   \end{align}
   Assume that this bound holds on the conformal class. Then, \eqref{bound_conformal_class_mu} becomes
   \begin{align*}
       \frac{1}{X}\geq \frac{5}{\sqrt{c(X)}},
   \end{align*}
   that we rewrite as 
   \begin{align*}
       25X^2\leq c(X)&=\frac{1}{4(m-4)^2+16\lambda_n+4X^2}\bigg((m-4)^4+8\big((m-4)^2+4\big)\lambda_n\\
        &+16\lambda_n^2-2\big((m-2)^2+4+4\lambda_n\big)X^2-X^4\bigg),
   \end{align*}
   or
   \begin{align*}
       25(4(m-4)^2X^2+16\lambda_nX^2+4X^4)&\leq (m-4)^4+8\big((m-4)^2+4\big)\lambda_n
        +16\lambda_n^2-2\big((m-2)^2+4+4\lambda_n\big)X^2\\
        &-X^4
   \end{align*}
   that we finally rewrite as
   \begin{align*}
       2\big(50(m-4)^2+(m-2)^2+4+204\lambda_n\big)X^2+101X^4\leq (m-4)^4+8\big((m-4)^2+4\big)\lambda_n+16\lambda_n^2,
   \end{align*}
   and finally gives us the estimate
   \begin{align*}
       &X^2\leq -\big(50(m-4)^2+(m-2)^2+4+204\lambda_n\big)\\
       &+\sqrt{\big(50(m-4)^2+(m-2)^2+4+204\lambda_n\big)^2+(m-4)^4+8\big((m-4)^2+4\big)\lambda_n+16\lambda_n^2}.
   \end{align*} 
   Therefore, the condition on the conformal class becomes
   \scriptsize
   \begin{align}\label{bound_conformal_class_mu3}
       &\log\left(\frac{b}{a}\right)\geq\\
       &\frac{\pi\sqrt{\big(50(m-4)^2+(m-2)^2+4+204\lambda_n\big)+\sqrt{\big(50(m-4)^2+(m-2)^242+204\lambda_n\big)^2+(m-4)^4+8\big((m-4)^2+4\big)\lambda_n+16\lambda_n^2}}}{\sqrt{(m-4)^4+8\big((m-4)^2+4\big)\lambda_n+16\lambda_n^2}}\nonumber.
   \end{align}
   \normalsize
   The important point is that this number is bounded in $n\in\Z$.

   Finally, we treat the remaining case $m=4$, $n=0$.

   \textbf{Final Case: $m=4$ and $n=0$}. Then, the roots of the characteristic polynomial are given by 
   \begin{align*}
       \left\{\begin{alignedat}{4}
           r_1&=&&\sqrt{2-\frac{\mu}{2}+\frac{1}{2}\sqrt{(\mu-4)^4}}&=&&\frac{1}{\sqrt{2}}\sqrt{4-\mu+|4-\mu|}\\
           r_2&=-&&\sqrt{2-\frac{\mu}{2}+\frac{1}{2}\sqrt{(\mu-4)^4}}&=-&&\frac{1}{\sqrt{2}}\sqrt{4-\mu+|4-\mu|}\\
           r_3&=&&\sqrt{2-\frac{\mu}{2}-\frac{1}{2}\sqrt{(\mu-4)^4}}&=&&\frac{1}{\sqrt{2}}\sqrt{4-\mu-|4-\mu|}\\
           r_4&=-&&\sqrt{2-\frac{\mu}{2}-\frac{1}{2}\sqrt{(\mu-4)^4}}&=-&&\frac{1}{\sqrt{2}}\sqrt{4-\mu-|4-\mu|}.
       \end{alignedat}\right.
   \end{align*}
   Therefore, we distinguish three cases.
   
   \textbf{Sub-case 1: $\mu< 4$.} Then 
   \begin{align*}
       \left\{\begin{alignedat}{4}
           r_1&=&&\sqrt{4-\mu}\\
           r_2&=-&&\sqrt{4-\mu}\\
           r_3&=&&0\\
           r_4&=&&0.
       \end{alignedat}\right.
   \end{align*}
   Therefore, $Y$ is given by
   \begin{align*}
       Y(t)=\mu_1e^{\sqrt{4-\mu}\,t}+\mu_2e^{-\sqrt{4-\mu}\,t}+\mu_3+\mu_4t.
   \end{align*}
   Therefore, we have thanks to the boundary conditions if $\lambda_1=\sqrt{4-\mu}$
   \begin{align*}
       A\begin{pmatrix}
           \mu_1\\
           \mu_2\\
           \mu_3\\
           \mu_4
       \end{pmatrix}=\begin{pmatrix}
           a^{\lambda_1} & a^{-\lambda_1} & 1 & \log(a)\\
           b^{\lambda_1} & b^{-\lambda_1} & 1 & \log(b)\\
           \lambda_1a^{\lambda_1} & -\lambda_1a^{-\lambda_1} & 0 & 1\\
           \lambda_1b^{\lambda_1} & -\lambda_1b^{-\lambda_1} & 0 & 1
       \end{pmatrix}\begin{pmatrix}
           \mu_1\\
           \mu_2\\
           \mu_3\\
           \mu_4
       \end{pmatrix}=0.
   \end{align*}
   As $Y\neq 0$, we have $\det(A)=0$. Now, expanding on the third column, we compute
   \begin{align*}
       \det(A)&=-\begin{vmatrix}
           b^{\lambda_1} & b^{-\lambda_1} & \log(b)\\
           \lambda_1a^{\lambda_1} & -\lambda_1a^{-\lambda_1}  & 1\\
           \lambda_1b^{\lambda_1} & -\lambda_1b^{-\lambda_1}  & 1
       \end{vmatrix}+\begin{vmatrix}
           a^{\lambda_1} & a^{-\lambda_1} & \log(a)\\
           \lambda_1a^{\lambda_1} & -\lambda_1a^{-\lambda_1}  & 1\\
           \lambda_1b^{\lambda_1} & -\lambda_1b^{-\lambda_1}  & 1
       \end{vmatrix}\\
       &=-\left(-\lambda_1\left(\frac{b}{a}\right)^{\lambda_1}+\lambda_1-\lambda_1^2\bigg(\frac{a}{b}\bigg)^{\lambda_1}\log(b)+\lambda_1^2\left(\frac{b}{a}\right)^{\lambda_1}\log(b)-\lambda_1\bigg(\frac{a}{b}\bigg)^{\lambda_1}+\lambda_1\right)\\
       &-\lambda_1+\lambda_1\left(\frac{b}{a}\right)^{\lambda_1}-\lambda_1^2\bigg(\frac{a}{b}\bigg)^{\lambda_1}\log(a)+\lambda_1^2\left(\frac{b}{a}\right)^{\lambda_1}\log(a)-\lambda_1+\lambda_1\bigg(\frac{a}{b}\bigg)^{\lambda_1}\\
       &=-4\lambda_1+2\lambda_1\left(\left(\frac{b}{a}\right)^{\lambda_1}+\bigg(\frac{a}{b}\bigg)^{\lambda_1}\right)-\lambda_1^2\left(\left(\frac{b}{a}\right)^{\lambda_1}-\bigg(\frac{a}{b}\bigg)^{\lambda_1}\right)\log\left(\frac{b}{a}\right)\\
       &=-\lambda_1\left(4-2\left(\left(\frac{b}{a}\right)^{\lambda_1}+\bigg(\frac{a}{b}\bigg)^{\lambda_1}\right)+\lambda_1\left(\left(\frac{b}{a}\right)^{\lambda_1}-\bigg(\frac{a}{b}\bigg)^{\lambda_1}\right)\log\left(\frac{b}{a}\right)\right).
   \end{align*}
   However, as $\lambda_1>0$, this function of $0<a<b<\infty$ never vanished thanks to the analysis that follows \eqref{determinant_equality_case4}. Therefore, we get a contradiction since $Y\neq 0$ by hypothesis. Let us now move to the second alternative.

   \textbf{Sub-case 2: $\mu=4$.} Then, all roots are equal to $0$, which shows that
   \begin{align*}
       Y(t)=\mu_1+\mu_2\,t+\mu_3\,t^2+\mu_4\,t^3.
   \end{align*}
   Writing to simplify $X=\log(a)$ and $Y=\log(b)$, the boundary conditions show that
   \begin{align*}
       A\begin{pmatrix}
           \mu_1\\
           \mu_2\\
           \mu_3\\
           \mu_4
       \end{pmatrix}=\begin{pmatrix}
           1 & X & X^2 & X^3\\
           1 & Y & Y^2 & Y^3\\
           0 & 1 & 2X & 3X^2\\
           0 & 1 & 2Y & 3Y^2
       \end{pmatrix}\begin{pmatrix}
           \mu_1\\
           \mu_2\\
           \mu_3\\
           \mu_4
       \end{pmatrix}=0.
   \end{align*}
   Expanding on the first column, we get
   \small
   \begin{align*}
       &\det(A)=\begin{pmatrix}
           Y & Y^2 & Y^3\\
           1 & 2X & 3X^2\\
           1 & 2Y & 3Y^2
       \end{pmatrix}-\begin{pmatrix}
           X & X^2 & X^3\\
           1 & 2X & 3X^2\\
           1 & 2Y & 3Y^2
       \end{pmatrix}\\
       &=6XY^3+3X^2Y^2+2Y^4-2XY^3-3Y^4-6X^2Y^2-\left(6X^2Y^2+3X^4+2X^3Y-2X^4-3X^2Y^2-6X^3Y\right)\\
       &=-(X^4+Y^4)+4XY(X^2+Y^2)-6X^2Y^2=-\left(X-Y\right)^4=-\log^4\left(\frac{b}{a}\right)\neq 0.
   \end{align*}
   \normalsize
   Therefore, we have $Y=0$.

   \textbf{Sub-case 3: $\mu>4$.} Then, the roots of $P$ are given by 
   \begin{align*}
       \left\{\begin{alignedat}{2}
           r_1&=r_2&&=0\\
           r_3&=&&i\sqrt{\mu-4}\\
           r_4&=-&&i\sqrt{\mu-4}
       \end{alignedat}\right.
   \end{align*}
   and $Y$ is given by
   \begin{align*}
       Y(t)=\mu_1+\mu_2t+\mu_3e^{i\lambda_1t}+\mu_4e^{-i\lambda_1t}.
   \end{align*}
   where $\lambda_1=\sqrt{\mu-4}$. The boundary conditions imply that
   \begin{align*}
       A\begin{pmatrix}
           \mu_1\\
           \mu_2\\
           \mu_3\\
           \mu_4
       \end{pmatrix}=\begin{pmatrix}
           1 & \log(a) & a^{i\,\lambda_1} & a^{-i\,\lambda_1}\\
           1 & \log(b) & b^{i\,\lambda_1} & b^{-i\,\lambda_1}\\
           0 & 1 & i\,\lambda_1a^{i\,\lambda_1} & -i\,\lambda_1a^{-i\,\lambda_1}\\
           0 & 1 & i\,\lambda_1b^{i\,\lambda_1} & -i\,\lambda_1b^{-i\,\lambda_1}
       \end{pmatrix}\begin{pmatrix}
           \mu_1\\
           \mu_2\\
           \mu_3\\
           \mu_4
       \end{pmatrix}=0.
   \end{align*}
   We have
   \begin{align*}
       \det(A)&=\begin{vmatrix}
           \log(b) & b^{i\,\lambda_1} & b^{-i\,\lambda_1}\\
           1 & i\,\lambda_1a^{i\lambda_1} & -i\,\lambda_1a^{-i\,\lambda_1}\\
           1 & i\,\lambda_1b^{i\lambda_1} & -i\,\lambda_1b^{-i\,\lambda_1}
       \end{vmatrix}-\begin{vmatrix}
           \log(a) & a^{i\lambda_1} & a^{-i\,\lambda_1}\\
           1 & i\,\lambda_1a^{i\,\lambda_1} & -i\,\lambda_1a^{-i\,\lambda_1}\\
           1 & i\,\lambda_1b^{i\,\lambda_1} & -i\,\lambda_1b^{-i\,\lambda_1}
       \end{vmatrix}\\
       &=\lambda_1^2\bigg(\frac{a}{b}\bigg)^{\lambda_1}\log(b)-i\,\lambda_1\left(\frac{b}{a}\right)^{\lambda_1}+i\,\lambda_1-i\,\lambda_1\bigg(\frac{a}{b}\bigg)^{i\,\lambda_1}+i\,\lambda_1-\lambda_1^2\left(\frac{b}{a}\right)^{i\,\lambda_1}\log(b)\\
       &-\left(\lambda_1^2\bigg(\frac{a}{b}\bigg)^{i\,\lambda_1}\log(a)-i\,\lambda_1+i\,\lambda_1\left(\frac{b}{a}\right)^{i\,\lambda_1}-i\,\lambda_1+i\,\lambda_1\bigg(\frac{a}{b}\bigg)^{i\,\lambda_1}-\lambda_1^2\left(\frac{b}{a}\right)^{i\,\lambda_1}\right)\\
       &=4\,i\,\lambda_1-2\,i\,\lambda_1\left(\left(\frac{b}{a}\right)^{i\,\lambda_1}+\bigg(\frac{a}{b}\bigg)^{i\,\lambda_1}\right)-\lambda_1^2\left(\left(\frac{b}{a}\right)^{i\,\lambda_1}-\bigg(\frac{a}{b}\bigg)^{i\,\lambda_1}\right)\log\left(\frac{b}{a}\right)\\
       &=4\,i\,\lambda_1-4\,i\,\lambda_1\cos\left(\lambda_1\log\left(\frac{b}{a}\right)\right)-2\,i\,\lambda_1^2\sin\left(\lambda_1\log\left(\frac{b}{a}\right)\right)\log\left(\frac{b}{a}\right)\\
       &=2\,i\,\lambda_1\left(2\left(1-\cos\left(\lambda_1\log\left(\frac{b}{a}\right)\right)\right)-\lambda_1\sin\left(\lambda_1\log\left(\frac{b}{a}\right)\right)\log\left(\frac{b}{a}\right)\right).
   \end{align*}
   Therefore, writing $\theta=\lambda_1\log\left(\frac{b}{a}\right)$, we are brought to the study of the first positive zero of the function
   \begin{align*}
       f(\theta)=2(1-\cos(\theta))-\theta\sin(\theta),
   \end{align*}
   which is elementary. We have $f(0)=0$, and
   \begin{align*}
       f'(\theta)&=\sin(\theta)-\theta\cos(\theta)\\
       f''(\theta)&=\theta\sin(\theta).
   \end{align*}
   Therefore, $f''(\theta)>0$ for all $0<\theta<\pi$ and $f''(\theta)<0$ for all $\pi<\theta<2\pi$. The function $f'$ is strictly increasing on $[0,\pi]$ and strictly decreasing on $[\pi,2\pi]$. We have $f'(\pi)=\pi>0$, which implies in particular that $f$ is strictly decreasing on $[0,\pi]$. Since $f'(2\pi)=-2\pi<0$, there exists $\pi<\theta_0<2\pi$ such that $f'(\theta)>0$ for all $0<\theta<\theta_0$ and $f'(\theta)<0$ for all $\theta_0<\theta<2\pi$. Finally, since $f(2\pi)=0$, we deduce that $f(\theta)>0$ for all $0<\theta<2\pi$, and the first non-trivial zero $\theta_1$ of the equation $\det(A)=0$ is $2\pi$. Therefore, we get
   \begin{align*}
       \lambda_1\log\left(\frac{b}{a}\right)=2\pi,
   \end{align*}
   or
   \begin{align*}
       \sqrt{\mu-4}=\frac{2\pi}{\log\left(\frac{b}{a}\right)},
   \end{align*}
   which finally gives us
   \begin{align*}
       \mu=4+\dfrac{4\pi^2}{\log^2\left(\frac{b}{a}\right)}.
   \end{align*}
   We can now move to the global bound on the eigenvalues and the computation of the smallest eigenvalue for large enough conformal class. Since the case $d=4$ is special, we start by this case.

   \textbf{Step 2: Lower bound on all eigenvalues in dimension $\mathbf{4}$.}

   We have 
   \begin{align*}
       \mu_0=4+\frac{4\pi^2}{\log^2\left(\frac{b}{a}\right)},
   \end{align*}
   and for all $n\geq 1$, 
   \begin{align*}
       \mu_n>\frac{16\lambda_n^2+\left(2+\lambda_n\right)\dfrac{32\pi^2}{\log^2\left(\frac{b}{a}\right)}+\dfrac{16\pi^4}{\log^4\left(\frac{b}{a}\right)}}{16\lambda_n+\dfrac{16\pi^2}{\log^2\left(\frac{b}{a}\right)}}=\frac{\left(4\lambda_n+4X^2\right)^2+64X^2}{16\lambda_n+16X^2}=\lambda_n+X^2+\dfrac{4X^2}{\lambda_n+X^2},
   \end{align*}
   where $\lambda_n=n(n+2)$ for all $n\in\N$, and where we denoted $X=\dfrac{\pi}{\log\left(\frac{b}{a}\right)}$ for simplicity. Let $\alpha>0$ and
   \begin{align*}
       f(t)=t+\frac{4\alpha^2}{t+\alpha^2}.
   \end{align*}
   We have
   \begin{align*}
       f'(t)=1-\frac{4\alpha^2}{(t+\alpha^2)^2}>0
   \end{align*}
   if and only if
   \begin{align*}
       t> 2\alpha-\alpha^2=\alpha(2-\alpha).
   \end{align*}
   Since $\alpha(2-\alpha)\leq 1<\lambda_1=3$, we deduce that 
   \begin{align*}
       \inf_{n\geq 1}\left(\lambda_n+X^2+\dfrac{4X^2}{\lambda_n+X^2}\right)=\lambda_1+X^2+\frac{4X^2}{\lambda_1+X^2}=3+X^2+\frac{4X^2}{3+X^2}.
   \end{align*}
   Then, we have for all $X\in \R$
   \begin{align*}
       4+4X^2\geq 3+X^2+\frac{4X^2}{3+X^2}
   \end{align*}
   since this inequality is equivalent to
   \begin{align*}
       3+10X^2+3X^4\geq 4X^2.
   \end{align*}
   Therefore, we deduce that
   \begin{align*}
       \inf_{n\in\N}\mu_{n}>3+\frac{\pi^2}{\log^2\left(\frac{b}{a}\right)}+\frac{4\pi^2}{\pi^2+3\log^2\left(\frac{b}{a}\right)}.
   \end{align*}
   \textbf{Step 3: Lower bound on all eigenvalues in dimension $\mathbf{3}$.}

   Thanks to \eqref{lower_bound_mu}, we have
   \begin{align*}
       \mu>\inf_{n\in\N}\mu_n&=\inf_{n\in\N}\ens{\frac{(4n(n+1)-3)^2+(5+4n(n+1))\dfrac{8\pi^2}{\log^2\left(\frac{b}{a}\right)}+\dfrac{16\pi^4}{\log^4\left(\frac{b}{a}\right)}}{4+16n(n+1)+\dfrac{16\pi^2}{\log^2\left(\frac{b}{a}\right)}}}\\
       &=\mu_1=\frac{25+\dfrac{104\pi^2}{\log^2\left(\frac{b}{a}\right)}+\dfrac{16\pi^4}{\log^4\left(\frac{b}{a}\right)}}{36+\dfrac{16\pi^2}{\log^2\left(\frac{b}{a}\right)}}
   \end{align*}
   if the conformal class is large enough. Indeed, let $\alpha=\dfrac{4\pi^2}{\log^2\left(\frac{b}{a}\right)}$ and $t=4n(n+1)$, and introduce on $\R_+$ the function
   \begin{align*}
       f(t)=\frac{(t-3)^2+2\alpha(5+t)+\alpha^2}{1+\alpha+t}=9+\alpha+t-\frac{16t}{1+\alpha+t}.
   \end{align*}
   Therefore, we get
   \begin{align*}
       f'(t)=1-\frac{16(1+\alpha)}{(1+\alpha+t)^2}.
   \end{align*}
   Therefore, $f$ is strictly decreasing on $[0,4\sqrt{1+\alpha}-1-\alpha]$ and strictly increasing  on $[4\sqrt{1+\alpha}-1-\alpha,\infty[$ provided that $4\sqrt{1+\alpha}\geq 1+\alpha$, or
   \begin{align*}
       \alpha^2-14\alpha-15\leq 0,
   \end{align*}
   which is verified if and only if $\alpha\leq 15$. Since for all $n\geq 1$, we have $4n(n+1)\geq 8>4\sqrt{1+\alpha}-1-\alpha$ for all $\alpha>0$, we deduce that
   \begin{align*}
       \inf_{n\in\N}f(4n(n+1))=\min\ens{f(0),f(8)}=\min\ens{9+\alpha,\frac{25+26\alpha+\alpha^2}{9+\alpha}}.
   \end{align*}
   Since $(9+\alpha)^2-(25+26\alpha+\alpha^2)=56-8\alpha=8(7-\alpha)$, we deduce that the condition $\alpha\leq 7$, or
   \begin{align*}
       \log\left(\frac{b}{a}\right)\geq \frac{2\pi}{\sqrt{7}}
   \end{align*}
   implies that
   \begin{align*}
       \inf_{n\in\N}f(4n(n+1))=\min\ens{f(0),f(8)}=\frac{25+26\alpha+\alpha^2}{9+\alpha}.
   \end{align*}
   
   \textbf{Step 4: Lower bound on all eigenvalues in dimension $m\geq 5$.}

   Introduce on $\R_+$ the function
   \begin{align*}
       f(t)=\frac{(m(m-4)+t)^2+2((m-2)^2+4+t)\alpha+\alpha^2}{(m-4)^2+t+\alpha}=m^2+\alpha+t-\frac{16t}{(m-4)^2+t+\alpha},
   \end{align*}
   where $t$ stands for $4n(n+m-2)$ ($n\in\N$).
   We have
   \begin{align*}
       f'(t)=1-\frac{16((m-4)^2+\alpha)}{((m-4)^2+t+\alpha)^2}.
   \end{align*}
   For $m\geq 8$, the function $f$ is strictly increasing on $\R_+$, which shows that
   \begin{align*}
       \inf_{n\in\N}f(4n(n+m-2))=f(0)=\frac{m^2(m-4)^2+2((m-2)^2+4)\alpha+\alpha^2}{(m-4)^2+\alpha}.
   \end{align*}
   If $5\leq m\leq 7$, since $4n(n+m-2)>4\sqrt{(m-4)^2+\alpha}-(m-4)^2-\alpha$ for all $n\geq 1$, we deduce that 
   \begin{align*}
       &\inf_{n\in\N}f(4n(n+m-2))=\min\ens{f(0),f(4(m-1))}\\
       &=\min\ens{\frac{m^2(m-4)^2+2((m-2)^2+4)\alpha+\alpha^2}{(m-4)^2+\alpha},\frac{(m+2)^2(m-2)^2+2(m+2)^2\alpha+\alpha^2}{(m-4)^2+4(m-1)+\alpha}}.
   \end{align*}
   We have
   \small
   \begin{align*}
       &\frac{(m+2)^2(m-2)^2}{(m-4)^2+4(m-1)}-\frac{m^2(m-4)^2}{(m-4)^2}=\frac{1}{(m-4)^2+4(m-1)}\left((m+2)^2(m-2)^2-((m-4)^2+4(m-1))\right)\\
       &=\frac{4}{(m-4)^2+4(m-1)}\left(m^3-5m^2+4\right)=\frac{4(m-1)}{(m-4)^2+4(m-1)}\left((m-2)^2-8\right)>0
   \end{align*}
   \normalsize
   for all $m\geq 5$. Therefore, 
   \begin{align}\label{min_end_second_eigen}
       \inf_{n\in\N}f(4n(n+m-2))=f(0)=\frac{m^2(m-4)^2+2((m-2)^2+4)\alpha+\alpha^2}{(m-4)^2+\alpha}
   \end{align}
   if the conformal class is large enough. We have
   \begin{align*}
       &\frac{(m+2)^2(m-2)^2+2(m+2)^2\alpha+\alpha^2}{(m-4)^2+4(m-1)+\alpha}-\frac{m^2(m-4)^2+2((m-2)^2+4)\alpha+\alpha^2}{(m-4)^2+\alpha}\\
       &=\frac{1}{((m-4)^2+\alpha)((m-4)^2+4(m-1)+\alpha)}\bigg\{4(m-1)(m-4)^2((m-2)^2-8)\\
       &+8(m^3-11m^2+24m-6)\alpha-4(3m-1)\alpha^2\bigg\}>0
   \end{align*}
   if and only if
   \small
   \begin{align*}
       \alpha\leq \frac{1}{3m-1}\left(m^3-11m^2+24m-6+\sqrt{(m^3-11m^2+24m-6)^2+(m-1)(3m-1)(m-4)^2\left((m-2)^2-8\right)}\right).
   \end{align*}
   Therefore, if
   \small
   \begin{align*}
       \log\left(\frac{b}{a}\right)\geq \frac{4\pi(3m-1)}{m^3-11m^2+24m-6+\sqrt{(m^3-11m^2+24m-6)^2+(m-1)(3m-1)(m-4)^2\left((m-2)^2-8\right)}},
   \end{align*}
   \normalsize
   then \eqref{min_end_second_eigen} holds.
   \end{proof}

    \nocite{}
	 \bibliographystyle{plain}
	 \bibliography{biblio_full}

 \end{document}